\documentclass[12pt]{article}

\usepackage{amsmath}
\usepackage{amsthm}
\usepackage{amssymb}
\usepackage{caption}
\usepackage{subcaption}
\usepackage[curve]{xypic}
\usepackage{cite}
\usepackage{enumitem}
\usepackage{color}
\usepackage{url}
\usepackage[usenames,dvipsnames]{xcolor}
\usepackage{graphicx}
\usepackage{tikz}
\usepackage{tikz-cd}
\usepackage{hyperref} 
\usepackage{verbatim}
\usepackage[nottoc,notlot,notlof]{tocbibind} 

\makeatletter 
\def\@cite#1#2{{\m@th\upshape\bfseries%
[{#1\if@tempswa{\m@th\upshape\mdseries, #2}\fi}]}}
\makeatother 

\theoremstyle{plain}
\newtheorem{thm}{Theorem}[section]
\newtheorem{cor}[thm]{Corollary}
\newtheorem{ass}[thm]{Assumption}
\newtheorem{prop}[thm]{Proposition}
\newtheorem{lem}[thm]{Lemma}
\newtheorem{sublem}[thm]{Sublemma}

\newtheorem{conv}[thm]{Convention}

\theoremstyle{definition}
\newtheorem{defn}[thm]{Definition}
\newtheorem{war}[thm]{Warning}
\newtheorem{ex}[thm]{Example}

\newtheorem{prob}[thm]{Problem}
\newtheorem{conj}[thm]{Conjecture}
\theoremstyle{remark}
\newtheorem{rem}[thm]{Remark}

\numberwithin{equation}{subsection}
\renewcommand{\bold}[1]{\medskip \noindent {\bf #1 }\nopagebreak}

\newcommand{\nc}{\newcommand}
\newcommand{\rnc}{\renewcommand}

\usepackage[marginpar=2.5cm]{geometry}

\newcommand{\red}{\color{black}}

\newcommand{\black}{\color{black}}

\nc\bA{\mathbb{A}}
\nc\bB{\mathbb{B}}
\nc\bC{\mathbb{C}}
\nc\bD{\mathbb{D}}
\nc\bE{\mathbb{E}}
\nc\bF{\mathbb{F}}
\nc\bG{\mathbb{G}}
\nc\bH{\mathbb{H}}
\nc\bI{\mathbb{I}}
\nc{\bJ}{\mathbb{J}} 
\nc\bK{\mathbb{K}}
\nc\bL{\mathbb{L}}
\nc\bM{\mathbb{M}}
\nc\bN{\mathbb{N}}
\nc\bO{\mathbb{O}}
\nc\bP{\mathbb{P}}
\nc\bQ{\mathbb{Q}}
\nc\bR{\mathbb{R}}
\nc\bS{\mathbb{S}}
\nc\bT{\mathbb{T}}
\nc\bU{\mathbb{U}}
\nc\bV{\mathbb{V}}
\nc\bW{\mathbb{W}}
\nc\bY{\mathbb{Y}}
\nc\bX{\mathbb{X}}
\nc\bZ{\mathbb{Z}}
\nc\cA{\mathcal{A}}
\nc\cB{\mathcal{B}}
\nc\cC{\mathcal{C}}
\rnc\cD{\mathcal{D}}
\nc\cE{\mathcal{E}}
\nc\cF{\mathcal{F}}
\nc\cG{\mathcal{G}}
\rnc\cH{\mathcal{H}}
\nc\cI{\mathcal{I}}
\nc{\cJ}{\mathcal{J}} 
\nc\cK{\mathcal{K}}
\rnc\cL{\mathcal{L}}
\nc\cM{\mathcal{M}}
\nc\cN{\mathcal{N}}
\nc\cO{\mathcal{O}}
\nc\cP{\mathcal{P}}
\nc\cQ{\mathcal{Q}}
\rnc\cR{\mathcal{R}}
\nc\cS{\mathcal{S}}
\nc\cT{\mathcal{T}}
\nc\cU{\mathcal{U}}
\nc\cV{\mathcal{V}}
\nc\cW{\mathcal{W}}
\nc\cY{\mathcal{Y}}
\nc\cX{\mathcal{X}}
\nc\cZ{\mathcal{Z}}
\nc\bfA{\mathbf{A}}
\nc\bfB{\mathbf{B}}
\nc\bfC{\mathbf{C}}
\nc\bfD{\mathbf{D}}
\nc\bfE{\mathbf{E}}
\nc\bfF{\mathbf{F}}
\nc\bfG{\mathbf{G}}
\nc\bfH{\mathbf{H}}
\nc\bfI{\mathbf{I}}
\nc{\bfJ}{\mathbf{J}} 
\nc\bfK{\bk}
\nc\bfL{\mathbf{L}}
\nc\bfM{\mathbf{M}}
\nc\bfN{\mathbf{N}}
\nc\bfO{\mathbf{O}}
\nc\bfP{\mathbf{P}}
\nc\bfQ{\mathbf{Q}}
\nc\bfR{\mathbf{R}}
\nc\bfS{\mathbf{S}}
\nc\bfT{\mathbf{T}}
\nc\bfU{\mathbf{U}}
\nc\bfV{\mathbf{V}}
\nc\bfW{\mathbf{W}}
\nc\bfY{\mathbf{Y}}
\nc\bfX{\mathbf{X}}
\nc\bfZ{\mathbf{Z}}

\newcommand{\bk}{{\mathbf{k}}}

\nc{\dmo}{\DeclareMathOperator}
\nc{\wt}{\widetilde}
\rnc{\Re}{\operatorname{Re}}
\rnc{\Im}{\operatorname{Im}}
\rnc{\span}{\operatorname{span}}
\dmo{\rank}{rank}
\dmo{\End}{End}
\dmo{\Hom}{Hom}
\dmo{\Jac}{Jac}
\dmo{\Id}{Id}
\dmo{\Ann}{Ann}
\dmo{\Area}{Area}
\dmo{\CP}{\bC P^1}
\dmo{\rk}{rank}
\dmo{\rel}{rel}
\dmo{\ra}{\rightarrow}

\rnc{\Col}{\operatorname{Col}}

\nc{\ColOne}{\Col_{\bfC_1}}
\nc{\ColOneX}{\ColOne(X,\omega)}

\nc{\ColTwo}{\Col_{\bfC_2}}
\nc{\ColTwoX}{\ColTwo(X,\omega)}

\nc{\ColOneTwo}{\Col_{\bfC_1, \bfC_2}}
\nc{\ColOneTwoX}{\ColOneTwo(X,\omega)}

\nc{\MOne}{\cM_{\bfC_1}}
\nc{\MTwo}{\cM_{\bfC_2}}
\nc{\MOneTwo}{\cM_{\bfC_1, \bfC_2}}

\dmo{\For}{\cF}

\nc{\GL}{\mathrm{GL}^+(2, \bR)}

\title{Reconstructing orbit closures from their boundaries}
\author{Paul~Apisa \and Alex~Wright}

\date{}

\begin{document}
\maketitle
\thispagestyle{empty}

\begin{abstract}
We introduce and study \emph{diamonds} of $\GL$-invariant subvarieties of Abelian and quadratic differentials, which allow us to recover information on an invariant subvariety by simultaneously considering two degenerations, and which provide a new tool for the classification of invariant subvarieties. We classify a surprisingly rich collection of diamonds where the two degenerations are contained in ``trivial" invariant subvarieties. \red Our main results have been applied  to classify large collections of invariant subvarieties; the statement of those results do not involve diamonds, but their proofs rely on them.\black  
\end{abstract}

\renewcommand{\thefootnote}{\fnsymbol{footnote}} 
\renewcommand{\thefootnote}{\arabic{footnote}}

\tableofcontents

\section{Introduction}\label{S:intro}

The $\GL$-orbit closure of a translation surface is a properly immersed smooth suborbifold \cite{EM, EMM} and algebraic variety \cite{Fi2}. Conversely, every subvariety of translation surfaces that is $\GL$-invariant and irreducible is an orbit closure, so we use ``invariant subvariety" as a synonym for ``orbit closure", it being implicit that our subvarieties are irreducible unless otherwise indicated. 

This paper concerns the classification of invariant subvarieties. Previous classification results in genus 2 \cite{Mc5}, and subsequent classification results in genus 3 and higher, recalled below, give hope for strong, general results, but recently discovered examples \cite{MMW, EMMW} underscore the difficulty of obtaining such results. 

Here we develop new tools for the classification problem. Our study advances an emerging paradigm, which is that invariant subvarieties may be studied inductively, using their boundary. While considering a single degeneration is often insufficient, we show that one can often completely determine the structure of an invariant subvariety from two degenerations that form what we call a diamond. Our methods provide a framework for further analysis, and our results \red are \black crucial ingredients in two \red subsequent \black papers on classification  \red\cite{ApisaWrightGemini, ApisaWrightHighRank}. \black

The broader goal of this paper and the \red subsequent \black papers is to realize a portion of Mirzakhani's vision for classification: there should be easily verified conditions which imply that an invariant subvariety is ``trivial", and which are so broadly applicable that one could say they solve a major portion of the classification problem; see Remark \ref{R:Mirzakhani} for more details.

\subsection{Diamonds}\label{SS:diamonds}

Before discussing our main results (Theorems \ref{T:IntroFull} and \ref{T:DoubleIntro}) in the next two subsections, we must introduce the setup. 

Given a collection $\bfC$ of parallel cylinders on a translation surface $(X,\omega)$, we define the standard cylinder dilation $a_t^\bfC(X,\omega)$ to be the result of rotating the surface so the cylinders are horizontal, applying 
$$a_t=\begin{pmatrix} 1 & 0 \\ 0 & e^t \end{pmatrix}$$
only to the cylinders in $\bfC$, and then applying the inverse rotation. We define  
$$\Col_{\bfC}(X,\omega) = \lim_{t\to -\infty} a_t^\bfC(X,\omega).$$
This is the result of collapsing the cylinders in $\bfC$ in the direction perpendicular to their core curves, while keeping their circumferences constant and leaving the rest of the surface otherwise unchanged. We will be almost exclusively interested in the case when the collapse causes the surface to degenerate. The limit is taken in the What You See Is What You Get partial compactification studied in \cite{MirWri, ChenWright}. 

If $(X,\omega)$ is contained in an invariant subvariety $\cM$, there are many choices of $\bfC$ for which $a_t^\bfC(X,\omega)\in \cM$ for all $t\in \bR$ \cite{Wcyl}. In this case we say the standard dilation  of $\bfC$ remains in $\cM$, and we obtain that
$\Col_{\bfC}(X,\omega)$ is contained in \red an \black invariant subvariety $\cM_{\bfC}$ in the boundary of $\cM$. 

Suppose now that $(X,\omega)\in \cM$ has two collections of cylinders $\bfC_1$ and $\bfC_2$ such that 
\begin{enumerate}
\item $\bfC_1$ and $\bfC_2$ are disjoint, and moreover do not share any boundary saddle connections, 
\item the standard dilations of each $\bfC_i$ remain in $\cM$, and
\item  the collapses of each $\bfC_i$ do indeed cause the surface to degenerate.
\end{enumerate}
Motivated by Figure \ref{F:FirstDiamond}, we call this data $((X,\omega), \cM, \bfC_1, \bfC_2)$ a \emph{diamond}. In the first point we view each $\bfC_i$ as a subset of the surface (rather than a set of cylinders on the surface). 

$\bfC_2$ gives rise to a collection of cylinders on $\ColOneX$, which we denote $\ColOne(\bfC_2)$, and similarly with the indices swapped. We may define
$$\ColOneTwoX=\Col_{\ColTwo(\bfC_1)}\ColTwoX=\Col_{\ColOne(\bfC_2)}\ColOneX,$$
and this surface is contained in an invariant subvariety $\MOneTwo$ that is simultaneously in the boundary of $\cM$, $\MOne$, and $\MTwo$. (As we specify in Convention \ref{CV:connected}, although the surfaces in $\cM$ are typically assumed to be connected, we allow $\MOneTwo$, and sometimes $\MOne$ and $\MTwo$, to consist of multi-component surfaces.) 

\begin{figure}[h]\centering
\begin{tikzcd}[column sep=0em]
 & (X, \omega) \in \cM \arrow[dr, dash] \arrow[dl, dash, swap] &  \\
\ColOneX \in \MOne  \arrow[dr, dash] & & \ColTwoX \in \MTwo \arrow[dl, dash, swap]\\
& \ColOneTwoX \in \MOneTwo 
\end{tikzcd}
\caption{}
\label{F:FirstDiamond}
\end{figure}

While the surface and cylinders are necessary to codify the relation between these invariant subvarieties, we think of the essential part of a diamond as the four invariant subvarieties. 

Frequently we will demand that our diamonds be \emph{generic}; see Definition \ref{D:GenericDiamond}. This is a very mild assumption, and one can always obtain a generic diamond from a non-generic diamond. 

We also consider diamonds of quadratic differentials, which are defined exactly as above. 

\subsection{Full loci of covers}\label{SS:IntroFull}

We now consider (branched) covers of (half) translation \red surfaces\black, as defined in Definition \ref{D:Covering}. We require our covers to be branched only over marked points and zeros. This is of course not a true restriction, since one can simply declare the branch points to be marked. 

For any cover, and any small deformation of the base, one obtains a deformation of the cover. Let $\cM$ and $\cN$ be invariant subvarieties of Abelian or quadratic differentials. We say that $\cM$ is a \emph{full} locus of covers of $\cN$ if every surface in $\cM$ is a  cover of a surface in $\cN$ in such a way that all deformations of the codomain in $\cN$ give rise to covers in $\cM$. If $\cN$ is a connected component of a stratum of Abelian or quadratic differentials, we simply say that $\cM$ is a full locus of covers. 

Our analysis begins with the Diamond Lemma, see Lemma \ref{L:diamond}. Under the assumption that $\MOne$ and $\MTwo$ consist of covers of (typically lower genus) surfaces, if additional assumptions hold, the Diamond Lemma implies that $\cM$ similarly consists of covers. This leaves open the possibility that, despite consisting of covers, $\cM$ could be an unexpected and complicated invariant subvariety properly contained in a full locus of covers. 

We will say that a cover of translation or half-translation surfaces satisfies \emph{Assumption CP} (for Cylinder Preimage) if the preimage of every cylinder is a union of cylinders. Here our conventions, stated in Definition \ref{D:CylAndBoundary}, are crucial: cylinders do not contain their boundary, and  their boundary must be a union of saddle connections. These conventions imply in particular that if the preimage of a cylinder $C$ is a union of cylinders, then each cylinder in the preimage has the same height as $C$.  

If Assumption CP is not satisfied, then there must be a preimage of a marked point or pole that is an unmarked non-singular point. In particular, since we do not allow branching over unmarked non-singular points, any cover of a translation surface without marked points automatically satisfies Assumption CP. 

Our first main result is the following. 

\begin{thm}\label{T:IntroFull}
If $((X, \omega), \cM, \bfC_1, \bfC_2)$ forms a generic diamond where $\cM_{\bfC_1}$ and $\cM_{\bfC_2}$ are full loci of covers satisfying Assumption CP and $\MOneTwo$ consists of connected surfaces,  then $\cM$ is a full locus of covers of a stratum of Abelian or quadratic differentials. 
\end{thm}

In fact we obtain the conclusion of Theorem \ref{T:IntroFull} in most situations where the surfaces in $\MOneTwo$, and even $\MOne$ and $\MTwo$, are disconnected; see Theorem \ref{T:FullV2} for a more detailed statement. The possibility of disconnected surfaces adds significant extra difficulty, but is important since connected surfaces can and frequently do degenerate to disconnected surfaces.  

The simple statement of Theorem \ref{T:IntroFull}  belies surprising subtlety. For example, the degree of the covers for $\cM$ can be twice the degree of the covers for $\cM_{\bfC_1}$ and $\cM_{\bfC_2}$, as discussed in the proof of Lemma \ref{L:VS-Agree}. And Assumption CP may fail for $\cM$, even though it holds for $\cM_{\bfC_1}$ and $\cM_{\bfC_2}$, as discussed in  Remark \ref{R:LostCP}.

Part of this subtlety is associated with  the example illustrated in Figure \ref{F:QQnotQdiamond}, showing a diamond where both $\MOne$ and $\MTwo$ are strata of quadratic differentials, but $\cM$ is not. 

\begin{figure}[h]\centering
\includegraphics[width=\linewidth]{QQnotQdiamond.pdf}
\caption{A diamond of quadratic differentials, where both sides are strata but the top is not.
Here $\cM$ is a codimension 1 hyperelliptic locus locally defined by $a=b$, which is equivalent to $c=d$.}
\label{F:QQnotQdiamond}
\end{figure}

We emphasize the generality of Theorem \ref{T:IntroFull}. If we assumed $\MOneTwo$ does not consist of torus covers and we dealt only with Abelian differentials (excluding quadratic differentials), the proof would be short. The more general statement, although vastly more difficult, is crucial for applications and to obtain meaningful insight into the richness of invariant subvarieties. 

\subsection{Abelian and quadratic doubles}\label{SS:full}

Diamonds where $\MOne$ and $\MTwo$ are full loci of covers \emph{not} satisfying Assumption CP are much more difficult to understand, and it seems entirely possible that their analysis could result in the discovery of new invariant subvarieties.  Here we only begin such an analysis. Our main result in this direction  \red is crucial for the subsequent papers \cite{ApisaWrightGemini, ApisaWrightHighRank}\black, and, although it only concerns certain degree two covers, it is broad enough to illustrate an interesting phenomenon which is typically incompatible with Assumption CP. 

For the next definition, we emphasize that we allow strata to parameterize surfaces with marked points; we treat marked points as zeros of order zero.

We define an \emph{Abelian double} to be a full locus of covers of a component of a stratum of Abelian differentials such that the covering maps have degree two, the covers are connected, and all preimages of marked points are either singularities or marked points. 

We define a \emph{quadratic double} to be a full locus of covers of a component of a stratum of quadratic differentials such that the covering maps are the holonomy double cover and all preimages of marked points are marked points. The preimage of a pole may be marked or unmarked. We assume the quadratic differentials have non-trivial holonomy, so again the covers are connected.

In the Abelian case, different choices of degree two covering map might lead to different  Abelian doubles associated to the same component of a stratum. In the quadratic case, different choices of which preimages of poles to mark might lead to different quadratic doubles associated to the same component of a stratum. 

While Abelian doubles must satisfy Assumption CP, quadratic doubles can fail to satisfy this assumption, if not all preimages of poles are marked. 

\begin{thm}\label{T:DoubleIntro}
If $((X, \omega), \cM, \bfC_1, \bfC_2)$ forms a generic diamond where $\cM_{\bfC_1}$ and $\cM_{\bfC_2}$ are Abelian or quadratic doubles, then $\cM$ is a full locus of covers of a stratum of Abelian or quadratic differentials. 

Moreover, $\cM$ is one of the following: an Abelian or quadratic double, or a codimension one locus in a full locus of double covers of a component of a stratum of Abelian differentials. 
\end{thm}

This is an abbreviated form of Theorems \ref{P1} and \ref{TP2},  which describe the codimension one loci that occur.

The interesting phenomenon that appears here but not in Theorem \ref{T:IntroFull} is that $\MOne$ might be an Abelian double while $\MTwo$ is a quadratic double; see Figure \ref{F:ExtraPossibilityNoW}. 
\begin{figure}[h]\centering
\includegraphics[width=\linewidth]{ExtraPossibilityNoW.pdf}
\caption{A diamond  where $\MOne$ is an Abelian double, $\MTwo$ is a quadratic double, and $\cM$ is neither.  The invariant subvariety $\cM$ is described in Theorem \ref{TP2} case \eqref{I:ExtraSymmetry}.}
\label{F:ExtraPossibilityNoW}
\end{figure}
In this case the degree two covering maps defined on surfaces in $\MOne$ and $\MTwo$ give rise to distinct covering maps on surfaces in $\MOneTwo$. In fact, $\MOneTwo$ must be simultaneously an Abelian and quadratic double. Simultaneous  Abelian and quadratic doubles  arise from certain hyperelliptic connected components of strata of quadratic differentials; see Section \ref{SS:simul}.

\subsection{Additional remarks}\label{SS:context}

\bold{Context.} We start by discussing the relation to previous work. 

\begin{rem}\label{R:Mirzakhani} 
Mirzakhani conjectured a single statement that, if true, would be a major part of the classification of invariant subvarieties. This conjecture is often described as ``all $GL(2, \bR)$-orbit closures of rank at least two are trivial", and a version was first recorded in \cite[Conjectures 1.6, 1.7]{Wfield}. Here ``trivial" should mean ``a locus of covers", but it was not completely clear at the time the conjecture was made how to make a precise definition,  because it was not known what relations might be imposed on the branch points of the covering maps. This issue has been clarified in \cite{Apisa, ApisaWright}, so that we can now confidently propose that, in the language above,  trivial should mean a full locus of covers of a component of a stratum of Abelian or quadratic differentials.\footnote{The second author first learned of this conjecture in October 2012. The exact wording of the conjecture in \cite{Wfield} states only that  ``every translation surface in $\cM$ [the invariant subvariety] covers a quadratic differential (half-translation surface) of smaller genus", but correspondence and conversations between Mirzakhani and the second author suggest that a stronger conjecture was intended, i.e. that the locus of surfaces being covered would be a component of a stratum of Abelian or quadratic differentials. The stronger version is more in line with the discussion of the conjecture in \cite{Wcyl,ANW, AN, aulicino2016rank, Apisa2}.}

Mirzakhani's Conjecture was disproven in \cite{MMW, EMMW}, but the counterexamples are all rank two. It is plausible that all invariant subvarieties of rank at least three might be full loci of covers of connected components of strata of Abelian or quadratic differentials. 
\end{rem} 

Recent progress on the classification problem includes finiteness results \cite{EFW, BHM, LNW}, strong results in genus 3 \cite{NW, ANW, AN, aulicino2016rank, Ygouf} and for hyperelliptic components \cite{Apisa2, Apisa:Rank1}, and classification of full rank invariant subvarieties \cite{MirWri2}. Especially important here will be results considering cylinder deformations \cite{Wcyl}, the boundary of invariant subvarieties \cite{MirWri, ChenWright}, and marked points \cite{Apisa, ApisaWright}.

Adjacent recent developments include progress on the isoperiodic foliation \cite{McM:iso, CDF, Ygouf2, HamErg, HooperWeiss}, compactifications \cite{Many, lms, Benirschke}, the unipotent flow \cite{BSW, CSW}, and Prym eigenforms \cite{LNcomponents, Wg4, LanneauMoller}. Surveys of the field include \cite{ForniMatheus:Intro, Z, MT, Wsurvey}.

\bold{Techniques.} The main technique in this paper is induction. Given a hypothetical counterexample to one of our main results, we try to degenerate cylinders disjoint from $\bfC_1$ and $\bfC_2$ to produce a smaller counterexample. The results of \cite{MirWri, ChenWright} allow us to understand the invariant subvarieties containing these degenerations, and the results of \cite{Apisa, ApisaWright} prove surprisingly useful whenever the degenerations produce new marked points. Base cases are handled using diverse techniques: some can be ruled out by surprisingly easy numerology powered by \cite{AEM}; some are ruled out using marked point results; and some are handled in various ways using the existence and non-existence of certain cylinder deformations. For example, sometimes we use a new technique in which we ``overcollapse" a collection of cylinders $\bfC$ to produce a deformation that changes the modulus of some disjoint cylinders $\bfD$, contradicting a partial generalization of the Veech dichotomy proved in \cite{MirWri} and recalled below in Corollary \ref{C:ConstantRatio}. We call this ``attacking $\bfD$ with $\bfC$".

Omnipresent in our analysis are Masur and Zorich's results on generically parallel, or ``hat-homologous", cylinders and saddle connections on quadratic differentials \cite{MZ}. On a generic Abelian differential, all cylinders are simple. In contrast, we summarize in Theorem \ref{T:MZ} the five types of cylinders that, according to Masur and Zorich, may appear on generic quadratic differentials. This richness in behaviour contributes significantly to the length of this paper.


Beyond showcasing how our new ``attacking" technique can be profitably combined with many other techniques, the broader novelty of this paper is that it introduces diamonds as a paradigm for a more systematic study of the classification problem. Our results \red are \black illustrated in our \black subsequent \red work \cite{ApisaWrightGemini, ApisaWrightHighRank}, where the statements \red do \black not involve diamonds but the proofs \red rely on \black them.

\bold{Organization.} In Section \ref{S:DiamondLemma}, we define generic diamonds and prove the Diamond Lemma, which is the starting point for all our analysis. Sections \ref{S:Pre} and \ref{S:PreQD} establish definitions and preliminaries, which the reader may refer back to as necessary. 

In Section \ref{S:DiamondEasy} we classify  the easiest diamonds, namely those where one side is a component of a stratum of Abelian differentials. Before turning to harder diamonds, in Section \ref{S:ComplexEnvelope} we classify certain codimension one invariant subvarieties of quadratic differentials. This is a key tool for subsequent results, and suggests some open problems, listed in Subsection \ref{SS:OpenProblems}. 

Section \ref{S:DiamondQ} classifies diamonds where both sides are quadratic doubles. Section \ref{S:FullLoci} proves Theorem \ref{T:IntroFull}, and concludes in Section \ref{SS:OpenProblems2} with related open problems. Section \ref{S:Q-hyp} gives preliminaries concerning hyperelliptic strata of quadratic differentials. These preliminaries are used in Section \ref{S:DiamondQDouble}, which completes the proof of Theorem \ref{T:DoubleIntro} by classifying diamonds where one side is an Abelian double and one side is a quadratic double. 

This paper is highly modular. In particular, the only statements from each of Sections \ref{S:DiamondEasy}, \ref{S:ComplexEnvelope}, and  \ref{S:DiamondQ} that are used elsewhere in the paper are Proposition \ref{P:DiamondWithH}, Theorem \ref{T:complex-gluing0}, and  Theorem \ref{P1} respectively. The results from Section \ref{S:FullLoci} are not required elsewhere in the paper. (We use Lemma \ref{L:StrongTheorem1.1} for convenience once in Section \ref{S:DiamondQDouble}, but the reader may also supply a more direct argument).

\bold{Conventions.} Cylinders do not include their boundary saddle connection (Definition \ref{D:CylAndBoundary}); with important exceptions, most surfaces are assumed to be connected (Convention \ref{CV:connected}); quadratic differentials are assumed to have non-trivial holonomy (Convention \ref{CV:NoTrivHol}); and, for translation covers, the fiber of a marked point must contain a marked point or singular point (Definition \ref{D:Covering}). 

\bold{Acknowledgments.}  During the preparation of this paper, the first author was partially supported by NSF Postdoctoral Fellowship DMS 1803625, and the second author was partially supported by a Clay Research Fellowship,  NSF Grant DMS 1856155, and a Sloan Research Fellowship.  

\section{The Diamond Lemma}\label{S:DiamondLemma}

In this section, we establish a versatile result that allows one to conclude that an orbit closure is a locus of covers. \red We begin with this topic to immediately illustrate one of the key ideas in the paper, but some readers may prefer to start instead with the background material in Sections \ref{S:Pre} and \ref{S:PreQD}. \black

We will use notation that is typical for Abelian differentials, but the results will  apply equally well to quadratic differentials.  \red We build on the definitions of cylinder collapses and diamonds in Section \ref{SS:diamonds}. \black

Given a diamond
$$((X,\omega), \cM, \bfC_1, \bfC_2)$$
where both $\MOne$ and $\MTwo$ consist of covers, our goal is to conclude that $\cM$ is a locus of covers. So we assume that each $\Col_{\bfC_i}(X,\omega)$ admits a half translation cover 
$$f_i : \Col_{\bfC_i}(X,\omega)\to (Y_i, q_i).$$
Note that $f_i(\overline{\Col_{\bfC_i}(\bfC_{i+1})})$ is the closure of a union of  cylinders parallel to the cylinders in $\bfC_{i+1}$. We will assume that 
$$\overline{\Col_{\bfC_i}(\bfC_{i+1})}=f_i^{-1}(f_i(\overline{\Col_{\bfC_i}(\bfC_{i+1})})).$$
This assumption gives that any standard cylinder deformation of ${\Col_{\bfC_i}(\bfC_{i+1})}$ on $\Col_{\bfC_i}(X,\omega)$ covers the corresponding deformation of the  cylinders whose closure is $f_i(\overline{\Col_{\bfC_i}(\bfC_{i+1})})$ on $(Y_i, q_i)$; see Subsection \ref{SS:CylinderDeformations} for the definition of ``standard".

\begin{rem}
See Figure \ref{F:PreimageOfImage.pdf} for an example showing why we use the closure of $\Col_{\bfC_i}(\bfC_{i+1})$, keeping in mind the conventions in Definitions \ref{D:Covering} and \ref{D:CylAndBoundary} highlighted in the introduction. One does not need to use closures if $(Y_i, q_i)$ does not contain marked points or poles.
\begin{figure}[h]\centering
\includegraphics[width=0.8\linewidth]{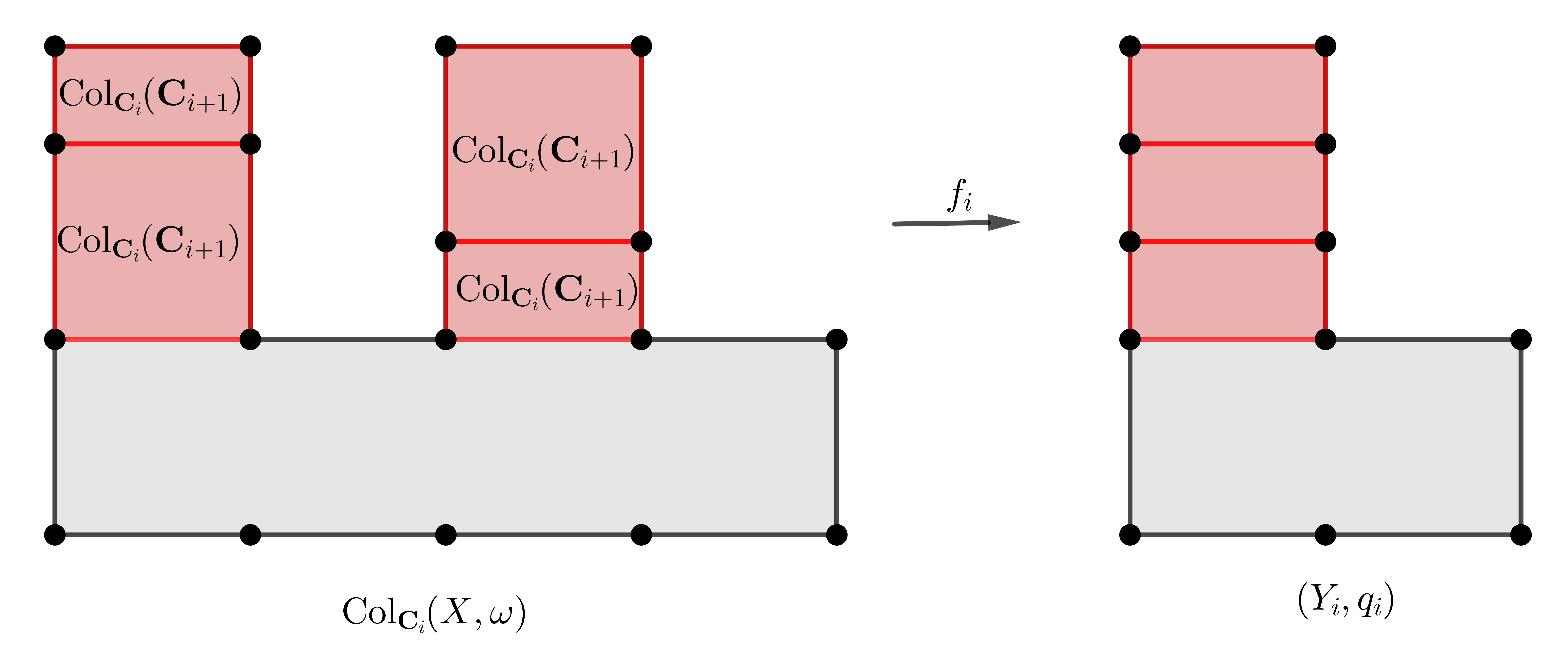}
\caption{An example where $\overline{\Col_{\bfC_i}(\bfC_{i+1})}$ is the preimage of its image, but $f_i^{-1}(f_i({\Col_{\bfC_i}(\bfC_{i+1})})) - \Col_{\bfC_i}(\bfC_{i+1})$ is two saddle connections, each joining a marked point to itself. Here $\Col_{\bfC_i}(\bfC_{i+1})$ consists of four cylinders and $(Y_i, q_i)\in \cH(2,0^2)$.}
\label{F:PreimageOfImage.pdf}
\end{figure}
\end{rem}

As we will see presently, we get a limiting map 
$$\Col_{\Col_{\bfC_i}(\bfC_{i+1})}(f_i) : \Col_{\Col_{\bfC_i}(\bfC_{i+1})} \Col_{\bfC_i}(X,\omega)\to \Col_{f_i(\Col_{\bfC_i}(\bfC_{i+1}))} (Y_i, q_i).$$

Here we use $\Col_{f_i(\Col_{\bfC_i}(\bfC_{i+1}))}$ to denote the collapse of the collection of cylinders whose closure is $f_i(\overline{\Col_{\bfC_i}(\bfC_{i+1})})$.

\begin{lem}\label{L:DefinitionCol(f)}
Suppose that $f: (X, \omega) \ra (Y, \eta)$ is a half-translation covering. Let $\bfC\subset (X,\omega)$ be a collection of parallel cylinders such that $f^{-1}(f(\overline\bfC)) = \overline\bfC$ and $\overline\bfC \neq (X,\omega)$. Then there is a half-translation surface covering map $$\Col_{\bfC}(f): \Col_{\bfC}(X, \omega) \ra \Col_{f(\bfC)}(Y, \eta)$$ of the same degree.
\end{lem}

Since we will require an explicit understanding of $\Col_{\bfC}(f)$, we give an explicit proof. 

\begin{proof}
Notice that 
\[ \Col_{\bfC}(X, \omega) - \Col_{\bfC}(\bfC) = (X, \omega) - \overline{\bfC}\] and
\[\Col_{f(\bfC)}(Y, \eta) - \Col_{f(\bfC)}(f(\bfC)) = (Y, \eta) -  \overline{f(\bfC)}. \]
These equalities and the condition that $f^{-1}(f(\overline\bfC)) = \overline\bfC$ implies that we can define $\Col_{\bfC}(f)$ on $\Col_{\bfC}(X, \omega) - \Col_{\bfC}(\bfC)$ to be given by the restriction of $f$ to  $ (X, \omega) - \overline{\bfC}$. 

$\Col_{\bfC}(\bfC)$  is a union of saddle connections. Consider a point $p$ in $\Col_{\bfC}(\bfC)$ that isn't a singularity or marked point. We can extend the definition of $\Col_{\bfC}(f)$ to $p$ as follows. 

The set $\Col_{\bfC}^{-1}(p)$ that collapses to $p$ is a line segment $S_p$ contained in $\overline{\bfC}$. This line segment is mapped by $f$ to a line segment $f(S_p)$ in $f(\bfC)$, and we may define $\Col_{\bfC}(f)(p)$ to be $\Col_{f(\bfC)}(f(S_p))$. 

This defines an extension of $\Col_{\bfC}(f)$ to the complement of a finite set of points, on which $\Col_{\bfC}(f)$ can be defined by continuity.
\end{proof}
 
Even when they are isomorphic, there isn't always a canonical way to identify the codomains of the maps $\Col_{\ColOne(\bfC_{2})}(f_1)$ and $\Col_{\ColTwo(\bfC_{1})}(f_{2})$, because the codomains may have automorphisms. But they do have the same domain, namely $\ColOneTwoX$. We will say that \emph{$f_1$ and $f_2$ agree at the base of the diamond} if $\Col_{\ColOne(\bfC_{2})}(f_1)$ and $\Col_{\ColTwo(\bfC_1)}(f_{2})$ have the same fibers (each fiber of one of these maps is also a fiber for the other). We will write ``$\Col(f_1) = \Col(f_2)$" as shorthand for this condition. 

The main result of this section verifies the intuition that, if these two maps agree, one should be able to somehow glue them together to obtain a map whose domain is $(X,\omega)$. 
 
\begin{lem}[The Diamond Lemma]\label{L:diamond}
Given a diamond using the notation above, with maps $f_i$ as above such that $$\overline{\Col_{\bfC_i}(\bfC_{i+1})}=f_i^{-1}(f_i(\overline{\Col_{\bfC_i}(\bfC_{i+1})})),$$ assume that $f_1$ and $f_2$ agree at the base of the diamond.

Then $(X,\omega)$ admits a covering map $f$ to a quadratic differential, with $\overline\bfC_{i}=f^{-1}(f(\overline\bfC_{i}))$, and $f_i = \Col_{\bfC_{i}}(f)$.
\end{lem}

\begin{cor}\label{C:DenseDiamond}
If additionally the orbit closure of $(X,\omega)$ is $\cM$, then  every surface in $\cM$ is a cover of a half translation surface in such a way that each $f_i$ is a limit of  associated covering maps.
\end{cor}

\begin{proof}
 $\cM$ must be contained in a  locus of half-translation covers, since such loci are closed and invariant. 
\end{proof}

\begin{proof}[Proof of Lemma \ref{L:diamond}] \red
Consider the equivalence relation $\sim_i$ on $\Col_{\bfC_i}(X,\omega)$ whose equivalence classes are exactly the fibers of $f_i$. Roughly speaking, we will ``glue together" $\sim_1$ and $\sim_2$ to get an equivalence relation $\sim$ on $(X,\omega)$, and show that $(X,\omega)/{\sim}$ has the structure of a quadratic differential. 

The definition of the collapse maps implies that we have the inclusions illustrated in Figure \ref{F:SubsurfaceInclusions}. 
\begin{figure}[h]\centering
\begin{tikzcd}[column sep=-1.5em]
 & (X, \omega) &  \\
\ColOneX - \ColOne({\bfC_1})  \arrow[ur, hook] & & \ColTwoX - \ColTwo({\bfC_2}) \arrow[ul, hook']\\
& \ColOneTwoX - \ColOneTwo({\bfC_1 \cup \bfC_2}) \arrow[ul, hook'] \arrow[ur, hook]
\end{tikzcd}
\caption{}
\label{F:SubsurfaceInclusions}
\end{figure}

\noindent Via these inclusions, 
$$ \bigcup_{i=1,2} \left( \Col_{\bfC_i}(X,\omega) - \Col_{\bfC_i}(\bfC_i) \right) = (X,\omega) - \overline{\bfC}_1 \cap \overline{\bfC}_2.$$

We can give the outline of the proof more precicely as follows. We will show in the next sublemma that $\Col_{\bfC_i}(X,\omega) - \Col_{\bfC_i}(\bfC_i)$ is closed under $\sim_i$. This will allow us to glue together the equivalence relations on these sets to obtain an equivalence relation $\sim$ on $(X,\omega)$ minus the finite set $\overline{\bfC}_1 \cap \overline{\bfC}_2$. 

\begin{rem}
The notation $\Col_{\bfC_i}(\bfC_i)$ can be understood via the map $\Col_{\bfC_i} : (X,\omega) \to \Col_{\bfC_i}(X,\omega)$ that is in general multi-valued. (This map can be viewed in two steps: An honest collapse map, and then a step that deletes nodes and fills in punctures. The composition is multi-valued on the subset of $(X,\omega)$ that collapses to a node via the initial honest collapse map. See \cite{ChenWright} for more discussion.)  
\end{rem}

\begin{sublem}\label{SL:DLkey}
$f_i^{-1}(f_i( \Col_{\bfC_i}(\bfC_i))) = \Col_{\bfC_i}(\bfC_i).$
\end{sublem}

\begin{proof}
Because the collapse map may be multi-valued, $\Col_{\bfC_i}(\bfC_i)$, a priori, is a finite set of saddle connections and isolated points. Our first claim is that there are no isolated points. This follows from the fact that gluing in the cylinders of $\bfC_i$ into the saddle connections of $\Col_{\bfC_i}(\bfC_i)$ does not involve the hypothetical isolated points. 

Suppose in order to find a contradiction that the sublemma is false; so 
$$f_i^{-1}(f_i( \Col_{\bfC_i}(\bfC_i)))-\Col_{\bfC_i}(\bfC_i)$$ 
contains a saddle connection. Because $\bfC_1$ and $\bfC_2$ do not share boundary saddle connections, collapsing $\Col_{\bfC_i}(\bfC_{i+1})$ cannot cause this saddle connection to merge with $\Col_{\bfC_i}(\bfC_i)$. Hence, we get 
$$\Col(f_i)^{-1}(\Col(f_i)(  \Col_{\Col_{\bfC_i}(\bfC_{i+1})} \Col_{\bfC_i}(\bfC_i))) - \Col_{\Col_{\bfC_i}(\bfC_{i+1})} \Col_{\bfC_i}(\bfC_i)$$
contains a saddle connection. (Here we write $\Col(f_i)$ instead of $\Col_{\Col_{\bfC_i}(\bfC_{i+1})}(f_i)$.) 

Since $\Col(f_1) = \Col(f_2)$, we get the same statement with $\Col(f_i)$ replaced with $\Col(f_{i+1})$; namely that
$$\Col(f_{i+1})^{-1}(\Col(f_{i+1})(  \Col_{\Col_{\bfC_i}(\bfC_{i+1})} \Col_{\bfC_i}(\bfC_i))) - \Col_{\Col_{\bfC_i}(\bfC_{i+1})} \Col_{\bfC_i}(\bfC_i)$$
contains a saddle connection.

Because of the subtle multi-valued nature of the  collapse maps, it isn't clear whether  
$$\Col_{\Col_{\bfC_1}(\bfC_{2})} \circ \Col_{\bfC_1} = \Col_{\Col_{\bfC_2}(\bfC_{1})} \circ \Col_{\bfC_2}.$$
Nonetheless, the  definition of the collapse maps implies that this commutativity holds off the preimage of the singular points and marked points in $\ColOneTwoX$. Hence we get that 
$$\Col(f_{i+1})^{-1}(\Col(f_{i+1})(  \Col_{\Col_{\bfC_{i+1}}(\bfC_{i})} \Col_{\bfC_{i+1}}(\bfC_i))) - \Col_{\Col_{\bfC_{i+1}}(\bfC_{i})} \Col_{\bfC_{i+1}}(\bfC_i)$$
contains a saddle connection.

 Since we have assumed that 
$$\overline{\Col_{\bfC_{i+1}}(\bfC_{i})}=f_{i+1}^{-1}(f_{i+1}(\overline{\Col_{\bfC_{i+1}}(\bfC_{i})})),$$
the definition of the $\Col(f_{i+1})$ implies that, at worst up to a finite set of points,  $$\Col(f_{i+1})^{-1}(\Col(f_{i+1})(  \Col_{\Col_{\bfC_{i+1}}(\bfC_{i})} \Col_{\bfC_{i+1}}(\bfC_i))) = \Col_{\Col_{\bfC_{i+1}}(\bfC_{i})} \Col_{\bfC_{i+1}}(\bfC_i),$$ which is a contradiction.
%
\end{proof}

The image of the inclusion 
$$\ColOneTwoX - \ColOneTwo({\bfC_1 \cup \bfC_2})\hookrightarrow \ColOneX$$
is equal to $$\ColOneX - \ColOne({\bfC_1})-\overline{\ColOne(\bfC_2)}.$$
Since $\overline{\ColOne(\bfC_2)}$ is closed under $\sim_1$ by assumption and $\ColOne({\bfC_1})$ is closed under $\sim_1$ by Sublemma \ref{SL:DLkey}, we see that the image of the inclusion is closed under $\sim_1$ (and similarly on the other side of the diamond). 

The assumption that $\Col(f_1)=\Col(f_2)$ implies that the restrictions of the $\sim_i$ to 
$$\ColOneTwoX - \ColOneTwo({\bfC_1 \cup \bfC_2})$$
obtained via these inclusions agree. 

Viewing all these sets as subsets of $(X,\omega)$, it follows that there is an equivalence relation $\sim$ on 
$(X,\omega)-\overline{\bfC}_1 \cap \overline{\bfC}_2$ which restricts to $\sim_i$. All of the sets obtained via the above inclusions are closed under $\sim$. (Note that $\overline{\bfC}_1 \cap \overline{\bfC}_2$ is contained in the set of singularities of $(X,\omega)$.)

By construction, $\sim$ has the property that as a point moves in the complement of the singularities, each point of the equivalence relation moves with slope $\pm1$ (see Definition \ref{D:slope} for the definition of the ``slope" of a point marking).  As we review in the next paragraph, it follows that $((X,\omega) \setminus \overline{\bfC}_1 \cap \overline{\bfC}_2)/{\sim}$ can be endowed with a half translation surface structure in such a way that the map 
$$(X,\omega) \setminus \overline{\bfC}_1 \cap \overline{\bfC}_2\to ((X,\omega) \setminus \overline{\bfC}_1 \cap \overline{\bfC}_2)/{\sim}$$
is a map of half translation surfaces with punctures. This extends to a half-translation surface map defined on $(X,\omega)$. 

The construction of the half translation surface structure on $((X,\omega) \setminus \overline{\bfC}_1 \cap \overline{\bfC}_2)/{\sim}$  is \black very similar to the proof of \cite[Lemma 2.8]{ApisaWright}, but we review  the details here. The quotient map is a  covering map, and restrictions of the quotient map to small balls where the map is injective can be used to endow the quotient with an atlas of charts whose transition maps are of the form $z\mapsto \pm z+C$. In the neighborhood of each puncture, the map, being a local isometry, must have a standard form, and we can fill in the punctures to get a map of closed surfaces. 
\end{proof}

\section{Preliminaries on orbit closures}\label{S:Pre}

We will briefly review some facts about invariant subvarieties. In the remainder of the section, $\cM$ will denote a connected  \red $\GL$-invariant \black subvariety and $(X, \omega)$ will be a point in $\cM$. 

\subsection{Rank and rel}\label{SS:RankAndRel}
First we recall that the tangent space of $\cM$ at a point $(X, \omega) \in \cM$ is naturally identified with a subspace of $H^1(X, \Sigma; \mathbb{C})$, where $\Sigma$ denotes the set of zeros of $\omega$ on $X$. Let $p: H^1(X, \Sigma; \mathbb{C}) \ra H^1(X; \mathbb{C})$ denote the projection from relative to absolute cohomology. The subspace $\ker(p)\cap T_{(X, \omega)} \cM$ is called the \emph{rel subspace  of $\cM$ at $(X,\omega)$}. We call $$\mathrm{rel}(\cM)= \dim_\bC \ker(p)\cap T_{(X, \omega)} \cM$$ \emph{the rel of $\cM$}.

By Avila-Eskin-M\"oller \cite{AEM}, for any $(X,\omega) \in \cM$, $p\left( T_{(X, \omega)} \cM \right)$ is a complex symplectic vector space, in particular its complex dimension is even. The \emph{rank} of $\cM$ is defined to be half the complex-dimension of $p\left( T_{(X, \omega)} \cM \right)$, which is independent of the choice of $(X, \omega) \in \cM$.

An invariant subvariety $\cM$ in a stratum of connected genus $g$ Abelian differentials is called \emph{full rank} if its rank is $g$. The main result of  \cite{MirWri2} is the following, \red where it is implicit the surfaces do not have marked points. \black 

\begin{thm}[Mirzakhani-Wright]\label{T:MirWriFullRank}
Let $\cM$ be a full rank invariant subvariety.
Then $\cM$ is either a connected component of a stratum, or the locus
of hyperelliptic translation surfaces therein.
\end{thm}

\subsection{Field of definition and translation covers}

\begin{defn}\label{D:Covering}
A \emph{translation covering} from $(X, \omega)$ to $(Y, \eta)$ is defined to be a holomorphic map $f: X \ra Y$ branched only over singularities and marked points such that:
\begin{enumerate}
\item $f^* \eta = \omega$, 
\item\label{D:Covering:2} all marked points on $(X,\omega)$ map to marked points on $(Y,\eta)$, and
\item\label{D:Covering:3}  each marked point on $(Y,\eta)$ has at least one preimage on $(X,\omega)$ that is a singular or marked point. 
\end{enumerate}
A \emph{half-translation surface covering}, from a translation surface or half-translation surface to a half-translation surface, is defined similarly, with the additional stipulations that:
\begin{enumerate} 
\setcounter{enumi}{3}
\item\label{D:Covering:4} a marked point may map to a simple pole, 
\item\label{D:Covering:5} but poles need not have a preimage that is a singular or marked point. 
\end{enumerate}
\end{defn}

The requirements concerning marked points \red(especially items \ref{D:Covering:2}, \ref{D:Covering:3}, and  \ref{D:Covering:5}) \black   are not standard, but will be convenient here. 

\red

Without item \ref{D:Covering:2}, one could deform the domain $(X,\omega)$ without changing the codomain $(Y,\eta)$; and without item \ref{D:Covering:3} one could deform the codomain without deforming the domain (both while remaining in an appropriate locus of covers). So \ref{D:Covering:2} and \ref{D:Covering:3} combined ensure that  deformations of the codomain surface correspond locally to deformations of the covering map, and to deformations of the domain surface (again remaining in an appropriate locus of covers). 
Items \ref{D:Covering:4} and \ref{D:Covering:5} give us the flexibility either to mark preimages of poles or not to. \black 

A translation surface is \emph{square-tiled} if it is a translation cover of square torus with only one marked point (so that the cover is branched at most over that one point). It is called a \emph{torus cover} if it admits a map to genus one translation surface with any number of marked points. 

The \emph{field of definition} $\bk(\cM)$ of an invariant subvariety $\cM$ is the smallest subfield of $\mathbb{R}$ such that $\cM$ can be defined by equations in $\bk(\cM)$ in any local period coordinate chart. 

\begin{lem}\label{L:DenseSquares}
If $\bk(\cM) = \mathbb{Q}$, then square-tiled surfaces are dense in $\cM$. 
\end{lem}

\begin{proof}
Any surface whose period coordinates are contained in $\bQ[i]$ is in particular square-tiled. \red (See  \cite[Section 1.5.2]{HS-IntroVeech} for a similar proof and discussion.)\black
\end{proof}

\begin{lem}\label{L:R1Arithmetic}
$\cM$ is a locus of torus covers if and only if it has rank 1 and $\bk(\cM) = \mathbb{Q}$.
\end{lem}
\begin{proof}
This follows from the fact that a translation surface is a torus cover if and only if its absolute periods span a $\bZ$-module of rank 2.
\end{proof}

We conclude this subsection with the following result of M\"oller \cite[Theorem 2.6]{M2} and its extension in Apisa-Wright \cite[Lemma 3.3]{ApisaWright}. We emphasize that it applies to connected surfaces, although later we will use it to get some information for multi-component surfaces.  

\begin{thm}\label{T:MinimalCover}
Suppose that $(X, \omega)$ is not a torus cover. There is a unique translation surface $(X_{min}, \omega_{min})$ and a translation covering $$\pi_{X_{min}}: (X, \omega) \rightarrow (X_{min}, \omega_{min})$$ such that any translation cover from $(X, \omega)$ to translation surface is a factor of $\pi_{X_{min}}$. 

Additionally, there is a quadratic differential $(Q_{min}, q_{min})$ with a degree
1 or 2 map $(X_{min}, \omega_{min}) \rightarrow (Q_{min}, q_{min})$ such that any map from $(X, \omega)$ to a quadratic differential is a factor of the composite map $\pi_{Q_{min}} :
(X, \omega) \rightarrow (Q_{min}, q_{min})$.
\end{thm}

\subsection{Point markings}

This section recalls some definitions and results from \cite{ApisaWright}. 

Let $\cM$ be an invariant subvariety of a stratum $\cH$. Define $\cH^{*n}$ to be the same stratum with $n$-additional marked points, and let $\pi:\cH^{*n}\to \cH$ be the map that forgets marked points. Define a \emph{$n$-point marking over $\cM$} to be an invariant subvariety $\cN\subset \cH^{*n}$ such that $\pi(\cN)$ is a dense subset of $\cM$.

A $1$-point marking over $\cM$ of the same dimension as $\cM$ is called a \emph{periodic point}. Similarly, if $(X,\omega)$ has orbit closure $\cM$, a periodic-point on $(X,\omega)$ is one such that $(X,\omega)$ is contained in a periodic point for $\cM$. 

A $n$-point marking is called \emph{irreducible} if it is not obtained by combining a $k$-point marking and a $(n-k)$-point marking for some $0<k<n$; see \cite{ApisaWright} for the precise definition. 

Given a point marking, we call a marked point \emph{free} if it can be deformed freely in $\cN$ without changing the unmarked surface or the positions of the other marked point. Any $n$-point marking with $n>1$ that has a free point is reducible, since it arises from combining a 1-point marking with the free marked point and an $(n-1)$-point marking arising from the remaining marked points. 

Similarly, any $n$-point marking with $n>1$ that has a periodic point is reducible, since it arises from the 1-point marking giving that periodic point together with an $(n-1)$-point marking arising from the remaining marked points. 

For a 2-point marking $\cN$ over $\cM$, one of the following is true: 
\begin{itemize}
\item $\dim \cN=\dim \cM$, and both marked points are periodic,
\item $\dim \cN=\dim \cM+2$, and both marked points are free,
\item $\dim \cN=\dim \cM+1$, one marked point is free and the other is periodic, 
\item $\dim \cN=\dim \cM+1$ and $\cN$ is irreducible. 
\end{itemize}

In the last case, fixing the surface in $\cM$, there is 1-dimension of freedom to change the position of the two marked points, and the position of each marked point locally determines the position of the other. The slope, which we now define, describes the relative speed at which the two marked points move. 

\begin{defn}\label{D:slope}
Let $\cN$ be an irreducible $2$-point marking over $\cM$ and let $(X, \omega; \{p_1, p_2\})$ denote a generic surface in $\cN$. If $\gamma_i$ is a path from a zero $z$ to $p_i$ for $i\in \{1,2\}$, then there is a constant $a$ such that at all points in $\cN$ near $(X, \omega; \{p_1, p_2\})$, 
\[ \int_{\gamma_1} \omega = a \int_{\gamma_2} \omega + \int_\gamma \omega \]
where $\gamma \in H_1(X, \Sigma; \bC)$ and $\Sigma$ is the set of zeros and marked points of $(X,\omega)$ (excluding $p_1, p_2$). The \emph{slope} of $\cN$ is defined to be $a$ or $1/a$, whichever is larger. 

If $\cM$ belongs to a stratum of quadratic differentials, the slope is defined similarly but is only defined up to sign. 
\end{defn}

The following is a version of \cite[Theorem 2.8]{ApisaWright}, which follows immediately from its proof. 

\begin{lem}[Apisa-Wright]\label{L:ApisaWright}
Suppose that $\cM$ is an invariant subvariety of Abelian differentials, but is not a locus of torus covers, and that $\cN$ is an irreducible 2-point marking over $\cM$. Let $(X,\omega, \{p_1, p_2\}) \in \cN$. 

If $\cN$ has slope 1, there is a translation cover $f: (X, \omega) \ra (Y, \eta)$ to a translation surface $(Y, \eta)$ with $f(p_1) = f(p_2)$.

If $\cN$ has slope $-1$, there is a half-translation cover $f: (X, \omega) \ra (Y, \eta)$ to a quadratic differential $(Y, \eta)$ with $f(p_1) = f(p_2)$.

If $\cM$ consists of quadratic differentials, and the loci of holonomy double covers isn't a locus of torus covers, and the slope of $\cN$ is $\pm1$, then the same conclusion holds.
\end{lem}

The following is one of the main results of \cite{Apisa} for Abelian differentials and of Apisa-Wright \cite{ApisaWright} for quadratic differentials. 

\begin{thm}[Apisa,  Apisa-Wright]\label{T:StrataMarkedPoints}
Connected components of strata of Abelian or quadratic differentials that have rank at least two do not have periodic points unless they are hyperelliptic components, in which case the periodic points are the Weierstrass points. 
\end{thm}

\begin{defn}\label{D:F}
Suppose that $(X, \omega)$ is an Abelian or quadratic differential with marked points that belongs to an invariant subvariety $\cM$. Then $\For(X, \omega)$ will denote $(X, \omega)$ once marked points are forgotten. Similarly, we will define $\For(\cM)$ to be the invariant subvariety that is the closure of $\{ \For(Y, \eta) : (Y, \eta) \in \cM \}$.
\end{defn}

\subsection{Cylinder deformations}\label{SS:CylinderDeformations}
In this section we recall some definitions and results from \cite{Wcyl}, together with supplemental results from \cite{MirWri}.

\begin{defn}\label{D:CylAndBoundary}
A \emph{cylinder} on a  translation or half-translation surface is \red an \black isometric embedding of $\bR/(c \bZ) \times (0,h)$ into the surface, which is not the restriction of an isometric embedding of a larger cylinder. The circumference of the cylinder is defined to be $c$, and its height is defined to be $h$. 

The map extends to a map of $\bR/(c \bZ) \times [0,h]$ to the surface, which is not in general an embedding. The images of  $\bR/(c \bZ) \times \{0\}$ and $\bR/(c \bZ) \times \{h\}$ are the two \emph{boundary components} of the cylinder; they consist of saddle connections, and together they form the \emph{boundary} of the cylinder. The \emph{multiplicity} of a saddle connection in a component of the boundary is the number of preimages it has in the corresponding $\bR/(c \bZ) \times \{0\}$ or $\bR/(c \bZ) \times \{h\}$; this must be \red at most 2 \black (see Figure \ref{F:CylinderTypes} for an example).
\end{defn}

Two parallel cylinders on  $(X, \omega)\in\cM$ are called \emph{$\cM$-parallel} or \emph{$\cM$-equivalent} if they remain parallel on all nearby surfaces in $\cM$. A maximal collection $\bfC$ of $\cM$-parallel cylinders on $(X, \omega)$ is called an \emph{$\cM$-equivalence class}. 

At all surfaces near $(X, \omega)$ in $\cM$ on which the cylinders in $\bfC$ persist they remain $\cM$-parallel. However, it is possible that at nearby surfaces $\bfC$ is only a subset of an $\cM$-equivalence class. In other words, new cylinders might appear on nearby surfaces that are $\cM$-equivalent to cylinders in $\bfC$. 

If $\bfC = \{C_1, \hdots, C_n\}$ is a collection of parallel cylinders on a flat surface $(X, \omega)$ with core curves $\{ \gamma_1, \hdots, \gamma_n\}$, then we will say that the core curves are consistently oriented if their holonomy vectors are positive real multiples of each other.  

If $\bfC = \{C_1, \hdots, C_n\}$ is an $\cM$-equivalence class of cylinders on $(X, \omega) \in \cM$ with consistently oriented core curves $\{\gamma_1, \hdots, \gamma_n\}$, then the \emph{standard shear} in $\bfC$ is defined to be 
$$\sigma_{\bfC} := \sum_{i=1}^n h_i \gamma_i^*$$ 
where $h_i$ denotes the height of cylinder $C_i$ and $\gamma_i^*$ is the intersection number with $\gamma_i$.  The following is the main theorem of \cite{Wcyl}, restated in a form closer to \cite[Theorem 4.1]{MirWri}.

\begin{thm}[Wright]\label{T:CDT}
The standard shear $\sigma_{\bfC}$ belongs to the tangent space of $\cM$ whenever $\bfC$ is an $\cM$-equivalence class.
\end{thm}

Notice that when $\bfC$ is a collection of horizontal cylinders, the straight line path in $\cM$ determined by the tangent direction $\sigma_\bfC$ at $(X, \omega)$ determines a family of translation surfaces formed from $(X, \omega)$ by applying $$\begin{pmatrix} 1 & t \\ 0 & 1 \end{pmatrix}$$ for $t \in \bR$ to the cylinders in $\bfC$ on $(X, \omega)$ while fixing the rest of the surface. Similarly, if $\bfC$ is a collection of horizontal cylinders the straight line path in $\cM$ determined by the tangent direction $i \sigma_\bfC$ at $(X, \omega)$ gives the standard dilation in $\bfC$.

If $\bfC = \{C_1, \hdots, C_n\}$ is an $\cM$-equivalence class of cylinders with core curves $\{\gamma_1, \hdots, \gamma_n\}$ on a surface $(X, \omega)$ in $\cM$, then the \emph{twist space of $\bfC$}, denoted $\mathrm{Twist}(\bfC)$, is the collection of complex linear combinations of $\{\gamma_1^*, \hdots, \gamma_n^*\}$ that belong to the tangent space of $\cM$ at $(X, \omega)$. Mirzakhani and Wright showed the following partial converse to Theorem \ref{T:CDT} \cite[Theorem 1.5]{MirWri}; see Lemma \ref{L:PD} for an alternate proof. 

\begin{thm}[Mirzakhani-Wright]\label{T:CylECTwistSpace}
Every element of $\mathrm{Twist}(\bfC)$ can be written uniquely as the sum of a  multiple of the standard shear on $\bfC$ and an element of $\ker(p)$.
\end{thm}

The following is derived from Theorem \ref{T:CylECTwistSpace} in \cite[Corollary 1.6]{MirWri}. 

\begin{cor}[Mirzakhani-Wright]\label{C:ConstantRatio}
$\cM$-parallel cylinders in invariant subvarieties with no rel have a constant ratio of moduli. 
\end{cor}

\begin{defn}\label{D:subeq} 
We will say that a collection of cylinders $\bfC = \{C_1, \hdots, C_n\}$ on $(X, \omega)$ is an $\cM$-\emph{subequivalence class} if all of the cylinders in $\bfC$ are $\cM$-parallel and if $\bfC$ is a minimal collection of cylinders such that the standard shear $\sigma_{\bfC}$ belongs to $T_{(X, \omega)} \cM$. 
\end{defn}

Every $\cM$-equivalence class is a union of $\cM$-subequivalence classes. A related but different definition of subequivalent was used in \cite{Apisa:Rank1}; our definition is better suited to the general study of invariant subvarieties. We illustrate the definition with the following lemma. 

\begin{lem}\label{L:RankTest}
If a surface in $\cM$ has two disjoint non-parallel subequivalence classes, then $\cM$ has rank at least two. 
\end{lem}

The example of $\cM=\cH(0,0)$ shows that the non-parallel assumption is necessary. 

\begin{proof}
Suppose $\cM$ has rank 1. Suppose in order to find a contradiction that $(X,\omega)\in \cM$ has two disjoint non-parallel subequivalence classes, $\bfC_1$ and $\bfC_2$. 

Then \cite[Theorem 1.5]{Wcyl} gives in particular that $(X,\omega)$ is periodic in the $\bfC_1$ direction. Deforming $\bfC_2$ does not change the circumference of the cylinders in $\bfC_1$, but will change the circumference of some cylinders parallel to $\bfC_1$. This shows that not all parallel cylinders are $\cM$ parallel, which contradicts \cite[Theorem 1.10]{Wcyl}.
\end{proof}

As an example, on a surface $(X, \omega)$ in a component $\cH$ of a stratum of Abelian differentials, two cylinders are equivalent if and only if their core curves are homologous to each other. However, every cylinder on $(X, \omega)$ can be dilated while remaining in $\cH$. Therefore, each $\cH$-subequivalence class is a singleton.

\begin{defn}\label{D:GenericCyl}
A cylinder on a translation surface $(X, \omega)$ in $\cM$ is said to be \emph{generic} if all the saddle connections on its boundary remain parallel to the core curve of the cylinder on all nearby surfaces in $\cM$. 
\end{defn}

\begin{rem}\label{R:GenericCylInStrata}
It is not hard to see that if a cylinder is not generic on $(X, \omega)$ then it becomes generic on almost every surface in a neighborhood of $(X, \omega)$ in $\cM$. The generic cylinders in strata of Abelian differentials are simple cylinders, i.e. cylinders where each component of the boundary consists of a single saddle connection. 
\end{rem}

\begin{lem}\label{L:AbelianDoubleSEC}
Subequivalence classes of generic cylinders in Abelian doubles are either pairs of simple cylinders or a single complex cylinder.
\end{lem}

 \red Per Definition \ref{D:CylinderTypes}, a complex cylinder is one with two saddle connections of equal length on each boundary. \black 

\begin{proof}
We have already remarked that subequivalence classes of cylinders in strata of Abelian differentials are sets containing a single cylinder.  Therefore, on an Abelian double, subequivalence classes of generic cylinders consist of preimages of simple cylinders, which are either a pair of simple cylinders or a single complex cylinder. In the case that the stratum being covered has free marked points this uses the stipulation in the definition of ``Abelian double" that the preimage of every free marked point be marked.
\end{proof}

Finally, we will record the following definition for future use. 

\begin{defn}\label{D:CylinderAdjacent}
We will say that two cylinders are \emph{non-adjacent} if they are disjoint and share no boundary saddle connections. Similarly, we will say that a cylinder is \emph{not adjacent} to a saddle connection if the cylinder and saddle connection are disjoint and the saddle connection is not a boundary saddle connection of the cylinder.  
\end{defn}

\subsection{The boundary of an invariant subvariety}

We now discuss results from \cite{MirWri, ChenWright} about the boundary of an invariant subvariety $\cM$. 

Let $(X_t, \omega_t), t\in (0,1]$ be a continuous family of translation surfaces which degenerates as $t\to 0$. Suppose that it is possible to present $(X_1, \omega_1)$ via polygons in the plane in such a way that all $(X_t, \omega_t), t\in (0,1]$ can be presented by changing the edge vectors of these polygons. Suppose furthermore that, as $t\to 0$, the polygons converge to limit polygons in such a way that some edges reach 0 length, and all other edges converge to edges of the limiting polygons. Then 
$$\lim_{t\to 0} (X_t, \omega_t)$$
exists in the WYSIWYG partial compactification, and is equal to the surface given by the limit polygons \cite[Definition 2.2, Proposition 9,8]{MirWri}; for the specific example of cylinder degenerations see also \cite[Lemma 3.1]{MirWri}. This motivates the name What You See Is What You Get, since in such nice situations the limit is obtained naively from limiting polygons. 

The following is a special case of the main result \cite{MirWri} when the limit is connected, and in \cite{ChenWright} when it is disconnected. 

\begin{thm}[Mirzakhani-Wright, Chen-Wright]\label{T:BoundaryTangent}
\red Suppose \black a path $(X_t, \omega_t), t\in (0,1]$ as above is contained an orbit closure $\cM$. Then $\lim_{t\to 0}(X_t, \omega_t)$ is contained in a component $\cM'$ of the boundary \red of $\cM$, and \black $\cM'$  is locally defined as follows by linear equations. 

Using local coordinates consisting of edge vectors for the polygons defining $(X_1, \omega_1)$, consider the linear equations locally defining $\cM$. Delete (or replace by zero) all terms corresponding to edges that do not give rise to edges of the limit. These equations locally define $\cM'$. 
\end{thm}

This again is a naively intuitive statement: After an edge reaches zero length, we should just plug in 0 for the corresponding variable in the equations defining $\cM$ in order to obtain the equations defining $\cM'$. 

Here we have stated Theorem \ref{T:BoundaryTangent} in a simpler way than it is first presented in \cite{MirWri, ChenWright}, since we have no need for the more complicated situation of limits of arbitrary sequences in $\cM$;  our presentation here follows ``Other points of view" in  \cite[Section 1]{ChenWright}. We have also chosen to use polygonal presentations, which places a small burden on the reader later in the paper to imagine polygonal presentations. Readers familiar with \cite{MirWri, ChenWright} will understand that polygonal presentations are in fact not necessary.

Note that connected surfaces may degenerate to disconnected surfaces. 

\begin{conv}\label{CV:connected} 
Unless otherwise specified, all surfaces in this paper will be connected. \red Abelian and quadratic \black doubles by definition consist of connected surfaces, so in Theorem \ref{T:DoubleIntro} the invariant subvarieties $\MOne$ and $\MTwo$ consist of connected surfaces; but $\MOneTwo$ might consist of disconnected surfaces. The assumption in Theorem  \ref{T:IntroFull} that the surfaces in $\MOneTwo$ are connected implies that the surfaces in $\MOne$ and $\MTwo$ are also connected, but as we indicate in Theorem  \ref{T:FullV2} we obtain the same conclusion in most cases when the surfaces in $\MOne$ and $\MTwo$ are possibly disconnected covers of connected surfaces (in which case it turns out that the surfaces in $\MOneTwo$ are also possibily disconnected covers of connected surfaces).
\end{conv}

We now illustrate some of the ways Theorem \ref{T:BoundaryTangent} will be applied throughout the paper. 

\begin{cor}\label{C:StillGeneric}
Suppose that $\bfC_1$ and $\bfC_2$ are disjoint subequivalence classes of cylinders on a surface $(X, \omega)$ in an invariant subvariety $\cM$, that $\bfC_1$ and $\bfC_2$ don't share boundary saddle connections, and that the cylinders in $\bfC_1$ are $\cM$-generic. If $\MTwo$ is the component of the boundary of $\cM$ containing $\ColTwoX$, then the cylinders in $\ColTwo(\bfC_1)$ are $\MTwo$-generic.
\end{cor}
\begin{proof}
One can consider polygonal presentations where each cylinder in $\bfC_1$ and $\bfC_2$ is, for example, a union of triangles. There are linear equations defining $\cM$ that directly give that all the boundary saddle connections of $\bfC_1$ are generically parallel, and these give rise to corresponding equations for $\MTwo$. 
\end{proof}

\begin{cor}\label{C:ParallelSC}
Similarly, if $\bfC$ is a generic subequivalence class such that $\mathrm{Twist}(\bfC)$ is one dimensional, then the saddle connections in $\Col_{\bfC}(\bfC)$ are $\cM_{\bfC}$-parallel. 
\end{cor}
Recall that $\mathrm{Twist}(\bfC)$ is defined before Theorem \ref{T:CylECTwistSpace}. The condition that this space is one \red dimensional \black indicates that the only deformation of $\bfC$ that \red remains \black in $\cM$ is the standard deformation.
\begin{proof}
Every saddle connection in $\Col_{\bfC}(\bfC)$ arises from a saddle connection in $\overline\bfC$. 

We have assumed that there is a saddle connection in $\overline\bfC$ that is perpendicular to its core curves. Since the only deformation of $\bfC$ that remains in $\cM$ is the standard deformation, every saddle connection in $\overline{\bfC}$ has holonomy given as a linear combination of perpendicular cross curve and the circumference. Since the perpendicular cross curves collapses (has zero holonomy in the limit), this gives the result.
\end{proof}

\begin{lem}\label{L:codim1doubles}
Suppose that $\cM'$ consists of connected surfaces and is a codimension one boundary component of an Abelian (resp. quadratic) double $\cM$. Suppose too that $\cM'$ contains a surface of the form $\Col_{\bfC}(X, \omega)$ where $(X, \omega) \in \cM$ and $\bfC$ is a subequivalence class of cylinders on $(X, \omega)$. Then $\cM'$ is an Abelian (resp. quadratic) double. 
\end{lem}
\begin{proof}
For concreteness we prove this in the Abelian double case. 
One can pick a $T$-invariant triangulation of $(X,\omega)$, where $T$ is the translation involution, in such a way that $\bfC$ is a union of triangles. For each edge $v$, there is an equation $v=Tv$ locally defining $\cM$. These give rise to equations for $\cM'$ showing that $\cM'$ consists of degree two covers of Abelian differentials. 

Since $\cM$ is a full locus of covers, we get that $\cM'$ is also, since any boundary component of a component of stratum of Abelian differentials is again a component of stratum of Abelian differentials. The set of marked points on $\Col_{\bfC}(X, \omega)$ is $T$ invariant since the set of marked points and zeros on $(X,\omega)$ is $T$-invariant. 
\end{proof}

\begin{lem}\label{L:ExtendingPaths}
Suppose that $(X, \omega)$ is a translation surface in an invariant subvariety $\cM$, and that  $\bfC$ is a subequivalence class of cylinders on $(X,\omega)$. Let $\cM_{\bfC}$ be the component of the boundary containing $\Col_{\bfC}(X,\omega)$. 

Let $(X_t, \omega_t)$ be a path in the stratum containing $(X,\omega)$ with $(X_0, \omega_0)=(X,\omega)$ and along which $\bfC$ not only persists but, for each $t$, is a rotated and scaled copy of $\bfC$ on $(X,\omega)$  (including all saddle connections in the boundary of $\bfC$). 

Suppose that, in a continuous way depending on $t$, $\bfC$ can be collapsed using standard cylinder deformations to give a surface in  $\cM_{\bfC}$. Then the path lies in $\cM$. 
\end{lem}

\begin{proof}
Again consider a triangulation where each cylinder in $\bfC$ is a union of triangles. 

Since the standard deformation of $\bfC$ remains in $\cM$, one can locally define $\cM$ using the following two types of linear equations: those using only edges not in $\bfC$, and those using only edges in $\bfC$. \red (Here it is important that we view $\bfC$ as an open subset of the surface, which does not contain the boundary saddle connections.) \black

The first type of equations hold along the path because, for each $t$, the collapse of $\bfC$ lies in $\cM_{\bfC}$, and these equations correspond to equations in $\cM_{\bfC}$. The second type of equations hold along the path because $\bfC$ only gets rotated and scaled. 
\end{proof}

\subsection{Generic diamonds}\label{SS:GenericDiamonds}

In the sequel we will mainly use the following type of diamond, whose definition builds on Definitions \ref{D:subeq} and \ref{D:GenericCyl}: 

\begin{defn}\label{D:GenericDiamond}
A diamond will be called a \emph{generic} if 
\begin{enumerate}
\item\label{E:genericSE} each $\bfC_i$ is a subequivalence class of generic cylinders, and 
\item\label{E:one} $\cM_{\bfC_i}$ has dimension exactly one less than $\cM$ for each $i \in \{1, 2\}$.
\end{enumerate}
\end{defn} 


\red

In the remainder of this section we explain that given a generic diamond, there is no harm in assuming the surface has dense orbit, and that generic diamonds are abundant. We need a few lemmas for this. 

\begin{lem}\label{L:DensePerturbation}
Suppose that $\bfC_1$ and $\bfC_2$ are disjoint subequivalence classes of generic cylinders on $(X, \omega)$ in an invariant subvariety $\cM$. Then there is an arbitrarily nearby surface $(X'', \omega'')$ with dense orbit in $\cM$ on which $\bfC_1$ and $\bfC_2$ persist and remain subequivalence classes of generic cylinders. 
\end{lem}

This lemma confronts the possibility that a perturbation may cause a subequivalence class to cease being a subequivalence class; although one can transport the standard deformation of  the subequivalence class to the deformation, if the ratios of heights of the cylinders have changed it might no longer be standard. In all the applications in this paper and its sequels, we will know  a priori that this difficulty cannot occur. Since additionally the proof of the lemma, in general, is  a bit technical, some readers may wish to skip it. 

\begin{proof}
We will work entirely within a neighborhood $U$ of $(X, \omega)$ on which the cylinders in $\bfC_1 \cup \bfC_2$ persist and remain generic. (Since $\bfC_j$ consists of generic cylinders, this can be accomplished by letting $U$ be a neighborhood in which all of the saddle connections on the boundary of cylinders in $\bfC_1 \cup \bfC_2$ persist.) Let $(X', \omega') \in U$ be a surface close to $(X,\omega)$ with dense orbit in $\cM$. Let $\bfD_j$ be the equivalence class of $\bfC_j$. 

The cylinders in $\bfD_j$ persist on $(X', \omega')$ and remain equivalent to each other, but additional parallel cylinders may have been created in the passage from $(X,\omega)$ to $(X', \omega')$.  So we let $\bfD_j'$ denote the cylinders on $(X', \omega')$ equivalent to those persisting from $\bfD_j$, keeping in mind that $\bfD_j'$ may have more cylinders than $\bfD_j$. 

Set $w_j := \sigma_{\bfD_j} -\sigma_{\bfD_j'}$. (Here $\sigma_{\bfD_j'}$ is the standard shear of $\bfD_j'\subset (X', \omega')$ and $\sigma_{\bfD_j}$ is the standard shear of $\bfD_j\subset (X,\omega)$, and we identify the relative cohomology groups of $(X,\omega)$ and $(X',\omega')$ via parallel transport.) Let $\alpha_j$ be a unit modulus complex number that is perpendicular to the period of the core curves of $\bfD_j'$, so deforming $(X', \omega')$ in the $\alpha_j w_j$ direction corresponds to dilating the cylinders without shearing them and deforming in the $i \alpha_j w_j$  direction corresponds to twisting the cylinders. 

Picking the sign of $\alpha_j$ appropriately ensures that, on $(X', \omega') + \alpha_j w_j$, the cylinders in $\bfC_j$ have returned to their original heights. (One can think of this as subtracting off the appropriate multiple of $\sigma_{\bfD_j'}$ to collapse all the cylinders in $\bfD_j'$, and then adding the appropriate multiple of $\sigma_{\bfD_j}$ to restore the cylinders in $\bfD_j$ to their original heights.) 

Using the Cylinder Finiteness Theorem of \cite[Theorem 4.1]{MirWri} or \cite[Theorem 5.3]{ChenWright}, we see that the circumferences of cylinders in $\bfD_j'$ are bounded independently of the perturbation $(X', \omega')$. Since  $(X', \omega')$ is close to $(X, \omega)$, this implies moreover that any ``new" cylinders in $\bfD_j'$ that don't arise from $\bfD_j$ have small height. These observations can be used to show that  $(X', \omega') + \alpha_1 w_1 + \alpha_2 w_2$ is close to $(X', \omega')$ and hence also close to $(X,\omega)$. 

For all sufficiently small real numbers $a_1$ and $a_2$, 
\[ (X', \omega') + \alpha_1 w_1 + \alpha_2 w_2 + i(\alpha_1 a_1 w_1 + \alpha_2 a_2 w_2)\]
remains in $U$ and has the property that the cylinders in $\bfC_1$ and $\bfC_2$ have the same heights as on $(X,\omega)$ and hence continue to form a subequivalence class. If one of these surfaces has dense orbit in $\cM$, then we are done. 

Suppose therefore that this does not occur. Then there is a proper invariant subvariety $\cM'$ contained in $\cM$ such that $i\alpha_1 w_1$ and $i \alpha_2 w_2$ are tangent to $\cM'$ at $(X', \omega') + \alpha_1 w_1 + \alpha_2 w_2$. This in particular means that $(X', \omega')$ is contained in $\cM'$, which contradicts the assumption that it has dense orbit in $\cM$.
\end{proof}

\begin{lem}\label{L:ShearStillDense}
Suppose that $\bfC$ is a subequivalence class of cylinders on a surface $(X,\omega)$ whose orbit is dense in $\cM$. Then, for almost every $s$, every surface obtained via a standard shear in $\bfC$ from $a_s^{\bfC}(X,\omega)$ has dense orbit in $\cM$. 
\end{lem}

Here $a_s^{\bfC}(X,\omega)$ is the result of dilating the cylinders in $\bfC$, as in the introduction. 

\begin{proof}
Pick any $s$ so that the $\bQ$-span of the moduli of the cylinders in $\bfC$ on $a_s^{\bfC}(X,\omega)$ has trivial intersection with the $\bQ$-span of the moduli of all other cylinders parallel to $\bfC$. 

Fix $t\in \bR$, and let $u_t^{\bfC} a_s^{\bfC}(X,\omega)$ be the result of shearing the cylinders in $\bfC$ on $a_s^{\bfC}(X,\omega)$ by $t$. 
Since $\bfC$ is a subequivalence class, $u_t^{\bfC} a_s^{\bfC}(X,\omega)\in \cM$.

As in \cite[Lemma 3.1]{Wcyl} or \cite[Lemma 4.6]{MirWri}, the orbit closure of $u_t^{\bfC} a_s^{\bfC}(X,\omega)$ contains $(X,\omega)$. (The proof of the first cited lemma shows that the shear in $\bfC$ remains in the orbit closure of $u_t^{\bfC} a_s^{\bfC}(X,\omega)$; the dilation is obtained as a complex multiple of the shear.) Hence $u_t^{\bfC} a_s^{\bfC}(X,\omega)$ has dense orbit in $\cM$, since $(X,\omega)$ has dense orbit. 
\end{proof}

\begin{cor}\label{C:wologDense}
Given a generic diamond $((X, \omega), \cM, \bfC_1, \bfC_2)$,  there is a perturbation $(X', \omega')$ of $(X,\omega)$ with subequivalence classes $\bfC_i'$ of generic cylinders arising from deforming $\bfC_i$ such that $((X', \omega'), \cM, \bfC_1', \bfC_2')$ is also a generic diamond and the surface $(X', \omega')$ has dense orbit in $\cM$. 

Furthermore,  $\Col_{\bfC_i'}(X',\omega')$ is a small perturbation of $\Col_{\bfC_i}(X,\omega)$, and we have $\cM_{\bfC_i}= \cM_{\bfC_i'}$.
\end{cor}
\begin{proof}
Let $s_i$ be a saddle connection in $\overline{\bfC}_i$ perpendicular to the core curves. 

By Lemma \ref{L:DensePerturbation}, there is a surface $(X', \omega')$ with dense orbit in $\cM$ that is arbitrarily close to $(X, \omega)$ and on which $\bfC_1$ and $\bfC_2$ remain generic subequivalence classes of cylinders. For clarity, when thinking of these subequivalence classes on $(X', \omega')$, we denote them  $\bfC_1'$ and $\bfC_2'$ respectively. 

Because the cylinders were generic, each $s_i$ is still a saddle connection on $(X', \omega')$, and it is contained in $\overline{\bfC}_i'$. However, it may no longer be perpendicular to the core curves. By applying Lemma \ref{L:ShearStillDense}, we may correct this by applying standard cylinder deformations to $\bfC_1'$ and $\bfC_2'$ while ensuring that the resulting surface still has dense orbit. Thus, without loss of generality, we assume that, actually, each $s_i$ is perpendicular to the core curves of $\bfC_i'$ on $(X', \omega')$. 

Since $\cM_{\bfC_i}$ is a codimension 1 degeneration, all saddle connections parallel to $s_i$ that are contained in $\overline{\bfC}_i$ are generically parallel to each other. Assuming the above perturbations are small, the corresponding statement holds on the perturbation. 

Since $\cM_{\bfC_i}$ and $\cM_{\bfC_i'}$ are both codimension 1 degenerations that degenerate $s_i$, they must be equal. 
%
%
\end{proof}

Almost every surface in $\cM$ does not have any saddle connections perpendicular to a cylinder, which is why Lemma \ref{L:ShearStillDense} is required above. Alternatively, one can avoid this issue by modifying the definition of diamonds, choosing, for each $i$, a choice of direction in which there is a saddle connection in $\overline\bfC_i$, and using cylinder degenerations that collapse these directions. See \cite{ApisaWrightGemini} for details, where we call the resulting notion \emph{skew diamonds}. \red Skew diamonds are just as good as diamonds, but more flexible, and the only reason we use regular diamonds in this paper is to avoid the notational annoyance of always having to specify the choice of direction. 

\begin{rem}\label{R:CouldBeDense}
Keeping Corollary \ref{C:wologDense} in mind, we will frequently assume without loss of generality that the surface in our generic diamonds has dense orbit. 
\end{rem}
\black

We close the section with a result that will not be used in this paper, but illustrates the ubiquity of diamonds. 

\begin{lem}\label{L:ManyDiamonds}
Let $\cM$ be an invariant subvariety of rank at least 2. Then there exists a surface $(X,\omega)\in \cM$ with collections $\bfC_1, \bfC_2$ of cylinders that form a generic diamond.  

Moreover, up to shearing the $\bfC_i$, this $(X,\omega)$ may be assumed to be any surface in $\cM$ on which all parallel saddle connections are $\cM$-parallel, and $\bfC_1$ may be any subequivalence class. 
\end{lem}

\red The shearing is not necessary if one uses skew diamonds. Diamonds never exist when $\cM$ is rank 1 and has (complex) dimension 2 or 3. If $\cM$ is rank 1 and has dimension at least 4, they often but not always exist. 
\black

The lemma does not guarantee that the $\cM_{\bfC_i}$ consist of connected surfaces; even though we assume $(X,\omega)$ is connected, $\ColOneX$ and $\ColTwoX$ might have multiple components. 

\begin{proof}
Start with any $(X,\omega) \in \cM$ on which all parallel saddle connections are $\cM$-parallel. (Such surfaces are dense.) 


Let $\bfD_1$ be an equivalence class of $\cM$-parallel cylinders on $(X,\omega)$. \red Since $\cM$ is not rank 1 and parallel saddle  connections are $\cM$-parallel, the cylinders of $\bfD_1$ do not cover $(X,\omega)$. This follows from  the proof of \cite[Theorem 1.7]{Wcyl}, and can also be established by contradiction as follows. Assume that $\bfD_1$ is horizontal and covers $(X,\omega)$. By assumption all horizontal saddle connections are $\cM$-parallel. Therefore, any real deformation in $\cM$ that preserves the length of one horizontal saddle connection preserves the lengths of all horizontal saddle connections. Such deformations necessarily belong to $\mathrm{Twist}(\bfD_1)$, which projects to a one-dimensional subspace of absolute cohomology by Theorem \ref{T:CylECTwistSpace}. Since the collection of real deformations projects to a $2\cdot \mathrm{rank}(\cM)$ real-dimensional subspace of $p\left( T_{(X, \omega)} \cM \right)$, and since $\mathrm{Twist}(\bfD_1)$ is by assumption codimension one in the space of real deformations, we get that $\cM$ has rank one, which is a contradiction.\black

\red  Recall from \cite[Corollary 6]{SW2} that the horocycle flow orbit closure of any surface contains a surface covered by horizontal cylinders. Applying this fact as \black in \cite[Section 8]{Wcyl} or \red \cite[Lemma 8.3]{ApisaWrightHighRank}, we get the existence of \black  a cylinder $E$ disjoint from $\bfD_1$. Since we have assumed that all parallel saddle connections are generically parallel, $E$ cannot be parallel to $\bfD_1$. Let $\bfD_2$ be the equivalence class of $E$. It is easy to see using the Cylinder Deformation Theorem \cite[Theorem 1.1]{Wcyl} that no cylinder of $\bfD_1$ intersects a cylinder of $\bfD_2$; see also \cite[Proposition 3.2]{NW}, which is sometimes called the Cylinder Proportion Theorem. 

The Cylinder Deformation Theorem implies that the standard dilation in each $\bfD_i$ remains in $\cM$. For each $i$, let $\bfC_i$ be a minimal subset of $\bfD_i$ such that the standard dilation of $\bfC_i$ remains in $\cM$. (One expects $\bfC_i = \bfD_i$, since in general there is no reason to believe that the standard dilation of any strict subset of $\bfD_i$ remains in $\cM$.)

Now, shear each $\bfD_i$ so that $\overline\bfC_i$ contains a saddle connection $s_i$ perpendicular to core curves of cylinders in $\bfC_i$. Replace $(X,\omega)$ with this sheared surface. 

\red By definition, $\bfC_1$ and $\bfC_2$ are disjoint. Since they are not parallel, they cannot share boundary saddle connections. The existence of the $s_i$ imply that collapsing either $\bfC_i$ does indeed cause the surface to degenerate. So $((X,\omega), \cM, \bfC_1, \bfC_2)$ defines a diamond.

The $\bfC_i$ are subequivalence classes by definition, so to see that this diamond is generic it suffices to check that \black each $\cM_{\bfC_i}$ has dimension exactly one less than $\cM$. That follows from the main results of \cite{MirWri, ChenWright}, as recalled in Theorem \ref{T:BoundaryTangent}, using that all the saddle connections \red in $\bfC_i$ \black parallel to $s_i$ are $\cM$-parallel to $s_i$, \red which follows from our original assumption that parallel saddle connections are $\cM$-parallel. \black 
\end{proof}

\section{Preliminaries on strata}\label{S:PreQD}
Throughout this section $\cQ$ will denote a connected component of a stratum of quadratic differentials, possibly with marked points. 

\begin{conv}\label{CV:NoTrivHol}
For convenience,  we will insist that the term ``quadratic differential" will never apply to the square of a holomorphic 1-form. Thus, in our convention, quadratic differentials will never have trivial holonomy. We will however use ``half-translation surface" to mean either an Abelian or quadratic differential, so half-translation surface may have trivial or non-trivial holonomy. 
\end{conv}

By Lanneau \cite{LanneauHyp} and Chen-M\"oller \cite{CM}, aside from a finite explicit collection of strata, every stratum of quadratic differentials is either connected or has two components that are distinguished by hyperellipticity. 

\begin{defn}\label{D:HypLocus}
Given a component of a stratum of Abelian or quadratic differentials, the hyperelliptic locus therein is the collection of surfaces with a half-translation map to a genus zero quadratic differential, such that the associated hyperelliptic involution preserves the set of marked points. 

A component of a stratum is called hyperelliptic if it coincides with its hyperelliptic locus. 
\end{defn}

We will  study hyperelliptic components of strata in detail in Section \ref{S:Q-hyp}, but for now we recall the following from \cite[Lemma 4.5]{ApisaWright}. When $\cQ$ is a component of a stratum of quadratic differentials the following is due to Lanneau \cite[Theorem 1]{LanneauHyp}. \red Recall from Definition \ref{D:F}  that if $\cQ$ has marked points, then $\For(\cQ)$ denotes the same stratum with marked points forgotten. \black

\begin{lem}\label{L:InvolutionImpliesHyp-background}
The generic element of a component $\cQ$ of a stratum of
Abelian or quadratic differentials admits a non-bijective half-translation
cover to another translation or half-translation surface if and only if $\For(\cQ)$ is hyperelliptic, in which case the hyperelliptic involution
yields the only such map when $\cQ$ has rank at least two.
%
 
\end{lem}

In the sequel we will also need the rank and rel of a stratum of quadratic differentials, which we recall from \cite[Lemma 4.2]{ApisaWright}. 

\begin{lem}\label{L:Q-rank}
Let $\cQ(\kappa)$ where $\kappa = (k_1, \hdots, k_n)$ be a stratum of quadratic differentials. Let $m_{odd}$ be the number of odd numbers in $\kappa$ and $m_{even}$ the number of even numbers. Let $g$ be the genus. The rank and rel of the component is then
\[ \rk(\cQ) = g + \frac{m_{odd}}{2} - 1 \qquad \text{and} \qquad \mathrm{rel}(\cQ) = m_{even}. \]
\end{lem}

Since the following result requires a more detailed understanding of hyperelliptic components, we defer its proof to Section \ref{S:Q-hyp} (see Lemma \ref{L:QHolonomy-redux}).

\begin{lem}\label{L:QHolonomy}
Let $(X, \omega)$ be a generic surface in a quadratic double of a component $\cQ$ of a stratum of quadratic differentials. If $\For(\cQ) \ne \cQ(-1^4)$, then there is a unique involution $J$ of derivative $-\mathrm{Id}$ such that $(X, \omega)/J$ is a generic surface in a component of a stratum of quadratic differentials. 

If $\For(\cQ) = \cQ(-1^4)$ and $(X, \omega)$ has at least one marked point, then there is a unique marked-point preserving involution $J$ of derivative $-\mathrm{Id}$ such that $(X, \omega)/J$ is a generic surface in a component of a stratum of quadratic differentials.
\end{lem}

\begin{defn}\label{D:holinv}
The involution in Lemma \ref{L:QHolonomy} will be called \emph{the holonomy involution}.
\end{defn}

\subsection{Cylinders}

We will need to understand what cylinders that are generic in the sense of Definition \ref{D:GenericCyl} look like in strata of quadratic differentials. Recall our conventions on cylinders, and the definition of multiplicity, from Definition \ref{D:CylAndBoundary}.

\begin{defn}\label{D:CylinderTypes}
A cylinder on a translation or half-translation surface is called a
\begin{enumerate}
\item \emph{simple cylinder} if each boundary consists of a single saddle connection,
\item \emph{half-simple} cylinder if one boundary is a single saddle connection, and the other is two distinct saddle connections of equal length, 
\item \emph{complex cylinder} if each boundary consists of two distinct saddle connections of equal length, 
\item \emph{simple envelope} if one boundary is a single saddle connection, and the other boundary is a single saddle connection with multiplicity two, 
\item \emph{complex envelope} if one boundary is two distinct saddle connections of equal length, and the other boundary is a single saddle connection with multiplicity two.
\end{enumerate}
See Figure \ref{F:CylinderTypes}
\end{defn}

\begin{figure}[h!]\centering
\includegraphics[width=1\linewidth]{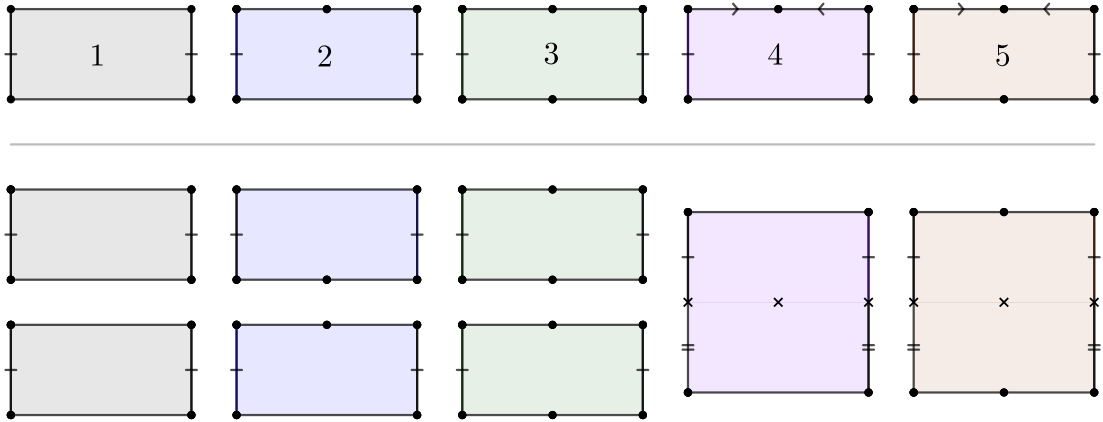}
\caption{Top: The five cylinder types of Definition \ref{D:CylinderTypes}, drawn as rectangles with left and right sides identified to give horizontal cylinders. Bottom: The preimage on the holonomy double cover, if these cylinders appear in a generic quadratic differential. In possibilities 4 and 5, the preimage can have 0, 1, or 2 marked points. The bottom right possibility looks different after a half Dehn twist (illustrated in Figure \ref{F:CylinderType5HalfTwist}).}
\label{F:CylinderTypes}
\end{figure}

The final two possibilities can of course only occur on half-translation surfaces. Using this language, the following foundational result is a consequence of \cite{MZ}, \red as we discuss in Remark \ref{R:MZguide}. \black

\begin{thm}[Masur-Zorich]\label{T:MZ}
Let $C$ be a generic cylinder on a quadratic differential in any stratum other than $\cQ(-1^4)$. Then:
\begin{enumerate}
\item\label{T:MZ:5} $C$ is one of the five possibilities in Definition \ref{D:CylinderTypes}.

\item\label{T:MZ:TrivHolComponent} If $C$ has two distinct saddle connections on one of its boundary components, then cutting those saddle connections disconnects the surface into two pieces, exactly one of which has trivial linear holonomy. The piece with trivial linear holonomy is the component that does not contain the interior of the original cylinder. 

\item\label{T:MZ:GenericAdjacency} If $C$ shares a boundary saddle connection with another generic  cylinder $C'$, and this saddle connection does not join a marked point to itself, then possibly after switching $C$ and $C'$ we have that $C'$ is simple and does not share a boundary saddle connection with any other cylinder, and $C$ has two saddle connections in the given boundary component that borders $C'$. 
\begin{figure}[h!]\centering
\includegraphics[width=.9\linewidth]{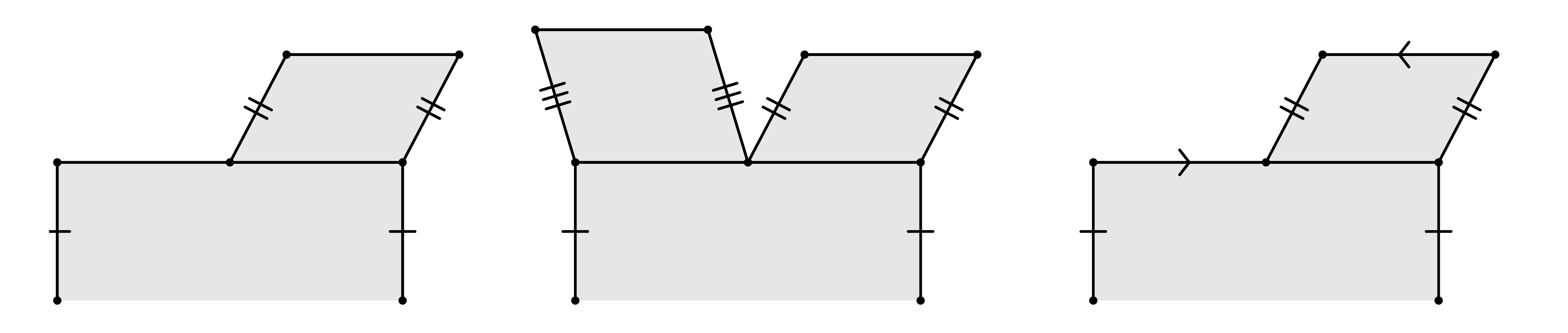}
\caption{The left and right images indicate the two possible configurations of $C$ and $C'$ in Theorem  \ref{T:MZ} (3). The middle images shows that there may also be another cylinder adjacent to $C$. }
\label{F:3hat}
\end{figure}
\end{enumerate}
\end{thm}

Recall that two saddle connections on a quadratic differential are said to be \emph{hat-homologous} in a stratum $\cQ$ of quadratic differentials if they remain parallel at all nearby half-translation surfaces, or in other words if they are $\cQ$-parallel. A tool in the proof of Theorem \ref{T:MZ} is the following \cite[Theorem 1]{MZ}. 

\begin{thm}[Masur-Zorich]\label{T:HatHomologous}
Two saddle connections are hat-homologous if and only if a component of their complement has trivial holonomy. If such a component exists, it is unique.
\end{thm}

\begin{rem}\label{R:MZguide}\red
Since we do not use exactly the same language as Masur and Zorich, we now explain how Theorem \ref{T:MZ} is contained in \cite{MZ}. The assumption that $C$ is generic implies that all the boundary saddle connections of $C$ are hat-homologous. Thus Theorem \ref{T:MZ} follows from the classification of configurations of hat homologous saddle connections in \cite[Theorem 2]{MZ}. (It can also be obtained more directly.) 
\black

Masur and Zorich do not consider marked points, but the statements for surfaces with marked points can be easily derived from the statements for unmarked surfaces.
\end{rem}

We conclude with the following basic observations.

\begin{lem}\label{L:StrataBasics}
Let $\cQ$ be a component of a stratum of quadratic differentials and let $\bfC$ be a generic cylinder on $(X, q) \in \cQ$ \red containing a saddle connection perpendicular to its core curve. \black 
\begin{enumerate}
\item\label{L:StrataBasics:codim1} We have $\dim \cQ_{\bfC} = \dim \cQ -1$, and $\cQ_{\bfC}$ is a component of a stratum of Abelian or quadratic differentials. 
\item\label{L:StrataBasics:TrivHolOrNot} If $\bfC$ is a simple cylinder, simple envelope, or half-simple cylinder, then $\cQ_{\bfC}$ consists of connected quadratic differentials, and if $\bfC$ is a complex cylinder, then  $\cQ_{\bfC}$ consists of connected Abelian differentials.
\item\label{L:StrataBasics:hyp} If $\cQ$ is hyperelliptic, then so is $\cQ_{\bfC}$. 
\end{enumerate}
\end{lem}

Here $\cQ_{\bfC}$ is defined to be the component of the boundary of $\cQ$ that contains $\Col_{\bfC}(X, q)$.

\begin{proof}
The first claim follows since deformations in $\cQ$ correspond to deformations in $\cQ_{\bfC}$ as well as shearing $\bfC$. 

\red We now discuss the second claim. If $\bfC$ is a simple cylinder, simple envelope, or half-simple cylinder, it is immediate that $\Col_\bfC(X,q)$ is connected. To see that $\Col_\bfC(X,q)$ has nontrivial holonomy, note that any loop on $(X,q)$ with nontrivial holonomy can be modified to still have nontrivial holonomy and be disjoint from the perpendicular saddle connection in $\bfC$, thus giving rise to a loop with non-trivial holonomy on $\Col_\bfC(X,q)$.  If $\bfC$ is a complex cylinder, then Theorem \ref{T:MZ} \eqref{T:MZ:TrivHolComponent} implies that cutting either boundary of $\bfC$ produces a translation surface with boundary. Therefore, $\Col_{\bfC}(X, q)$ consists of these two translation surfaces glued together along their boundary. \black

The final claim follows because the hyperelliptic involution on $(X,\omega)$ degenerates to a hyperelliptic involution on  $\Col_{\bfC}(X, q)$
\end{proof}

\begin{cor}\label{C:codim1}
Suppose $\cM\subset \cQ$, and $\bfC$ is an $\cM$-generic cylinder on $(X, q) \in \cQ$. Suppose that $\bfC$ is a simple cylinder,  a simple envelope, or a half-simple cylinder. Then if $\cM_{\bfC}$ is a component of a stratum of quadratic differentials, then $\cM=\cQ$. 
\end{cor}

Section \ref{S:ComplexEnvelope} is entirely devoted to understanding the extent to which this fails when $\bfC$ is a complex envelope. Section \ref{SS:OpenProblems} discusses the corresponding  unresolved problem when $\bfC$ is a complex cylinder.  

\begin{proof}
We first claim that $\bfC$ is in fact $\cQ$-generic. This is immediate if $\bfC$ is a simple cylinder or simple envelope. In the case that $\bfC$ is a half-simple cylinder, note that $\Col_\bfC(\bfC)$ consists of two saddle connections that, by Theorem \ref{T:BoundaryTangent}, are $\cM_{\bfC}$-parallel. Since $\cM_{\bfC}$ is a stratum, this means these two saddle connections are hat homologous. By Theorem \ref{T:HatHomologous}, a component of their complement has trivial holonomy. This implies that the saddle connections bounding $\bfC$ are hat-homologous, giving the claim. 

Since $\cM_{\bfC}$ is in the boundary of $\cM$, we have $\dim \cM_{\bfC} \leq \dim \cM-1$.  Lemma \ref{L:StrataBasics} gives $ \dim \cM_{\bfC} =\dim \cQ -1$, so we get that $\dim \cM = \dim \cQ$ and hence $\cM=\cQ$. 
\end{proof}

\subsection{Hyperelliptic components of strata}

Cylinder types are restricted in hyperelliptic components.

\begin{lem}\label{L:hyp-cyl}
Suppose that $\cQ \ne \cQ(-1^4)$ is hyperelliptic. Then every generic cylinder is a complex envelope, complex cylinder, or simple cylinder. 

The hyperelliptic involution acts by translations on complex envelopes and complex cylinders and by rotation on simple cylinders. 
\end{lem}

By Definition \ref{D:CylAndBoundary}, cylinders do not include their boundary, so all cylinders have trivial holonomy. A translation preserves holonomy on the cylinder, and a rotation negates it. 

\begin{proof}
The hyperelliptic involution fixes each cylinder, and either preserves holonomy in that cylinder or negates it. 
 Half-simple cylinders and simple envelopes have no such involution symmetry. A simple cylinder only has a nontrivial rotation involution; it does not have a translation involution. Conversely a complex envelope has only a translation involution. 
 
 Complex cylinders have both types of symmetry. Masur-Zorich (Theorem \ref{T:MZ}) implies that deleting the cylinder's core curve disconnects the surface into two components. The hyperelliptic involution must fix each component, so it must fix each boundary component of the complex cylinder. Hence the involution must act as a translation on complex cylinders. 
\end{proof}

By Kontsevich-Zorich \cite{KZ}, the only strata of Abelian differentials that have hyperelliptic components are $\cH(2g-2)$ and $\cH(g-1,g-1)$, where $g$ is any positive integer. We will denote these components by $\cH^{hyp}(2g-2)$ and $\cH^{hyp}(g-1,g-1)$. Given a surface in one of these strata, its quotient by the hyperelliptic involution is a genus zero quadratic differential with a single zero when $g > 1$; when $g=1$ the quotient belongs to $\cQ(-1^4)$ or $\cQ(-1^4, 0)$ in the case of $\cH(0)$ and $\cH(0,0)$ respectively.

\begin{lem}\label{L:HypParallelism}
In a hyperelliptic component of a stratum of Abelian differentials, two distinct saddle connections are generically parallel if and only if they are exchanged by the hyperelliptic involution.
\end{lem}
\begin{proof}
Suppose that $s$ and $s'$ are distinct generically parallel saddle connections on a translation surface $(X, \omega)$ that belongs to a hyperelliptic component of a stratum. Let $J$ denote the hyperelliptic involution on $(X, \omega)$. Since the claim is obvious in $\cH(0)$ and $\cH(0,0)$, suppose for notational clarity that $(X,\omega)$ has genus at least two. Suppose in order to derive a contradiction that $J(s) \ne s'$. 

It follows that $s/J$ and $s'/J$ are distinct generically parallel saddle connections on $(X, \omega)/J$, which is a genus zero quadratic differential with a single zero. (There is a zero of positive order on $(X,\omega)/J$ since $(X,\omega)$ has genus at least 2.)

By Masur-Zorich (Theorem \ref{T:HatHomologous}), a component of their complement has trivial linear holonomy. Since $(X, \omega)/J$ has genus zero, this component is a cylinder whose core curve is necessarily nullhomologous. However, this contradicts the fact that the unique zero of $(X, \omega)/J$ lies on both $s/J$ and $s'/J$, because these saddle connections lie on opposite sides of the cylinder.
\end{proof}

\begin{lem}\label{L:HypCuttingS}
Suppose that $s$ is a saddle connection that is not fixed by the hyperelliptic involution $J$ on a translation surface in a hyperelliptic component of a stratum of Abelian differentials. Cutting along $s$ and $J(s)$ divides the translation surface into two components that are both fixed by $J$. 
\end{lem}
\begin{proof}
Since $s$ and $J(s)$ are generically parallel, they are homologous, and so cutting them divides the translation surface into two components. The two components must be fixed by $J$ since $J$ negates holonomy.
\end{proof}

Recall that subequivalence classes are defined in Definition \ref{D:subeq}.

\begin{lem}\label{L:QBoundaryHolonomy}
Let $\bfC$ be a subequivalence class of generic cylinders on a generic surface $(X, \omega)$ in a quadratic double $\cM$. If $\Col_{\bfC}(X, \omega)$ is disconnected, then $\cM_{\bfC}$ is an antidiagonal embedding of a component $\cH$ of a stratum of Abelian differential into $\cH \times \cH$ \red defined by $(X,\omega)\mapsto ((X,\pm \omega), (X,\mp\omega))$. \black 

The marked-point preserving affine symmetry group of generic surfaces in $\cM_{\bfC}$ is  $\bZ/2\bZ$ if $\cH$ is not hyperelliptic and $\bZ/2\bZ \times \bZ/2\bZ$ if $\cH$ is  hyperelliptic. 
\end{lem}
\begin{rem}
In the case that $\cH$ is hyperelliptic, by Lemma \red \ref{L:HypParallelism}, \black two saddle connections \red on a surface in $\cM_\bfC$ \black are generically parallel if and only if they are exchanged by elements of the affine symmetry group. 
\end{rem}
\begin{proof}
Let $J$ denote the holonomy involution on $(X, \omega)$. By Lemma \ref{L:StrataBasics}, if $\Col_{\bfC}(X, \omega)$ is disconnected, then $\bfC/J$ is a complex cylinder or a complex envelope.  It follows that $\Col_{\bfC/J}\left( (X, \omega)/J \right)$ is generic in a component $\cH$ of a stratum of Abelian differentials. \red (See Remark \ref{R:TwoDegenerations} for illustrations of the two ways to degenerate a complex cylinder.) \black Notice that $\Col_{\bfC}(J)$ is an involution with derivative $-I$ on $\Col_{\bfC}(X, \omega)$. This involution necessarily exchanges the connected components of $\Col_{\bfC}(X, \omega)$ \red and has derivative \black $-I$, implying that $\cM_{\bfC}$ is an antidiagonal embedding of $\cH$ in $\cH \times \cH$.

Notice that each component of $\Col_{\bfC}(X, \omega)$ has a zero or marked point. By Lemma \ref{L:InvolutionImpliesHyp-background}, the generic surface in a stratum of Abelian differentials with a zero or marked point has a marked-point preserving affine symmetry group that is either trivial (if the stratum is not hyperelliptic) or $\bZ/2\bZ$ if the stratum is hyperelliptic. This establishes the final statement.  
\end{proof}

\subsection{$S$-paths}

\begin{defn}
Two saddle connections are said to be \emph{disjoint} if they have no intersection points in their interiors.
\end{defn}

\begin{defn}\label{D:Spath}
Given a collection of disjoint saddle connections $S$ on a translation surface $(X, \omega)$ belonging to a stratum $\cH$, an $S$-path is continuous map $\gamma: [0, 1] \rightarrow \cH$ such that $\gamma(0) = (X, \omega)$ and the saddle connections in $S$ remain disjoint saddle connections on each $\gamma(t)$. 
\end{defn}

Given a pair $S$ of homologous saddle connections on a connected translation surface $(X, \omega)$, cutting them produces two connected components each with two equal length boundary saddle connections. Gluing together the two boundary saddle connections of each component, we may obtain a surface $(X_S, \omega_S)$ with two connected components. On $(X_S, \omega_S)$, there is a pair $S'$ of saddle connections, one on each component, resulting from gluing the boundary saddle connections. 

Similarly, given a collection $S$ of $n$  pairwise disjoint pairs of homologous saddle connections, cutting and re-gluing gives a translation surface $(X_S, \omega_S)$ with $n+1$ connected components and with a collection $S'$ of $n+1$ pairs of saddle connections. 

\begin{lem}\label{L:ExtendingPerturbations}
Let $S$ be a collection of pairwise disjoint pairs of homologous saddle connections on a translation surface $(X, \omega)$, and let $\gamma'$ be a $S'$-path with $\gamma'(0)=(X_S, \omega_S)$. Then there is an associated $S$-path $\gamma$ such that 
\begin{enumerate}
\item $\gamma(0)=(X,\omega)$, 
\item for all $t$,  $\gamma(t)_S$ and $\gamma'(t)$ are equal up to individually rotating and scaling connected components. 
\end{enumerate}
\end{lem}
Here  $\gamma(t)_S$ is the result of cutting and regluing the saddle connections in $S$ on the translation surface $\gamma(t)$. 
\begin{proof}
By induction, it suffices to consider the case when $S$ consists of a single pair of homologous saddle connections. In this case, $\gamma(t)$ can be defined by first rotating and scaling one of the components of $\gamma'(t)$ so the two saddle connections in $S$ have the same holonomy, and then cutting and re-gluing to obtain a connected surface. 
\end{proof}

\section{Diamonds with a stratum of Abelian differentials}\label{S:DiamondEasy}

The purpose of this section is to prove the following, where we use the notation $\cF$ from Definition \ref{D:F} for forgetting marked points. 

\begin{prop}\label{P:DiamondWithH}
Let $((X,\omega), \cM, \bfC_1, \bfC_2)$ be a generic diamond of Abelian differentials. Suppose that $\MOne$ is a component of a stratum of connected Abelian differentials. Then: 
\begin{enumerate}
\item\label{P:DiamondWithH:NoAbelianDouble} $\MTwo$ cannot be an Abelian double. 
\item\label{P:DiamondWithH:HH} If $\MTwo$ is a component of a stratum of Abelian differentials, so is $\cM$. 
\item\label{P:DiamondWithH:HQ} If $\MTwo$ is a quadratic double, then $\cM$ has rank at least two, and either
\begin{enumerate}
    \item\label{P:DiamondWithH:HQ:NonHypStrat}  $\cM$ is a non-hyperelliptic component of a stratum of Abelian differentials with no marked points, or
    \item\label{P:DiamondWithH:HQ:HypStrat} $\For(\cM)$ is a hyperelliptic component and there is at most one free marked point on surfaces in $\cM$ with the remaining marked points \red the empty set or \black  a collection of  one marked point fixed or two marked points exchanged by the hyperelliptic involution, or 
    \item\label{P:DiamondWithH:HQ:Codim1} $\For(\cM)$ is a codimension one hyperelliptic locus and there is at most one marked point, which is free.
\end{enumerate}
\end{enumerate}
\end{prop}

\begin{proof}[Proof of Proposition \ref{P:DiamondWithH} parts \eqref{P:DiamondWithH:NoAbelianDouble} and \eqref{P:DiamondWithH:HH}] 
Since the diamond is generic, $\ColOne(\bfC_2)$ is a subequivalence class of generic cylinders on $\ColOne(X, \omega)$. 
Since $\MOne$ is a component of a stratum of Abelian differentials it follows that $\ColOne(\bfC_2)$ consists of a single simple cylinder. Since collapsing a simple cylinder on a connected surface produces a connected surface, it follows that $\MOneTwo$ is a locus of connected surfaces. It is necessarily also a component of a stratum of Abelian differentials.

Since no locus is both a component of a stratum of Abelian differentials and an Abelian double, we see that $\MTwo$ is not an Abelian double. This establishes \eqref{P:DiamondWithH:NoAbelianDouble}. 

Similarly, if $\MTwo$ is also a component of a stratum of Abelian differentials, then both $\bfC_1$ and $\bfC_2$ are simple cylinders. Since gluing a simple cylinder into a generic surface in a stratum of Abelian differentials produces another generic surface in a stratum of Abelian differentials, this proves \eqref{P:DiamondWithH:HH}. 
\end{proof}

The rest of this section is devoted to the case when $\MTwo$ is a quadratic double. 
Recall that subequivalence classes are defined in Definition \ref{D:subeq}.

\begin{proof}[Proof of Proposition \ref{P:DiamondWithH} part \eqref{P:DiamondWithH:HQ}]
 Since $\MOneTwo$ is both a component of a stratum of Abelian differentials and a quadratic double, $\MOneTwo$ is a hyperelliptic component of a stratum of Abelian differentials by Lemma \ref{L:InvolutionImpliesHyp-background}. The marked points on surfaces in $\MOneTwo$ must be free (since $\MOneTwo$ is a component of a stratum) and fixed by the hyperelliptic involution (since $\MOneTwo$ is a quadratic double). This implies that the surfaces in $\MOneTwo$ have no marked points except when $\MOneTwo$ is $\cH(0)$ or $\cH(0,0)$. 
 
 It follows that the hyperelliptic involution is the unique marked point preserving involution on the generic surface in $\MOneTwo$. Letting $J_2$ denote the holonomy involution on $\ColTwoX$ we have that $\Col_{\ColTwo(\bfC_1)}(J_2)$ must be the hyperelliptic involution.

\begin{sublem}\label{SL:MIsNotRankOne}
$\cM$ has rank at least two. Moreover, if $\MOneTwo$ has rank one, then $\ColOneTwo(\bfC_1)$ and $\ColOneTwo(\bfC_2)$ are not generically parallel.  
\end{sublem}
\begin{proof}
We prove the second statement first. Assume $\MOneTwo$ has rank one. Since it is hyperelliptic, $\MOneTwo$ is either $\cH(0)$ or $\cH(0,0)$.

Suppose that $\ColOneTwo(\bfC_1)$ and $\ColOneTwo(\bfC_2)$ are generically parallel. Since $\cH(0)$ does not have pairs of generically parallel saddle connections, this implies that $\MOneTwo = \cH(0,0)$ and that each of $\ColOneTwo(\bfC_1)$ and $\ColOneTwo(\bfC_2)$ is a single saddle connection that joins a marked point to itself. However, $\ColOneTwo(\bfC_1)$ is fixed by the hyperelliptic involution on $\ColOneTwo(X, \omega)$ since $\ColTwo(\bfC_1)$ is fixed by the holonomy involution on $\ColTwo(X, \omega)$. This is a contradiction since in $\cH(0,0)$ a saddle connection joining a marked point to itself is not fixed by the hyperelliptic involution. 

We now use the second statement to prove the first statement. If $\cM$ is rank one, then so is $\MOneTwo$. Since $\bfC_1$ and $\bfC_2$ are not parallel, Lemma \ref{L:RankTest} gives a contradiction.  
\end{proof}

\begin{sublem}\label{SL:Periodic}
$\bfC_1$ consists of one or two simple cylinders. If the cylinders are adjacent, then the only singularity or marked point on their common boundary is a single periodic marked point. 

\end{sublem}
\begin{proof}
The largest number of homologous saddle connections on a surface in $\MOneTwo$ is two since in hyperelliptic components of strata of Abelian differentials two saddle connections are homologous if and only if they are exchanged by the hyperelliptic involution by Lemma \ref{L:HypParallelism}. Therefore, $\ColOneTwo(\bfC_i)$ is a collection of at most two saddle connections. 

$\ColTwo(\bfC_1)$ is a subequivalence class of generic cylinders in $\ColTwo(X, \omega)$, which is a surface in a quadratic double. Since $\ColOneTwo(\bfC_1)$ consists of at most two saddle connections, using Masur-Zorich (Theorem \ref{T:MZ}) and examining the bottom of Figure \ref{F:CylinderTypes}, it follows that $\ColTwo(\bfC_1)$ is either 
\begin{enumerate}
\item two simple nonadjacent cylinders or 
\item a simple cylinder or a complex cylinder that has possibly been divided into two cylinders by marking the preimage of $n \leq 2$ poles on its core curve. 
\end{enumerate} 
(Following our conventions, in the second case it is correct to say these $n$ points lie on the boundary of $\ColTwo(\bfC_1)$, but we emphasize they lie in between the two cylinders of $\ColTwo(\bfC_1)$ and on the interior of a cylinder in $\For(\ColTwoX)$.) 

If $\bfC_1$ consists of two simple nonadjacent cylinders there is nothing more to prove; so suppose that we are in the second case.

We now show that, when $\bfC_2$ is glued into $\ColTwo(\bfC_2)$ to obtain $(X,\omega)$ from $\ColTwoX$, the $n$ marked points  on the boundary of $\ColTwo(\bfC_1)$ remain marked points on $(X, \omega)$; so the $n$ marked points on $\ColTwoX$ arise from $n$ marked points on $(X,\omega)$.   To see this, notice that $\ColTwo(\bfC_2)$ does not contain a saddle connection that belongs to the boundary of $\bfC_1$ nor does it contain a saddle connection that intersects a cylinder in $\ColTwo(\bfC_1)$.

Let $P$ denote the $n$ marked points on the boundary of $\bfC_1$. If $n > 0$, then since the two cylinders in $\bfC_1$ have generically identical heights, it follows that the points in $P$ are periodic points. Let $\For'(X, \omega)$ be the translation surface $(X, \omega)$ once the points in $P$ have been forgotten. Let $\For'(\bfC_i)$ denote the images of the cylinders in $\bfC_i$ on this surface. 

We will now show that $\For'(\bfC_1)$ is not a complex cylinder. Since $\ColOneTwo(\bfC_1)$ consists of two saddle connections that are exchanged by the hyperelliptic involution we have that $\Col_{\ColTwo(\bfC_1)/J_2}\left( \ColTwo(\bfC_1)/J_2 \right)$ consists of a single saddle connection. As discussed in Remark \ref{R:TwoDegenerations}, since $\ColTwo(\bfC_1)/J_2$ is a generic complex envelope, this implies that $\ColOneTwoX$ is disconnected, which is a contradiction. 


It remains to show that $|P| \leq 1$ when $\For'(\bfC_1)$ is a simple cylinder. If not, then, letting $Z$ denote the zeros of $\omega$, $\ColOneTwo(P)$ contains a periodic point that is not contained in  $\ColOneTwo(Z)$ (see the cylinder labelled $4$ in Figure \ref{F:CylinderTypes}). Strata do not have such periodic points, so this is a contradiction.  
\end{proof}

We now verify the proposition up to determining the number of free marked points. \red Let $\cH$ be the connected component of the stratum containing $\cM$. \black

\begin{sublem}\label{SL:full-rank-gluing}
One of the following occurs:
\begin{enumerate}
    \item $\cM = \cH$ (if $\bfC_1$ is a single simple cylinder, then this case occurs).
    \item\label{I:HyperellipticMPs} $\For(\cM) = \For(\cH)$ is a hyperelliptic component and all marked points are free on surfaces in $\cM$ except for a collection of either one marked point fixed or two marked points exchanged by the hyperelliptic involution. 
    \item\label{I:Codim1} $\For(\cM)$ is a codimension one hyperelliptic locus in $\For(\cH)$ and all marked points in $\cM$ are free.  
\end{enumerate}
\end{sublem}
\begin{proof}
If $\bfC_1$ consists of a single simple cylinder, $\cM=\cH$, so suppose (by Sublemma \ref{SL:Periodic}), that $\bfC_1$ consists of two simple cylinders. 

Then $\Col_{\bfC_1}(\bfC_1)$ consists of one or two saddle connections and, if two, then two generically parallel ones. Since the only generically parallel saddle connections in a stratum of Abelian differentials are homologous ones it follows that the two cylinders in $\bfC_1$ are homologous on $(X, \omega)$. It follows that  $\cM$ is defined by a single equation relating the heights of the two cylinders in $\bfC_1$ and hence $\cM$ is codimension one in $\cH$.

Suppose first that $\For(\cM) = \For(\cH)$. Keeping in mind that $\cM$ has rank at least two (by Sublemma \ref{SL:MIsNotRankOne}) and that $\cM\subset \cH$ is codimension 1,  \cite[Theorem 1.5]{Apisa} gives that $\For(\cM)$ is a hyperelliptic component and the marked points are as described in (\ref{I:HyperellipticMPs}). 
 
 Suppose finally that $\For(\cM) \ne \For(\cH)$. Since $\cM$ is codimension 1 in $\cH$,  this implies that $\For(\cM)$ is codimension one in $\For(\cH)$ and that all the marked points in $\cM$ are free. By Mirzakhani-Wright \cite{MirWri2} (recalled in Theorem \ref{T:MirWriFullRank}), $\For(\cM)$ is a hyperelliptic locus of $\For(\cH)$. So $\cM$ is as described in \eqref{I:Codim1}.
\end{proof}

Let $Q$ denote the collection of free marked points on $(X, \omega)$. It remains to see that $Q$ contains no marked points when $\For(\cH)$ is non-hyperelliptic and at most one point otherwise. 

\begin{sublem}\label{SL:MarkedPointConclusion}
$\ColTwo(Q)$ contains no free marked points.
\end{sublem}
\begin{proof}
Suppose not in order to deduce a contradiction. A quadratic double can only contain free marked points if it is $\cH(0)$ or $\cH(0,0)$. If this occurs, then $\ColTwo(\bfC_1)$ and hence $\bfC_1$ consists of a single simple cylinder, implying that $\cM = \cH$ by Sublemma \ref{SL:full-rank-gluing}. 

Since $\cM$ has rank at least two, and $\MTwo$ coincides with $\cH(0,0)$ or $\cH(0)$ and has dimension exactly one less than $\cM$, it follows that $\cM = \cH(2)$. This is a contradiction since it implies that $Q$ is empty. 
\end{proof}

Since the marked points in $Q$ are free, whenever one collides with a zero it is a codimension one degeneration. Since $\MTwo$ has dimension exactly one less than $\cM$, at most one point in $Q$ can coincide with a zero on $\ColTwoX$. This shows that $Q$ contains at most one point (by Sublemma \ref{SL:MarkedPointConclusion}) and that when $Q$ contains one point, $\ColTwoX$ is formed by moving that point into a zero, which shows that $\For(\cM) = \For(\MTwo)$. When $\For(\cM)$ is nonhyperelliptic this cannot occur (by Lemma \ref{L:InvolutionImpliesHyp-background}) since $\MTwo$ is a quadratic double. So $Q$ is empty when $\For(\cM)$ is nonhyperelliptic. 
\end{proof}

\section{Gluing in a complex envelope }\label{S:ComplexEnvelope}

The main result of the section is the following. 

\begin{thm}\label{T:complex-gluing0}
Suppose that $(Z,\zeta)$ is contained in an invariant subvariety $\cM$ of quadratic differentials, and $\bfC$ is an $\cM$-generic cylinder on $(Z,\zeta)$ that is a complex envelope, and that the standard dilation of $\bfC$ remains in $\cM$. 

If $\cM_{\bfC}$ is a component of a stratum of Abelian or quadratic differentials, then one of the following occurs:
\begin{enumerate}
    \item $\cM$ is  a component of a stratum, \red or \black
    \item $\For(\cM)$ is a hyperelliptic component of a stratum and all marked points on $(X, q)$ are free except for a pair of points on the boundary of $\bfC$ that are exchanged by the hyperelliptic involution, or
    \item $\For(\cM)$ is a codimension one hyperelliptic locus in a non-hyperelliptic connected component of a stratum and all marked points are free. 
\end{enumerate}  
The latter two cases occurs if and only if \red $\cF\left(\Col_{\bfC}(Z,\zeta)\right)$ \black belongs to a hyperelliptic component of a stratum of quadratic differentials and $\Col_{\bfC}(\bfC)$ is a single saddle connection that is generically fixed by the hyperelliptic involution. 
\end{thm}

Recall that generic cylinders were defined in Definition \ref{D:GenericCyl}, complex envelopes were defined in Definition \ref{D:CylinderTypes}, the notation $\cF$ for forgetting marked points was defined in Definition \ref{D:F}, and hyperelliptic loci were defined in Definition \ref{D:HypLocus}.

\begin{rem}\label{R:TwoDegenerations}
There are two ways to collapse a complex envelope, illustrated in Figure \ref{F:TwoWaysToCollapseComplexEnvelope}.
\begin{figure}[h]\centering
\includegraphics[width=.8\linewidth]{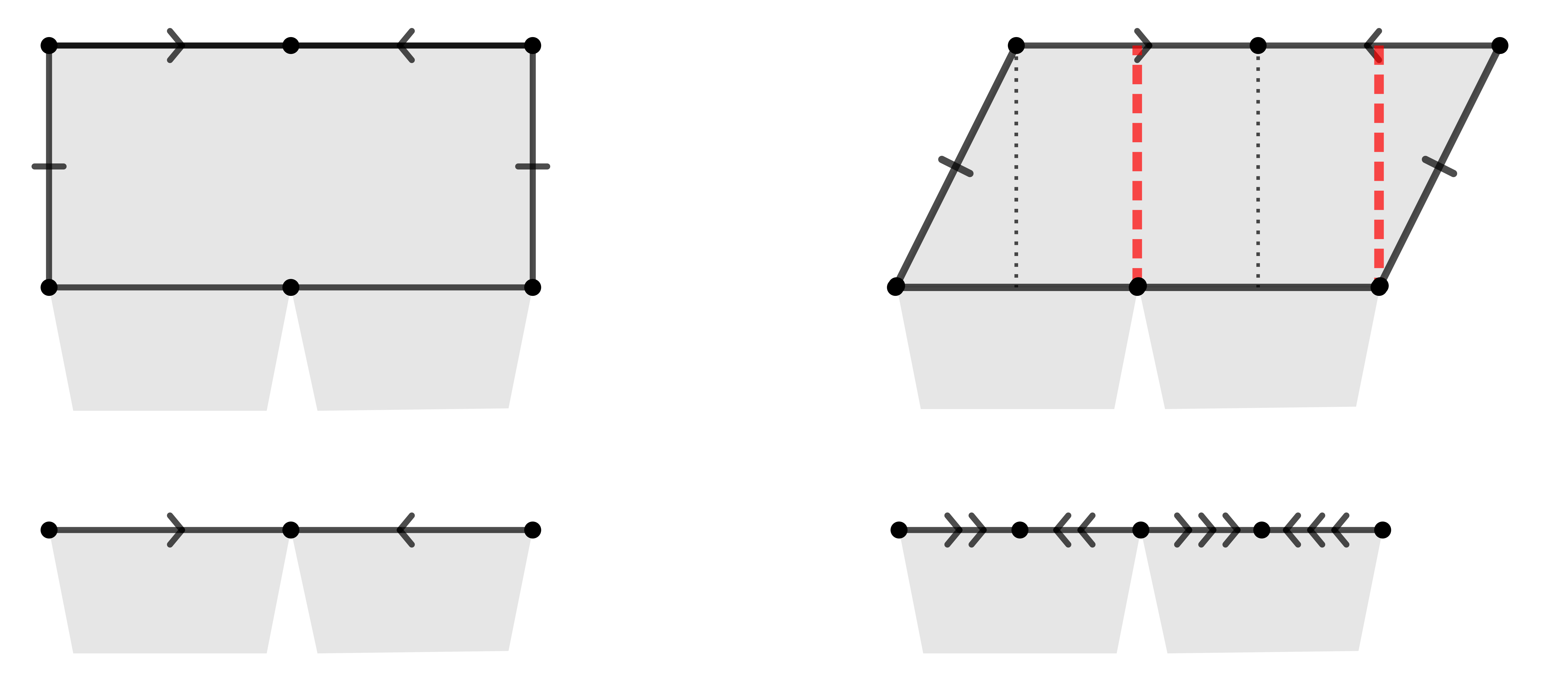}
\caption{The two ways of collapsing a complex envelope (left and right). The top illustrates the envelope before collapse, and the bottom illustrates after the collapse.}
\label{F:TwoWaysToCollapseComplexEnvelope}
\end{figure}
In one, $\Col_{\bfC}(\bfC)$ is a single saddle connection. In the other, $\Col_{\bfC}(\bfC)$ consists of two saddle connections that each connect a singularity or marked point to a pole. 

When $\cM$ is a component of a stratum of quadratic differentials, Theorem \ref{T:MZ} gives that the complement of $\bfC$ has trivial holonomy. In this case, the first degeneration has trivial holonomy. Of course the second degeneration has non-trivial holonomy because it has simple poles. Hence, when $\cM$ is a component of a stratum, the two degenerations of $\bfC$  can be distinguished by the corresponding degeneration on the holonomy double cover of $(X, q)$. The first degeneration \red disconnects \black the holonomy double cover, while the second \red leaves it connected. \black
\end{rem}

We will prove Theorem \ref{T:complex-gluing0} primarily by considering an operation that reverses one of the collapses of $\bfC$.  Given a quadratic differential $(Q, q)$ in a stratum $\cQ$ with a saddle connection $s$, cutting $s$ produces a surface with two parallel boundary saddle connections of equal length, $s_1$ and $s_2$. Now, consider a parallelogram with one edge parallel to and of the same length as $s$. Glue in the parallelogram as in Figure \ref{F:GlueInComplex} to obtain a new quadratic differential. We say that this quadratic differential is the result of gluing in a complex envelope to $s$, keeping in mind that the result depends on the choice of parallelogram. 

\begin{figure}[h]\centering
\includegraphics[width=.5\linewidth]{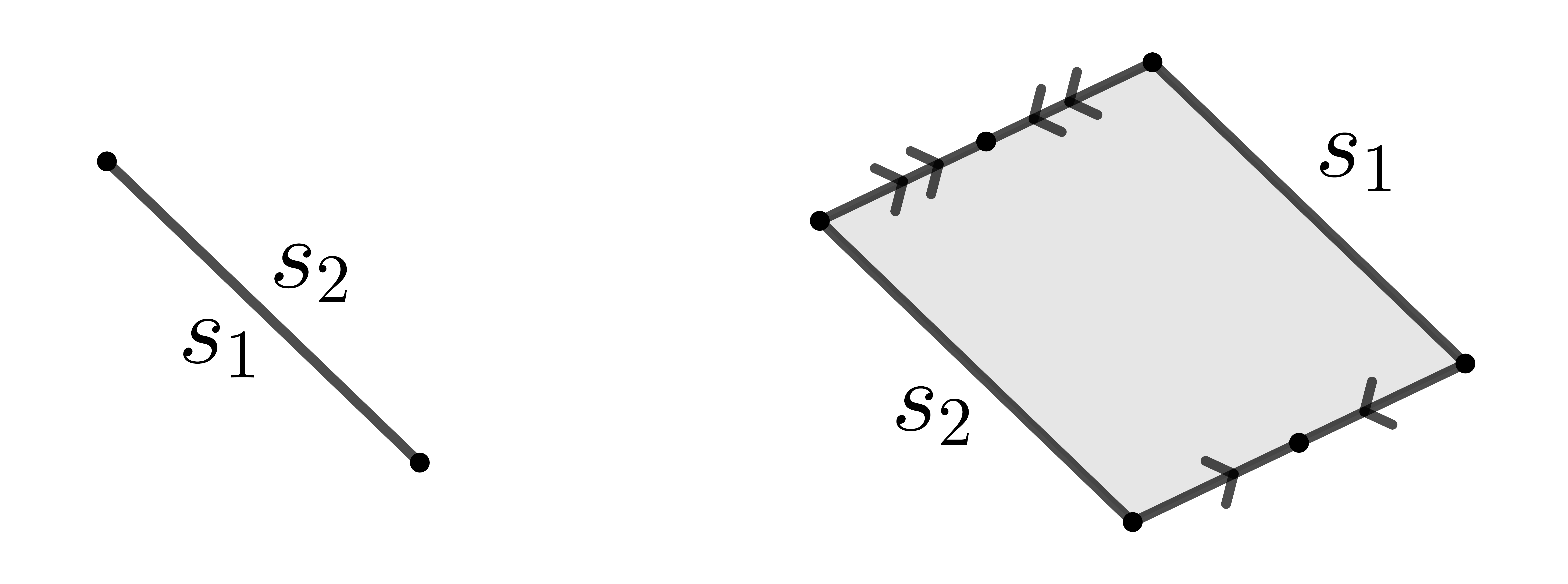}
\caption{}
\label{F:GlueInComplex}
\end{figure}

We now define an invariant subvariety $\cM_s$ as follows. If there is an invariant subvariety locally defined by the equation $s_1=s_2$, then we define $\cM_s$ to be this invariant subvariety. Otherwise, we define $\cM_s$ to be the connected component of the stratum that contains the result of gluing in a complex envelope to $s$.

Most of the section will involve proving the following, where, following Convention \ref{CV:NoTrivHol}, the quadratic differential is implicitly assumed to have non-trivial holonomy.

\begin{thm}\label{T:complex-gluing}
Let $s$ be a saddle connection on a quadratic differential $(Q,q)$, and let $\cQ$ be the connected component of a stratum containing $(Q,q)$. If $\For(\cQ)$ is hyperelliptic and $s$ is generically fixed by the hyperelliptic involution, then one of the following occurs:
\begin{enumerate}
    \item $\For(\cM_s)$ is a hyperelliptic component of a stratum and all marked points  are free except for a pair of points on the boundary of the glued-in complex envelope that are exchanged by the hyperelliptic involution, or
    \item $\For(\cM_s)$ is a codimension one hyperelliptic locus in a non-hyperelliptic connected component of a stratum and all marked points are free. 
\end{enumerate} 
\red Otherwise, $\cM_s$ is a connected component of a stratum. \black
%
\end{thm}

\begin{rem}\label{R:notHH}
The analogue of Theorem \ref{T:complex-gluing} for Abelian differentials is trivial.
Indeed, if $(Q,q)$ had trivial linear holonomy, then $s_1$ and $s_2$ would be hat-homologous, so the equation $s_1=s_2$ would be vacuous and $\cM_s$ would be a whole component of a stratum. 

However, in this section, because of our assumption that $(Q,q)$ has non-trivial holonomy,  Masur-Zorich (Theorem \ref{T:HatHomologous}) gives that the $s_1$ and $s_2$ are not hat-homologous and hence that the equation $s_1=s_2$ is not vacuous. 
\end{rem}

\begin{rem}\label{R:MarkedEndpoint}
Being generically fixed by the hyperelliptic involution means that on all surfaces in a neighborhood of $(Q, q)$, $s$ is fixed by the hyperelliptic involution. This condition is not satisfied if an endpoint of $s$ is a marked point, so the theorem gives that in this case $\cM_s$ is a connected component of a stratum.
\end{rem}

\begin{rem}\label{R:GluingStratum}
Let $(Q', q')$ be a quadratic differential formed by gluing in a complex envelope to the saddle connection $s$ on a quadratic differential $(Q, q)$ in a stratum $\cQ$. If $s$ connects a singularity of order $a$ to itself, then write $\cQ = \cQ(a, \kappa)$. If $s$ connects two singularities of order $a$ and $b$, then write $\cQ = \cQ(a, b, \kappa)$. From Figure \ref{F:GlueInComplex}, it is clear that the surface $(Q', q')$ belongs to the stratum $\cQ(a+2, -1^2, \kappa)$ in the first case and $\cQ(a+1, b+1,-1^2, \kappa)$ in the second. This shows that gluing in a complex envelope increases the dimension of the stratum by two. 
\end{rem}

\begin{proof}[Proof of Theorem \ref{T:complex-gluing0} given Theorem \ref{T:complex-gluing}.]
Let us begin with the first type of collapse, illustrated on the left of Figure \ref{F:TwoWaysToCollapseComplexEnvelope}. 
If $\Col_{\bfC}(Z,\zeta)$ has trivial holonomy, then $\cM$ is a stratum of quadratic differentials by Remark \ref{R:notHH}. 

\red  To complete our analysis of the first type of collapse, we now consider the case when $\Col_{\bfC}(Z,\zeta)$ has nontrivial holonomy. Remark \ref{R:GluingStratum} gives that the stratum containing $(Z,\zeta)$ has dimension two greater than the stratum containing $\Col_{\bfC}(Z,\zeta)$. Since $\cM$ has dimension one greater than $\cM_\bfC$, we get that $\cM$ is codimension 1 in a component of a stratum. The equation $s_1=s_2$ holds on $\cM$ (because $\bfC$ is $\cM$-generic) and this equation is non-trivial by Remark \ref{R:notHH}. Hence  $\cM$ is locally defined by $s_1=s_2$. Theorem \ref{T:complex-gluing} now gives the result.  \black

Now we conclude by considering the second type of collapse, illustrated on the right of Figure \ref{F:TwoWaysToCollapseComplexEnvelope}. 
As in the proofs of Lemma \ref{L:StrataBasics} and Corollary \ref{C:codim1}, we get that $\cM$ is a connected component of a stratum. 
%
%
\end{proof}

We now give the proof of the easy part of Theorem \ref{T:complex-gluing}.

\begin{lem}\label{L:complex-gluing-easy}
Suppose $\For(\cQ)$ is hyperelliptic and $s$ is \red generically \black fixed by the hyperelliptic involution. Then one of the following occurs:
\begin{enumerate}
    \item $\For(\cM_s)$ is a hyperelliptic component of a stratum and all marked points  are free except for a pair of points on the boundary of $\bfC$ that are exchanged by the hyperelliptic involution, or
    \item $\For(\cM_s)$ is a codimension one hyperelliptic locus in a non-hyperelliptic connected component of a stratum and all marked points are free. 
\end{enumerate} 
\end{lem}

\begin{proof}
The involution extends to the result of gluing in a complex envelope, by taking the involution on the glued-in parallelogram to be rotation by $\pi$. The quotient can be seen explicitly to be genus 0 since it is the result of gluing in a simple envelope to a genus 0 surface. Thus, $\For(\cM_s)$ is contained in a hyperelliptic locus. 

If either endpoint of $s$ is a pole then both are since $s$ is fixed by the hyperelliptic involution and no pole is. 

If both endpoints of $s$ are poles, $\For(\cM_s) = \For(\cQ)$ is a hyperelliptic component of a stratum and all marked points on $(X, q)$ are free except for a pair of points on the boundary of $\bfC$ that are exchanged by the hyperelliptic involution.


If neither endpoint of $s$ is  a pole, then the saddle connections $s_1$ and $s_2$ do not have marked points as endpoints. The stratum containing $\For(\cM_s)$ cannot be hyperelliptic since, if that were the case, the equation $s_1=s_2$ would be vacuous, contradicting Remark \ref{R:notHH}. \red Since $\cM_s$ is defined by $s_1 = s_2$, $\For(\cM_s)$ is as well. Therefore $\For(\cM_s)$ is a codimension one hyperelliptic locus in $\For(\cQ)$. \black 
\end{proof}

\subsection{Preliminaries on horizontally periodic surfaces}

The purpose of this section is to establish the following variant of \cite[Theorem 1.10]{Wcyl}. It applies to Abelian or quadratic differentials, but we use the notation of Abelian differentials.

\begin{prop}\label{P:ArithmeticCylinder}
Suppose that $(X, \omega)$ is an Abelian or quadratic differential that is contained in an  invariant subvariety $\cM$ with field of definition $\bk(\cM) = \mathbb{Q}$. Let $s$ be a horizontal saddle connection on $(X, \omega)$. Then there are horizontally periodic surfaces $(X', \omega')$ arbitrarily close to $(X, \omega)$ on which $s$ persists and remains horizontal and for which the following hold:
\begin{enumerate}
    \item\label{I:FirstBound} The number of horizontal cylinders is at least $\rk(\cM) + \mathrm{rel}(\cM) - 1$.
    \item\label{I:TwistDim} The core curves of the horizontal cylinders span a subspace of $T_{(X', \omega')}(\cM))^*$ of dimension at least $\rk(\cM)$. 
    \item\label{I:SecondBound} The number of horizontal cylinders is at least $\mathrm{\rk}(\cM)$. (This is an improvement on the first bound if  $\mathrm{rel}(\cM) = 0$.) 
    \item\label{I:GenericallyParallelBound}  If $s$ is generically parallel to one of the horizontal cylinders, then there are at least $\rk(\cM) + \mathrm{rel}(\cM)$ many horizontal cylinders.
\end{enumerate}
\end{prop}
Here $V^*$ denotes the dual of a vector space $V$. 

\begin{proof}
We will work entirely with Abelian differentials, but the proof will imply the same bounds for quadratic differentials by taking holonomy double covers. 

Let $(X,\omega)\in \cM$ be horizontally periodic. Define the \emph{twist space} of $(X,\omega)$, denoted $\mathrm{Twist}((X, \omega), \cM)$, to be the intersection of $T_{(X,\omega)}\cM$ with the span of the duals of the core curves of horizontal cylinders, using the duality between the homology of the punctured surface and relative cohomology, as in \cite[Section 4.1]{MirWri}. 

The dimension of the twist space is a lower bound for the number of cylinders on $(X,\omega)$. If $\cM$ consists of holonomy double covers of quadratic differentials, its dimension also gives a lower bound for the number of cylinders on the associated quadratic differential. 

The \emph{cylinder preserving space} of $(X,\omega)$, denoted $\mathrm{Pres}((X, \omega), \cM)$, is the subspace of $T_{(X,\omega)}\cM$ that evaluates to zero on the core curves of the horizontal cylinders. Note that 
$$ \mathrm{Twist}((X, \omega), \cM) \subset \mathrm{Pres}((X, \omega), \cM).$$
Consider also $\mathrm{Pres}_s((X, \omega), \cM)$, which we define to be the subspace of $\mathrm{Pres}((X, \omega), \cM)$ that evaluates to 0 on $s$, and note that 
$$\mathrm{Twist}((X, \omega), \cM) \subset \mathrm{Pres}_s((X, \omega), \cM).$$

Since $\bk(\cM) = \mathbb{Q}$, Lemma \ref{L:DenseSquares} implies that square-tiled surfaces are dense in $\cM$. So it is possible to deform $(X,\omega)$ to a nearby horizontally periodic surface $(X', \omega')$ on which $s$ persists and remains horizontal. Let $(X', \omega')$ be such a deformation with as many horizontal cylinders as possible. 

\begin{lem}\label{L:twist}
$\mathrm{Twist}((X', \omega'), \cM) = \mathrm{Pres}_s((X', \omega'), \cM).$
\end{lem}

\begin{proof}
Otherwise, as in \cite[Sublemma 8.7]{Wcyl}, it is possible to deform $(X', \omega')$ to some surface $(X'', \omega'')$ on which the horizontal cylinders of $(X', \omega')$ persist and stay horizontal but don't cover $(X'', \omega'')$, and where $s$ persists and is horizontal. Since square-tiled surfaces are dense, we may assume this $(X'', \omega'')$ is horizontally periodic. Since the horizontal cylinders from $(X', \omega')$ don't cover the surface, there must be more horizontal cylinders on $(X'', \omega'')$, which is a contradiction. 
\end{proof}

Let $d$ be the dimension of the span in $p(T_{(X',\omega')}\cM)^*$ (or equivalently in $(T_{(X',\omega')}\cM)^*$) of the core curves of the horizontal cylinders. As discussed in \cite[Proof of Theorem 1.10]{Wcyl}, we have $d \leq \rank \cM$.

Note that \cite[Lemma 8.8]{Wcyl} states
\begin{equation}\label{E:CPlower}
\dim \mathrm{Pres}((X', \omega'), \cM) = \dim \cM - d.
\end{equation}

We also have the following, \red which applies to any horizontally periodic surface. \black

\begin{lem}\label{L:PD}
$\dim p(\mathrm{Twist}((X', \omega'), \cM)) = d.$ 
\end{lem}

The $d=1$ case of Lemma \ref{L:PD} is essentially equivalent to Theorem \ref{T:CylECTwistSpace}, and it is possible to extract from the proof of Lemma \ref{L:PD} a shorter proof of Theorem \ref{T:CylECTwistSpace} than was given in \cite{MirWri}. 

\begin{proof}
\red First we will show $\dim p(\mathrm{Twist}((X', \omega'), \cM)) \leq d.$ We start with a purely linear algebraic statement, where we use use notation that suggests how it will be applied. 

\begin{sublem}\label{SL:SymplecticLinearAlgebra}
Let $H^1$ be a symplectic vector space, and let $(H^1)^*$ be its dual space. Consider the isomorphism $PD_H: H^1 \to (H^1)^*$ given by
$$  \eta\mapsto \langle \eta, \cdot \rangle, $$
where $\langle \cdot, \cdot \rangle$ is the symplectic pairing on $H^1$. 
Let $T\subset H^1$ be a symplectic subspace, so, similarly, we have an isomorphism $PD_T: T \to T^*$ defined by the same formula.

Let $\operatorname{res}_T: (H^1)^* \to T^*$ be the restriction to $T$ map. Let $\operatorname{proj}_T$ be the projection onto $T$ whose 
kernel is the symplectic perp of $T$. Then the following diagram commutes. 
\begin{figure}[h]\centering
\begin{tikzcd}
  H^1 \arrow[r, "\mathrm{PD}_H"] \arrow[d, "\mathrm{proj}_T"] & (H^1)^* \arrow[d, "\mathrm{res}_T"] \\
 T \arrow[r,"\mathrm{PD}_T"] & T^*
\end{tikzcd}
\end{figure}
\end{sublem}

The proof of this statement is left to the reader. We now apply this statement with $H^1=H^1(X', \bR)$ and $T= p(T_{(X', \omega')}(\cM)) \cap H^1(X', \bR)$. In this case  $(H^1)^*$ is $H_1(X,\bR)$ and $PD_H$ is the Poincare duality isomorphism.   

Let $R \subset H_1(X,\mathbb{R})$ be the subspace spanned by the core curves of horizontal cylinders. By definition, $d=\dim \mathrm{res}_T(R)$. 

Also by definition, $p(\mathrm{Twist}((X', \omega'), \cM))$ is contained in the (complexification of) $PD_H^{-1}(R)\cap T$. 

The commutative diagram above gives that 
$$ \mathrm{res}_T(R) \supset  PD_T \circ \operatorname{proj}_T (  PD_H^{-1}(R)\cap T).$$
Since $\operatorname{proj}_T$ is injective on $T$ and since $PD_T$ is an isomorphism, we get 

$$d = \dim R \geq \dim PD_H^{-1}(R)\cap T \geq \dim p(\mathrm{Twist}((X', \omega'), \cM)) ,$$
which is exactly the desired inequality.

 Next we will show that $\dim p(\mathrm{Twist}((X', \omega'), \cM)) \geq d.$ 
If two cylinders are $\cM$-equivalent then their core curves are collinear in $T^*$ (by  \cite[Lemma 4.7]{Wcyl}). Therefore, the span of the core curves of an $\cM$-equivalence class $\bfC$ in $T^*$ is precisely the line generated by $\mathrm{res}_T \circ PD_H\left( p(\sigma_{\bfC}) \right)$. Recall that $\sigma_{\bfC}$ is the standard shear in $\bfC$. Since $\sigma_{\bfC}$ belongs to $\mathrm{Twist}((X', \omega'), \cM)$ for every equivalence class of horizontal cylinders $\bfC$, we have shown that $\mathrm{PD}_T \circ \mathrm{proj}_T \left( p(\mathrm{Twist}((X', \omega'), \cM)) \right)$ contains $\mathrm{res}_T(R)$ (this uses Sublemma \ref{SL:SymplecticLinearAlgebra}). Since $\operatorname{proj}_T$ is injective on $T$ and since $PD_T$ is an isomorphism, we have the desired inequality.
\black
\end{proof}

A corollary of Lemma \ref{L:PD} is that
\begin{equation}\label{E:TwistUpper}
\dim \mathrm{Twist}((X', \omega'), \cM) \leq  d+ (\dim \cM - 2\rank \cM).
\end{equation}
Lemma \ref{L:PD} may be viewed as an improvement of \cite[Lemma 8.10]{Wcyl}, and inequality \eqref{E:TwistUpper} an improvement of \cite[Corollary 8.11]{Wcyl}.

If  $\mathrm{Pres}_s((X', \omega'), \cM)= \mathrm{Pres}((X', \omega'), \cM)$, then Lemma \ref{L:twist} and Equation \eqref{E:CPlower} give that 
$$\dim \mathrm{Twist}((X', \omega')) =  \dim \cM - d.$$
Since $d\leq \rank \cM$, this shows that there are at least $\dim \cM - \rank \cM$ cylinders. Conditional on $\mathrm{Pres}_s((X', \omega'), \cM)= \mathrm{Pres}((X', \omega'), \cM)$, this gives the first, third, and fourth claims; Inequality \eqref{E:TwistUpper} gives $d=\rank \cM$, which gives the second claim. 

So we may now assume that $\mathrm{Pres}_s((X', \omega')$ is codimension 1 in $\mathrm{Pres}((X', \omega'), \cM)$. Lemma \ref{L:twist} and Equation \eqref{E:CPlower} give that 
$$\dim \mathrm{Twist}((X', \omega') = \dim \cM - d -1.$$
Since $d\leq \rank \cM$, this gives the first claim.

Combining this with Inequality \eqref{E:TwistUpper}, we get $2d\geq 2\rank(\cM)-1$, which, since $d\leq \rank(\cM)$ is an integer, implies $d=\rank(\cM)$. This gives the second claim. Since there must be $d$ many horizontal cylinders, this also gives the third claim.  

For the fourth claim, note that if $s$ is $\cM$ is generically parallel to one of the horizontal cylinders, then $s$ is automatically zero on all of $\mathrm{Pres}((X', \omega'), \cM)$, so we are in the case of $\mathrm{Pres}_s((X', \omega'), \cM)= \mathrm{Pres}((X', \omega'), \cM)$ where we obtained the improved bound. 
\end{proof}

\subsection{Abundance of diamonds}
We now prove the following, which will power an inductive approach to the non-hyperelliptic case of Theorem \ref{T:complex-gluing}. Recall that multiplicity and boundary components of cylinders are defined in Definition \ref{D:CylAndBoundary}, and $\{s\}$-paths are defined in Definition \ref{D:Spath}.

\begin{prop}\label{P:Avoidance}
Let $(X, q)$ belong to a stratum $\cQ$ of quadratic differentials without marked points such that 
 $\cQ$ has rank at least two and is not $\cQ(1^2, -1^2)$ or $\cQ(2,1^2)$.
Then for any saddle connection $s$ on $(X, q)$ there is an $\{s\}$-path to another surface $(X', q')$ in $\cQ$ where there are two generic cylinders $C_1$ and $C_2$ such that the following statements about $C_1$ and $C_2$ hold:
\begin{enumerate}
    \item They do not intersect and do not have any boundary saddle connections in common.
    \item They are one of the following: a simple cylinder, a simple envelope, a half-simple cylinder.
    \item They do not intersect $s$.
\end{enumerate}
\end{prop}

\begin{proof}
Assume without loss of generality that $s$ is horizontal. By Proposition  \ref{P:ArithmeticCylinder}, after passing to nearby surface, we may suppose without loss of generality that $(X, q)$ is horizontally periodic with at least $$\max(\rk(\cQ) + \mathrm{rel}(\cQ) - 1, \rk(\cQ))$$ horizontal cylinders. Let $\cC$ denote the collection of horizontal cylinders.  Again without loss of generality, after passing to an arbitrarily close surface, we may suppose that every cylinder in $\cC$ is generic (outside of certain low dimensional strata, this will cause $(X, q)$ to no longer be horizontally periodic). 

\noindent \textbf{Case 1: $\cC$ contains a complex envelope $C$.}

Without loss of generality, $C$ is horizontal and contains a vertical saddle connection that joins a pole to a zero. By Masur-Zorich (Theorem \ref{T:MZ}), $\Col_{C}(X, q)$ has trivial holonomy and hence is contained in a stratum $\cH$ of Abelian differentials. 

If $\cH = \cH(0^n)$ then either the saddle connection in $\Col_{C}(C)$ connects a zero to itself, and hence $\cQ = \cQ(2, -1^2, 0^{n-1})$,  or it connects two distinct zeros, and hence $\cQ = \cQ(1^2, -1^2, 0^{n-2})$. Since $(X, q)$ does not have marked points and does not belong to a rank one stratum or to $\cQ(1^2, -1^2)$, we have a contradiction. 

Therefore, $\cH$ has rank at least two. So Proposition  \ref{P:ArithmeticCylinder} allows us to obtain a nearby surface with two simple cylinders that are disjoint from each other and from \red $\Col_C(s)$. \black These two cylinders remain cylinders at surfaces in $\cQ$ near $\Col_{C}(X, q)$. 

(Since the core curves of these cylinders are not nullhomologous on $\Col_{C}(X, q)$ the same is true on nearby surfaces in $\cQ$. By Masur-Zorich (Theorem \ref{T:MZ}), this implies that the two cylinders are generically simple as desired.)

Since the surface does not have marked points, two simple cylinders cannot share boundary saddle connections.

\noindent \textbf{Case 2: $\cC$ contains a complex cylinder $C$.}

By Masur-Zorich (Theorem \ref{T:MZ}), the surface $(X, q) - C$ consists of two disjoint translation surfaces with boundary. So all cylinders in $\cC$ other than $C$ are simple. Since there are no marked points, no two of the simple cylinders can share boundary saddle connections. If $\cM$ has rank at least 3, or has rank 2 and rel at least 2, $\cC$ has at least three cylinders, and we are done. So assume $\cM$ is rank 2 and rel 0 or 1. 

Denote the components of $(X, q) - C$ by $\Sigma_1$ and $\Sigma_2$, and suppose without loss of generality that $\Sigma_1$ does not contain $s$. Let $\Sigma_i^{glue}$ denote $\Sigma_i$ with its boundary saddle connections identified. If $\Sigma_1^{glue}$ belongs to a rank two stratum, then by Proposition  \ref{P:ArithmeticCylinder}, $\Sigma_1$ contains two disjoint simple cylinders, giving the result.

Therefore, using this fact and the assumption that $(X, q)$ has no marked points, $\Sigma_1^{glue}$ belongs to either $\cH(0)$ or $\cH(0,0)$.

\noindent \emph{Case 2a: $\Sigma_1^{glue}$ belongs to $\cH(0)$.}

In this case, $\Sigma_1$ is covered by a simple cylinder $C_1$ and hence $C_1$ and $C$ are generically parallel to each other (they are glued to each other as in the surface on the right of Figure \ref{F:3hat}). \red In particular, $\cQ$ has rel at least $1$ (this can be seen by applying either Theorem \ref{T:CylECTwistSpace} or Lemma \ref{L:Q-rank}). \black Since $\cQ$ has rank at least two, it follows from Proposition  \ref{P:ArithmeticCylinder} (\ref{I:TwistDim}) that $\cC$ has three cylinders.  

\noindent \emph{Case 2b: $\Sigma_1^{glue}$ belongs to $\cH(0,0)$.}

Since $(X, q)$ has no marked points, $C$ must be glued to $\Sigma_1$ along a saddle connection that joins the two distinct marked points of $\Sigma_1^{glue}$.


It follows that the $\Sigma_1$ side of the complex cylinder has two zeros of order $1$. We have that $\cQ = \cQ(1,1,\kappa)$ where $\kappa$ is a collection of positive integers (having a complex cylinder precludes the possibility of having poles and we have assumed that surfaces in $\cQ$ have no marked points). Since $\cQ$ has rank two it follows from Lemma \ref{L:Q-rank} that
\[ 2 = g + \frac{m_{odd}}{2} - 1, \]
where $g$ is the genus of $X$ and $m_{odd}$ is the number of odd order zeros of $q$. Because there is a complex cylinder, $g \geq 2$. Since we know $m_{odd} \geq 2$, actually $g = 2$ and $\kappa$ consists entirely of even integers. Since the sum of the orders of the zeros is $4g-4$, this implies  that $\kappa = (2)$, which contradicts our assumption that $\cQ \ne \cQ(2,1,1)$. 

\noindent \textbf{Case 3: $\cC$ does not contain complex cylinders or complex envelopes.}

Proposition  \ref{P:ArithmeticCylinder} (\ref{I:TwistDim}) gives that $\cC$ contains a pair of non-parallel cylinders. This pair proves the result.
\end{proof}

\subsection{Proof of Theorem \ref{T:complex-gluing} when $\cF(\cQ)$ is hyperelliptic}

The main result of the subsection will be the proof of Theorem \ref{T:complex-gluing} in the case that $\For(\cQ)$ is a hyperelliptic component of a stratum of quadratic differentials.

\begin{prop}\label{P:complex-gluing1}
Theorem \ref{T:complex-gluing} holds when $\For(\cQ)$ is a hyperelliptic component. 
\end{prop}

We begin with some special cases. 

\begin{lem}\label{L:PillowcaseBase}
Let $(Q, q)$ be a surface in $\cQ = \cQ(-1^4, 0)$ or $\cQ = \cQ(-1^4, 0^2)$ and let $s$ be a saddle connection. Suppose that \red $\cM_s$ \black  is proper. Then $s$ must be fixed by a hyperelliptic involution.
\end{lem}
\begin{rem}\label{R:PolePoleFixed}
Notice that $\cQ(-1^4)$ is the unique hyperelliptic stratum of quadratic differentials for which the hyperelliptic involution is not unique (see Lemma \ref{L:HypSymmetries}). Moreover, any saddle connection joining two poles is fixed by some hyperelliptic involution. 
\end{rem}
\begin{proof}
\red Suppose, in order to find a contradiction, that $s$ is a saddle connection that is not generically fixed by a hyperelliptic involution and that $\cM_s$ is proper. By Remark \ref{R:PolePoleFixed}, $s$ does not join two poles. Suppose without loss of generality that $s$ is horizontal. \black


Suppose first that $s$ is a saddle connection on a surface in $\cQ(-1^4, 0)$ joining the marked point to itself. Gluing in a complex envelope to $s$ produces a surface in $\cQ(2, -1^6)$. Suppose that $\cM_s$ is proper, which means that, letting $s_1$ and $s_2$ be the two saddle connections in the boundary of the complex envelope that do not end at a pole, $\cM_s$ is defined by the equation $s_1 = s_2$. This implies that all three horizontal cylinders are generically parallel in $\cM_s$; see Figure \ref{F3} (left). Since each of these cylinders forms its own subequivalence class,
it follows by Theorem \ref{T:CylECTwistSpace} that $\cM_s$ has rel at least $2$ since the twist space corresponding to the collection of horizontal cylinders has dimension at least $3$. This contradicts the fact that $\cM_s$ is contained in $\cQ(2, -1^6)$ which only has rel $1$ by Lemma \ref{L:Q-rank}. \red Note that this argument also handles the case where $s$ joins a marked point to itself on $\cQ(-1^4, 0^2)$. \black

Suppose next that $s$ is a saddle connection on a surface in $\cQ(-1^4, 0)$ joining a marked point to a pole. Gluing in a complex envelope to such a slit produces a surface in $\cQ(1, 0, -1^5)$; see Figure \ref{F3} (right). This stratum has rank $2$ and rel $1$ by Lemma \ref{L:Q-rank}. Since $\cM_s$ has codimension one, it is necessarily rank two rel zero, but this implies that if $(X, \omega)$ is a surface in $\cM_s$ with dense $\GL$-orbit, then the marked point on $(X, \omega)$ is a periodic point and $\For(X, \omega)$ has dense orbit in $\cQ(1, -1^5)$. This contradicts Theorem \ref{T:StrataMarkedPoints}, which states that non-hyperelliptic components of strata of rank at least 2 do not admit periodic points. Note that $\cQ(1, -1^5)$ is not hyperelliptic by the classification of hyperelliptic connected components \cite{LanneauHyp}, recalled in Section \ref{S:Q-hyp}. \red Note that this argument also handles the case where $s$ joins a marked point to a pole on $\cQ(-1^4, 0^2)$. \black

\begin{figure}[h]\centering
\includegraphics[width=.7\linewidth]{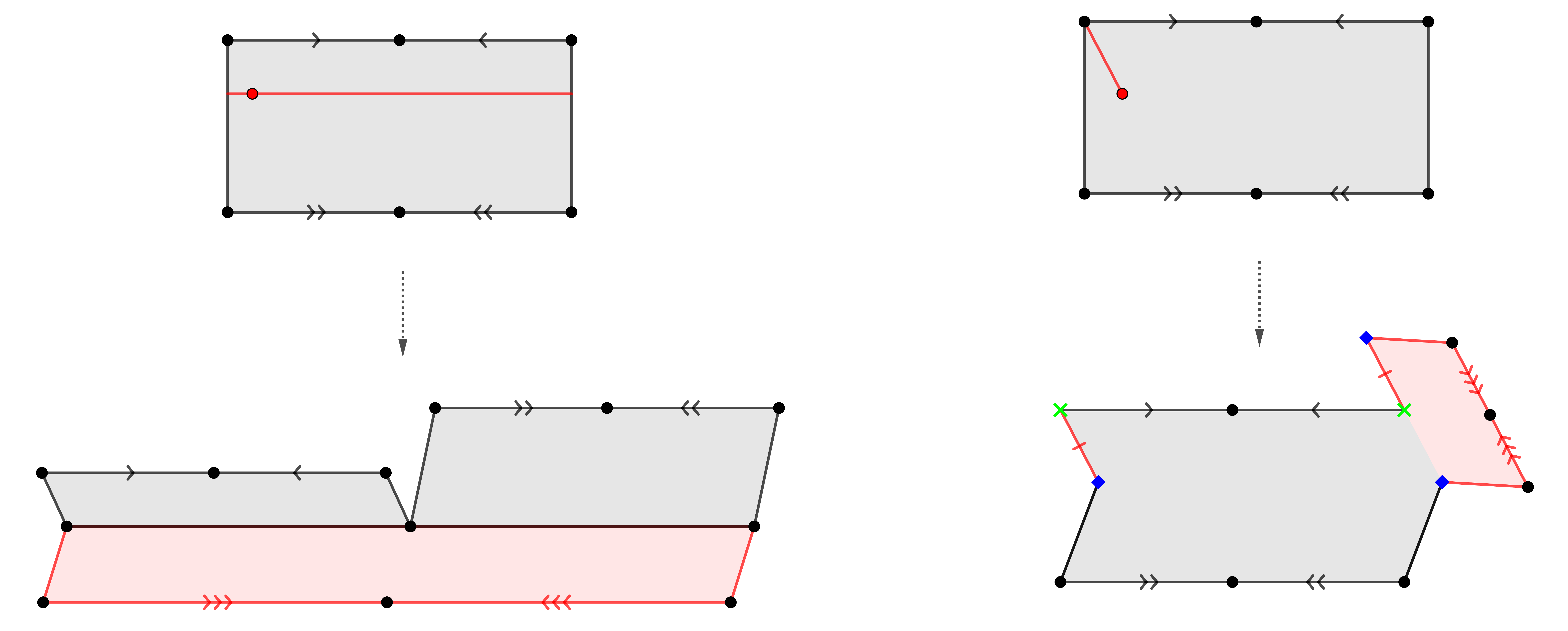}
\caption{Gluing in a complex envelope to surfaces in $\cQ = \cQ(-1^4, 0)$ and $\cQ = \cQ(-1^4, 0^2)$ }
\label{F3}
\end{figure}

Suppose finally that $s$ joins two \red distinct \black marked points together in $\cQ(-1^4, 0^2)$, begin by perturbing the surface so that $s$ lies in a periodic direction and not on a separatrix. Moving one endpoint in this periodic direction causes one endpoint to eventually collide with the other. Let $\cQ$ denote the component of the stratum that contains $\cM_s$. The previously described deformation is a degeneration of $\cQ$ to a codimension one boundary stratum $\cQ'$ and hence, by Lemma \ref{L:ExtendingPaths}, a degeneration of $\cM_s$ to boundary component $\cM_s'$. By Theorem \ref{T:BoundaryTangent}, $\cM_s'$ has dimension one less than $\cM_s$. It follows that $\cM_s'$ is a proper invariant subvariety that contains a surface formed by gluing in a complex envelope to a saddle connection $s$ on a surface in $\cQ(-1^4, 0)$ that joins a marked point to itself. This contradicts the results of the preceding paragraphs. 
\end{proof}

\begin{lem}\label{L:NonPillowcaseBase}
Let $(Q, q)$ be a surface in $\cQ = \cQ(2, -1^2)$ and let $s$ be a saddle connection. Suppose that $\cM_s$ is proper. Then $s$ must be fixed by the hyperelliptic involution.\end{lem}
\begin{proof}
Suppose not to a contradiction and suppose without loss of generality that $s$ is horizontal. There are three kinds of saddle connections on a surface in $\cQ(2, -1^2)$: ones that join two poles, ones that join a zero and a pole, and ones that join a zero to itself. 

Since $\cQ(2, -1^2)$ has one component, which is hyperelliptic, there is a half-translation double cover $\phi: (Q, q) \rightarrow (S^2, q')$ from $(Q, q)$ to an element $(S^2, q')$ of $\cQ(-1^4,0)$; this map is branched over three poles and one marked point (for a reference for these facts on hyperelliptic components see Section \ref{S:Q-hyp}).

There is at most one saddle connection in any given direction joining two given poles. Hence, any such saddle connection must be fixed by the hyperelliptic involution. Therefore, we may assume $s$ does not join two poles. 

Suppose now that $s$ joins a zero to itself and that $s$ is not fixed by the hyperelliptic involution. Since $s$ is not fixed by the hyperelliptic involution, $\phi(s)$ is a saddle connection joining the marked point to itself on $(S^2, q')$. Phrased differently, on $\For(S^2, q')$, $\phi(s)$ is a core curve of a cylinder. Since $\cQ(-1^4, 0)$ is rank one by Lemma \ref{L:Q-rank}, it follows from \cite[Theorem 1.5]{Wcyl} that $(S^2, q')$ is horizontally periodic. (This claim is also immediate from the fact that the holonomy double cover of a surface in $\cQ(-1^4, 0)$ is a flat torus with marked points).

It follows that $(Q, q)$ is horizontally periodic.  As in Lemma \ref{L:PillowcaseBase}, if a complex envelope is glued into $s$ with the stipulation that its boundary saddle connections are all generically parallel to each other, then all horizontal saddle connections on the resulting surface must be generically parallel to each other. This implies that $\cM_s$ is rank one. The dimension of $\cM_s$ is 4 (one greater than the dimension of $\cQ(2, -1^2)$), so $\cM_s$ must have rel 2. But, by Remark \ref{R:GluingStratum}, $\cM_s\subset \cQ(4, -1^4)$, which has rel 1, giving a contradiction. 

Suppose finally that $s$ joins a zero to a pole. By Remark \ref{R:GluingStratum}, gluing in a complex envelope to such a slit produces a surface in $\cQ(3, 0, -1^3)$, which is rank two rel one by Lemma \ref{L:Q-rank}. Since $\cM_s$ has codimension one, it is necessarily rank two rel zero, but this implies that the marked point is not free, which, as in the proof of Lemma \ref{L:PillowcaseBase}, is a contradiction to the main theorem of Apisa-Wright \cite{ApisaWright} (Theorem \ref{T:StrataMarkedPoints}). This uses the fact that $\cQ(3,-1^3)$ is not hyperelliptic (see for example Figure \ref{F:HyperellipticCheatSheet} and the surrounding discussion in Section \ref{S:Q-hyp}). For an illustration of the construction in this paragraph and the previous one see Figure \ref{F0}. 
\begin{figure}[h]\centering
\includegraphics[width=.7\linewidth]{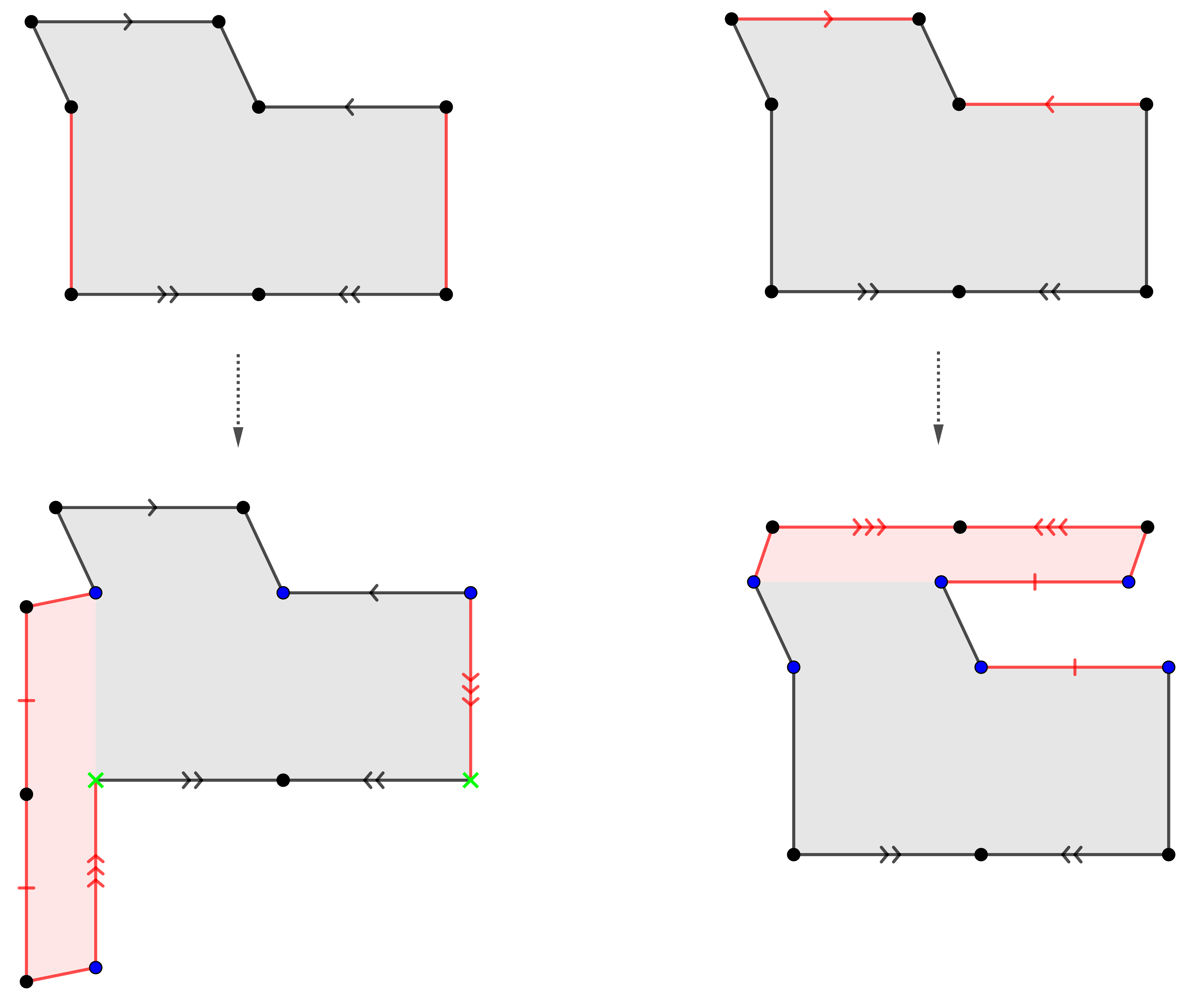}
\caption{Gluing in a complex envelope to surfaces in $\cQ = \cQ(2, -1^2)$. The saddle connection $s$ is labelled in red as is the complex envelope that is glued in.}
\label{F0}
\end{figure}
\end{proof}

\begin{proof}[Proof of Proposition \ref{P:complex-gluing1}]
\red By Lemma \ref{L:complex-gluing-easy}, if $s$ is generically fixed by the hyperelliptic involution then Proposition \ref{P:complex-gluing1} holds. \black  Therefore suppose \red in order to find \black a contradiction that $s$ is not fixed by the hyperelliptic involution but that $\cM_s$ is proper. Suppose moreover that  the component $\cQ$ of the stratum of quadratic differentials containing $(Q, q)$ is the smallest dimensional stratum for which this occurs. 

We begin by observing that this implies that if $\cQ$ has any marked points, they are at endpoints of $s$. Otherwise, \cite{MirWri} can be used to show that forgetting the marked points that aren't at endpoints of $s$ gives rise to a smaller example. In particular, we see that there are at most $2$ marked points.  

We also observe that, by Lemma \ref{L:PillowcaseBase}, $\For(\cQ) \ne \cQ(-1^4)$. By Lemma \ref{L:HypSymmetries}, this implies that on $\For(Q, q)$ there is a unique hyperelliptic involution, which we will denote $J$.

\begin{sublem}\label{SL:DisjointSimple}
Let $C$ be a simple cylinder on $(Q,q)$ \red whose boundary does not contain marked points. \black Then $C$ intersects $s$ or the boundary of $C$ consists precisely of $s$ and $J(s)$. 
\end{sublem}
\begin{proof}
Suppose not and let $\bfC = \{C\}$. We will show that collapsing $C$ gives rise to a smaller example, contradicting our minimality assumption. 

By Lemma \ref{L:StrataBasics}, $\cQ_{\bfC}$ is a stratum of connected quadratic differentials of dimension one less than $\cQ$ that contains $\Col_{\bfC}(Q,q)$. Since $\cQ_{\bfC}$ belongs to the boundary of $\cQ$, where $\For(\cQ)$ is a hyperelliptic component, it follows that $\For(\cQ_{\bfC})$ is a hyperelliptic component as well. 

By assumption, $s$ remains a saddle connection on $\Col_{\bfC}(Q,q)$, since it does not intersect $C$. \red Since the boundary of $C$ does not contain marked points, $J(C) = C$ and so $\Col_{\bfC}(J)$ is well-defined. The construction of $\Col_{\bfC}(J)$ shows, since $s \ne J(s)$ and since $s$ and $J(s)$ do not intersect $\overline{\bfC}$ (except possibly at endpoints), that $$\Col_{\bfC}(J)(\Col_{\bfC}(s)) \ne \Col_{\bfC}(s).$$ Note that if $J(C) \ne C$ (which can only happen if, contrary to our hypothesis, the boundary of $C$ contained a marked points) then there are examples where $\Col_{\bfC}(s)$ \emph{is} fixed by the hyperelliptic involution); see Figure \ref{F:sJs} for an illustration. 

\begin{figure}[h]\centering
\includegraphics[width=.25\linewidth]{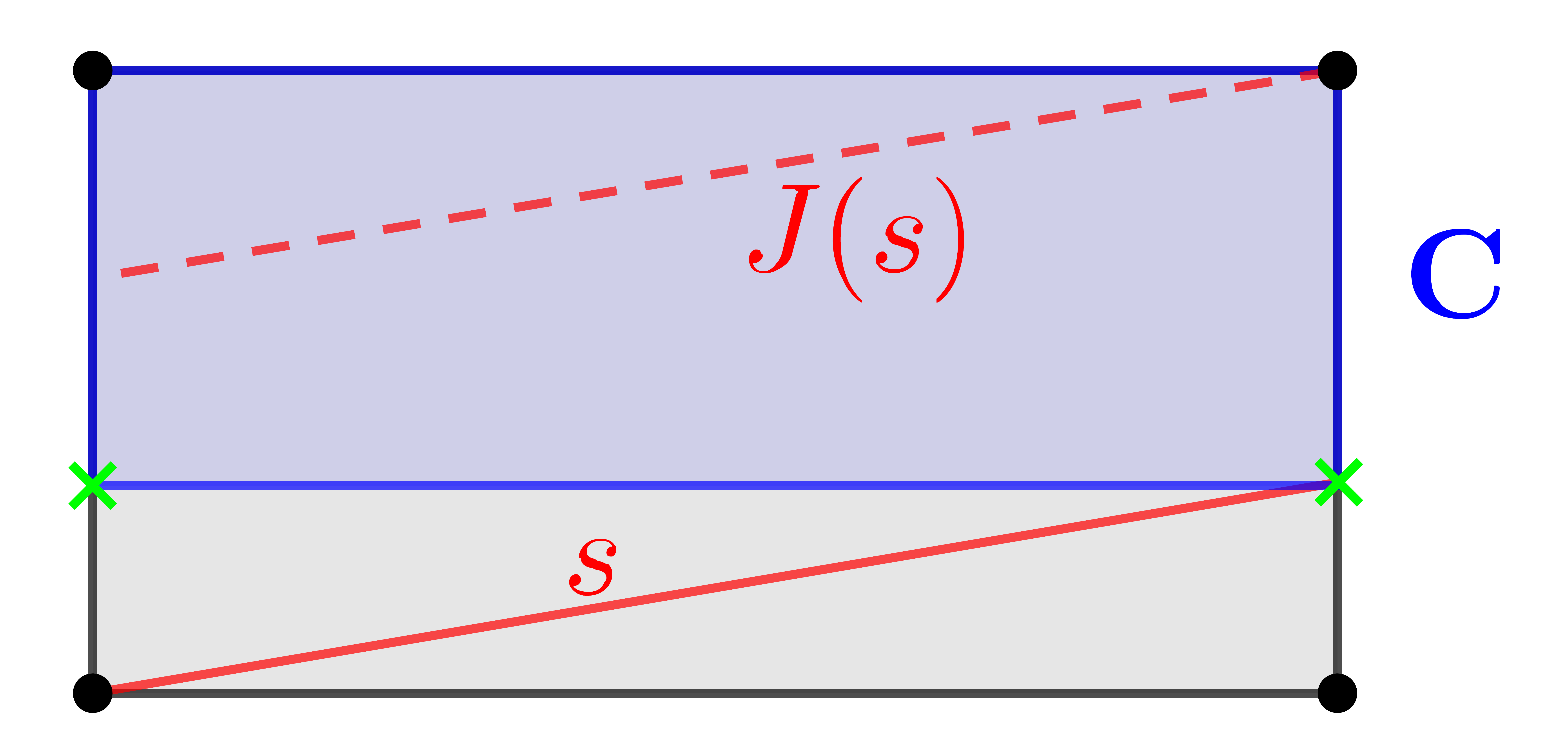}
\caption{In the case when $s$ is contained in a cylinder of $\For(Q,q)$ fixed by $J$ and $s$ has a marked point as an endpoint, if $\bfC$ is the part of the cylinder not intersecting $s$, then collapsing $\bfC$ can cause $s$ to become fixed by the hyperelliptic involution. Note that in this example, $J(s)$ is not a saddle connection since it has an unmarked endpoint.}
\label{F:sJs}
\end{figure}

It is clear that $\Col_{\bfC}(J)$ is a hyperelliptic involution. A priori, if $\For(\Col_{\bfC}(Q, q)) \in \cQ(-1^4)$, there could be a different hyperelliptic involution that fixed $\Col_{\bfC}(s)$, but we now breifly explain why this does not occur. Indeed, the only saddle connections on a surface in $\cQ(-1^4, 0^n)$ that are generically fixed by a hyperelliptic involution are those that go from a pole to another pole. Since $\Col_{\bfC}(J)$ fixes $\Col_{\bfC}(\bfC)$, this saddle connection is a pole-pole saddle connection. So, if a different hyperelliptic involution fixes the saddle connection $\Col_\bfC(s)$, then $\Col_\bfC(s)$ is also a pole-pole saddle connection. See Figure \ref{F:NotUniqueAnyMore}. In this case $(Q,q)$ belongs to $\cQ(2,-1^2)$ and the situation is ruled out by Lemma \ref{L:NonPillowcaseBase}. 


\begin{figure}[h]\centering
\includegraphics[width=.75\linewidth]{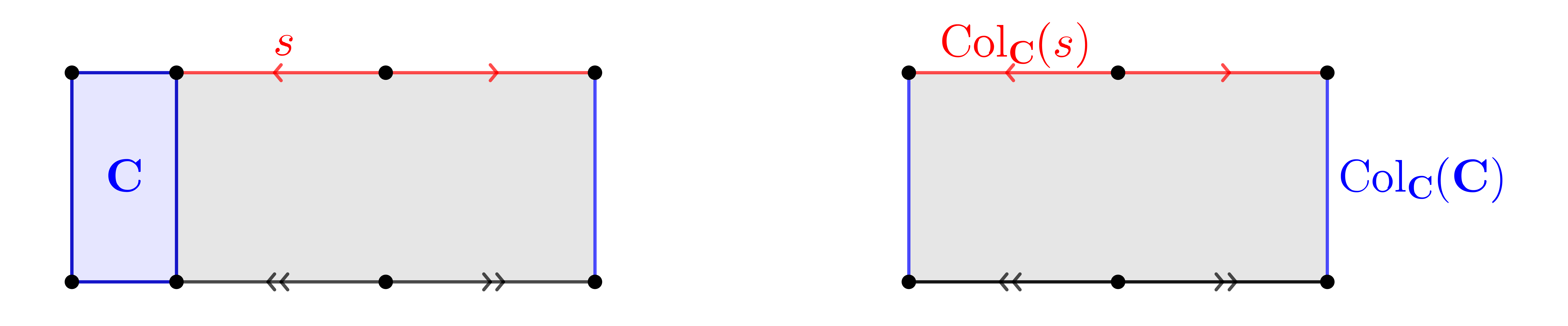}
\caption{A surface in $\cQ(2,-1^2)$ (left) degenerates to a surface in $\cQ(-1^4)$ (right). On the left surface, the hyperelliptic involution is unique and does not fix $s$. On the right surface, the hyperelliptic involution is not unique. }
\label{F:NotUniqueAnyMore}
\end{figure}


We now know that $\Col_\bfC(s)$ is not fixed by any hyperelliptic involution. 
Let $\mathbf{C}'$ denote the cylinder corresponding to $\bfC$ on the surface in $\cM_s$ formed by gluing in a complex envelope to $s$. The dimension of $\cM_s$ is one more than that of $\cQ$. Similarly, the dimension of $(\cM_s)_\mathbf{C'}$ is one greater than that of $\cQ_\mathbf{C}$. However, by Remark \ref{R:GluingStratum}, gluing in a complex envelope increases the dimension of a stratum of quadratic differentials by two; which shows that $(\cM_s)_\mathbf{C'}$ is codimension one in the stratum containing it. Since $(\cM_s)_\mathbf{C'} = (\cM_\bfC)_{\Col_{\bfC}(s)}$, we have a contradiction to the minimality assumption. \black
\end{proof}

\begin{sublem}\label{SL:PillowcaseMarkedPoints}
$\cQ$ has no marked points.
\end{sublem}

\begin{proof}
As observed above, any marked point is an endpoint of $s$. We will proceed by contradiction. 

\noindent\textbf{Case 1: Both endpoints of $s$ are distinct marked points.} In this case, we can move one endpoint to a zero to find a smaller counterexample, contradicting our assumption that $\cQ$ is the smallest counterexample. 

\noindent\textbf{Case 2: Exactly one endpoint of $s$ is a marked point.} Call the endpoint that is a marked point $p$, and the other end point $z$, so $z$ is a zero or pole. Perhaps perturbing the surface we can send $p$ into a zero or pole $z'$ not equal to $z$ or $J(z)$.
%
%
%
Such a zero always exists for hyperelliptic components (see Figure \ref{F:HyperellipticCheatSheet} in Section \ref{S:Q-hyp}).  Again we have reduced to a smaller counterexample where $s$ remains a saddle connection that is not fixed by the hyperelliptic involution.
(Note that this proof implicitly but crucially uses the uniqueness of the hyperelliptic involution, since we need to send $p$ into a zero or pole that does not equal to $z$ or its image under \emph{any} hyperelliptic involution.)

\noindent\textbf{Case 3: Both endpoints of $s$ are the same marked point.} In this case, $s$ is contained in a cylinder $C$. If $C$ has a boundary that consists of one multiplicity one saddle connection, we can move the marked point so $s$ becomes this saddle connection. Since $C$ is fixed by the hyperelliptic involution, this saddle connection cannot be fixed by the involution, so again we obtain the same contradiction. 

Otherwise, $C$ is a complex cylinder or complex envelope. Its complement has trivial holonomy by Masur-Zorich, so there is a simple cylinder in its complement, \red which contradicts Sublemma \ref{SL:DisjointSimple}. \black
\end{proof}

We now conclude the proof of Proposition \ref{P:complex-gluing1}.  

By Lemma \ref{L:DenseSquares}, since $\bk(\cQ) = \mathbb{Q}$, horizontally periodic surfaces are dense in $\cQ$. Let $U$ be a small neighborhood of $(Q, q)$ and let $(Q_1, q_1)\in U$ be a half-translation surface such that the following holds:
\begin{enumerate}
    \item $(Q_1, q_1)$ is horizontally periodic.
    \item $s$ remains horizontal on $(Q_1, q_1)$ and remains a saddle connection at all points in $U$. 
    \item\label{I:MaxCyl} There are maximally many horizontal cylinders on $(Q_1, q_1)$ subject to the condition that $s$ remain horizontal and that $(Q_1, q_1)$ belongs to $U$.
\end{enumerate}
Let $\cC$ denote the collection of horizontal cylinders on $(Q_1, q_1)$. Let $(Q_2, q_2)$ be a surface in $U$ where all the cylinders in $\cC$ persist and are generic. Recall that, by Lemma \ref{L:hyp-cyl}, a generic cylinder in a hyperelliptic component different from $\cQ(-1^4)$ is one of the following: a simple cylinder, complex cylinder, or complex envelope. 

By Sublemma \ref{SL:DisjointSimple}, if there is a simple cylinder disjoint from $s$ and $J(s)$, then its boundary consists of $s$ and $J(s)$. It follows that the only generic cylinders disjoint from $s$ and $J(s)$ (including their boundaries) are complex cylinders and complex envelopes. We will derive a contradiction.   

\noindent \emph{Case 1: One of the cylinders in $\cC$ is a complex cylinder on $(Q_2, q_2)$.}

By Masur-Zorich (Theorem \ref{T:MZ}), cutting along the two boundaries of the complex cylinder disconnects the surface into three surfaces with boundary, all of which have trivial linear holonomy. Since each component is fixed by the hyperelliptic involution, $s$ and $J(s)$ both lie in the same component. The other component that is not the interior of the original cylinder contains a simple cylinder that does not contain $s$ in its boundary, giving a contradiction. 

\noindent \emph{Case 2: One of the cylinders in $\cC$ is a complex envelope on $(Q_2, q_2)$.} 

Let $C$ denote the complex envelope. Let $D$ be any other cylinder in $\cC$. By assumption, for all surfaces in $U$, $D$ does not intersect $s$. 

By Masur-Zorich, specifically Theorem \ref{T:MZ} (3), $D$ must be simple on $(Q_2, q_2)$. Hence its boundary consists of $s$ and $J(s)$. This shows that $\cC$ contains only one cylinder apart from $C$ since otherwise there would be two simple cylinders $D$ and $D'$, both of which would have boundary $s$ and $J(s)$, 
implying that $\cQ = \cH(0,0)$ and contradicting the assumption that $\cQ$ is a stratum of quadratic differentials.

Therefore, $\cC$ contains exactly two cylinders, one of which (i.e. $D$) is generically parallel to $s$. Since there are maximally many horizontal cylinders on $(Q_1, q_1)$ subject to the condition that $s$ remain horizontal, it follows from Proposition  \ref{P:ArithmeticCylinder} that $\rk(\cQ) + \mathrm{rel}(\cQ) \leq 2$.

First suppose that $\rank(\cQ)=2$. Recalling the formula for rank and rel given in Lemma \ref{L:Q-rank}, we have $m_{even}=0$ and 
$$g+\frac{m_{odd}}2-1 = 2,$$
where $m_{odd}$ and $m_{even}$ are the number of odd and even order zeros (and poles) respectively. Since there are two poles (in the simple envelope), $m_{odd}\geq 2$. Since a quadratic differential cannot have 2 poles and no other zeros, we get $g=0, m_{odd}=6$ or $g=1, m_{odd}=4$. By the classification of hyperelliptic connected components \cite{LanneauHyp}, recalled in Section \ref{S:Q-hyp}, $\cQ = \cQ(1^2, -1^2)$. 

Since $(Q_1, q_1)$ is horizontally periodic and belongs to $\cQ(1^2, -1^2)$ it has four horizontal saddle connections. One joins two poles to each other and forms a boundary of $C$. Two more - $s$ and $J(s)$ - form the boundary of $D$. Therefore, one boundary of $C$ consists of a saddle connection $\sigma$ along with $s$ and $J(s)$. See Figure \ref{F:SimpleIntoComplexEnv} for a depiction of $(Q_2, q_2)$; the saddle connection with a single arrowhead on it corresponds to $\sigma$.

\begin{figure}[h]\centering
\includegraphics[width=.4\linewidth]{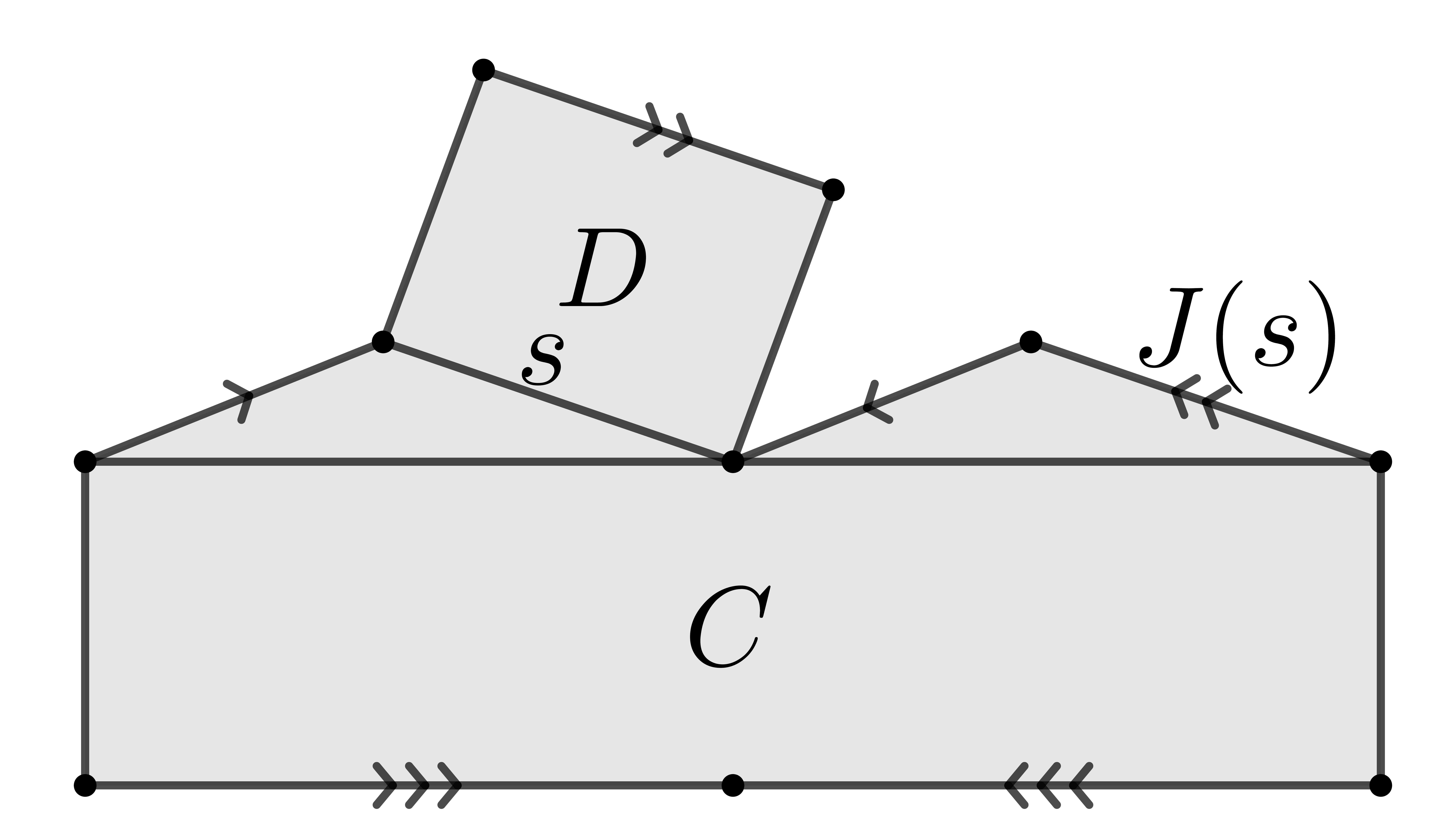}
\caption{The surface $(Q_2, q_2)$ }
\label{F:SimpleIntoComplexEnv}
\end{figure}

We may deform $(Q_1, q_1)$ through a continuum of horizontally periodic surfaces by lengthening $s$ and $J(s)$ (i.e. adding a positive real number $t$ to each of their periods) while shrinking $\sigma$ (i.e. subtracting $2t$ from its period). When the period of $\sigma$ reaches zero we have arrived at a surface in $\cQ(2, -1^2)$. As in Sublemma \ref{SL:DisjointSimple}, gluing in a complex envelope to $s$ on the surface we constructed in $\cQ(2, -1^2)$ would produce a proper orbit closure, contradicting our choice of $\cQ$ and also contradicting Lemma \ref{L:NonPillowcaseBase}. 

Next suppose $\rank(\cQ)=1$. In this case we can assume $(Q_2,q_2)=(Q_1, q_1)$, and this surface is a complex envelope (namely $C$) glued to a simple cylinder (namely $D$), so $\cQ = \cQ(2, -1^2)$. This cannot occur by Lemma \ref{L:NonPillowcaseBase}.

%
%





%
%
%

\noindent \emph{Case 3: All cylinders in $\cC$ are simple on $(Q_2, q_2)$.} 

By Sublemma \ref{SL:DisjointSimple}, since the cylinders in $\cC$ do not intersect $s$, the boundary of each cylinder in $\cC$ consists of $s$ and $J(s)$. Therefore, either $\cC$ contains one cylinder and $(Q_2, q_2)$ belongs to $\cH(0)$ or it contains two and $(Q_2, q_2)$ belongs to $\cH(0,0)$. In either case, we have contradicted the assumption that $\cQ$ is a stratum of quadratic differentials.
\end{proof}

\subsection{The Hyperelliptic Diamond Lemma}

We now turn to diamonds whose sides are hyperelliptic components of strata of quadratic differentials together with marked points, prove a result that will be applied in the proof of Sublemma \ref{SL:hyperelliptic-diamond2} below. Specifically we make the following assumption.

\begin{ass}\label{A:HypDiamond}
Let $\cM$ be an invariant subvariety in a stratum of quadratic differentials. Suppose that $(X, q)$ belongs to $\cM$ and that the collection of cylinders $\bfC_1$ and $\bfC_2$ forms a generic diamond satisfying the following conditions:
\begin{enumerate}
    \item\label{I:BigBoundary} $\MOne$ and $\MTwo$ are components of strata of connected quadratic differentials. 
    \item\label{I:HypBoundary} $\For(\MOne)$ and $\For(\MTwo)$ are hyperelliptic components of strata different from $\cQ(-1^4)$. 
    \item\label{I:HypAgreement} The hyperelliptic involutions on  $\For \left( \ColOneTwo (X,q) \right)$ obtained from those on  $\For \left( \ColOne (X,q) \right)$ and $\For \left( \ColTwo (X,q) \right)$ agree (this trivially holds unless $\MOneTwo = \cQ(-1^4, 0^n)$ for some integer $n$).
\end{enumerate}
\end{ass}
\begin{rem}
Since the diamond is generic condition $(\ref{I:BigBoundary})$ implies that $\bfC_1$ and $\bfC_2$ each consist of a single cylinder. We will abuse notation and allow $\bfC_i$ to denote these single cylinders. \end{rem}

At first glance this situation seems like one in which the Diamond Lemma can be immediately applied. After all, we have a diamond in which the sides admit hyperelliptic involutions (at least up to forgetting marked points) and where these involutions agree at the base of the diamond. However, we have not assumed the $\Col_{\bfC_i}(\bfC_j)$ are invariant under the hyperelliptic involution, so we must consider if the ``preimage of the image" assumption in the Diamond Lemma is satisfied. It turns out this would be automatically satisfied if there were no marked points, but the caveat ``up to forgetting marked points" prevents the Diamond Lemma from being used to conclude that $\For(\cM)$ is a hyperelliptic component of a stratum. To see what goes wrong, consider the example in Figure \ref{F13}. 

\begin{figure}[h]\centering
\includegraphics[width=.7\linewidth]{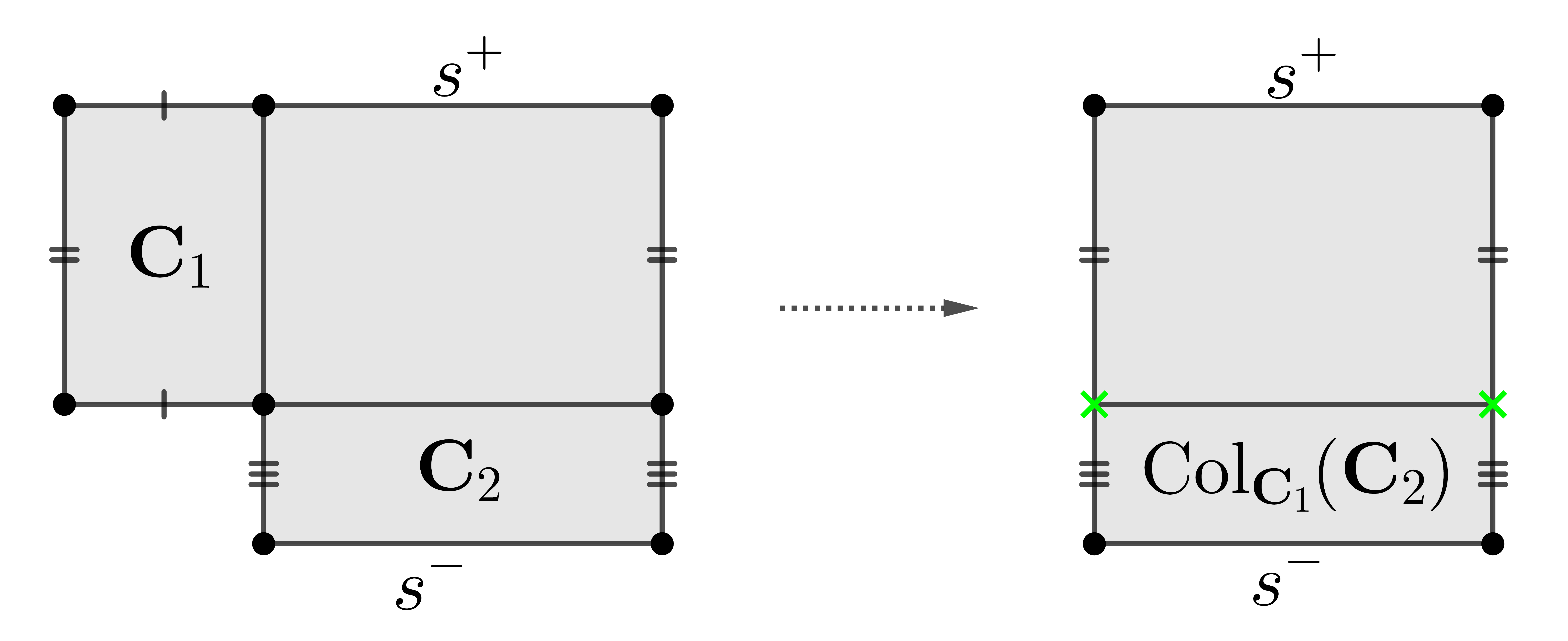}
\caption{Form a  half-translation surface $(X, q)$ as follows: take a quadratic differential $(X', q')$ in a hyperelliptic component, choose a saddle connection $s$ on $(X', q')$ that is fixed by the hyperelliptic involution, slit $s$, and glue in the flat subsurface on the left. The resulting surface is no longer hyperelliptic since the subsurface with boundary that was spliced into $(X', q')$ is not fixed by an involution. Nevertheless, degenerating $\bfC_1$ and $\bfC_2$ produce surfaces that, after forgetting marked points, belong to hyperelliptic components. Notice that degenerating $\bfC_1$ produces a simple cylinder (once marked points are forgotten) that contains $\ColOne(\bfC_2)$ but on which $\ColOne(\bfC_2)$ is not fixed by the hyperelliptic involution.}
\label{F13}
\end{figure}

We will now isolate the behavior that prevents $\For(\cM)$ from being a hyperelliptic component. 

\begin{defn}\label{D:BadConfiguration}
Given a cylinder $C$ on a half-translation surface $(Y, \eta)$ with marked points, we will let $\For(C)$ denote the cylinder on $\For(Y, \eta)$ that contains the image of $C$. Say that $\bfC_1$ and $\bfC_2$ are in a \emph{bad configuration} if, perhaps after relabeling the cylinders, on $\For( \ColOne (X,q) )$, $\For(\ColOne(\bfC_2))$ is a \red generically \black simple cylinder that strictly contains $\ColOne(\bfC_2)$ (see Figure \ref{F13}). Otherwise, say that $\bfC_1$ and $\bfC_2$ are in a \emph{good configuration}.  
\end{defn}

\begin{rem}
The definition implies that $\bfC_2$ is simple. It does not require that $\bfC_1$ is simple (although it is possible to show using Assumption \ref{A:HypDiamond} that it must be).
%
%
%
%
%
\end{rem}
 
 We show now that bad configurations are precisely the obstacle to $\For(\cM)$ being a hyperelliptic component. 

\begin{lem}\label{L:hyperelliptic-diamond}
Under Assumption \ref{A:HypDiamond}, if $\bfC_1$ and $\bfC_2$ are in a good configuration then $\For(\cM)$ is contained in a hyperelliptic locus. 
\end{lem}

\begin{proof}
Let $T_i$ denote the hyperelliptic involution on $\Col_{\bfC_i}(X, q)$ (this exists by Assumption \ref{A:HypDiamond} \eqref{I:HypBoundary}). We note that (by Lemma \ref{L:HypSymmetries}) this is uniquely defined since $\For(\cM_{\bfC_i}) \ne \cQ(-1^4)$.

\begin{sublem}\label{SL:PreimageOfImage}
$\ColOne(\bfC_2)$ is fixed by $T_1$. Moreover, $\ColOne(\bfC_2)/T_1$ is a simple cylinder or simple envelope. The analogous statement holds for $\ColTwo(\bfC_1)$. 
\end{sublem}
\begin{proof}
\red If $\For\left( \ColOne(\bfC_2) \right) = \ColOne(\bfC_2)$ then the result is immediate. Suppose therefore that $\For\left( \ColOne(\bfC_2) \right)$ strictly contains $\ColOne(\bfC_2)$. \black

Since $\bfC_1$ and $\bfC_2$ are in a good configuration, $\For\left( \ColOne(\bfC_2) \right)$ is not a \red generically \black simple cylinder. By Lemma \ref{L:hyp-cyl}, on an unmarked surface in a hyperelliptic component different from $\cQ(-1^4)$, the hyperelliptic involution acts by translation on all cylinders except for \red generically \black simple cylinders, on which it acts by rotation. Since (by Assumption \ref{A:HypDiamond} \eqref{I:HypBoundary}), $\For(\MOne) \ne \cQ(-1^4)$ we have that $T_1$ acts by translation on $\For(\ColOne(\bfC_2))$ and hence that it fixes $\ColOne(\bfC_2)$; see Figure \ref{F:HypDiamondTranslate} for an example.

\begin{figure}[h]\centering
\includegraphics[width=.3\linewidth]{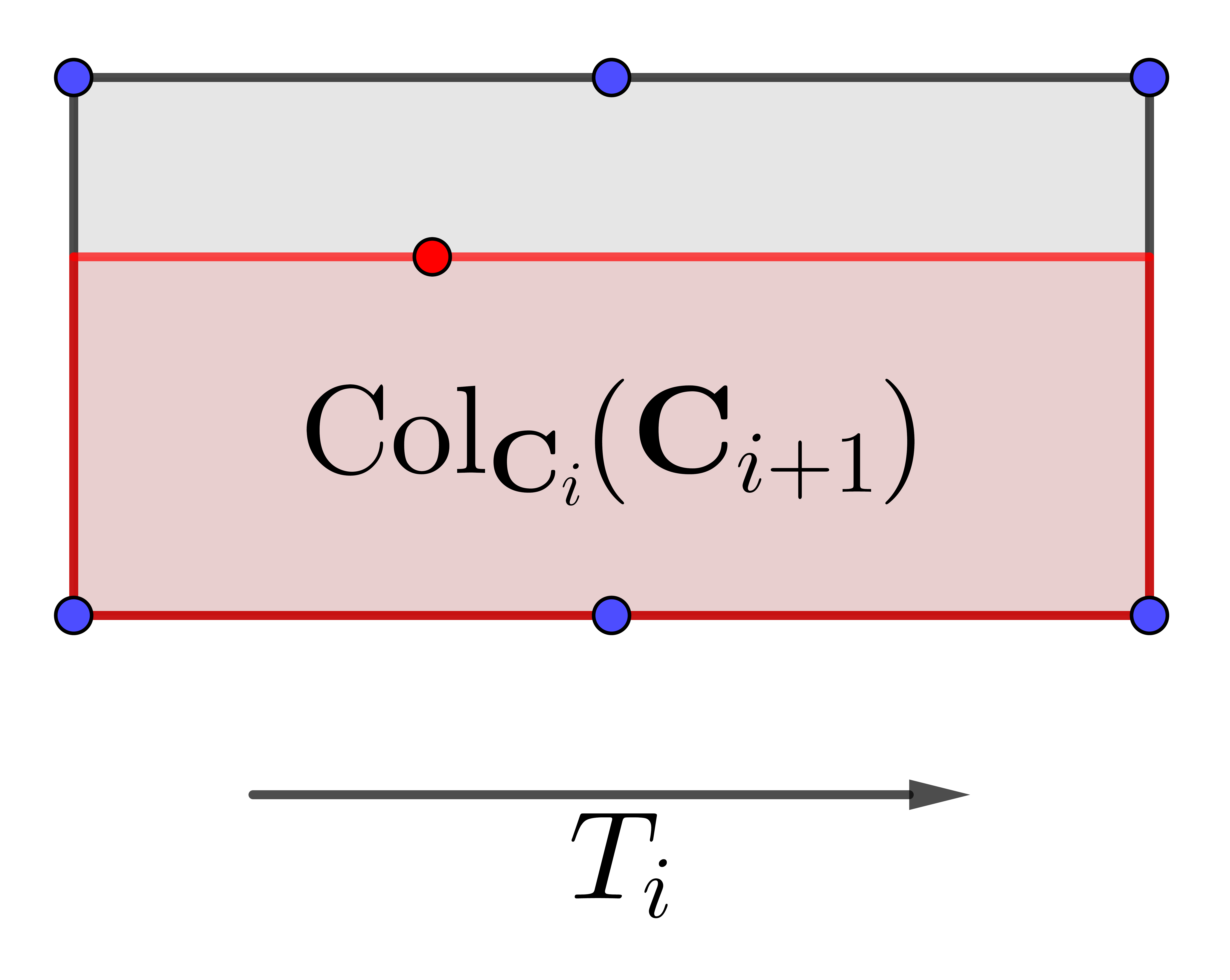}
\caption{The proof of Sublemma \ref{SL:PreimageOfImage}.}
\label{F:HypDiamondTranslate}
\end{figure}

Now we will turn to the second claim. Since the diamond is generic, all the saddle connections on the boundary of $\bfC_1$ and $\bfC_2$ are generically parallel to the cylinders they border. The same holds in $\MTwo$ and $\MOne$ respectively by Corollary \ref{C:StillGeneric}. By Assumption \ref{A:HypDiamond} \eqref{I:BigBoundary} and \eqref{I:HypBoundary}, $\For(\MOne)$ is a hyperelliptic component, which implies that $\ColOne(\bfC_2)/T_1$ is a generic cylinder in a genus zero stratum. This stratum is different from $\cQ(-1^4)$, so $\ColOne(\bfC_2)/T_1$ is a simple cylinder or simple envelope.
\end{proof}

Since $\ColOne(\bfC_2)$ is fixed by $T_1$ and $\ColTwo(\bfC_1)$ is fixed by $T_2$, the involutions $\Col_{\ColOne(\bfC_2)}(T_1)$ and $\Col_{\ColTwo(\bfC_1)}(T_2)$ are well-defined (by Lemma \ref{L:DefinitionCol(f)}) and equal (by Assumption \ref{A:HypDiamond} \eqref{I:HypAgreement}). Let $T$ denote this involution on $\ColOneTwo(X,q)$.

By the Diamond Lemma (Lemma \ref{L:diamond}) there is an involution $T_0$ on $(X, q)$ that preserves each $\bfC_i$ and such that $\Col_{\bfC_i}(T_0) = T_i$. Notice that the hypotheses of the Diamond Lemma are met since $\Col(T_1) = \Col(T_2)$ and by Sublemma \ref{SL:PreimageOfImage}.

By slightly nudging $(X, q)$, we may suppose that $(X, q)$ has dense orbit in $\cM$ and that $\bfC_1$ and $\bfC_2$ remain cylinders in a good configuration. Therefore, it suffices to show that $(X, q)/T_0$ is a quadratic differential on a sphere (rather than a higher genus surface).

By the proof of the Diamond Lemma, $(X, q)/T_0$ is formed by gluing the cylinder $\bfC_1/T_0$ into $\ColOne(\bfC_1)/T_1$ on $\ColOne(X, q)/T_1$. Since $\bfC_1/T_0$ is a simple cylinder or simple envelope (by Sublemma \ref{SL:PreimageOfImage}), $\ColOne(\bfC_1)/T_1$ is a single saddle connection on a sphere, $\ColOne(X, q)/T_1$. Gluing in an envelope to a sphere produces a sphere, so it remains to consider the case where $\bfC_1/T_0$ is a simple cylinder. 

Gluing in a simple cylinder to a sphere produces a sphere provided that the simple cylinder is glued into a closed loop. A saddle connection on a sphere is a closed loop if and only if cutting it disconnects the surface. Therefore, it suffices to show that cutting $\ColOne(\bfC_1)/T_1$ disconnects $\ColOne(X, q)/T_1$.

Since $\ColTwo(\bfC_1)/T_2$ is a cylinder on a sphere, we have that cutting it disconnects the surface. Hence cutting $\ColOneTwo(\bfC_1)/T$ also disconnects $\ColOneTwo(X, q)/T$. One of the resulting components contains $\ColOneTwo(\bfC_2)/T$. Since $\ColOne(X, q)/T_1$ is formed by gluing $\ColOne(\bfC_2)/T_1$ into $\ColOneTwo(\bfC_2)/T$ we see that cutting $\ColOne(\bfC_1)/T_1$ on $\ColOne(X, q)/T_1$ still disconnects the surface as desired.
\end{proof}

\subsection{Conclusion of the proof of Theorem \ref{T:complex-gluing}}

\begin{proof}[Proof of Theorem \ref{T:complex-gluing}:]
The result when $\cF(\cQ)$ is hyperelliptic has already been established in Proposition \ref{P:complex-gluing1}. Suppose in order to find a  contradiction that $\cQ$ is the smallest dimensional component of a stratum of quadratic differentials for which the statement fails, \red i.e. that there is a surface $(X, q) \in \cQ$ with a saddle connection $s$ so that $\cM_s$ is not a component of a stratum. \black Note that $\cF(\cQ)$ is not hyperelliptic. \red If $\cQ$ had marked points, then, as in the proof of Sublemma \ref{SL:PillowcaseMarkedPoints}, it would be possible to move them into the zeros (or poles) and hence contradict the minimality of $\cQ$. Hence \black $\cQ$ has no marked points, so $\cF(\cQ)=\cQ$. 


\begin{sublem}\label{SL:CollapsingInduction}
Suppose that $\bfC$ is a generic cylinder on $(X, q)$ that does not intersect $s$ and which is one of the following: a simple cylinder, a simple envelope, or a half-simple cylinder. Suppose also that if $s$ is a  multiplicity one boundary component of $\bfC$ then $\bfC$ is a simple cylinder. Then $\For(\Col_{\bfC}(X, q))$ belongs to a hyperelliptic component of a stratum of connected quadratic differentials and $\Col_{\bfC}(s)$  is generically fixed by the hyperelliptic involution on $\Col_{\bfC}(X, q)$.
\end{sublem}

Recall that the condition that $s$ is a multiplicity one boundary component of $\bfC$ means that one boundary consists entirely of $s$, and that $s$ occurs only once in that boundary (ruling out the possibility that $s$ joins two simple poles). See Figure \ref{F:BadMultOne}. 

\begin{figure}[h]\centering
\includegraphics[width=.5\linewidth]{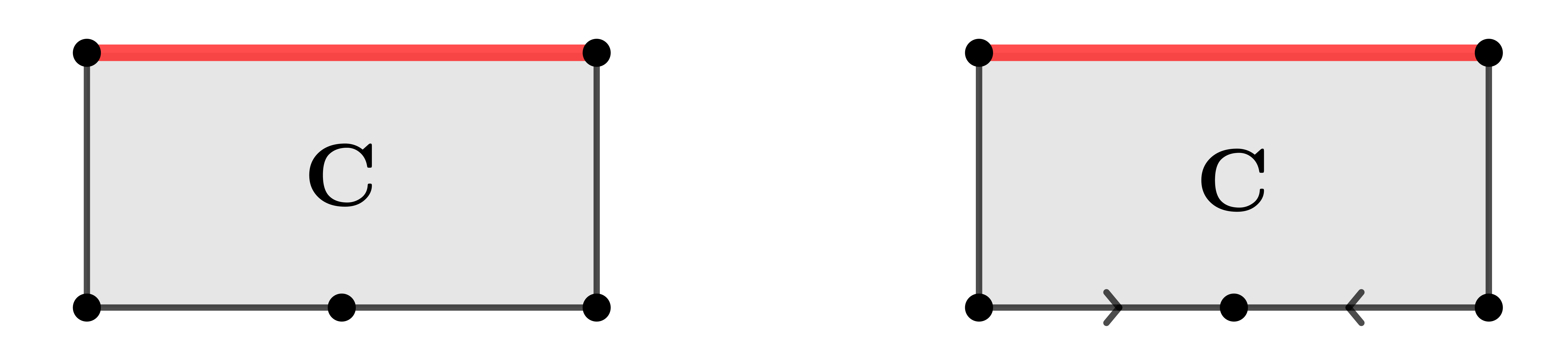}
\caption{When $\bfC$ is a half simple cylinder (left) or a simple envelope (right), the assumptions in Sublemma \ref{SL:CollapsingInduction} do not allow $s$ to be the highlighted saddle connection.}
\label{F:BadMultOne}
\end{figure}

\begin{proof}
The assumption that $\bfC$ is a simple cylinder, simple envelope, or half-simple cylinder is included to guarantee that $\Col_{\bfC}(X, q)$ has non-trivial holonomy. The assumption that if $s$ is a multiplicity one boundary component of $\bfC$ then $\bfC$ is a simple cylinder is included so $s$ gives a single saddle connection on $\Col_{\bfC}(X, q)$, instead of being split into two. See Figure \ref{F:BadMultOne}. 

Let $(Y, q_Y)$ be the surface belonging to $\cM_s$ formed by gluing in a complex envelope to $s$ on $(X, q)$. Let $\bfC_1$ be the cylinder corresponding to $\bfC$ on $(Y, q_Y)$ and let $\bfC_2$ denote the complex envelope that was glued in. By Corollary \ref{C:codim1}, $\left( \cM_s \right)_{\bfC_1}$ remains a proper invariant subvariety of the stratum of quadratic differentials containing it. Since it contains $\Col_{\bfC_1}(Y, q_Y)$, which is the result of gluing in the complex envelope $\Col_{\bfC_1}(\bfC_2)$ to $\Col_{\bfC}(X, q)$, the minimality assumption on the dimension of $\cQ$ implies that $\For(\Col_{\bfC}(X, q))$ belongs to a hyperelliptic component of a stratum of connected quadratic differentials and $s$ is generically fixed by the hyperelliptic involution on $\Col_{\bfC}(X, q)$. 
\end{proof}

We now show the following:

\begin{sublem}\label{SL:GCC}
There is an $\{s\}$-path from $(X, q)$ to a half-translation surface $(X', q')$ with cylinders $\bfC_1$ and $\bfC_2$ that form a generic diamond, where 
\begin{enumerate}
\item $\bfC_1$ and $\bfC_2$ are each a simple cylinder, a simple envelope, or a half-simple cylinder, 
\item  $s$ doesn't intersect $\bfC_i$, and 
\item if $s$ is a multiplicity one boundary component of $\bfC_i$ then $\bfC_i$ is a simple cylinder. 
\end{enumerate}
\end{sublem}

\begin{proof}
We will derive this from Proposition \ref{P:Avoidance}. That proposition excludes rank 1 strata, $\cQ(2,1,1)$, and $\cQ(1,1,-1,-1)$, but these strata need not be considered \red since they are connected (by Lanneau \cite[Theorem 1.2]{Lconn}) \black and hence hyperelliptic (see Section \ref{S:Q-hyp}).

Proposition \ref{P:Avoidance} thus gives $(X', q')$ together with cylinders $\bfC_1$ and $\bfC_2$ that are each a simple cylinder, a simple envelope, or a half-simple cylinder, such that the $\bfC_i$ don't share boundary saddle connections and don't intersect $s$. 

It suffices to show the final claim. Suppose without loss of generality that $s$ is in the boundary of $\bfC_1$. Hence $s$ is not in the boundary of $\bfC_2$. Applying Sublemma \ref{SL:CollapsingInduction} to $\bfC_2$, we see that $\For(\ColTwoX)$ is hyperelliptic and $\ColTwo(s)$ is generically fixed by the hyperelliptic involution. Since $\ColTwo(s)$ is generically fixed by the hyperelliptic involution, it cannot have marked points as endpoints (see Remark \ref{R:MarkedEndpoint}).  

In a hyperelliptic stratum without marked points, every cylinder is a simple cylinder, a complex cylinder, or a complex envelope (by Lemma \ref{L:hyp-cyl}). 
Since all marked points are free in a stratum, and $\bfC_1$ is generic, we see that if $\ColTwo(\bfC_1)$ is not simple then it must be contained in a complex cylinder or complex envelope and have one boundary component consisting of a loop from a marked point to itself; see Figure \ref{F:sGood}. 
\begin{figure}[h]\centering
\includegraphics[width=.7\linewidth]{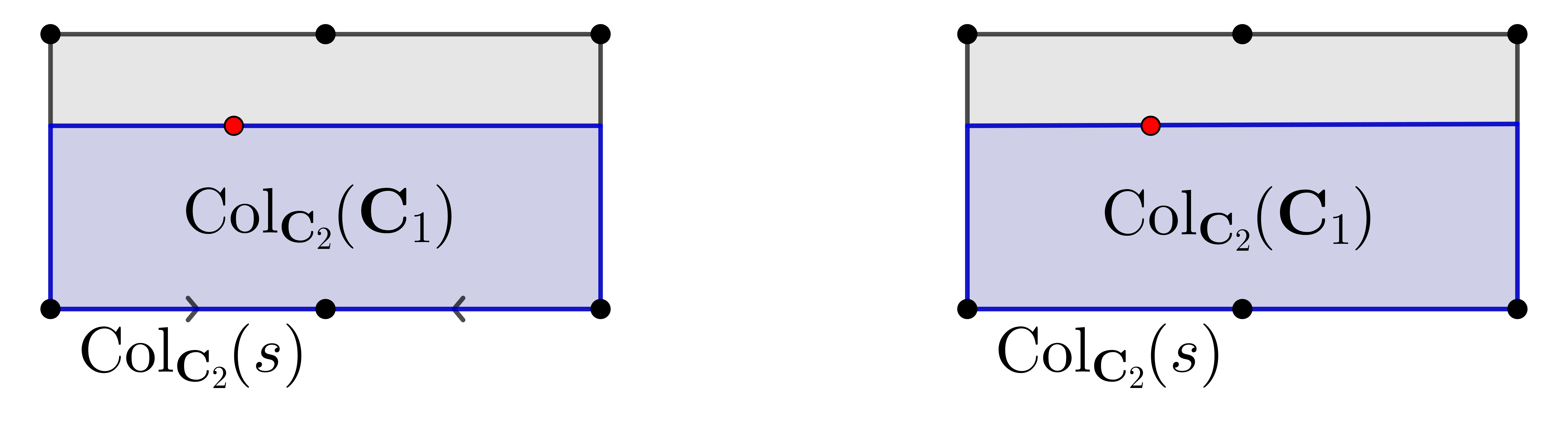}
\caption{Left: $\ColTwo(\bfC_1)$ contained in a complex envelope. Right: $\ColTwo(\bfC_1)$ contained in a complex cylinder. Both show a possibility for $\ColTwo(s)$ allowed for in the proof of Sublemma \ref{SL:GCC} (but one can show that the right possibility does not occur).}
\label{F:sGood}
\end{figure}
In this case $s$ must be on the other side of $\bfC_1$, so $s$ is not a multiplicity one boundary component. 
\end{proof}

Let $\bfC_1$ and $\bfC_2$ denote the cylinders produced by Sublemma \ref{SL:GCC}. These cylinders satisfy (\ref{I:BigBoundary}) of Assumption \ref{A:HypDiamond}. By Sublemma \ref{SL:CollapsingInduction}, $\For(\cQ_{\bfC_1})$ and $\For(\cQ_{\bfC_2})$ are hyperelliptic components and $\Col_{\bfC_i}(s)$ is fixed by a hyperelliptic involution $J_i$ on $\Col_{\bfC_i}(X, q)$ for $i\in \{1,2\}$.

We now show that $J_1$ and $J_2$ agree on $\ColOneTwo(X, q)$. Indeed,  the hypotheses on $\bfC_1$ and $\bfC_2$ imply that $\ColOneTwo(s)$ is a single saddle connection. Since $\Col_{\bfC_i}(s)$ is fixed by $J_i$, it follows that $\ColOneTwo(s)$ is fixed by both $\Col(J_1)$ and $\Col(J_2)$, implying that these two involutions are equal.

Hence part (\ref{I:HypAgreement}) of Assumption \ref{A:HypDiamond} holds. Therefore, Assumption \ref{A:HypDiamond} is  satisfied unless $\For(\cQ_{\bfC_1})$ or $\For(\cQ_{\bfC_2})$ is $\cQ(-1^4)$. By the Hyperelliptic Diamond Lemma (Lemma \ref{L:hyperelliptic-diamond}), since $\cQ$ is assumed to be a nonhyperelliptic component of a stratum, one of the two possibilities occur:
\begin{enumerate}
    \item at least one of $\For(\cQ_{\bfC_1})$ and $\For(\cQ_{\bfC_2})$ is $\cQ(-1^4)$, or
    \item the previous possibility does not occur, but $\bfC_1$ and $\bfC_2$ are in a bad configuration (see Definition \ref{D:BadConfiguration}). 
\end{enumerate}

We will now rule out the second case.

\begin{sublem}\label{SL:hyperelliptic-diamond2}
At least one of $\For(\cQ_{\bfC_1})$ and $\For(\cQ_{\bfC_2})$ is $\cQ(-1^4)$.
\end{sublem}
\begin{proof}
Suppose not. The previous discussion implies that $\bfC_1$ and $\bfC_2$ are in a bad configuration. In particular, suppose that $\ColOne(X, q)$ contains a marked point and that  $\For(\ColOne(\bfC_2))$  is a simple cylinder (see Definition \ref{D:BadConfiguration} for the definition of $\For$ of a cylinder). Set $S_1 := \ColOne(\bfC_1)$. 

By Corollary \ref{C:ParallelSC}, the saddle connections in $S_1$ are generically parallel to each other (notice that since $\bfC_i$ consists of a single cylinder, $\mathrm{Twist}(\bfC_i)$ is one-dimensional). Moreover, there is a saddle connection in $S_1$ whose endpoint is a marked point and that is not a closed loop joining the marked point to itself (the marked point borders $\ColOne(\bfC_2)$ on one side and the saddle connections in $S_1$ do not intersect $\ColOne(\bfC_2)$  by assumption, and saddle connections in $S_1$ are not boundary saddle connections of $\ColOne(\bfC_2)$). A saddle connection that joins a marked point to a distinct point cannot be generically parallel to any other saddle connection. Therefore, $S_1$ is a single saddle connection, ruling out the possibility that $\bfC_1$ is a half-simple cylinder. So  $\bfC_1$ is a simple cylinder or simple envelope.

Since $\ColOne(s)$ is \red generically \black fixed by the hyperelliptic involution on $\ColOne(X, q)$ (see Sublemma \ref{SL:CollapsingInduction}) we have that $\ColOne(s)$ does not have any marked points as endpoints by Remark \ref{R:MarkedEndpoint}.

\noindent \textbf{Case 1: $S_1$ is not contained entirely in $\For(\ColOne(\bfC_2))$.}

In this case, it is possible to retract the saddle connection $S_1$ until it lies entirely outside of $\For(\ColOne(\bfC_2))$ (see Figure \ref{F:retraction}). This deformation does not change the fact that $\ColOne(s)$ is a saddle connection since $\ColOne(s)$ does not have marked points as endpoints. In this way, $\ColOne(\bfC_2)$ has now expanded to be the simple cylinder $\For(\ColOne(\bfC_2))$ that does not intersect $\ColOne(s)$.

\begin{figure}[h]\centering
\includegraphics[width=.6\linewidth]{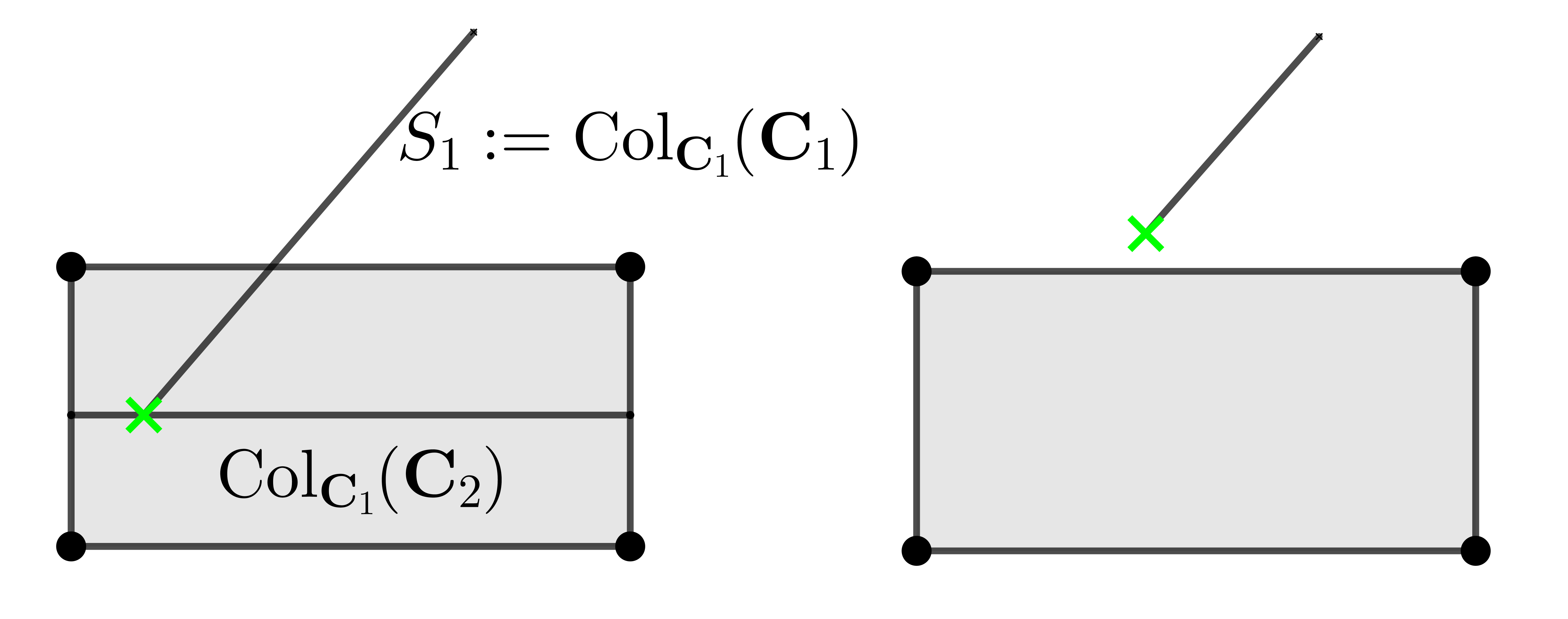}
\caption{Retracting $S_1$}
\label{F:retraction}
\end{figure}

 Notice that now $\bfC_1$ and $\bfC_2$ do not form a bad configuration. As in Figure \ref{F:NotBad}, we see that even after collapsing $\bfC_2$ there is a zero (not just a marked point) on either side of $\bfC_1$.  Lemma \ref{L:hyperelliptic-diamond} yields the contradiction that $\For(\cQ)$ is hyperelliptic.  
 
\begin{figure}[h]\centering
\includegraphics[width=.45\linewidth]{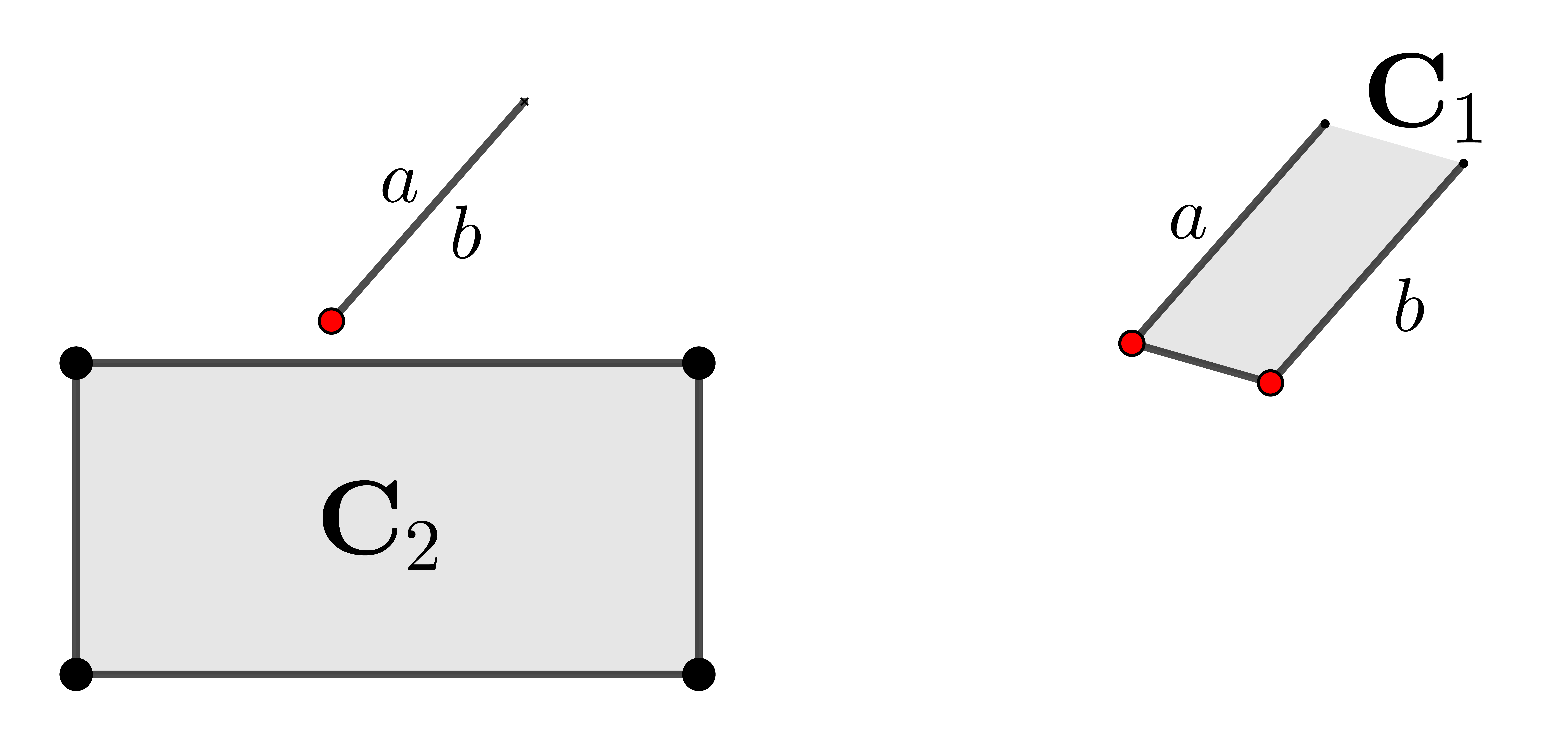}
\caption{Cylinders on the deformation of $(X, q)$ (illustrated when $\bfC_1$ is simple). }
\label{F:NotBad}
\end{figure}

\noindent \textbf{Case 2: $S_1$ is contained entirely in $\For(\ColOne(\bfC_2))$ and joins a marked point on the boundary of $\ColOne(\bfC_2)$ to another marked point in the interior of $\For(\ColOne(\bfC_2))$.}

In this case, there is a simple cylinder $\bfC_3$ on $(X, q)$ such that $\ColOne(\bfC_3)$ lies between the second marked point and the top boundary of $\For(\ColOne(\bfC_2))$ (see Figure \ref{F:TwoCylinders}). 

\begin{figure}[h]\centering
\includegraphics[width=.3\linewidth]{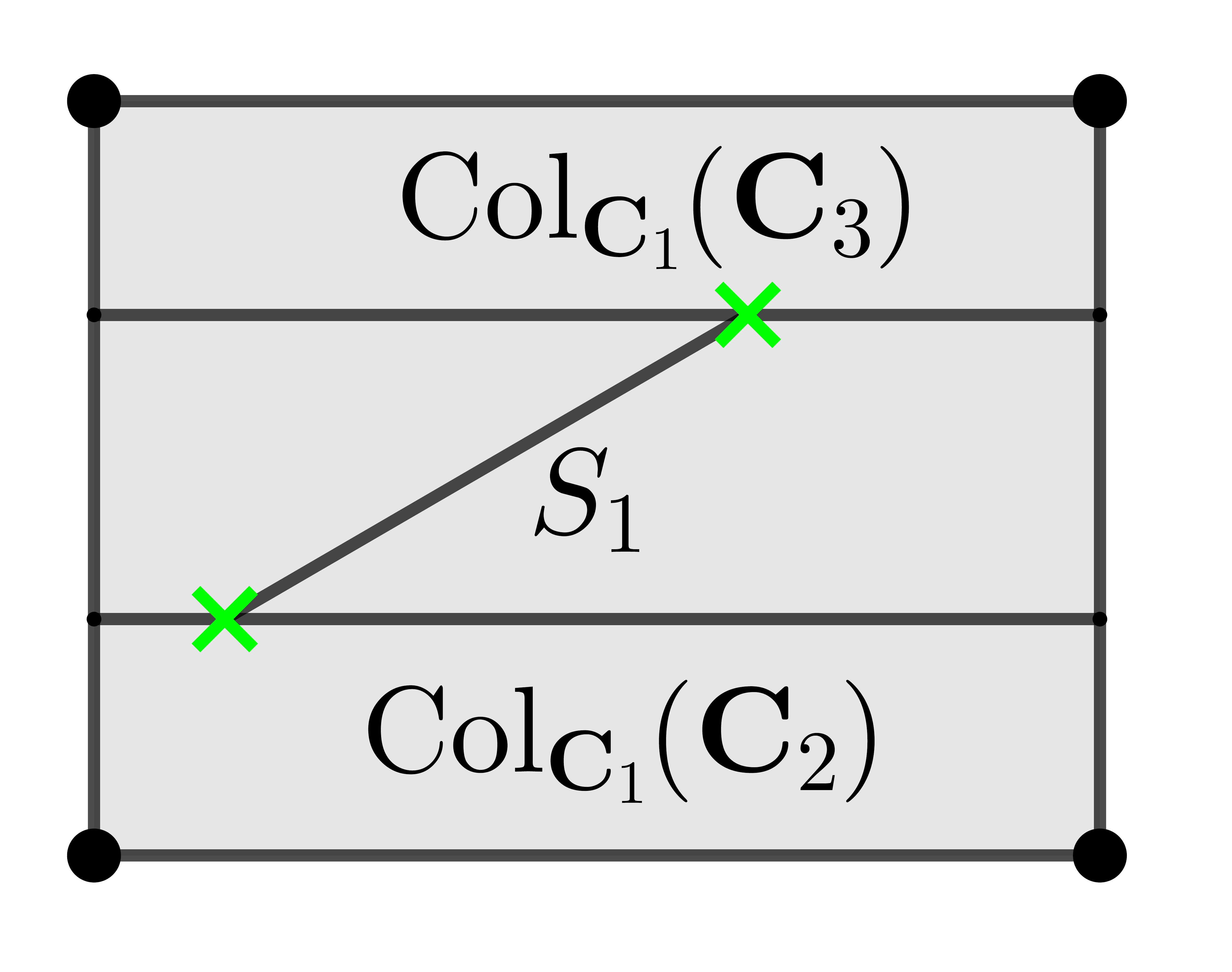}
\caption{An illustration of $\ColOne(\bfC_3)$ in $\For(\ColOne(\bfC_2))$.}
\label{F:TwoCylinders}
\end{figure}

Since the hyperelliptic involution acts on \red $\For(\ColOne(\bfC_2))$ \black by rotation, and since $\ColOne(s)$ is fixed by the hyperelliptic involution on $\ColOne(X, q)$ it follows that if $s$ intersects $\bfC_3$ then it would intersect $\bfC_2$, which is a contradiction. Therefore, $\bfC_3$ is a simple cylinder on $(X, q)$ that is disjoint from $s$. 

Since $\bfC_2$ and $\bfC_3$ are both simple they do not border each other in $(X, q)$, since $(X, q)$ has no marked points. It is easy to see that collapsing one or the other cannot cause the other to border a marked point so these cylinders do not form a bad configuration. 

Notice that neither $\Col_{\bfC_2}(X, q)$ nor $\Col_{\bfC_3}(X, q)$ belong to $\cQ(-1^4, 0^n)$ for any integer $n$. This is clear since $\For\left(\Col_{\bfC_1, \bfC_2, \bfC_3}(X, q) \right)$ belongs to the same stratum of quadratic differentials as $\For(\ColOne(X, q))$ which we have assumed is not $\cQ(-1^4)$. By the hyperelliptic Diamond Lemma (Lemma \ref{L:hyperelliptic-diamond}) using $\bfC_2$ and $\bfC_3$, we see that $\cQ$ is hyperelliptic, which is a contradiction.

\noindent \textbf{Case 3: $S_1$ is contained entirely in $\overline{\For(\ColOne(\bfC_2))}$ and joins a marked point on the boundary of $\ColOne(\bfC_2)$ to a singularity on the boundary of $\For(\ColOne(\bfC_2))$; $\ColOne(s)$ is not a boundary component of $\For(\ColOne(\bfC_2))$.} 

Since $\ColOne(s)$ is generically fixed by the hyperelliptic involution, if $\ColOne(s)$ intersects $\For(\ColOne(\bfC_2))$ then it crosses it and in particular, $s$ must intersect $\bfC_2$ which is a contradiction. Letting $p$ be the marked point that is an endpoint of $S_1$ we see that since $\ColOne(s)$ does not intersect $\For(\ColOne(\bfC_2))$ and since it is not a boundary saddle connection that it is possible to move $p$ out of $\For(\ColOne(\bfC_2))$ while keeping $S_1$ a saddle connection and ensuring that along this movement $S_1$ and $\ColOne(s)$ remain disjoint (see Figure \ref{F:Twirling}). We conclude that $\For(\cQ)$ is hyperelliptic as in Case 1.

\begin{figure}[h]\centering
\includegraphics[width=.7\linewidth]{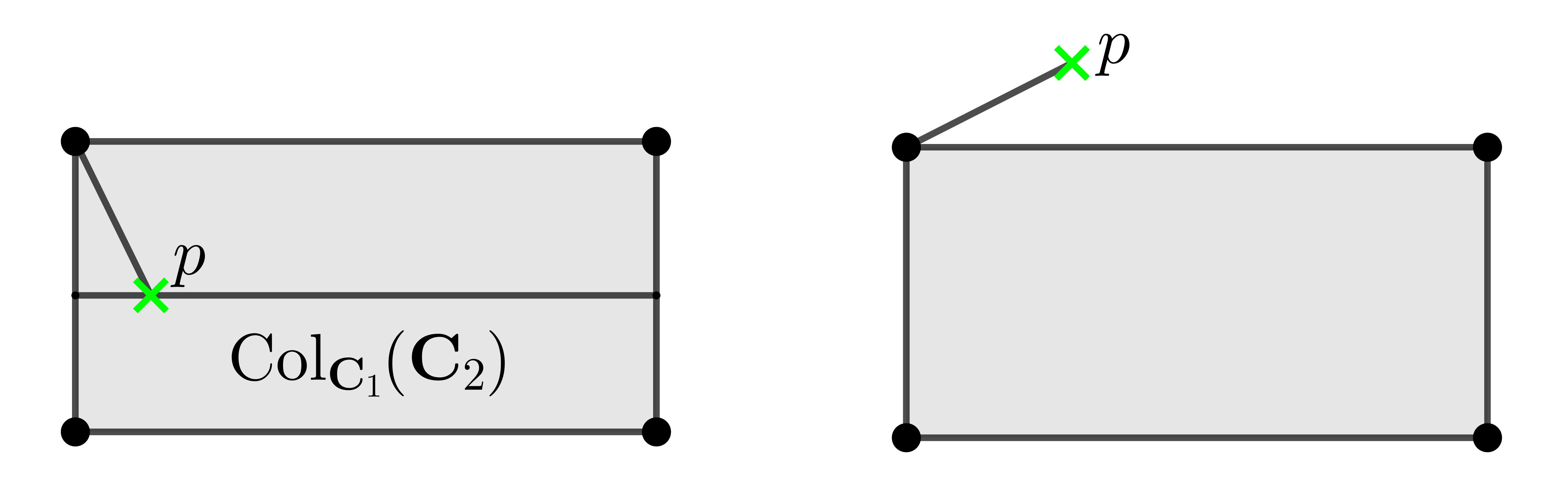}
\caption{Moving the marked point $p$ produces two cylinders that are no longer in a bad configuration.}
\label{F:Twirling}
\end{figure}

\noindent \textbf{Case 4: $S_1$ is contained entirely in $\overline{\For(\ColOne(\bfC_2))}$ and joins a marked point on the boundary of $\ColOne(\bfC_2)$ to a singularity on the boundary of $\For(\ColOne(\bfC_2))$; $\ColOne(s)$ is a boundary component of $\For(\ColOne(\bfC_2))$.} 

The hyperelliptic involution acts on $\For(\ColOne(\bfC_2))$ by rotation. Since $\ColOne(s)$ is a boundary component of $\For(\ColOne(\bfC_2))$ it follows that it is simultaneously the top and bottom boundary of $\For(\bfC_2)$ since $\ColOne(s)$ is fixed by the hyperelliptic involution. Therefore, $\ColOne(X, q)$ belongs to $\cH(0,0)$, contradicting our assumption that $\ColOne(X,q)$ has non-trivial holonomy.
%
%

The proof of Sublemma \ref{SL:hyperelliptic-diamond2} is now complete.
\end{proof}

It now suffices to suppose without loss of generality that $\ColOne(X, q)$ belongs to $\cQ(-1^4, 0^n)$ for some integer $n$. Since $\cQ_{\bfC_1}$ has dimension one less than $\cQ$ (Lemma \ref{L:StrataBasics}), which has rank at least two (since all rank 1 strata are hyperelliptic), it follows that $n \geq 1$. Since $\bfC_1$ is a simple cylinder, simple envelope, or half-simple cylinder and since $(X, q)$ has no marked points, $n \leq 2$. \red By Sublemma \ref{SL:CollapsingInduction}, $\ColOne(s)$ is generically fixed by a hyperelliptic involution and hence it joins two poles (by Remark \ref{R:MarkedEndpoint}). \black 


Set $S_1:= \ColOne(\bfC_1)$. The saddle connections in $S_1$ are all generically parallel to each other by Corollary \ref{C:ParallelSC}. We will now do a  case analysis to derive a contradiction.

\bold{Case 1: $S_1$ contains a saddle connection that goes from a marked point to itself.} See the right part of Figure \ref{F2}. This case does not occur, because it implies that on the original surface $\bfC_1$ and $\bfC_2$ share boundary saddle connections.

\begin{figure}[h]\centering
\includegraphics[width=.7\linewidth]{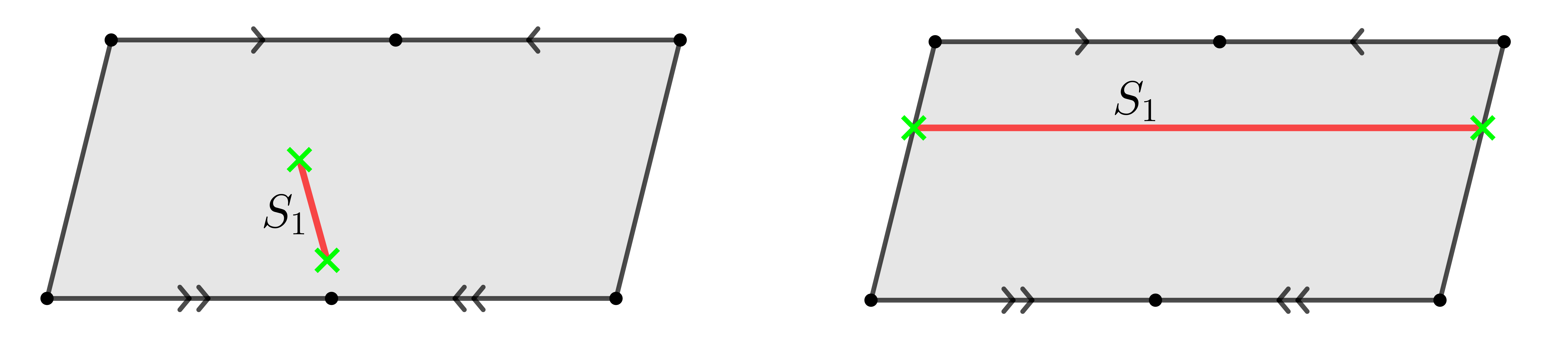}
\caption{Cases 1 and 3 in the final case analysis of the proof of Theorem \ref{T:complex-gluing}.}
\label{F2}
\end{figure}

At this point we may assume that $S_1$ contains a saddle connection that joins a marked point to a distinct point. Since no other saddle connection is generically parallel to such a saddle connection, it follows that $S_1$ is a singleton. This implies that $\bfC_1$ is a simple cylinder or simple envelope ($S_1$ would need to consist of two saddle connections if $\bfC_1$ were a half-simple cylinder).

\bold{Case 2: $S_1$ goes from a pole to a marked point.} See the top of Figure \ref{F5}.

\begin{figure}[h]\centering
\includegraphics[width=.8\linewidth]{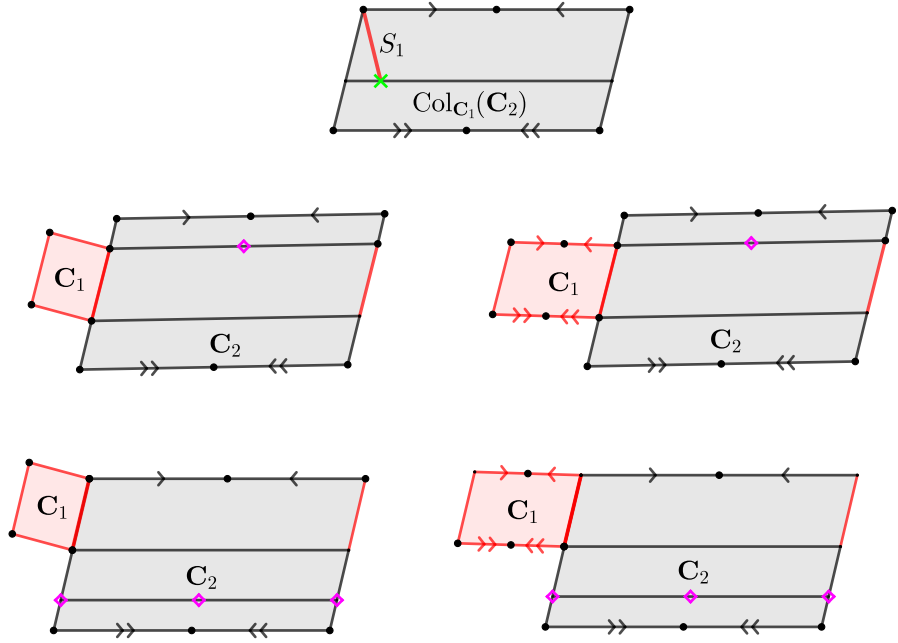}
\caption{Four possible surfaces in $\cM_s$. \red (In fact the top right picture isn't valid since $\bfC_1$ isn't a \emph{simple} envelope, but we treat it on equal footing with the other examples and derive a contradiction in all cases.) \black }
\label{F5}
\end{figure}

The four possibilities for the surfaces in $\cM_s$ that gluing in a complex envelope could produce  are recorded in Figure \ref{F5}. They are obtained via the following analysis: On the top of Figure \ref{F5}, the saddle connection $s$ could be the top saddle connection or the bottom one. Additionally, $\bfC_1$ can be a simple cylinder or a simple envelope. Hence we get four possibilities for the surface obtained by gluing in a complex envelope to $(X, q)$. 

Since $\cM_s$ is assumed to be a proper subvariety, it is defined in local period coordinates by the equation that states that the two boundary saddle connections of the complex horizontal envelope have the same periods. In the top two cases in Figure \ref{F5}, these equations imply that the marked point shown is a periodic point. In the bottom two cases, the equations imply that there is a irreducible pair of marked points. However, by Apisa-Wright \cite{ApisaWright} (restated as Lemma \ref{L:ApisaWright} and \ref{T:StrataMarkedPoints} in Section \ref{S:Pre}) there are no such point markings in non-hyperelliptic strata of quadratic differentials and so we have reached a contradiction.  

\bold{Case 3: $S_1$ joins distinct marked points.} See the left part of Figure \ref{F2}. Let $(X', q')$ denote the result of gluing in a complex envelope to $s$ on $(X, q)$. Let $\cQ'$ be the stratum containing $(X', q')$. Assume without loss of generality that $\ColOne(\bfC_2)$ is horizontal on $\ColOne(X, q)$. Since $s$ joins two poles and does not cross $\bfC_2$ on $\ColOne(X, q)$ it must also be a horizontal saddle connection. 

As in Figure \ref{F:s1EasyPart2}, it is possible to move one of the endpoints of $S_1$ into a pole while preserving the property that $S_1$ is a saddle connection that does not intersect $\ColOne(s)$. 
\begin{figure}[h]\centering
\includegraphics[width=.8\linewidth]{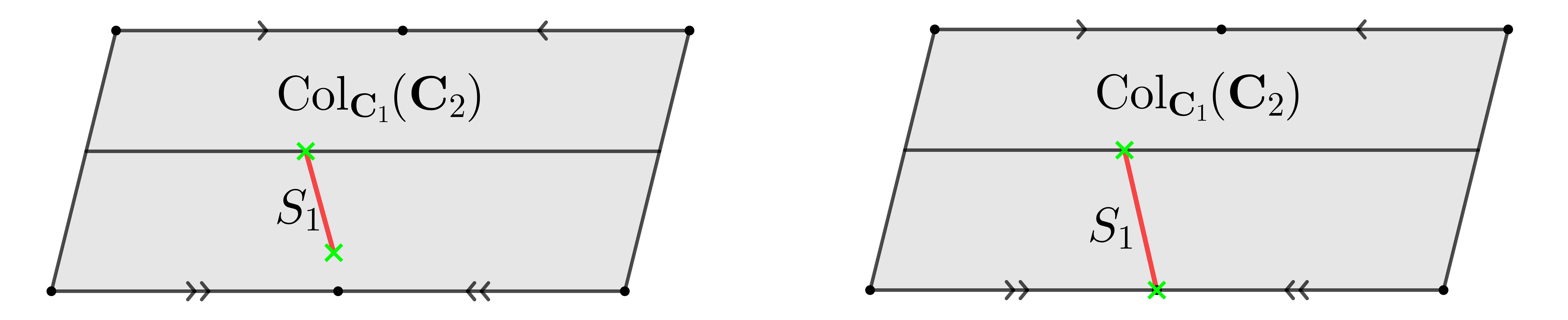}
\caption{The final argument in Case 3.}
\label{F:s1EasyPart2}
\end{figure}
This degeneration remains well-defined on $(X, q)$ and $(X', q')$, reducing to Case 2 and hence giving a contradiction. 
%
%
%
%
\end{proof}

\subsection{Open problems}\label{SS:OpenProblems}

Here are some possible next steps. 

\begin{prob}
Prove a version of Theorem \ref{T:complex-gluing} that considers the result of gluing in a complex cylinder (rather than an envelope) to a pair of hat homologous saddle connections (rather than a single saddle connection).
\end{prob}

Here one should keep in mind that there are in fact two ways to glue a complex envelope into a pair $\{s_1, s_2\}$ of equal length saddle connections.
The more natural way is to glue each boundary component to one boundary saddle connection arising from $s_1$ and one arising from $s_2$;  this can be done so that degenerating the complex cylinder gives the original surface.
The other way is to glue each boundary component of the complex cylinder to the pair of boundary saddle connections obtained by cutting either $s_1$ or $s_2$.  In either case, one is interested in a possible codimension one orbit closure defined by the condition that the four saddle connections bounding the complex cylinder have equal length. 

The following problems are more ambitious. 

\begin{prob}
Classify codimension one invariant subvarieties of components of quadratic differentials of rank at least 2. 
\end{prob}

\begin{prob}\label{P:HRQD}
Classify  invariant subvarieties $\cM$ of components of quadratic differentials $\cQ$ with $\rank(\cM)=\rank(\cQ)\geq 2$.  
\end{prob}

One of the invariant subvarieties $\cM$ constructed in \cite{EMMW} is contained in a locus of double covers of a component $\cQ$ of quadratic differentials with $\dim \cM = \dim \cQ -1$ and $\rank \cM = \rank \cQ$, so these last two problems are highly non-trivial. \red The main result of \cite{ApisaWrightHighRank} implies a solution to Problem \ref{P:HRQD} in the case when the stratum has at least six odd order zeros (or poles). \black

\section{Diamonds with quadratic doubles }\label{S:DiamondQ}

The following is the main result of the section. We use the notation $\cF$ from Definition \ref{D:F} for forgetting marked points. 

\begin{thm}\label{P1}
Let $((X,\omega), \cM, \bfC_1, \bfC_2)$ be a generic diamond of Abelian differentials.
Suppose that $\MOne$ and $\MTwo$ are quadratic doubles corresponding to components $\cQ_1$ and $\cQ_2$ of strata of quadratic differentials respectively. 

If $\ColOneTwoX$ is connected, then $\cM$ is a quadratic double.

More generally, making no assumption on the connectedness of $\ColOneTwoX$,  $\For(\cM)$ is either a quadratic double or the locus in a quadratic double that corresponds to the preimage of a codimension one hyperelliptic locus in the corresponding stratum of quadratic differentials. The marked points on $(X, \omega)$ consist of fixed points of the holonomy involution and pairs of points consisting of a non-periodic point and its image under the holonomy involution, with no further constraints.
\end{thm}

See Figure \ref{F:QQnotQdiamond} for an example where $\cM$ is not a quadratic double. 

\begin{rem} 
When $\For(\cM)$ \red corresponds to \black a codimension one hyperelliptic locus, the generic surface in $\For(\cM)$ admits a four-to-one map to a generic surface in a genus zero stratum of quadratic differentials; moreover, $\cQ_1$ and $\cQ_2$ are  hyperelliptic components. 
\end{rem}

\begin{rem}
To recall our standing assumptions, $\cM$ is permitted to have marked points. So are $\cQ_1$ and $\cQ_2$ provided that the marked points are free. Since $\MOne$ and $\MTwo$ are quadratic doubles, our definition requires that all preimages of a free marked point be marked. The preimage of a pole may or may not be marked. 
\end{rem}

\begin{rem}
Recall that all surfaces are assumed to be connected unless otherwise stated. In particular, the surfaces in $\MOne$ and $\MTwo$ are connected. However, it is not necessarily the case that the surfaces in $\ColOneTwo(\cM)$ are connected.

Recall also that, following Convention \ref{CV:NoTrivHol}, quadratic differentials are assumed to have non-trivial holonomy.  
\end{rem}

We will defer the proof of Theorem \ref{P1} until after the proof of the following related result.

\begin{thm}\label{T:P1}
Suppose that $\cM$ is an invariant subvariety in a component $\cQ$ of a stratum of quadratic differentials. Let $((X,q), \cM, \bfC_1, \bfC_2)$ be a generic diamond where $\MOne$ and $\MTwo$ are components $\cQ_1$ and $\cQ_2$ of strata of quadratic differentials respectively. Then $\For(\cM)$ is either $\For(\cQ)$ or a codimension one hyperelliptic locus in $\For(\cQ)$. The marked points on surfaces in $\cM$ are free. 

Moreover, if $\ColOneTwo(X, q)$ has nontrivial linear holonomy then $\cM$ is a component of a stratum of quadratic differentials. 
\end{thm}
\begin{proof}
We begin with the following reduction. 
Recall the types of cylinders from Definition \ref{D:CylinderTypes}.

\begin{sublem}\label{SL:MarkedPoints}
Suppose that $\bfC_1$ and $\bfC_2$ are both singletons containing either a complex cylinder or a complex envelope. Then the cylinder in $\bfC_i$ does not border a marked point in $(X, q)$, nor does $\Col_{\bfC_{i+1}}(\bfC_i)$ border one in $\Col_{\bfC_{i+1}}(X, q)$. Moreover, all marked points on $(X, q) \in \cM$ are free. 
\end{sublem}
\begin{proof}
Since the diamond is generic, specifically by Definition \ref{D:GenericDiamond} (\ref{E:genericSE}), the saddle connections on the boundary of $\Col_{\bfC_{i+1}}(\bfC_i)$ are generically parallel to $\Col_{\bfC_{i+1}}(\bfC_i)$. Since all marked points on $\Col_{\bfC_{i+1}}(X, q)$ are free and since each boundary of $\Col_{\bfC_{i+1}}(\bfC_i)$ consists of either two saddle connections or a single saddle connection joining two poles, it follows that the boundary of $\Col_{\bfC_{i+1}}(\bfC_i)$ does not contain any marked points.

We will now see that the boundary of $\bfC_i$ does not contain any marked points. This follows since if $z$ were such a point, then $\Col_{\bfC_{i+1}}(z)$ would be a marked point contained in the boundary of $\Col_{\bfC_{i+1}}(\bfC_i)$, which cannot occur.

Therefore, if $p$ is a marked point on $(X, q)$ it lies in the complement of the closures of the cylinders in $\bfC_1$ and $\bfC_2$. Therefore, $p$ is free since $\ColOne(p)$ is free on $\ColOne(X, q)$.
\end{proof}

We now  reduce to a special case of the situation considered in the previous sublemma. 

\begin{sublem}\label{SL:QQComplex}
If $\bfC_1$ or $\bfC_2$ is not a complex cylinder, then $\For(\cM)$ is either $\For(\cQ)$ or a codimension one hyperelliptic locus in $\For(\cQ)$ and the marked points on surfaces in $\cM$ are free. 

If additionally $\ColOneTwo(X, q)$ has nontrivial linear holonomy, then $\cM$ is a component of a stratum of quadratic differentials. 
\end{sublem}
\begin{proof}
By Definition \ref{D:GenericDiamond} (\ref{E:genericSE}), $\Col_{\bfC_{i+1}}(\bfC_i)$ is a subequivalence class of generic cylinders on $\Col_{\bfC_{i+1}}(X, q)$, and so, by Masur-Zorich (Theorem \ref{T:MZ}), $\Col_{\bfC_{i+1}}(\bfC_i)$ is one of the following: a simple cylinder, a simple envelope, a half-simple cylinder, a complex cylinder, or a complex envelope. The same is true for $\bfC_i$.

Suppose first, perhaps after re-indexing, that $\bfC_1$ is one of the first three types. Then $\cM = \cQ$ by Corollary \ref{C:codim1}.

If, again perhaps after re-indexing, $\bfC_1$ is a complex envelope, then $\For(\cM)$ is either $\For(\cQ)$ or a codimension one hyperelliptic locus in it by Theorem \ref{T:complex-gluing0}. Moreover, by Sublemma \ref{SL:MarkedPoints}, all the marked points on $(X, q)$ are free. 

Finally, we suppose $\ColOneTwo(X, q)$ has nontrivial holonomy and prove $\cM$ is a component of a stratum. It suffices to consider the case when $\bfC_1$ is a complex envelope. Applying Remark \ref{R:TwoDegenerations} to the complex envelope $\ColTwo(\bfC_1)$, we see that that $\ColOneTwo(\bfC_1)$ consists of two distinct saddle connections that have poles as endpoints, since otherwise $\ColOneTwoX$ has trivial holonomy. Hence we get that $\ColOne(\bfC_1)$ consists of two saddle connections. The final statement of Theorem \ref{T:complex-gluing0}, applied to the complex envelope $\bfC_1$, now gives that $\cM$ is a component of a stratum.
\end{proof}

Continuing with the proof of Theorem \ref{T:P1}, we may now assume that each $\bfC_i$ is a single complex cylinder. By Masur-Zorich (Theorem \ref{T:MZ} (\ref{T:MZ:TrivHolComponent})), if $\bfC$ is a generic complex cylinder on a quadratic differential $(Y, q_Y)$ then $\Col_{\bfC}(Y, q_Y)$ has trivial linear holonomy. Therefore, we are now exclusively in a situation where $\ColOneTwo(X, q)$ has trivial linear holonomy, which implies that we have established the final claim of Theorem \ref{T:P1}.

Suppose that Theorem \ref{T:P1} is false and assume that $((X,q), \cM, \bfC_1, \bfC_2)$ is a counterexample with the dimension of $\cM$ as small as possible.  
By Sublemma \ref{SL:MarkedPoints} we may suppose without loss of generality that $(X, q)$ has no marked points. We will get a contradiction in one of two ways: 
\begin{itemize}
\item in some cases, by considering the result of degenerating two different simple cylinders, we will argue that actually $\cM$ isn't a counterexample to Theorem \ref{T:P1} (this is the $n=3$ case below), or
\item in some cases the assumption that $\cM$ has smallest dimension will lead to a  precise description of surfaces in a degeneration of $\cM$, and an easy numerology argument will give a contradiction (this in the $n=4$ case below). 
\end{itemize} 

Since  $\Col_{\bfC_{i+1}}(\bfC_i)$ is a complex cylinder, for each $i$, $\ColOneTwo(\bfC_i)$ is a pair of saddle connections on $\ColOneTwo(X, q)$. Observe that 
$$(X, q) - \left( \bfC_1 \cup \bfC_2 \right) \quad\text{and}\quad  \ColOneTwo(X, q)- \left( \ColOneTwo(\bfC_1)\cup \ColOneTwo(\bfC_2)  \right) $$
are almost equal. (The left surface is the metric completion of the right surface; the right surface is ``missing its boundary saddle connections".) In particular, they have the same number of components. Let $n$ denote the number of components.  

\begin{sublem}\label{SL:Discussion}
Either $n=4$, or $n=3$ and there is one component bounded by $\ColOneTwo(\bfC_1)$, one component bounded by $\ColOneTwo(\bfC_2)$, and one component bounded by both.
\end{sublem}

\begin{proof}
Since $\ColOneTwo(X, q)$ has trivial holonomy, so must each of the $n$-components; so each of the $n$ components has at least two boundary saddle connections. After cutting four saddle connections in $\ColOneTwo(\bfC_1)\cup \ColOneTwo(\bfC_2)$, there are 8 boundary saddle connections on the resulting surface with boundary. So $n \leq 8/2$. 

We proceed in two cases. Suppose first that both saddle connections in $\ColOneTwo(\bfC_2)$ lie on the same component of $\ColOneTwo(X,q)- \ColOneTwo(\bfC_1)$. In this case, we see that there is a component of $\ColOneTwo(X, q)- \left( \ColOneTwo(\bfC_1)\cup \ColOneTwo(\bfC_2)  \right)$ that is bounded entirely by $\ColOneTwo(\bfC_1)$, and also a component that is bounded entirely by $\ColOneTwo(\bfC_2)$. Since there must be at least one other component, we get the result. 

Suppose next that $\ColOneTwo(\bfC_2)$ has one saddle connection on each component of $\ColOneTwo(X,q)- \ColOneTwo(\bfC_1)$. Since the two saddle connections in $\ColOneTwo(\bfC_2)$ are generically parallel, it follows that $\ColOneTwo(\bfC_1)$ and $\ColOneTwo(\bfC_2)$ are generically parallel to each other. (For example, if $\ColOneTwo(\bfC_1)$ is horizontal, one can shear each component of $\ColOneTwo(X,q)- \ColOneTwo(\bfC_1)$ individually to see that $\ColOneTwo(\bfC_2)$ must also be horizontal.) 

Hence the four saddle connections in $\ColOneTwo(\bfC_1)\cup \ColOneTwo(\bfC_2)$ are pairwise homologous (since $\MOneTwo$ is a stratum of Abelian differentials, generically parallel saddle connections are homologous). So cutting them disconnects the surface into four components. 
\end{proof}

Let $(\Sigma_j)_{j=1}^n$ denote the translation surfaces with boundary that comprise the components of $(X, q) - \left( \bfC_1 \cup \bfC_2 \right)$. 

For our next observation we will assume that on $(X,q)$, all parallel saddle connections are generically parallel (this assumption holds on a full measure set of $\cM$).

\begin{sublem}\label{SSL}
Whenever $\Sigma_j$ has two boundary saddle connections, it contains a simple cylinder $C$, and $\cM_{C}$ is a proper invariant subvariety in a stratum of connected quadratic differentials.

If either $C\neq \Sigma_j$, or $C= \Sigma_j$ and both boundary components of $C$ are boundary saddle connections of the same $\bfC_i$, then
$$(\Col_{C}(X,q), \cM_{C},  \Col_{C}(\bfC_1), \Col_{C}(\bfC_2)  )$$ 
is a generic diamond and both $(\cM_{C})_{\Col_{C}(\bfC_i)}$ are components of strata of quadratic differentials.  
\end{sublem}
\begin{proof}
Whenever $\Sigma_j$ has two boundary saddle connections, we can glue those two boundary saddle connections to get a single saddle connection on a translation surface $\Sigma_j^{glue}$ without boundary.

Every translation surface has a cylinder disjoint from any given saddle connection. Applying this to $\Sigma_j^{glue}$, we get a cylinder $C$. Our genericity assumption implies $C$ must be simple, since generically all cylinders on Abelian differentials are simple. Since $C$ is simple, if $\cM_{C}$ was a full component of a stratum of quadratic differentials, $\cM$ would be as well; hence $\cM_C$ is proper. 

Now assume  either that $C\neq \Sigma_j$ or that $C= \Sigma_j$ and both boundary components of $C$ are boundary saddle connections of  the same $\bfC_i$. This assumption is precisely what is required to ensure that $\Col_{C}(\bfC_1)$ and $\Col_{C}(\bfC_2)$ do not share boundary saddle connections, which is a requirement of diamonds. (If $C=\Sigma_j$ had one boundary component from $\bfC_1$ and one from $\bfC_2$, then $\Col_{C}(\bfC_1)$ and $\Col_{C}(\bfC_2)$ would both have $\Col_C(C)$ in their boundary.)

It is now easy to see that collapsing $C$ gives rise to another generic diamond. 
%
%
Since $\cM_{\bfC_i}$ is a component of a stratum, so is $(\cM_{C})_{\Col_{C}(\bfC_i)}= (\cM_{\bfC_i})_{\Col_{\bfC_i}(C)}$.
\end{proof}

\noindent \textbf{Case 1: $n = 3$.}

If $n = 3$, then by Sublemma \ref{SL:Discussion} there are two values of $j \in \{1, 2, 3\}$ for which $\Sigma_j$ has exactly two boundary saddle connections. Suppose that these subsurfaces are $\Sigma_1$ and $\Sigma_2$. See Figure \ref{TwoComplexCylsN3}.

\begin{figure}[h]\centering
\includegraphics[width=\linewidth]{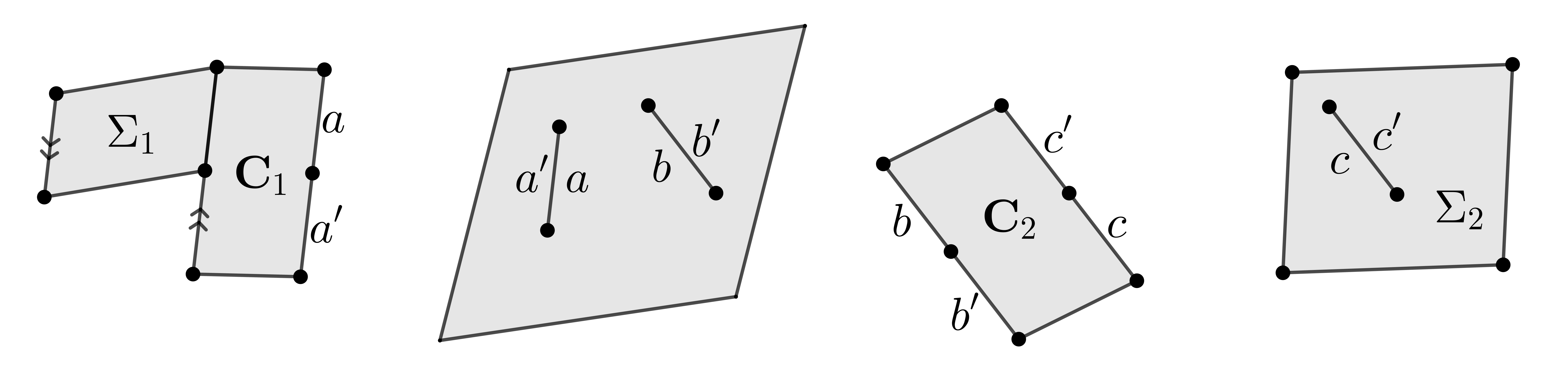}
\caption{An example with $n=3$ that must be ruled out. The equations $a=a'$ and $b=b'$ are all equivalent, and a priori might locally define a counterexample to Theorem \ref{T:P1}.}
\label{TwoComplexCylsN3}
\end{figure}

Applying Sublemma \ref{SSL} twice, we find simple cylinders $D_1$ and $D_2$ contained in $\Sigma_1$ and $\Sigma_2$ respectively. If necessary re-index so that the boundary saddle connections of $\Sigma_i$ belong to $\bfC_i$ for $i \in \{1, 2\}$.

\begin{sublem}
$\For(\cM_{D_i})$ is a codimension one hyperelliptic locus and, letting $J_i$ denote the hyperelliptic involution on $\For(\Col_{D_i}(X, q))$, $J_i$ fixes $\Col_{D_i}(\bfC_j)$ for $i$ and $j$ in $\{1, 2\}$.
\end{sublem}
\begin{proof}
By symmetry of hypotheses, it suffices to prove the statement for $i = 1$. By Sublemma \ref{SSL},
$$(\Col_{D_1}(X,q), \cM_{D_1},  \Col_{D_1}(\bfC_1), \Col_{D_1}(\bfC_2)  )$$
is a generic diamond. By our minimality assumption on the dimension of $\cM$, Theorem \ref{T:P1} holds for this generic diamond. So, since Sublemma \ref{SSL} gives that $\cM_{D_1}$ is proper, we get that $\For(\cM_{D_1})$  is a codimension one hyperelliptic locus and all marked points in $\For(\cQ_{D_1})$ are free. This proves the first claim. 

$\Col_{D_1}(\bfC_j)$ is a $\cM_{D_1}$-generic complex cylinder or complex envelope for both $j$, and since all marked points are free it cannot have marked points in its boundary. (The possibility of a complex envelope occurs when $D_1=\Sigma_1$.) In either case the hyperelliptic involution must fix both $\Col_{D_1}(\bfC_j)$, since generic genus zero quadratic differentials cannot have complex cylinders or envelopes.
\end{proof}


Since $J_i$ fixes $\Col_{D_i}(\bfC_j)$ it follows\footnote{This is because $\Col_{D_i}(\Sigma_j)$ can be identified with a component of $\Col_{D_i}(X, q) - \Col_{D_i}(\bfC_1 \cup \bfC_2)$. Each cylinder in $\Col_{D_i}(\bfC_1 \cup \bfC_2)$ is fixed by $J_i$ and each component of $\Col_{D_i}(X, q) - \Col_{D_i}(\bfC_1 \cup \bfC_2)$ is uniquely determined by which cylinders in $\bfC_1 \cup \bfC_2$ it borders.} that $J_i$ fixes $\Col_{D_i}(\Sigma_j)$ as well for all $i$ and $j$. Since $(X, q)$ has no marked points, it follows that $J_1$ is an involution on the subsurface $(X, q) - \Sigma_1 = \Col_{D_1}(X, q) - \Col_{D_1}(\Sigma_1)$.

  Notice that $\Col_{D_i}(\Sigma_j)$ is isometric to $\Sigma_j$ when $i \ne j$. Notice that the restriction of $J_1$ and $J_2$ to $\Col_{D_1, D_2}(\Sigma_3)$ are identical since both preserve this subsurface and exchange the two saddle connections on the boundary of this subsurface that lie in the boundary of $\bfC_1$. By defining an involution $J$ that acts by $J_1$ on $(X, q) - \Sigma_1$ and $J_2$ on $(X, q) - \Sigma_2$ we have a well-defined involution since we have already shown that $J_1$ and $J_2$ agree on the overlap (this proof is the same as the argument in the Diamond Lemma (Lemma \ref{L:diamond}), which we could also appeal to here).
  
  It is easy to see that the quotient of $(X, q)$ by this involution is a sphere (for instance since $(X, q)/J$ is a connected sum of three spheres since $\bfC_i/J$ is a simple cylinder and since $\Sigma_i/J$ is a sphere with boundary). This shows that $\cM$ is contained in a hyperelliptic locus. Since $\cM_{D_1}$ has codimension one in the stratum containing it, the same is true of $\cM$ since $(X, q)$ is formed by attaching a simple cylinder to $\Col_{D_1}(X, q)$. Since we have assumed that there are no marked points on $(X, q)$,  it follows that $\cM$ is a codimension one hyperelliptic locus in $\cQ$, which is a contradiction. 

\noindent \textbf{Case 2: $n = 4$.}

By successively applying Sublemma \ref{SSL} and then collapsing the simple cylinder thereby produced we may reduce to a surface $(X, q)$ on which every subsurface in $(\Sigma_j)_{j=1}^4$ is a simple cylinder. We will show that this is not possible. Unlike in the previous case we will not worry about whether collapsing a simple cylinder takes us to a hyperelliptic locus since, regardless of whether it does, we will rule out this case entirely.

Whenever a boundary component of a complex cylinder shares two boundary saddle connections with (not necessarily distinct) simple cylinders, that boundary contains a single singularity and the singularity has order two. This shows that $(X, q)$ belongs to $\cQ(2^4)$, which has rank two rel four by Lemma \ref{L:Q-rank}.


Since the boundary of every cylinder in $\left( \Sigma_j \right)_{j=1}^4$ is glued into the boundary saddle connections of either $\bfC_1$ or $\bfC_2$ we see that all six of these cylinders are parallel to each other. We suppose without loss of generality that they are all horizontal, see Figure \ref{F6}. 
\begin{figure}[h!]\centering
\includegraphics[width=.8\linewidth]{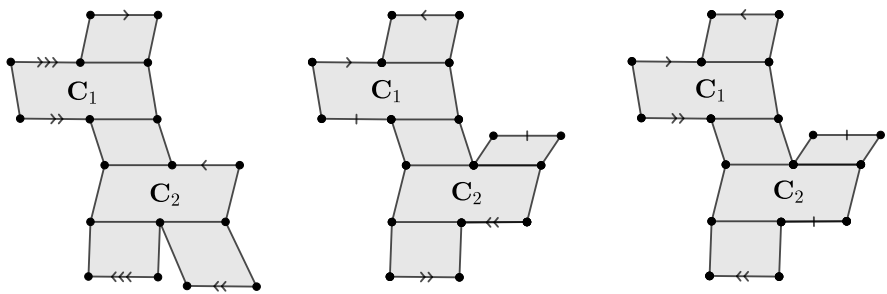}
\caption{Examples of surfaces that are ruled out in the $n=4$ case.}
\label{F6}
\end{figure}
Since $\MOne$ is a component of a stratum, it follows by Theorem \ref{T:BoundaryTangent} that the standard shear for every horizontal cylinder belongs to the tangent space of $\cM$. Under the projection $p$ to absolute cohomology\footnote{Tacitly we will work with the holonomy double cover of $\cM$ so that the tangent space is a subset of cohomology.}, the six dimensional subspace of the tangent space spanned by these shears projects to an isotropic subspace of $p\left( T_{(X, q)} \cM \right)$. This isotropic subspace has dimension at most two since $\cQ(2^4)$ has rank two. This shows that the dimension of $\ker(p) \cap T_{(X, q)} \cM$, which by definition is $\mathrm{rel}(\cM)$,  is at least four. Since $\cQ(2^4)$ has rel four, this shows that $\mathrm{rel}(\cM) = 4$. Therefore, the six dimensional subspace generated by standard shears in horizontal cylinders projects to a two dimensional Lagrangian subspace $p\left( T_{(X, q)} \cM \right)$, which shows, by definition of rank, that $\cM$ has rank at least two. Therefore, $\cM$ must have rank two rel four and hence coincide with $\cQ(2^4)$ contradicting our assumption that it is proper.  
\end{proof}

The idea of the proof of Theorem \ref{P1} is now to use the Diamond Lemma to reduce to Theorem \ref{T:P1}.

\begin{proof}[Proof of Theorem \ref{P1}:]
By perturbing, we may suppose without loss of generality that $(X, \omega)$ has dense orbit in $\cM$ (see Remark \ref{R:CouldBeDense}). Let $f_i: \Col_{\bfC_i}(X, \omega) \ra (Y_i, q_i)$ denote the quotient by the holonomy involution. 

\begin{lem}
$(X, \omega)$ is a double cover of a quadratic differential via a map $f$ satisfying $\Col_{\bfC_i}(f) = f_i$. 
\end{lem}




\begin{proof}
By the Diamond Lemma (Lemma \ref{L:diamond}), it suffices to show that $f_1$ and $f_2$ agree on the base of the diamond and that $f_i^{-1}\left( f_i\left( \overline{\Col_{\bfC_i}(\bfC_{i+1})} \right) \right) = \overline{\Col_{\bfC_i}(\bfC_{i+1})}$. The second of these two conditions holds for any quadratic double since $\Col_{\bfC_i}(\bfC_{i+1})$ is a subequivalence class and $f_i$ is the quotient by the holonomy involution.

The first condition (agreement at the base of the diamond) is slightly more complicated. If $\ColOneTwo(\cM)$ is a locus of connected surfaces, then it is a quadratic double (by Lemma \ref{L:codim1doubles}) with at least one marked point or zero. In this case, by Lemma \ref{L:QHolonomy}, the holonomy involution is unique on $\ColOneTwo(\cM)$ and hence $f_1$ and $f_2$ agree at the base. Suppose therefore that the surfaces in $\ColOneTwo(\cM)$ are disconnected. Then by Lemma \ref{L:QBoundaryHolonomy}, $\ColOneTwo(\cM)$ is the anti-diagonal embedding of a component $\cH'$ of a stratum of Abelian differentials into $\cH' \times \cH'$ and there is still a unique holonomy involution. Hence, in either case, $f_1$ and $f_2$ agree at the base. 
\end{proof}


In order to deduce that $(X, \omega)$ belongs to a quadratic double we must check in particular that for any marked point on $f\left( X, \omega \right)$, both of its preimages under $f$ are marked on $(X, \omega)$. Say that $p$ is a marked point on $(X, \omega)$ and suppose that $p$ is not fixed by the holonomy involution on $(X, \omega)$. 

If $p$ does not belong to the closure of a cylinder in $\bfC_1$, then $p$ remains a marked point, which we denote $\ColOne(p)$, on $\ColOne(X, \omega)$. Since $\ColOne(f) = f_1$, $\ColOne(p)$ is not fixed by the holonomy involution on $\ColOne(X, \omega)$. By assumption, both marked points in $f_1^{-1}(f_1(\ColOne(p)))$ are marked on $\ColOne(X, \omega)$. Since $\ColOne(f) = f_1$, both points in $f^{-1}(f(p))$ are marked on $(X, \omega)$. 

Therefore, we must only consider marked points that belong to the boundary of a cylinder in $\bfC_1$ and to the boundary of a cylinder in $\bfC_2$. However, since the cylinder in $\bfC_1$ and $\bfC_2$ are disjoint, no such marked points exist. Note that a cone point of cone angle greater than $2\pi$ may lie on the boundary of two disjoint cylinders, but a marked point, which is a cone point of cone angle exactly equal to $2\pi$, cannot. This concludes the proof that $\cM$ is contained in a quadratic double (since we have assumed that $(X, \omega)$ has dense orbit this follows by Corollary \ref{C:DenseDiamond}).

Let $\cM'$ denote the $\GL$-orbit closure of $f(X, \omega)$. Phrased differently, $\cM'$ is the collection of surfaces in $\cM$ after modding out by the holonomy involution. Then 
\[ \left( f(X, \omega), \cM', f(\bfC_1), f(\bfC_2) \right) \] forms a generic diamond where $(\cM')_{f(\bfC_i)}$ are components of strata of quadratic differentials for $i \in \{1, 2\}$. The result now follows from Theorem \ref{T:P1}. Notice that $\ColOneTwoX$ is disconnected if and only if $\Col_{f(\bfC_1), f(\bfC_2)}\left( f(X, \omega) \right)$ has trivial linear holonomy. 
\end{proof}

\section{Diamonds of full loci of covers}\label{S:FullLoci}

The goal of this section is to prove Theorem \ref{T:IntroFull}, but we will work in a more general setup, not requiring that $\ColOneTwoX$ or even the $\Col_{\bfC_i}(X,\omega)$ are connected.  Theorem \ref{T:IntroFull} is a special case of the following result. We say that two maps have the same fibers if every fiber of one is a fiber of the other.

The following assumption will be in effect throughout this section. 

\begin{ass}\label{A:FullV2}
Suppose $((X, \omega), \cM, \bfC_1, \bfC_2)$ forms a generic diamond where $\cM_{\bfC_1}$ and $\cM_{\bfC_2}$ are full loci of covers satisfying Assumption CP, and where we allow the covers to be disconnected as long as the surfaces they cover are connected. Let $f_i$ denote the covering map defined on $\Col_{\bfC_i}(X,\omega)$.
\end{ass}

\begin{thm}\label{T:FullV2}
 Under Assumption \ref{A:FullV2},  $\cM$ is a full locus of covers of a stratum of Abelian or quadratic differentials, except possibly if $\ColOneTwoX$ is disconnected and one of the following occurs: 
\begin{enumerate}
\item The codomains of the $\Col(f_i)$ are equal and contained in a hyperelliptic stratum of Abelian or quadratic differentials without marked points or in $\cH(0)$ or $\cH(0,0)$;  the codomain of at least one $f_i$ has non-trivial holonomy; and $\Col(f_1)$ and $\Col(f_2)$ do not have the same fibers but their restrictions to each component do have the same fibers.  
\item  Up to re-indexing the codomain of $\Col(f_1)$ is $\cH(0)$; the codomain of $f_1$ has non-trivial holonomy; and the codomain of $\Col(f_2)$ is $\cQ(-1^4)$. 
\end{enumerate} 
\end{thm}

The following example clarifies the first exceptional case. (We will see that essentially the same phenomenon also gives rise to the second exceptional case.) Additionally, it might be helpful to compare to the discussion before the Diamond Lemma (Lemma \ref{L:diamond}) and keep in mind we do not have a canonical identification of the codomains of $\Col(f_1)$ and $\Col(f_2)$.

\begin{ex}
Let $Y$ be a surface, and let $\tau:Y \to Y$ be a bijection that isn't the identity. Let $Y'$ denote two different copies of $Y$. Let $w_1: Y' \to Y$ be the map that identifies the two components via the identity map from $Y$ to itself, and let $w_2: Y' \to Y$ be the map that identifies the two components via $\tau$. Then $w_1$ and $w_2$ have the same fibers when restricted to each connected component, since each fiber of the restriction is a single point; but $w_1$ and $w_2$ do not have the same fibers. 
\end{ex}

The two exceptional cases of Theorem \ref{T:FullV2} are especially interesting, and we hope they will be studied in the future. We do not have any reason to believe that the conclusion does not hold in the two exceptional cases, but until they are analyzed the possibility remains that they may be associated with surprising new invariant subvarieties.

\subsection{Preliminaries}\label{SS:FullPrelims}

\begin{rem}\label{R:LostCP}
Before beginning our analysis, we first briefly sketch an example of a diamond $\left( (X, \omega), \cM, \bfC_1, \bfC_2 \right)$ where Assumption CP fails for $\cM$, but holds for $\MOne$ and $\MTwo$.  Pick $\cM$ to be a quadratic double where all but exactly two preimages of poles are marked. Let $(X, \omega)$ be a generic surface in $\cM$.  The quotient by the holonomy involution does not satisfy Assumption CP, since an envelope containing these two poles will lift to a cylinder of twice the height. (Envelopes are defined in Definition \ref{D:CylinderTypes}.) Pick the $\bfC_i$ so that $\cM_{\bfC_1}$ and $\cM_{\bfC_2}$ are quadratic doubles with all but exactly one preimage of poles marked, which ensures that Assumption CP does hold for $\cM_{\bfC_1}$ and $\cM_{\bfC_2}$. 
%
%
\end{rem}

We will prove Theorem \ref{T:FullV2}  primarily by applying the following.

\begin{lem}\label{L:VS-Agree}
If $\Col(f_1) = \Col(f_2)$, then $\cM$ is a full locus of covers of a stratum of Abelian or quadratic differentials.
\end{lem}

Recall that we write ``$\Col(f_1) = \Col(f_2)$" as a short form for ``$\Col(f_1)$ and  $\Col(f_2)$ agree at the base of the diamond", in the sense of Section \ref{S:DiamondLemma}.

In the proof of Theorem \ref{T:FullV2}, most of the work will be verifying the assumption in Lemma \ref{L:VS-Agree} that $\Col(f_1) = \Col(f_2)$. There is only one very special situation (the final case in the proof of Lemma \ref{VS-SpecialCase}) where we will have $\Col(f_1) \neq \Col(f_2)$ and hence will have to establish that $\cM$ is a full locus of covers by other arguments.

\begin{proof}[Proof of Lemma \ref{L:VS-Agree}]
By Assumption CP, $$f_i^{-1}\left(f_i\left( \Col_{\bfC_i}(\bfC_j) \right) \right) = \Col_{\bfC_i}(\bfC_j)$$ for any $i \ne j$. By hypothesis, $\Col(f_1) = \Col(f_2)$, so the Diamond Lemma (Lemma \ref{L:diamond}) implies that there is a (half)-translation cover $f: (X, \omega) \ra (Z, \zeta)$ such that $\Col_{\bfC_i}(f) = f_i$. 

Let $\cM'$ denote the orbit closure of $f(X, \omega)$. Then $(\cM', f(X, \omega), f(\bfC_1), f(\bfC_2))$ is a generic diamond where $\cM_{f(\bfC_1)}'$ and $\cM_{f(\bfC_2)}'$ are components of strata of Abelian or quadratic differentials. 

Suppose first that $\cM_{f(\bfC_1)}'$ is a component of a stratum of Abelian differentials and that $\cM_{f(\bfC_2)}'$ is a component of a stratum of quadratic differentials. Since subequivalence classes of cylinders in strata of Abelian differentials are singletons, $\Col_{f(\bfC_1)}(f(\bfC_2))$ is a single simple cylinder. Therefore, $\cM'$ is the result of gluing in a simple cylinder to a component of a stratum of quadratic differentials and so $\cM'$ is itself a component of a stratum of quadratic differentials by Corollary \ref{C:codim1}. This proves the result because $\cM$ is a full locus of covers of $\cM'$.

We may suppose now that $\cM_{f(\bfC_1)}'$ and $\cM_{f(\bfC_2)}'$ are both components of strata of Abelian differentials or both components of strata of quadratic differentials. By Proposition \ref{P:DiamondWithH} and Theorem \ref{T:P1}, $\cM'$ is either a stratum or a full locus of degree two covers of a genus zero stratum. In the first case the map $f$ proves that $\cM$ is a full locus of covers, and in the second case $f$ composed with the hyperelliptic involution proves that $\cM$ is a full locus of covers. 
\end{proof}

The following two basic lemmas are foundational to our analysis.  

\begin{lem}\label{L:ObviousCP}
Both $\Col(f_1)$ and $\Col(f_2)$, whose domains are the surface $\ColOneTwoX$ in $\MOneTwo$, satisfy Assumption CP. 
\end{lem}

\begin{proof}
If Assumption CP fails for a surface in $\MOneTwo$ for $\Col(f_i)$, then it fails in $\cM_{\bfC_i}$ after almost completely collapsing $\Col_{\bfC_i}(\bfC_{i+1})$, since heights of cylinders vary continuously.
\end{proof}

\begin{lem}\label{L:ColConnected}
The codomain of each $\Col(f_i)$ is connected. 
\end{lem}

\begin{proof}
By assumption, $f_i(\Col_{\bfC_i}(\bfC_j))$ is a generic cylinder with respect to the stratum in which $f_i(\Col_{\bfC_i}(X,\omega))$ lies. (Generic cylinders are defined in Definition \ref{D:GenericCyl}.) If it is a stratum of Abelian differentials, every generic cylinder is simple; if it is a stratum of quadratic differentials, the generic cylinders are given by Theorem \ref{T:MZ}. In either case, it is easy to see that collapsing a generic cylinder for a stratum does not disconnect the surface. Hence,  $\Col_{f_i(\Col_{\bfC_i}(\bfC_j))}(f_i(\Col_{\bfC_i}(X, \omega))$ is a connected surface. 
\end{proof}

We wish to apply Lemma \ref{L:VS-Agree}  in situations where we only know that $\Col(f_1)$ and $\Col(f_2)$ have the same fibers after restriction to each component of $\ColOneTwoX$, but first we must clarify a subtle issue with marked points. 

\begin{war} 
In general (without Assumption CP), the restriction of the $\Col(f_i)$ to a component of $\ColOneTwoX$ might not be a covering map in the sense of Definition \ref{D:Covering}, because there may be a marked point in the codomain all of whose preimages on a given component are non-singular unmarked points. (The preimage under $\Col(f_i)$ of a marked point must contain a singular point or a marked point, but possibly not on every connected component of $\ColOneTwoX$.) 
\end{war}

However, we now explain that this warning does not apply under the assumptions of Theorem \ref{T:FullV2}. 

\begin{lem}\label{L:RestrictionOK}
The restriction of $\Col(f_i)$ to each component satisfies Definition \ref{D:Covering}.
\end{lem}

\begin{proof}
By Lemma \ref{L:ObviousCP}, $\Col(f_i)$ satisfies Assumption CP. We may assume the codomain of $\Col(f_i)$ is a generic surface in its stratum, and hence if it contains a marked point $p$ then $p$ lies in the interior of some cylinder $C$ on $\cF\left(\Col(f_i)\left(\ColOneTwoX\right) \right)$. By Assumption CP, each component of the preimage of $C$ must contain a preimage of $p$ that is a marked point or singular point. 
\end{proof}

We can now explain how Lemma \ref{L:VS-Agree} can be applied when only the restrictions of the $\Col(f_i)$ to connected components are understood. 

\begin{lem}\label{L:FirstException}
If the restriction of $\Col(f_1)$ and $\Col(f_2)$ to each  connected component of $\ColOneTwoX$ have the same fibers, then $\cM$ is a full locus of covers  of a component of a stratum Abelian or quadratic differentials, except possibly in  first exceptional case of Theorem \ref{T:FullV2}.
\end{lem}

\begin{proof}
By Lemma \ref{L:ObviousCP}, it suffices to show that, except possibly in  first exceptional case of Theorem \ref{T:FullV2},  $\Col(f_1)=\Col(f_2)$.

The codomain of each $\Col(f_i)$ is a generic surface in a stratum. For each component of $\ColOneTwoX$ we obtain an half-translation map identifying the codomain of $\Col(f_1)$ with the codomain of $\Col(f_2)$, since both are quotients of the same component by the same equivalence relation (points are equivalent if they are in the same fiber). We now make two observations:

\begin{enumerate}
\item By Lemma \ref{L:RestrictionOK}, the identification map (which depends on the choice of a component of $\ColOneTwoX$) between the codomain of $\Col(f_1)$ and the codomain of $\Col(f_2)$ sends marked points to marked points.
\item If the codomain of $f_i$ has trivial holonomy, then the codomain of $\Col(f_i)$ is naturally an Abelian (rather than quadratic) differential, and $\Col(f_i)$ preserves holonomy (sends northward pointing tangent vectors to northward pointing tangent vectors). 
\end{enumerate}

\begin{rem}
In the second observation, it is not enough to assume the codomain of $\Col(f_i)$ has trivial holonomy; in this case we cannot guarantee that $\Col(f_i)$ preserves holonomy. There are two choices of square-roots of a quadratic differential with trivial holonomy. For each component of $\ColOneTwoX$, one can pick a choice of the square-root so that $\Col(f_i)$ restricted to this component preserves holonomy; but there is no guarantee that the choices corresponding to different components are the same.
\end{rem}

We must now show that either the identifications between the codomains of each $\Col(f_i)$ obtained from each component of $\ColOneTwoX$ agree, or that we are in  first exceptional case of Theorem \ref{T:FullV2}. Note that if the identifications do not agree, by composing them we obtain a non-trivial half-translation surface automorphism of the codomain. 

If the codomains of $f_1$ and $f_2$ both have trivial holonomy, then the identification is holonomy preserving by the second point above. The generic element of a stratum of Abelian differentials does not have any translation automorphisms preserving marked points (this follows from Lemma \ref{L:InvolutionImpliesHyp-background} in genus greater than 1, and a direct analysis in genus 1). Hence $\Col(f_1)=\Col(f_2)$. 

The only components of strata where the generic surface has a half-translation automorphism preserving marked points are hyperelliptic connected components of surfaces without marked points, $\cH(0)$, and $\cH(0,0)$. This follows when the rank of a stratum is at least two by Lemma \ref{L:InvolutionImpliesHyp-background}; when the stratum is Abelian and rank one, the stratum must be a quadratic double that coincides with $\cH(0^n)$ for some positive integer $n$, and these strata are precisely $\cH(0)$ and $\cH(0,0)$; when the stratum is quadratic and rank one this is essentially Lemma \ref{L:HypSymmetries}.
\end{proof}

The proof of Theorem \ref{T:FullV2} will now proceed in cases according to the size of $\MOneTwo$. 

\subsection{Higher rank}\label{SS:FullLociHigherRank}

The proof of Theorem \ref{T:FullV2} when $\MOneTwo$ has rank at least 2 will be easy once we have established the following. 

\begin{prop}\label{P:AtMostOneCP}
Let $(X_0, \omega_0)$ be a connected translation surface that isn't a torus cover. Then $(X_0, \omega_0)$ admits at most one map to a generic surface in a stratum of Abelian or quadratic differentials satisfying Assumption CP, in that any two such maps have the same fibers.
\end{prop}

The following example clarifies how crucial Assumption CP is in Proposition \ref{P:AtMostOneCP}. 

\begin{ex}\label{E:ManyCovers}
For any $N$, there is a translation surface that isn't a torus cover but has at least $N$ different maps to generic surfaces in strata of quadratic differentials. Indeed, consider $(Q, q)$ a generic surface in the genus zero stratum $\cQ(N-3, -1^{N+1})$.   Denote the poles by \red $z_1, \ldots, z_{N+1}$. \black For each $i \in  \{1, \ldots, N+1\}$, let $(Q_i, q_i)$ be the unique degree two cover of $(Q,q)$ branched over all of the poles except $z_i$; \red if $N$ is odd it is additionally branched over the zero of order $N-3$. \black As we recall in Section \ref{S:Q-hyp}, Lanneau showed that the $(Q_i, q_i)$ lie in hyperelliptic connected components defined in \cite{LanneauHyp}. Basic covering space theory gives the existence of a translation surface cover of $(Q,q)$ which is a common cover of all the $(Q_i, q_i)$.
\end{ex} 

\begin{proof}[Proof of Proposition \ref{P:AtMostOneCP}.]
By M\"oller (Theorem \ref{T:MinimalCover}), there is a unique half-translation cover $$\pi_{Q_{min}}: (X_0, \omega_0) \ra (Q_{min}, q_{min})$$ such that any other half-translation cover from $(X_0, \omega_0)$ to a half-translation surface is a factor of $\pi_{Q_{min}}$.  

By Lemma \ref{L:InvolutionImpliesHyp-background}, any quotient of $(X_0, \omega_0)$ that is generic in its stratum is either $(Q_{min}, q_{min})$ or a degree two cover of $(Q_{min}, q_{min})$ that lies in a hyperelliptic stratum after forgetting marked points.

\begin{lem}\label{L:piMinCP}
If $\pi_{Q_{min}}$  factors through a map to a surface $S$ \red that is generic in a stratum which, after forgetting marked points, is a hyperelliptic component of a stratum of Abelian or quadratic differentials, \black then $\pi_{Q_{min}}$ does not satisfy Assumption CP. 
\end{lem}

\begin{proof}
Suppose not in order to derive a contradiction. Since $(Q_{min}, q_{min})$ has genus 0, it must have at least two poles. In fact it has at least two poles whose preimage in $S$ are non-singular points: when $S$ is an Abelian differential, we can pick any two poles; and when it is a quadratic differential this follows from the classification of hyperelliptic components reviewed in Section \ref{S:Q-hyp}. 

Deforming if necessary, we can assume that these two poles are contained in an envelope. Considering the preimage in $S$ of this envelope, Assumption CP implies that the preimage of at least one of these two poles is marked; otherwise the preimage of the envelope has at least twice the height of the envelope. However, the preimages of the poles are fixed by the involution, and in particular cannot be free marked points. This contradicts the fact that $S$ is generic, since all marked points must be free in strata.
\end{proof} 

\begin{lem}\label{L:AtMostOneHyp}
If $\pi_{Q_{min}}$  factors through a map $\pi$ satisfying Assumption CP and whose codomain is a surface $S$ that has dense orbit in a stratum and such that $\For(S)$ is contained in a hyperelliptic component of a stratum of Abelian or quadratic differentials, then $\pi$ is the unique such map. 
\end{lem} 
\begin{proof}
The map $S \to (Q_{min}, q_{min})$ is a degree two cover of a sphere, and is hence determined by its branched points. Our approach is to first determine which poles and marked points are branched, and then show that the remaining branching data can be recovered from the classification of hyperelliptic components. 

\begin{sublem}\label{SL:PoleBranching}
The map $S\to (Q_{min}, q_{min})$ is branched over a pole $z$ of $(Q_{min}, q_{min})$ if and only if $\pi_{Q_{min}}^{-1}(z)$ \red does not contain \black a singular point or a marked point. 
\end{sublem}
\begin{proof}
First suppose $S\to (Q_{min}, q_{min})$ is not branched over $z$. We want to show that the preimage of $z$ under $\pi_{Q_{min}}$ contains a singularity or marked point. To do this, as in the proof of Lemma \ref{L:piMinCP}, it suffices to show that the two preimages of $z$ on $S$ are contained in an envelope. This envelope can be constructed by lifting an envelope from $(Q_{min}, q_{min})$ to $S$ that contains $z$ and a branched pole. (Such an envelope exists since there are poles that are branch points for the quotient by the hyperelliptic involution $S \ra (Q_{min}, q_{min})$ and since, for a genus zero surface, it is particularly easy to see that any two poles belong to a single envelope for a generic surface).


Next suppose that $S\to (Q_{min}, q_{min})$ is branched over $z$. The preimage $p$ of $z$ on $S$ is therefore not a singularity of the flat metric. Since $p$ is a fixed point of an involution on $S$, it is a periodic point. 

Since $S$ has dense orbit in a stratum, the only marked points on $S$ are free points. The only stratum in which a free point is simultaneously a periodic point is  $\cH(0)$. Since we have assumed that $(X_0, \omega_0)$ is not a torus cover, we see that $p$ is not marked on $S$. 

By our conventions on (half)-translation covers (see Definition \ref{D:Covering}) we get that $\pi_{Q_{min}}^{-1}(z)$ consists entirely of non-singular unmarked points, as desired. 
\end{proof}

\begin{sublem}\label{SL:MPBranching}
The map $S\to (Q_{min}, q_{min})$ is branched over a marked point $z$ of $(Q_{min}, q_{min})$ if and only if $\pi_{Q_{min}}^{-1}(z)$ consists entirely of singularities (zeros of positive order). 
\end{sublem}

\begin{proof}
If the map $S\to (Q_{min}, q_{min})$ is branched over $z$, then $z$ has a single preimage on $S$, which is a cone point of angle $4\pi$. Since $S\to (Q_{min}, q_{min})$ is a factor of $\pi_{Q_{min}}$, we see that $\pi_{Q_{min}}^{-1}(z)$ consists entirely of cone points with cone angles that are multiples of $4\pi$.

So suppose the map $S\to (Q_{min}, q_{min})$ is not branched over $z$; hence the preimage of $z$ on $S$ consists of two non-singular points. We claim that exactly one preimage of $z$ is a marked point. Indeed, at least one preimage must be marked by Definition \ref{D:Covering}. And if both were marked, then since these two points are interchanged by the hyperelliptic involution, and since $S$ is generic in its stratum, and since all marked points are free in strata, this gives a contradiction. (This uses that the stratum of $S$ is not $\cH(0,0)$, which is true because $(X_0, \omega_0)$ is not a torus cover.) 

Hence exactly one preimage of $z$ is marked on $S$; let $z'$ be the preimage that is unmarked. The preimages on $(X_0, \omega_0)$ of $z'$ must consist of unmarked non-singular points by  Definition \ref{D:Covering}, giving the result. 
\end{proof}
 
The classification of hyperelliptic components (see Lemma \ref{L:HyperellipticCheatSheet} for the quadratic case) gives that $(Q_{min}, q_{min})$ has either one or two zeros (that aren't poles or marked points). The case where there are no zeros only occurs when $\For(Q_{min}, q_{min})$ belongs to $\cQ(-1^4)$, which cannot occur here since $(X_0, \omega_0)$ is not a torus cover. If $(Q_{min}, q_{min})$ has two zeros (that aren't marked points or poles), the classification gives that there is a unique degree two cover of $(Q_{min}, q_{min})$ that is contained in a hyperelliptic component after forgetting marked points. (The double covers of $(Q_{min}, q_{min})$ that belong to hyperelliptic components are uniquely specified, when there are two zeros, since, whether or not a zero is a branch point is determined by its parity, as in Lemma \ref{L:HyperellipticCheatSheet}; all the poles are branch points; and none of the marked points are branched points. When there are fewer than two zeros there may be choices for which poles and marked points belong to the branch locus as in Example \ref{E:ManyCovers}.)

So assume there is exactly one zero (that isn't a pole or marked point). Since the number of branch points of a degree two cover is even, Sublemmas \ref{SL:PoleBranching} and \ref{SL:MPBranching} exactly determine the branch points of $S \to (Q_{min}, q_{min})$. Since a degree two cover of a sphere is determined by its branch points, this gives the result. 
\end{proof}

We will now see that the proof of Proposition \ref{P:AtMostOneCP} is complete. Suppose first that $\pi_{Q_{min}}$ satisfies Assumption CP. By Lemma \ref{L:piMinCP}, there is no factor of $\pi_{Q_{min}}$ with image $S$ that has dense orbit in a stratum and such that $\For(S)$ belongs to a hyperelliptic component of a stratum of Abelian or quadratic differentials. However, any map from $(X_0, \omega_0)$ to a surface $S$ with dense orbit in a stratum is either $\pi_{Q_{min}}$ or has the property that  $\For(S)$ belongs to a hyperelliptic component of a stratum of Abelian or quadratic differentials. Hence, $\pi_{Q_{min}}$ is the unique (half)-translation cover defined on $(X_0, \omega_0)$ that satisfies Assumption CP and whose image has dense orbit in its stratum. 

Suppose now that $\pi_{Q_{min}}$ does not satisfy Assumption CP, but that there is some (half)-translation cover from $(X_0, \omega_0)$ to a surface $S$ satisfying Assumption CP such that $S$ has dense orbit in a stratum. As before, $\For(S)$ belongs to a hyperelliptic component of a stratum of Abelian or quadratic differentials, and $(Q_{min}, q_{min})$ is the quotient by the hyperelliptic involution. Hence $\pi$ is unique by Lemma \ref{L:AtMostOneHyp}.
%
\end{proof}

\begin{proof}[Proof of Theorem \ref{T:FullV2} when $\MOneTwo$ has rank at least 2]
The cover $\ColOneTwoX$ might not be connected, but each of its components covers the connected surface $$\Col_{f_i(\Col_{\bfC_i}(\bfC_j))}(f_i(\Col_{\bfC_i}(X, \omega)).$$

We first remark that rank is defined in exactly the same way for invariant subvarieties of multi-component surfaces. So the rank of $\MOneTwo$ is equal to the rank of the invariant subvariety of connected surfaces covered by surfaces in $\MOneTwo$. (One could take this as the definition of the rank of $\MOneTwo$ for the purposes of this section. Although we do not require this fact, we remark that rank is an integer even for invariant subvarieties of multicomponent surfaces by the results in \cite[Section 7]{ChenWright}.) 

 Hence we can assume that none of the components of $\ColOneTwoX$ are torus covers (by Lemma \ref{L:R1Arithmetic}).

Proposition \ref{P:AtMostOneCP} gives that the restrictions of $\Col(f_1)$ and $\Col(f_2)$ to any connected component of $\ColOneTwoX$ have the same fibers. Hence the result follows from Lemma \ref{L:FirstException}. 
\end{proof}

\subsection{Rank 1, not dimension 2}

\begin{prop}\label{P:VS-One}
If $\MOneTwo$ is rank one, then either  $\Col(f_1) = \Col(f_2)$ or, possibly after re-indexing, the codomain of $f_1$ is an Abelian differential, and we can write $\Col(f_i)=g_i \circ g$, where $$g: \ColOneTwoX \ra (Y, \eta)$$ is a translation surface covering map to a flat torus with three or four marked points that differ by two torsion and no other marked points, $g_1$ is the unique four-to-one translation surface covering map from $(Y,\eta)$ to a surface in $\cH(0)$, and $g_2$ is the quotient by an involution fixing the marked points (see Figure \ref{F:g1g2}), except possibly in the two exceptional cases of Theorem \ref{T:FullV2}. 
\end{prop}

\begin{figure}[h]\centering
\includegraphics[width=.7\linewidth]{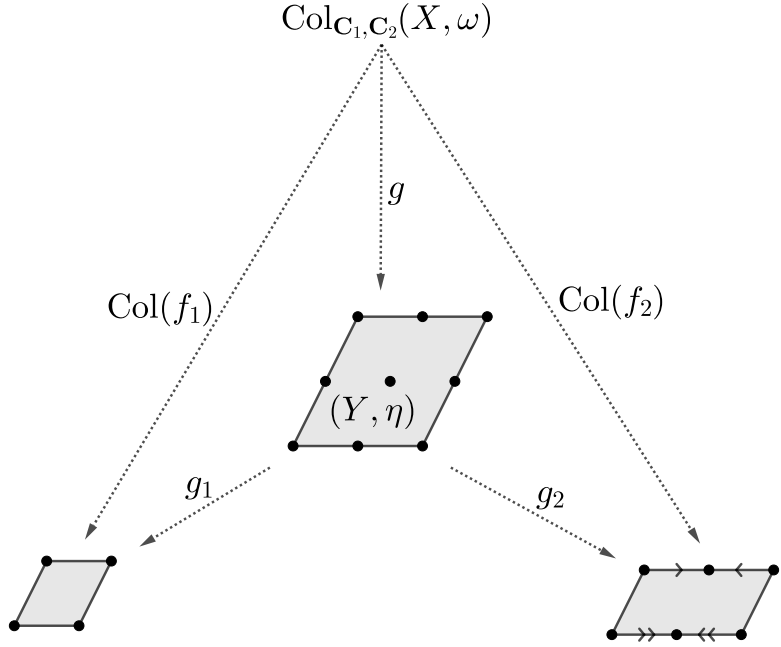}
\caption{The exceptional case in Proposition \ref{P:VS-One}. Up to reindexing, the codomain of $\Col(f_1)$ is in $\cH(0)$ and the codomain of $\Col(f_2)$ is in $\cQ(-1^4)$.}
\label{F:g1g2}
\end{figure}

Since the orbit closure of a flat torus with three or four two-torsion points marked is rank 1 and dimension 2, Proposition \ref{P:VS-One} and  Lemma \ref{L:VS-Agree} imply Theorem \ref{T:FullV2} when $\MOneTwo$ has rank 1 and dimension greater than 2. 

Before we begin the proof, it is helpful to recall that a connected surface $(Z, \eta)$ is a translation cover of a torus if and only if its $\bZ$-module $\Lambda_{abs}\subset \bC$ of absolute periods is a rank 2 lattice, in which case every translation map to a torus is given by $$p\mapsto \int_{\gamma_p} \omega + \Lambda\in \bC/\Lambda,$$ where $\Lambda$ is a lattice (rank two $\bZ$-submodule of $\bC$) containing $\Lambda_{abs}$, and $\gamma_p$ is a path from a fixed basepoint to $p$.

\begin{proof}
The rank 1 strata of dimension $2+n$ are
$$\cH(0^{n+1}), \cQ(-1^4, 0^{n}), \cQ(2, -1^2, 0^{n-1}), \text{ and } \cQ(2, 2, 0^{n-2}).$$
Here $n+1$ is also the maximum number of horizontal cylinders on a  horizontally periodic surface in these strata; and almost every surface in these strata has $n+1$ cylinders in every cylinder direction. Both $\Col(f_1)$ and $\Col(f_2)$ have a codomain that is one of these strata, where $2+n$ is the dimension of $\MOneTwo$. 

Without loss of generality we may assume that $\ColOneTwoX$ has the maximum number of horizontal cylinders of any surface contained in $\MOneTwo$. Let $(Y_i, \eta_i)$ denote the codomain of $\Col(f_i)$.

\begin{lem}\label{L:LowerRel:SameStratum}
If $n = 0$ and  $(Y_1, \eta_1)$ and $(Y_2, \eta_2)$ belong to different strata, then $\Col(f_1)$ and $\Col(f_2)$ both factor through a map $g$ as \red in the statement of Proposition \ref{P:VS-One}, \black except possibly in the second exceptional case of Theorem \ref{T:FullV2}.
\end{lem}
\begin{proof}
Up to reindexing suppose without loss of generality that $(Y_1, \eta_1) \in \cH(0)$ and $(Y_2, \eta_2) \in \cQ(-1^4)$. Keeping in mind the statement of the second exceptional case of Theorem \ref{T:FullV2}, we may assume that the codomain of $f_1$ has trivial holonomy. 

Let $(Y, \eta)$ be the holonomy double cover of $(Y_2, \eta_2)$.
So the map $$\Col(f_2): \ColOneTwoX\to  (Y_2, \eta_2)$$ factors through a translation surface covering map 
 $$g: \ColOneTwoX \ra (Y, \eta).$$  Thus, $(Y, \eta)$ is a torus, and we mark the image under $g$ of all the singularities and marked points of $\ColOneTwoX$. By Assumption CP, there are either three or four marked points on $(Y, \eta)$, and by construction the difference of any two of these is two torsion.

\begin{sublem}\label{SL:CPH0}
If a connected surface $(Z, \zeta)$ admits a translation map to a surface in $\cH(0)$ satisfying Assumption CP, then this map is unique. 
\end{sublem}
We emphasize that the map is assumed to preserve holonomy.
\begin{proof}
We will show that any such map must be the quotient by the lattice of \red relative \black periods. 

Every map to a torus is the quotient by some lattice $\Lambda\subset \bC$ containing the lattice of absolute periods. If the codomain is in $\cH(0)$, then (in total generality for any translation map to a  torus) $\Lambda$ must in fact contain the lattice of relative periods. 

If the map was the quotient by a lattice strictly containing the lattice of relative periods, the map would non-trivially factor through another map to a surface in $\cH(0)$, and  Assumption CP would not hold. This is because, for any non-identity map from a surface in $\cH(0)$ to a surface in $\cH(0)$, Assumption CP does not hold.
\end{proof}

Define $g_1$ to be the unique degree four translation covering map from $(Y, \eta)$ to a surface in $\cH(0)$. The map $g_1$ satisfies Assumption CP, and the map $g$ satisfies Assumption CP, so we conclude that the map $g_1 \circ g$ satisfies Assumption CP. Hence Sublemma \ref{SL:CPH0}, applied to the restriction of the various maps to the connected components of $\ColOneTwoX$, implies that $\Col(f_1)=g_1 \circ g$.
\end{proof}

\begin{lem}\label{L:HigherRel:SameStratum}
If $n>0$, the codomains of $\Col(f_1)$ and $\Col(f_2)$ lie in the same stratum.
\end{lem}

\begin{proof}
We first claim that if $n>1$ and one codomain lies in $\cH(0^{n+1})$, then so does the other. This will be true even if the codomains of one or both $f_i$ does not have trivial holonomy (in which case we do not know that $\Col(f_i)$ preserves holonomy).

If one codomain lies in $\cH(0^{n+1})$, then if $C$ and $D$ are horizontal cylinders on a component of $\ColOneTwoX$, and the top of $C$ shares a saddle connection with the bottom of $D$, then the top of every cylinder subequivalent to $C$ on the same component shares a saddle connection with the bottom of a cylinder subequivalent to $D$. (Recall that subequivalence classes are defined in Definition \ref{D:subeq}.) Since $n>1$, it is not possible that the bottom of a cylinder subequivalent to $C$ shares a saddle connection with the top of a cylinder subequivalent to $D$ on the same component. 

In contrast, for any pair of adjacent cylinders $C, D$ on the holonomy double cover of a quadratic differential, if the top of $C$ shares a saddle connection with the bottom of $D$, then the opposite is true for $J(C)$ and $J(D)$, where $J$ is the holonomy involution. It follows that for any locus of covers  satisfying Assumption CP of a rank 1 stratum of quadratic differentials, there are subequivalence classes that border each other both top-to-bottom and bottom-to-top on the same component. This proves the first claim. 

The remainder of the proof will proceed by considering the number of self-adjacent subequivalence classes, where we define a self-adjacent subequivalence class to be a subequivalence class that contains a pair of (here automatically distinct) cylinders that share a boundary saddle connection. 

When $n=1$, given a surface in $\cH(0^2)$ with as many horizontal cylinders as possible, every cover has 0 self-adjacent subequivalence classes of horizontal cylinders, whereas covers of such surfaces in  $\cQ(2, -1^2)$ or  $\cQ(-1^4, 0)$ have 1 or 2 respectively, as illustrated in Figure \ref{F:SelfAdjacent}. 
\begin{figure}[h]\centering
\includegraphics[width=.5\linewidth]{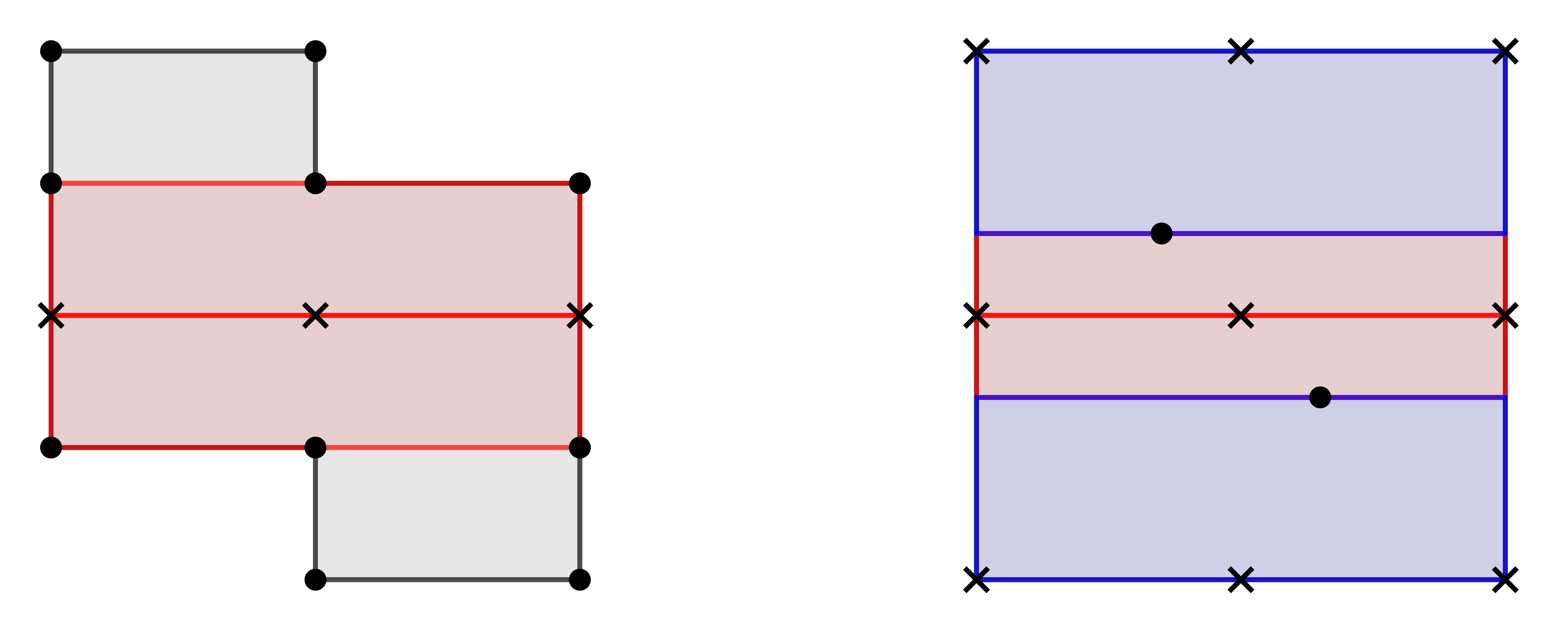}
\caption{Left: The holonomy double cover of a surface in $\cQ(2, -1^2)$, with one self-adjacent subequivalence class highlighted. (At least one of the two points labelled with an x should be marked, so the quotient map satisfies Assumption CP.) Right: The holonomy double cover of a surface in $\cQ(-1^4, 0)$, with two self-adjacent subequivalence classes highlighted. (At least three of the four points labelled with an $x$ should be marked, so the quotient map satisfies Assumption CP.) }
\label{F:SelfAdjacent}
\end{figure}

When $n>1$, covers of $\cQ(2,2,0^{n-2}), \cQ(2, -1^2, 0^{n-1})$ or $\cQ(-1^4, 0^n)$ have 0, 1, or 2 respectively.
\end{proof}

The previous two lemmas show that we can assume that $\Col(f_1)$ and $\Col(f_2)$ have codomains belonging to the same stratum. If we can show that these two maps have the same fibers on each component of $\ColOneTwoX$, then we may conclude with Lemma \ref{L:FirstException}.  We will break up this task according to which stratum the codomain of $\Col(f_i)$ belongs. 

\begin{lem}\label{L:VS-AgreeR1}
Suppose that $(Y_1, \eta_1)$ and $(Y_2, \eta_2)$ both belong to $\cH(0^{n+1})$ or both belong to $\cQ(-1^4, 0^n)$ for some  $n\geq 0$. Then $\Col(f_1)$ and $\Col(f_2)$ have the same fibers on each component of $\ColOneTwoX$.
\end{lem}
\begin{proof}
Without loss of generality assume that $\ColOneTwoX$ is horizontally and vertically periodic with $n+1$ subequivalence classes of cylinders in both those directions. Let $\{x_1, \hdots, x_{n+1} \}$ denote the heights of the cylinders in the vertical subequivalence classes and $\{y_1, \hdots, y_{n+1} \}$ the heights of the ones in the horizontal subequivalence classes.

\bold{Case 1: both $(Y_i, \eta_i)$ belong to  $\cH(0^{n+1})$.} First assume both $\Col(f_i)$ preserve holonomy.

Each subequivalence class of cylinders on one of the $(Y_i, \eta_i)$ must consist of a single cylinder, which by Assumption CP must have the same height as each of the cylinders in the corresponding subequivalence class on $\ColOneTwoX$. Hence, if 
$$\Lambda= \span_\bZ\left( \sum x_k, i \sum y_k\right) \subset \bC,$$
then $(Y_i, \eta_i)= \bC/\Lambda$, and the restriction of $\Col(f_i)$ to any component of $\ColOneTwoX$ must be the quotient by the lattice $\Lambda$. Such quotients are in general well defined up to post-composition with translations, but there are no translations preserving all cylinders on generic surfaces in $\cH(0^{n+1})$. Hence, $\Col(f_1)=\Col(f_2)$. 

Now we confront the possibility that the $\Col(f_i)$ might not preserve holonomy. Even in this case, the restriction of $\Col(f_i)$ to each component must be the above map up to post-composing with rotation by $\pi$, giving the result. 

\bold{Case 2: both $(Y_i, \eta_i)$ belong to  $\cQ(-1^4, 0^n)$.} In this case $\Col(f_i)$ can be written as $\phi_i \circ \tau_i$ where $\phi_i$ is a map from $\ColOneTwoX$ to a torus (the holonomy double cover of $(Y_i, \eta_i)$) and $\tau_i$ is the quotient by the holonomy involution, which is the unique involution that fixes the three or four two-torsion points that are the preimages of poles on $(Y_i, \eta_i)$. Since the $\phi_i$ are maps to holonomy double covers, they are holonomy preserving.

It suffices to show that the restriction of $\phi_1$ and $\phi_2$ to any component of $\ColOneTwoX$ have the same fibers. Since each subequivalence class on the torus (the holonomy double cover of $(Y_i, \eta_i)$) has two cylinders (interchanged by the holonomy involution) the torus is $\bC/2\Lambda$, where $$2\Lambda= \span_\bZ\left( 2\sum x_k, 2 i \sum y_k\right) \subset \bC.$$
This allows us to conclude as above. \red(It is worth noting that the $n=0$ case has an extra subtlety, since the holonomy double cover of surfaces in $\cQ(-1^4)$ with all preimages of poles marked does have translation involutions preserving subequivalence classes. The possibility of post-composing with such a translation involution creates some ambiguity for the map to the holononomy double cover; but after quotienting by the holonomy involution this does not change the fibers of the map.) \black
\end{proof}

\begin{lem}\label{L:HigherRelAgree}
Suppose that $(Y_1, \eta_1)$ and $(Y_2, \eta_2)$ both belong to  $\cQ(2, -1^2, 0^{n-1})$ for some $n > 0$ or both belong to $\cQ(2^2, 0^{n-2})$ for some $n > 1$. Then $\Col(f_1)$ and $\Col(f_2)$ have the same fibers on each component of $\ColOneTwoX$. 
\end{lem}
\begin{proof}
Because we have assumed $\ColOneTwoX$ has $n+1$ horizontal subequivalence classes of cylinders, it follows that $(Y_i, \eta_i)$ and $\For(Y_i, \eta_i)$ also have the maximum number of horizontal cylinders for surfaces in their strata. 

A short argument using Theorem \ref{T:MZ} implies that, for each of the two strata $\cQ(2, -1^2)$ and $\cQ(2^2)$, there is only one possibility for how the horizontal cylinders  are connected to each other on a surface with as many horizontal cylinders as possible; see Figure \ref{F:C2onQandCover} for $\cQ(2, -1^2)$ and Figure \ref{F:Q22} for $\cQ(2^2)$. (In other words, for each of these two strata there is just one \emph{cylinder diagram} with as many cylinders as possible; see, for example, \cite[Section 3]{ANW} for a precise definition.)

\begin{figure}[h]\centering
\includegraphics[width=.25\linewidth]{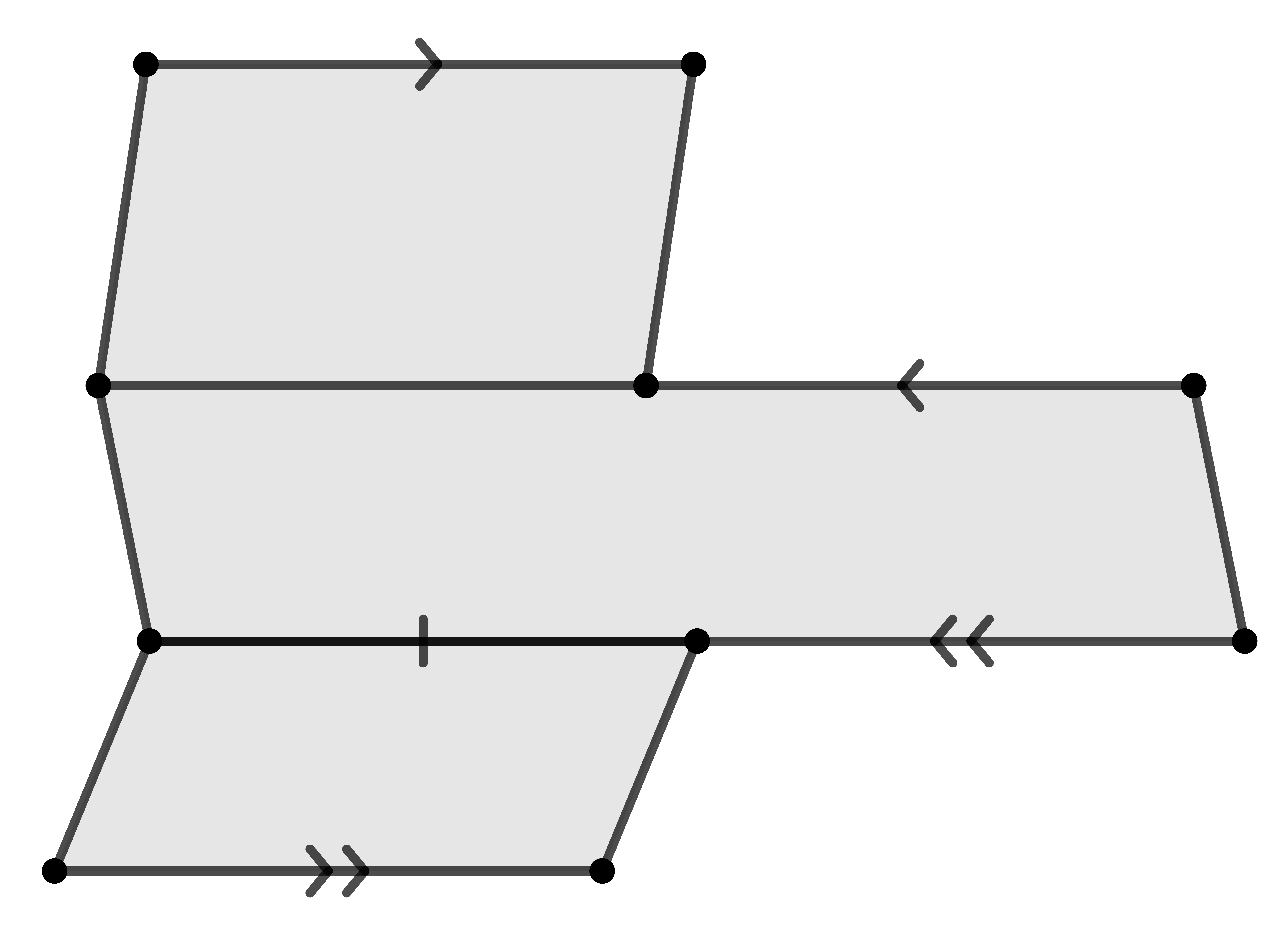}
\caption{A surface in $\cQ(2^2)$.}
\label{F:Q22}
\end{figure}

Notice that on $\For(Y_i, \eta_i)$ there is a unique longest horizontal cylinder; call its length $2\ell_i$. The other horizontal cylinders have length $\ell_i$. 

Let $d_i$ denote the degree of $\Col(f_i)$. It follows that there are two kinds of subequivalence class of horizontal cylinders on $\ColOneTwoX$: those where the sums of the lengths of the core curves of its elements is $d_i \ell_i$  and those where this sum is $2d_i \ell_i$. Call these two kinds of subequivalence classes ``short" and ``long" respectively. Notice that any saddle connection that lies in the boundary of cylinders in both a short and long subequivalence class has length $\ell_i$ for $i \in \{1, 2\}$. This shows that $\ell_1 = \ell_2$ and hence that $d_1 = d_2$.

By Assumption CP, subequivalent cylinders on $\ColOneTwoX$ have identical heights. If $\{h_1, \hdots, h_m\}$ is a list of heights of horizontal long subequivalence classes, then the unique longest horizontal cylinder on $\For(Y_i, \eta_i)$ has height $h_1 + \cdots + h_m$.   

When $(Y_1, \eta_1) \in \cQ(2, -1^2, 0^{n-1})$, if $\{h_{m+1}, \hdots, h_{n+1}\}$ is a list of heights of horizontal short subequivalence classes then the height of the unique shortest horizontal cylinder on $\For(Y_i, \eta_i)$ is $\sum_{j=m+1}^{n+1} h_j$. 

When $(Y_1, \eta_1) \in \cQ(2^2, 0^{n-2})$, there are three horizontal cylinders on $\For(Y_i, \eta_i)$, one of circumference $2\ell$ and two more of circumference $\ell$. The complement of the long cylinder on  $\For(Y_i, \eta_i)$ has two connected components, namely the two short cylinders.  Similarly, the short subequivalence classes of horizontal cylinders on $\ColOneTwoX$ are naturally partitioned into two ``clusters", with subequivalence classes belonging to the same cluster if they contain cylinders in the same component of the complement of the union of the long subequivalence classes. The height of a short cylinder on $\For(Y_i, \eta_i)$ is the sum of the heights of the subequivalence classes on $(Y_i, \eta_i)$ in the corresponding cluster. 

By shearing the horizontal subequivalence classes of cylinders on $\ColOneTwoX$ we can ensure that each subequivalence class contains a cylinder that contains a vertical cross curve, in which case the same is true on $(Y_i, \eta_i)$ and hence also $\For(Y_i, \eta_i)$. Since the circumference ($\ell$ or $2\ell$) and heights (computed as above) of the cylinders on $\For(Y_i, \eta_i)$ are known, we conclude that $\For(Y_1, \eta_1)=\For(Y_2, \eta_2)$. 

We now know that both $\Col(f_i)$ are branched covers of $\For(Y_1, \eta_1)=\For(Y_2, \eta_2)$, although we do not know yet that the branch locus is the same. 

Let $\ColOneTwoX'$ denote a connected component of $\ColOneTwoX$. Our goal is to show that $\Col(f_1)$ and $\Col(f_2)$ have the same fibers on $\ColOneTwoX'$. 

On $\For(Y_1, \eta_1)=\For(Y_2, \eta_2)$, consider a tangent vector $v$ based at a zero (of order 2) that points into a short subequivalence class in a direction parallel to the vertical foliation. For each \red short \black subequivalence class, there are two such $v$, which are exchanged by the hyperelliptic involution.

Let $\hat{v}$ be a preimage of $v$ on $\ColOneTwoX'$ \red(via either map). \black Its base is on the boundary of a long \red subequivalence \black class, and it points into a short \red subequivalence \black class in the cluster corresponding to the cylinder that $v$ points into. Consequently, both $\Col(f_1)$ and $\Col(f_2)$ map $\hat{v}$ to either $v$ or to its image under the hyperelliptic involution. Composing by this involution if necessary, we may assume that both map $\hat{v}$ to $v$. 

Since $\Col(f_1)$ and $\Col(f_2)$ are half-translation coverings, they are equal on a neighborhood of $\hat{v}$. By analytic continuation, they agree on the connected component of $\ColOneTwoX'$ containing $\hat{v}$.  
\end{proof}

\red By Lemmas \ref{L:LowerRel:SameStratum} and \ref{L:HigherRel:SameStratum}, the codomains of $\Col(f_1)$ and $\Col(f_2)$ are the same except possibly in the exceptional case mentioned in Proposition \ref{P:VS-One} or in the second exceptional case of Theorem \ref{T:FullV2}. Lemmas \ref{L:VS-AgreeR1} and \ref{L:HigherRelAgree} imply that when the codomains coincide the restrictions of $\Col(f_1)$ and $\Col(f_2)$ to any component of $\ColOneTwoX$ have the same fibers and hence $\Col(f_1) = \Col(f_2)$ except possibly in  first exceptional case of Theorem \ref{T:FullV2} (by Lemma \ref{L:FirstException}). This concludes the proof of Proposition \ref{P:VS-One}. \black
\end{proof}

\subsection{Rank 1, dimension 2}

The only remaining case is the following, since the other cases when $\MOneTwo$ has rank 1, dimension 2 follow from Proposition \ref{P:VS-One}.

\begin{lem}\label{VS-SpecialCase}
Suppose that the codomain of $f_1$ is an Abelian differential and that we can write $\Col(f_i)=g_i \circ g$, where $$g: \ColOneTwoX \ra (Y, \eta)$$ is a translation cover to a flat torus with three or four two-torsion points marked and no other marked points, $g_1$ is the unique four-to-one translation cover from $(Y,\eta)$ to a surface in in $\cH(0)$, and $g_2$ is the quotient by an involution fixing the marked points (see Figure \ref{F:g1g2}).  Then $\cM$ is a full locus of covers of a stratum of Abelian or quadratic differentials. In fact $\cM$ contains degree $2\deg(g)$ covers of the surfaces in the Prym locus of $\cH(4)$ illustrated in Figure \ref{F:H4Prym}. 
 
 Moreover, this statement holds even if the ``Assumption CP" restriction on $\MTwo$ is replaced by the slightly weaker assumption that $\MTwo$  is a full locus of covers and that $\overline{\ColTwo(\bfC_1)}$ is the full preimage of its image under $f_2$. 
\end{lem}

The final statement is not used in this paper; it follows from the proof of the first statement and will be used in \cite{ApisaWrightGemini}. 

\begin{rem}
The quadratic double of $\cQ(-1^4)$ with three or four preimages of poles marked is the only quadratic double that is a full locus of covers of a stratum of Abelian differentials via maps of degree greater than two that satisfy Assumption CP. This explains the occurrence of this special case (Lemma \ref{VS-SpecialCase}). 
\end{rem}

\begin{proof}
Since it has rank 1 rel 1, $\MOne$ is a full locus of covers of $\cH(0,0)$. Similarly, $\MTwo$ is a full locus of covers of $\cQ(-1^4, 0)$ or $\cQ(2, -1^2)$.

\noindent \textbf{Case 1: $\MTwo$ is a full locus of covers of $\cQ(-1^4, 0)$.}

Since $\bfC_1$ and $\bfC_2$ are disjoint and non-adjacent and since $\overline{\ColTwo(\bfC_1)}$ is the full preimage of its image, $f_2(\ColTwo(\bfC_1))$ and $f_2(\ColTwo(\bfC_2))$ are disjoint and non-adjacent. Therefore, $f_2(\ColTwo(\bfC_2))$ either consists of a single saddle connection joining a pole to another pole or a single saddle connection joining the marked point to a pole. In either case, $f_2(\ColTwo(\bfC_2))$, and hence $\Col(f_2)(\ColOneTwo(\bfC_2))$ consists of exactly one saddle connection (see Figure \ref{F:TwoSC}). It follows that $g(\ColOneTwo(\bfC_2))$ consists of at most two saddle connections. 

\begin{figure}[h]\centering
\includegraphics[width=.7\linewidth]{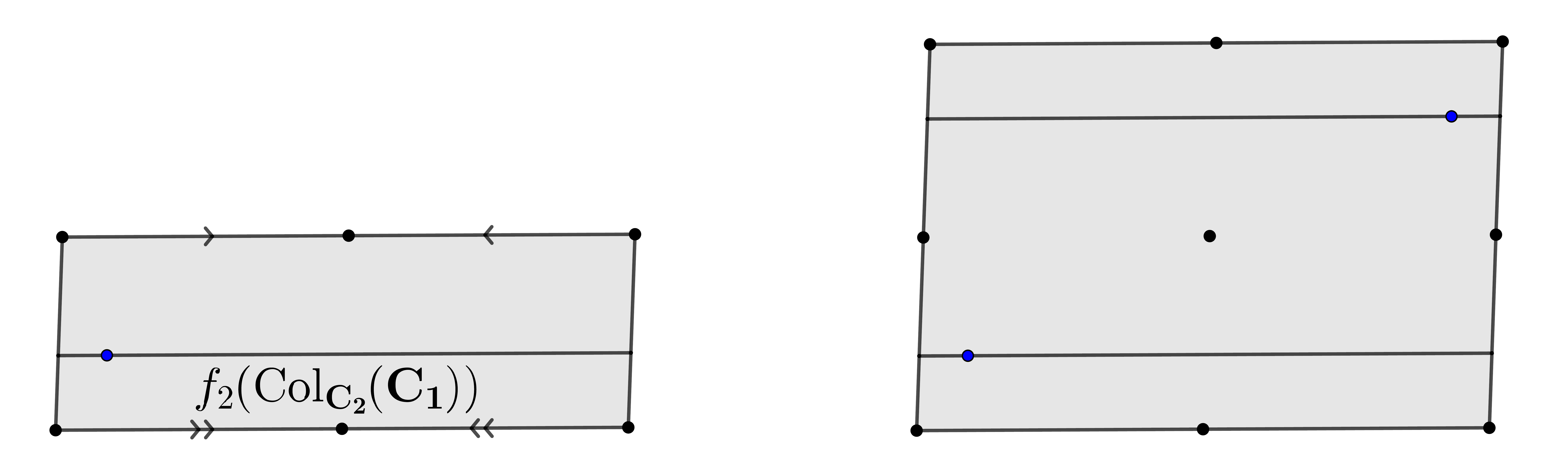}
\caption{On the left is the picture of the quadratic differential $f_2(\ColTwoX)$ and on the right is its holonomy double cover}
\label{F:TwoSC}
\end{figure}

The fact that $g(\ColOneTwo(\bfC_2))$ consists of at most two saddle connections  contradicts the fact that $\ColOne(\bfC_2)$ and hence $\ColOneTwo(\bfC_2)$ is the full preimage of its image under $f_1$ (resp. $\Col(f_1)$). Notice that the preimage of any saddle connection under $g_1$ must contain at least three saddle connections, whereas $g(\ColOneTwo(\bfC_2))$ consists of at most two saddle connections; see Figure \ref{F:ThreeOrFourSC}.

\begin{figure}[h]\centering
\includegraphics[width=.45\linewidth]{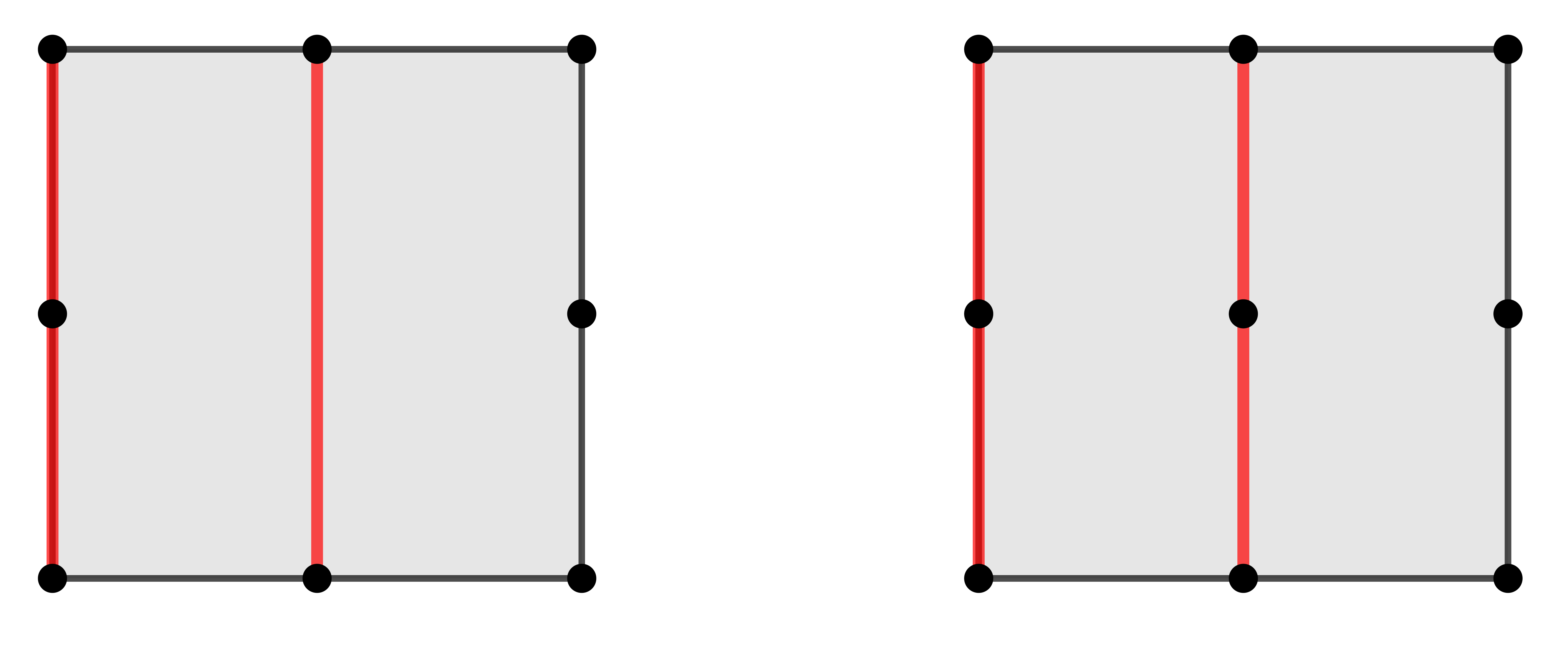}
\caption{The preimage of a saddle connection under $g_1$ has three saddle connections if the $(Y,\eta)\in\cH(0^3)$ (left) or four if $(Y,\eta)\in\cH(0^4)$ (right). }
\label{F:ThreeOrFourSC}
\end{figure}

\noindent \textbf{Case 2: $\MTwo$ is a full locus of covers of $\cQ(2, -1^2)$.}

In $\cQ(2, -1^2)$, every  subequivalence class of generic cylinders is either a single simple cylinder or a single complex envelope. Suppose first that $f_2(\ColTwo(\bfC_1))$ is a complex envelope, for example the larger horizontal cylinder in Figure \ref{F:C2onQandCover} (left). 

Because it is contained in the complement of $f_2(\ColTwo(\bfC_1))$, we see that $f_2(\ColTwo(\bfC_2))$ contains  one saddle collection. We now arrive at a contradiction as in the previous case. Indeed, it follows that $g(\ColOneTwo(\bfC_2))$ consists of at most two saddle connections, contradicting the fact that $\ColOne(\bfC_2)$ and hence $\ColOneTwo(\bfC_2)$ is the full preimage of its image under $f_1$.

\begin{figure}[h]\centering
\includegraphics[width=.7\linewidth]{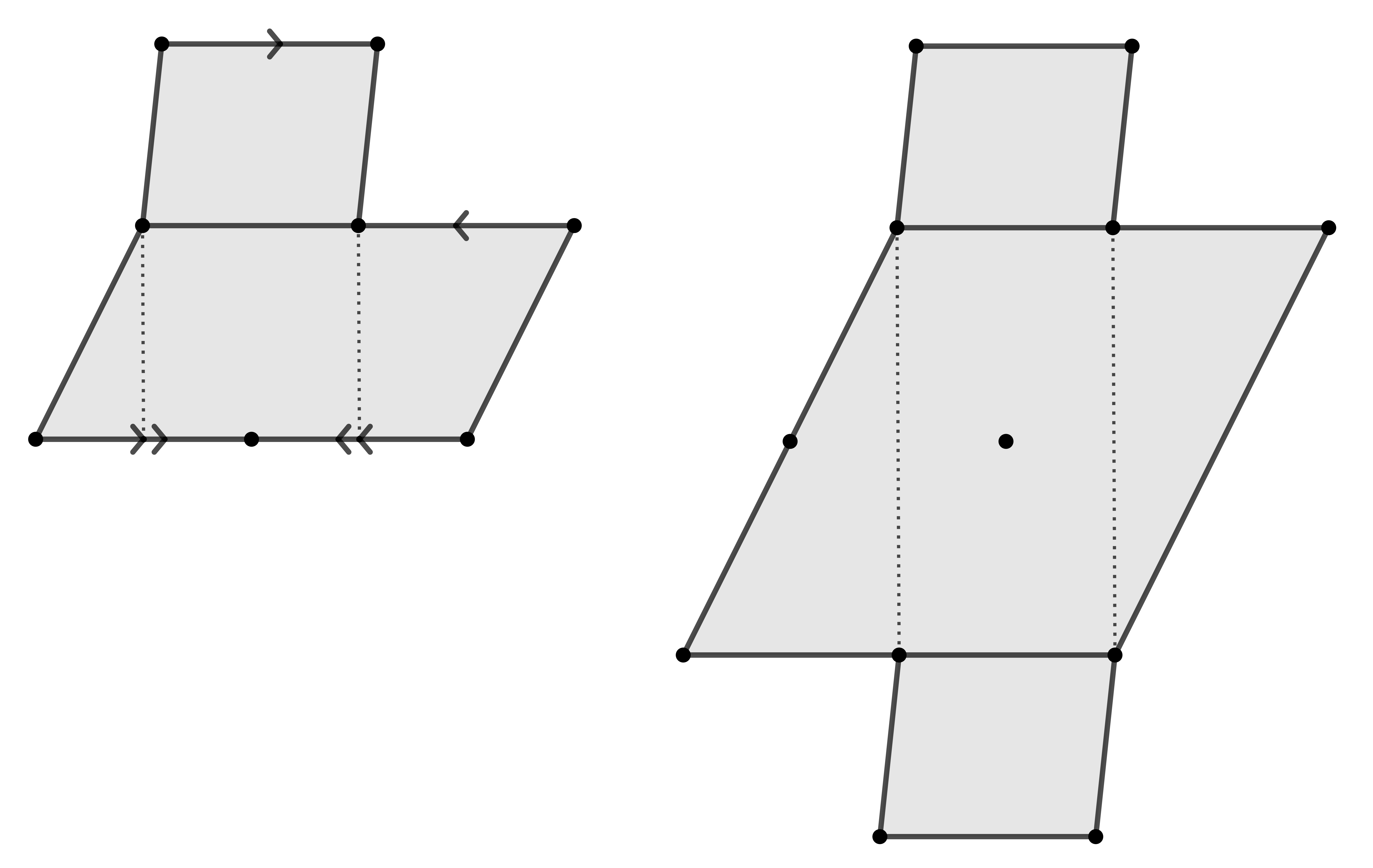}
\caption{On the left is the picture of the quadratic differential $f_2(\ColTwoX)$ and on the right its holonomy double cover. On the left surface, degenerating the longer horizontal cylinder gives a surface in $\cQ(-1^4)$. }
\label{F:C2onQandCover}
\end{figure}

Suppose therefore that $f_2(\ColTwo(\bfC_1))$ is a single simple cylinder, for example the smaller horizontal cylinder in Figure \ref{F:C2onQandCover}. Without loss of generality we suppose that it is horizontal. 

Because $\ColOneTwo(\bfC_2)$ is the preimage of its image under $\Col(f_1)$, it follows  that $g(\ColOneTwo(\bfC_2))$ is a collection of three or four saddle connections, as in Figure \ref{F:ThreeOrFourSC}. Without loss of generality we suppose that these saddle connections are vertical and that \red at least \black two of them have unit length. Moreover, since $\ColOne(\bfC_2)$ is a subequivalence class consisting of cylinders of the same height (by Assumption CP), assume also that the cylinders in $\bfC_2$ have height one.  

\begin{figure}[h]\centering
\includegraphics[width=.35\linewidth]{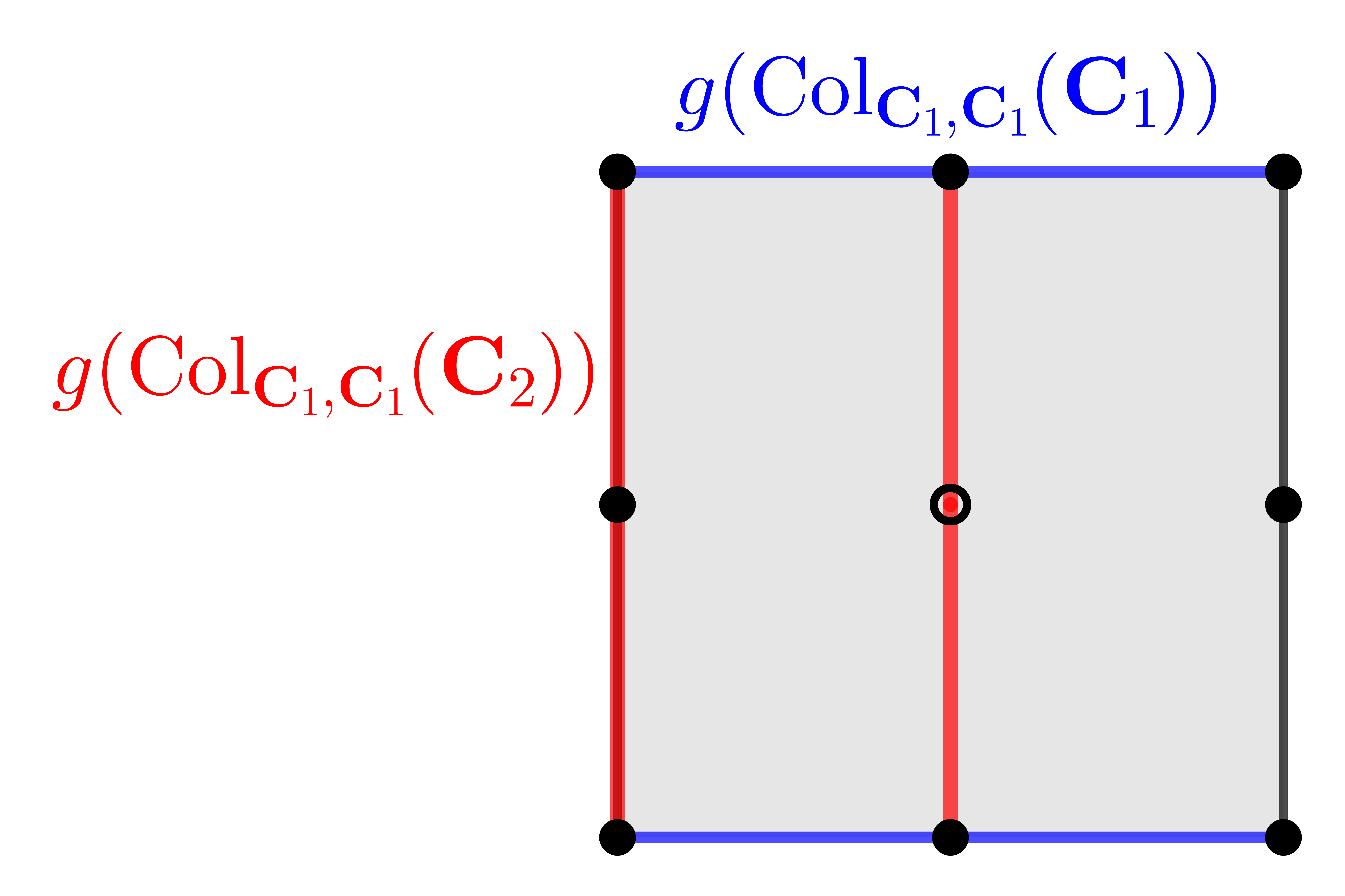}
\caption{$g(\ColOneTwoX)$. The middle point may or may not be marked.}
\label{F:PreGluing}
\end{figure}

At this point we understand that $\ColOneTwoX$ is a cover of the surface in Figure \ref{F:PreGluing} (via the map $g$), and we understand the effect of gluing in either the cylinders in $\bfC_1$ or $\bfC_2$ separately; see Figures \ref{F:H00Cover} and \ref{F:OtherCover}.
\begin{figure}[h]\centering
\includegraphics[width=.55\linewidth]{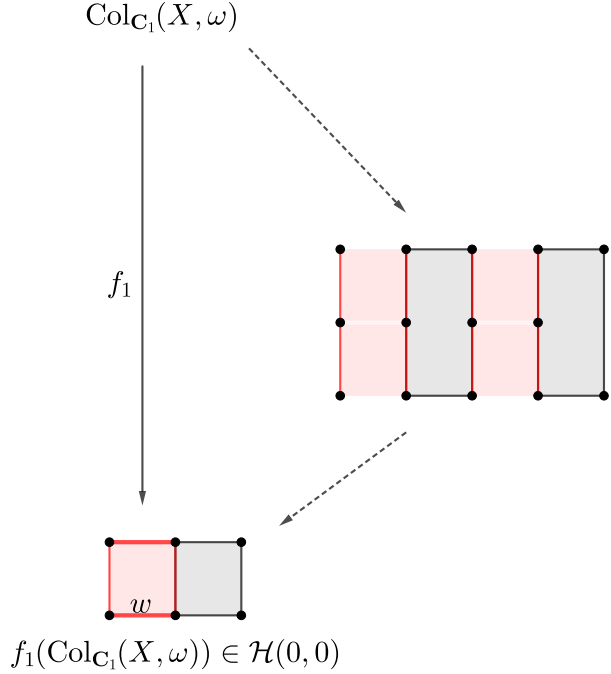}
\caption{If $w$ and its preimage are deleted, the map $f_1$ then factors through the right surface (which is not a closed surface); in this sense we can say $\ColOneX$ ``almost" looks like a cover of the right surface.}
\label{F:H00Cover}
\end{figure}
\begin{figure}[h]\centering
\includegraphics[width=.45\linewidth]{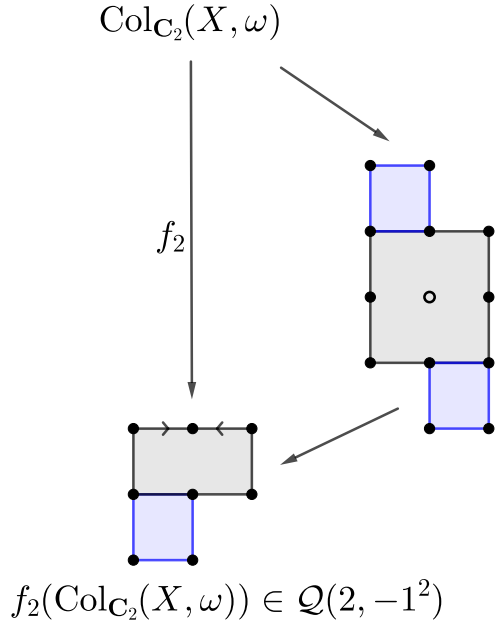}
\caption{The map $f_2$ factors through the holonomy double cover of its codomain.}
\label{F:OtherCover}
\end{figure}
Combining this understanding, we get that $(X,\omega)$ almost looks like a cover of the surface depicted in Figure \ref{F:H4Prym} (left), except that the missing horizontal edges may be glued in such a way that $(X,\omega)$ doesn't cover any surface.
\begin{figure}[h]\centering
\includegraphics[width=.75\linewidth]{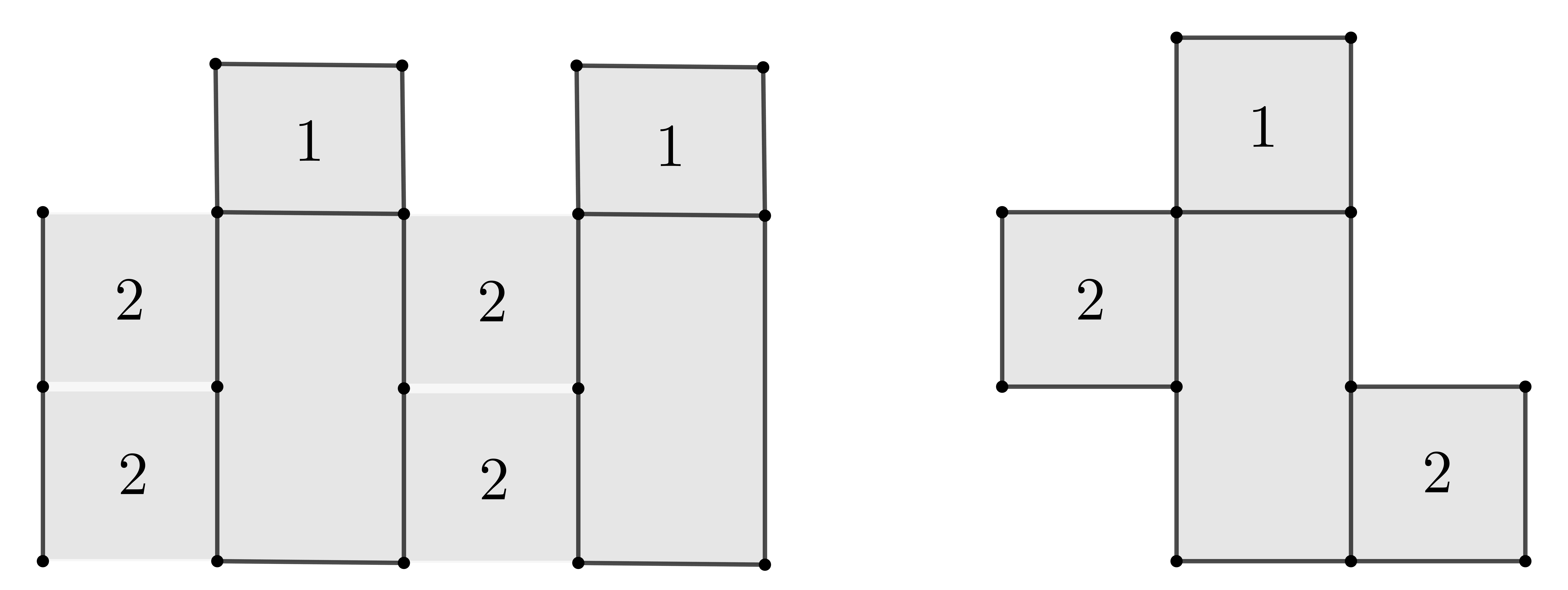}
\caption{Left: The cylinders labelled $1$ correspond to $\bfC_1$, and the rectangles labelled $2$ correspond to $\bfC_2$. Right: A surface in the Prym locus in $\cH(4)$.}
\label{F:H4Prym}
\end{figure}

\begin{sublem}
The surface $(X, \omega)$ is a cover of the surface in the Prym locus in $\cH(4)$ illustrated in Figure \ref{F:H4Prym} (right) if and only if after applying $u := \begin{pmatrix} 1 & 0 \\ 1 & 1 \end{pmatrix}$ to $\bfC_2$ there are no horizontal cylinders in the complement of $\bfC_1$ that border $\bfC_1$ on both their top and bottom boundary.
\end{sublem}
\begin{proof}
\red Let $S$ be the union of the horizontal saddle connections contained in $\bfC_2$. After deleting  $S$, $(X, \omega)$ is a cover of the surface $(W,w)$ depicted in Figure \ref{F:H4Prym} (left). (We emphasize that $(W,w)$ is not a closed surface, since the horizontal segments on the top and bottom of the squares labeled 2 have been deleted.) This allows us to label the components $\bfC_2 - S$ as ``top" or ``bottom" according to their image on $(W,w)$. The desired cover exists if and only if the top components are glued to other top components and similarly for bottom components. \black


\red If this does not occur, then \black after applying $u$ to $\bfC_2$, there will be a horizontal cylinder bordering $\bfC_1$ on its top and bottom boundary. 
\end{proof}

Suppose in order to derive a contradiction that $(X, \omega)$ is not a cover of a surface in a Prym locus in $\cH(4)$. Let $\bfD$ denote the horizontal cylinders in the complement of $\bfC_1$. After perhaps replacing $(X, \omega)$ by a surface where $u$ has been applied to cylinders in $\bfC_2$, we may assume that there is a cylinder in $\bfD$ that borders $\bfC_1$ along its top and bottom boundary. 

We note that since $\bfC_2$ has been sheared, $\ColTwoX$ is now a different surface than the one previously given the same name and which belonged to a cover of a stratum of quadratic differentials. Moreover, the orbit closure $\MTwo$ of $\ColTwoX$ may also have changed; however, it is still a rank one rel one invariant subvariety. We begin with the following sublemma.

\begin{sublem}\label{SL:DIdenticalHeight}
All the cylinders in $\bfD$ have unit height. 
\end{sublem}
\begin{rem}
If all the preimage of all marked points under $f_2$ were marked this would be immediate, but, a priori, shearing $\bfC_2$ may have created horizontal cylinders of height two. Compare to Figure \ref{F:ThreeOrFourSC}.
\end{rem} 
\begin{proof}
Notice that while $\MTwo$ has changed by shearing $\bfC_2$, $\MOne$ has not. Therefore, $\ColOneX$ still admits a map to a surface in $\cH(0,0)$ that satisfies Assumption CP. This proves the claim.
\end{proof}

Let $\Gamma$ denote the directed graph formed as follows. The vertices of $\Gamma$ are cylinders in $\bfD$. There is a directed edge from cylinder $D_1$ to a cylinder $D_2$ if the top boundary of $\ColTwo(D_1)$ includes a saddle connection on the bottom boundary of $\ColTwo(D_2)$.

\begin{sublem}\label{SL:DirectedLoops}
There are no directed loops in $\Gamma$.
\end{sublem}

\begin{proof}
Consider the rel deformation (for instance the linear combination of the standard shears in $\ColTwo(\bfC_1)$ and $\ColTwo(\bfD)$) in $\MTwo$ of $\ColTwoX$ that increases the heights of cylinders in $\ColTwo(\bfC_1)$ and decreases the height of those in $\ColTwo(\bfD)$. 

Given a directed loop of $\Gamma$, we can construct an absolute cycle $\gamma$ on $\ColTwoX$ as follows. Start with a cylinder corresponding to a vertex on the directed loop. Travel along the core curve (in either direction), and then vertically up into the cylinder that appears next in the directed loop. Continue in this way until returning to the starting cylinder, and then travel along the core curve to close up the path to a loop. See Figure \ref{F:H4PrymAbsoluteCycle} for an illustration in the simplest possible case. 

\begin{figure}[h]\centering
\includegraphics[width=.3\linewidth]{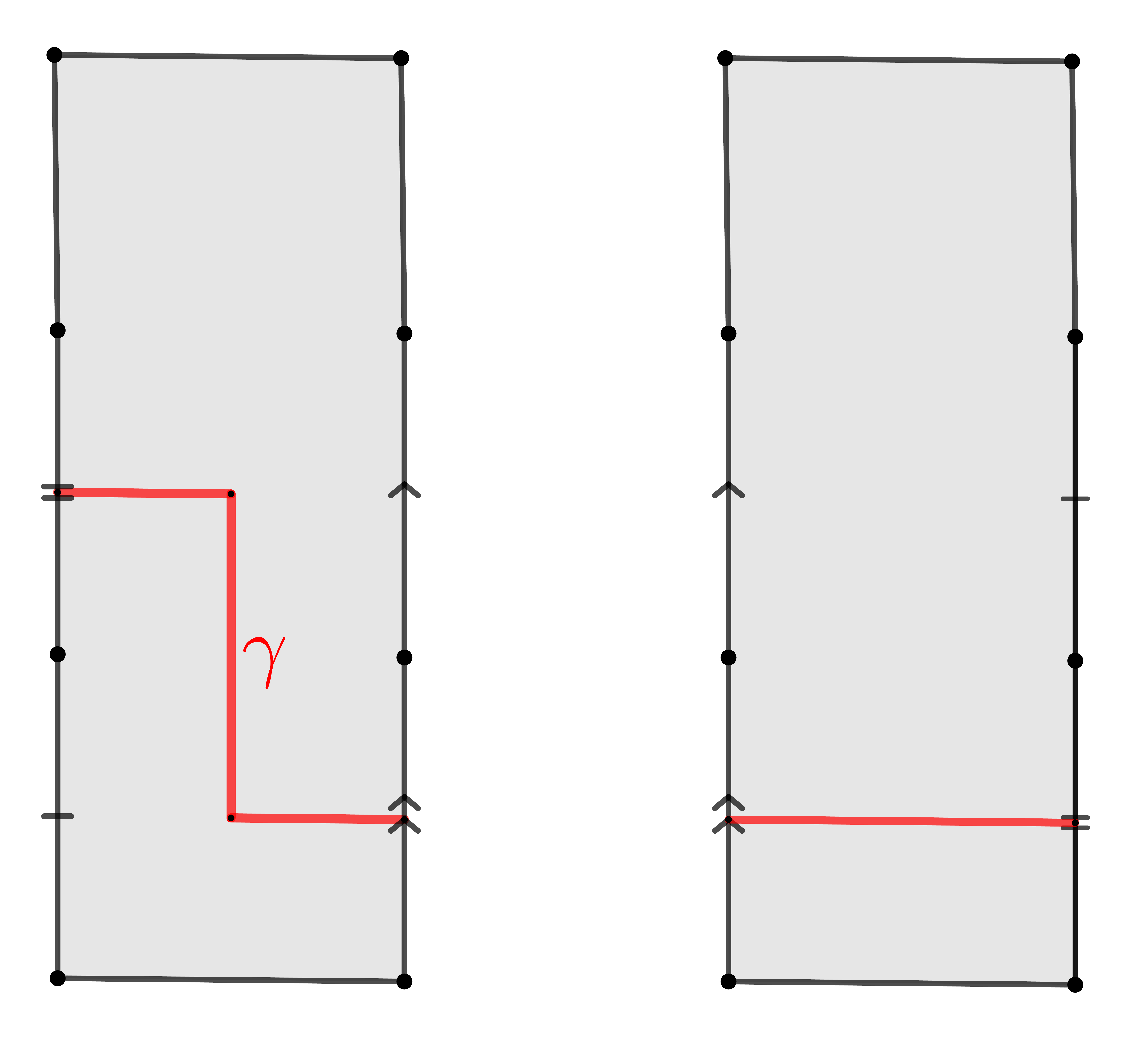}
\caption{The proof of Sublemma \ref{SL:DirectedLoops}, illustrated in a case where the map $g$ has degree 1. In this case $\Gamma$ has only one vertex.}
\label{F:H4PrymAbsoluteCycle}
\end{figure}

The period of the absolute homology class given by $\gamma$ changes under the above rel deformation, giving a contradiction. 
\end{proof}

Therefore, there are cylinders in $\bfD$ that are terminal in $\Gamma$, i.e. that have no outgoing directed edges. We will call these cylinders \emph{terminal cylinders}.

Note that a terminal cylinder only has its top boundary bordering $\bfC_1$. Indeed, if a saddle connection in its bottom boundary bordered $\bfC_1$, going vertically up from this saddle connection would show that the cylinder isn't terminal (this uses Sublemma \ref{SL:DIdenticalHeight}).

The following claim will immediately yield a contradiction. Recall that boundary components of cylinders are defined in Definition \ref{D:CylAndBoundary}.

\begin{sublem}\label{SL:H4PrymOvercollase}
All cylinders in $\bfD$ have exactly one boundary component that borders a cylinder in $\bfC_1$.
\end{sublem}
\begin{proof}
We can see this by ``overcollapsing $\bfC_1$ to attack cylinders in $\bfD$" as follows.

Note first that $\cM$ has rank at least 2 (by Lemma \ref{L:RankTest}) and that $\dim \cM=4$ (by definition of generic diamond, Definition \ref{D:GenericDiamond} \eqref{E:one}), so we see that $\cM$ has rank 2 and rel 0.

Shear $\bfC_1$ so that it does not contain a vertical saddle connection, then vertically collapse it to make the height of the cylinders in $\bfC_1$ zero. Since no zeros have collided at this point, we may continue this vertical collapse deformation, which moves the singularities on the boundary of $\bfC_1$ into the interiors of cylinders in $\bfD$. See Figure \ref{F:H4PrymOvercollapse}.

\begin{figure}[h]\centering
\includegraphics[width=.5\linewidth]{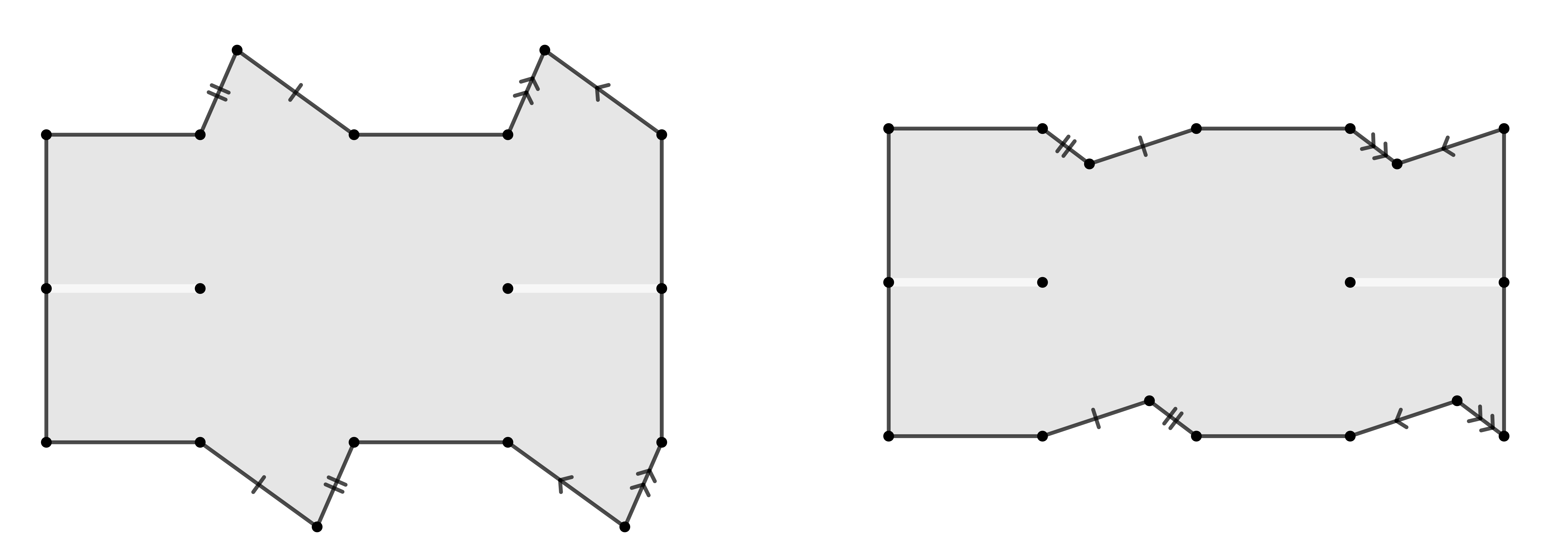}
\caption{The proof of Sublemma \ref{SL:H4PrymOvercollase}.}
\label{F:H4PrymOvercollapse}
\end{figure}

Each terminal cylinder in $\bfD$ borders $\bfC_1$ only along one boundary component. Suppose in order to find a contradiction that there is a cylinder in $\bfD$ that borders $\bfC_1$ along both its top and bottom boundary components. By Sublemma \ref{SL:DIdenticalHeight}, all cylinders in $\bfD$ have identical heights. So, along the the previously described overcollapse deformation, the ratio of moduli of two cylinders under consideration is not constant, which contradicts Lemma \ref{C:ConstantRatio}. 
\end{proof}
This contradicts our assumption that there is a cylinder in $\bfD$ that borders $\bfC_1$ along its top and bottom boundary.
\end{proof}

\subsection{A special case}\label{SS:FullSpecialCase}

We record a slight strengthening of Theorem \ref{T:IntroFull} in a special case where many of the difficulties of the proof do not occur.  

\begin{lem}\label{L:StrongTheorem1.1}
Suppose that $\left( (X, \omega), \cM, \bfC_1, \bfC_2 \right)$ is a generic diamond and that $\MOne$ and $\MTwo$  are full loci of covers of strata of Abelian differentials that satisfy Assumption CP. Then $\cM$ is a full locus of covers of a stratum of Abelian differentials. If $f$ (resp $f_i$) denotes the cover on $(X, \omega)$ (resp. $\Col_{\bfC_i}(X, \omega)$), then \red $\bfC_i = f^{-1}(f(\bfC_i))$ and \black  $\Col_{\bfC_i}(f) = f_i$.
\end{lem}
\begin{proof}
Let $f_i: \Col_{\bfC_i}(X, \omega) \ra (Y_i, \eta_i)$ be the covers that certify that $\cM_{\bfC_i}$ is a full locus of covers satisfying Assumption CP. As noted in Section \ref{SS:FullLociHigherRank}, the surfaces $\Col_{f_i(\Col_{\bfC_i}(\bfC_j))}(f_i(\Col_{\bfC_i}(X, \omega))$ are connected. 

If $\MOneTwo$ has rank at least two, then, working on each component of $\ColOneTwoX$, we see by Theorem \ref{T:MinimalCover} and Lemma \ref{L:InvolutionImpliesHyp-background} that there is a unique translation cover from $\ColOneTwoX$ to a generic surface in a stratum of Abelian differentials. This implies that $\Col(f_1) = \Col(f_2)$.

If $\MOneTwo$ has rank one, then the fibers of $\Col(f_1)$ and $\Col(f_2)$ are the same when restricted to each component of $\ColOneTwoX$ (by Lemma \ref{L:VS-AgreeR1}). Therefore, by a small subset of the proof of Lemma \ref{L:FirstException} (and using the fact that the codomains of $f_1$ and $f_2$ have trivial holonomy, which is not the case in the first exception to Theorem \ref{T:FullV2}) we have that $\Col(f_1) = \Col(f_2)$.

As in Lemma \ref{L:VS-Agree}, there is a translation cover $f: (X, \omega) \ra (Y, \eta)$ and the reduced diamond $\left( (Y, \eta), \cM', f(\bfC_1), f(\bfC_2) \right)$ is a generic diamond where both $\cM'_{f(\bfC_i)}$ are strata of Abelian differentials and hence  both $f(\bfC_i)$ are simple cylinders. This implies that $\cM'$ is a stratum of Abelian differentials where the degrees of the covers with domain $(X, \omega)$ and $\Col_{\bfC_i}(X, \omega)$ are all the same. Since $f$ is constructed using the Diamond Lemma (Lemma \ref{L:diamond}), it is immediate that $\Col_{\bfC_i}(f) = f_i$. 
\end{proof}

\red It seems likely that in this case $f$ satisfies Assumption CP, but we have not checked this. \black

\subsection{Open problems}\label{SS:OpenProblems2}

As previously indicated, we think the following problem is especially interesting: 

\begin{prob}
Determine if the conclusion of Theorem \ref{T:FullV2} holds even in the two exceptional cases. 
\end{prob}

It would be very interesting to prove a version of Theorem \ref{T:IntroFull} without Assumption CP. One of the strongest versions one could hope for would be the following, which replaces Assumption CP with one of the assumptions in Diamond Lemma. 

\begin{conj}[The Strong Diamond Conjecture]\label{C:StrongDiamond}
Suppose $((X, \omega), \cM, \bfC_1, \bfC_2)$ forms a generic diamond where $\cM_{\bfC_1}$ and $\cM_{\bfC_2}$ are full loci of covers, 
with covering maps $f_i$ (as in Section \ref{S:DiamondLemma}) such that $\overline{\Col_{\bfC_i}(\bfC_{i+1})}=f_i^{-1}(f_i(\overline{\Col_{\bfC_i}(\bfC_{i+1})}))$. Then $\cM$ is a full locus of covers.
\end{conj}

We state this as a conjecture to be provocative, and because Theorems \ref{T:IntroFull} and \ref{T:DoubleIntro}, as well as the general scarcity of known orbit closures, provide some evidence for it. But it is plausible that the conjecture might be false, and that efforts to prove it might lead to the discovery of new orbit closures. 

As in the proof of Lemma \ref{L:VS-Agree}, the results in this paper imply that the conjecture holds if  $f_1$ and $f_2$ agree at the base of the diamond. But this does not always hold even in the more restrictive situation of Theorem \ref{T:DoubleIntro}. 

A good first step towards Conjecture \ref{C:StrongDiamond}  would be the following: 
\begin{prob}
Determine if Conjecture \ref{C:StrongDiamond} is true when both $f_1$ and $f_2$ have degree 2. 
\end{prob}
In comparison to Theorem \ref{T:DoubleIntro}, the main difference would be that, in the new setting, fibers of $f_1$ or $f_2$ could consists of one marked point and one unmarked point. 

In a different direction, it would be interesting to \red analyze diamonds where the surfaces in $\MOne$ and $\MTwo$ cover disconnected surfaces. In this setting one might a priori arrive in a situation where the projection of $\cM_{\bfC_i}$ to any component is a full locus of covers (and hence is ``trivial") but $\cM_{\bfC_i}$ itself is not. 
This possibility is \black closely related to the following definition and conjecture, \red which indicate that in fact we do not believe this is possible. \black 

Fix $n>1$. For each $i\in \{1, \ldots, n\}$, let $\cS_i$ be a component of a stratum of Abelian or quadratic differentials. For simplicity we will assume there are no marked points, although one desires a statement with marked points allowed as well. Define a \emph{quasi-diagonal} in $\prod \cS_i$ to be a prime invariant subvariety 
$$\cM \subset \prod \cS_i$$
whose projection to each factor is dominant. Here prime means that $\cM$ is not a product; see \cite{ChenWright}. 

Say that a quasi-diagonal is \emph{trivial} if one of the following holds. 
\begin{enumerate}
\item For all $i,j$, we have $\cS_i = \cS_j$. At each point of $\cM$ all $n$ components are equal up to rotation by $\pi$ \red and rescaling. \black
\item For all $i,j$, either $\cS_i = \cS_j$, or one of $\{\cS_i, \cS_j\}$ is a hyperelliptic component, and the other is the associated genus zero stratum. At each point of $\cM$, all $n$ components are equal up to rotation by $\pi$, \red rescaling, \black and quotienting by the hyperelliptic involution. 
\end{enumerate}

\begin{conj}[The Quasi-diagonal Conjecture]\label{C:QuasiDiagonal}
All quasi-diagonals of rank at least two are trivial. 
\end{conj}

If this conjecture could be proved as stated, a version with marked points would follow immediately from the results of \cite{ApisaWright}. The connection with diamonds is that $\MOne$ or $\MTwo$ might consist of covers of disconnected surfaces, where the set of disconnected surfaces being covered is a quasi-diagonal.    Conjecture \ref{C:QuasiDiagonal} is also related to joinings of Masur-Veech measures.

\section{Hyperelliptic components}\label{S:Q-hyp}

In this section we will summarize facts about hyperelliptic components of strata. To start, let us restate Lemma \ref{L:InvolutionImpliesHyp-background} for convenience.

\begin{lem}\label{L:InvolutionImpliesHyp}
The generic element of a component $\cQ$ of a stratum of
Abelian or quadratic differentials admits a non-bijective half-translation
cover to another translation or half-translation surface if and only if $\For(\cQ)$ is hyperelliptic, in which case the hyperelliptic involution
yields the only such map when $\cQ$ has rank at least two.
\end{lem}


\subsection{Hyperelliptic components of quadratic differentials}\label{SS:Hyp}
 
 Fix $a, b \geq -1$, and consider a surface $(S, q)\in\cQ(a, b, -1^{a+b+4})$ with a zero of order $a$ labelled and a zero of order $b$ labelled. (This labelling is not required if $a$ and $b$ are distinct and not $-1$. We allow $a=0$ or $b=0$, which correspond to marked points.) 
 
The fundamental group of $S$ minus the set of singular points and marked points is generated by loops around the punctures. Let $\phi_{hyp}$ be the map from the fundamental group to $\bZ/2$ that maps the loops around the $a+b+4$ unlabelled poles to 1, maps the loop around $a$ to $a+1 \mod 2$, and maps the loop around $b$ to $b+1 \mod 2$. Let $\phi_{hol}$ be the the map that sends the loop around each singularity to the order of the zero mod 2; this is the holonomy representation of $(S,q)$. 

Let $(X,\omega)$ be the regular $\bZ/2 \times \bZ/2$ cover of $(S,q)$ corresponding to $\ker(\phi_{hyp})\cap \ker(\phi_{hol})$; it is easy to see that this cover has trivial holonomy. Let $J$ denote the involution of $(X,\omega)$ whose quotient is the cover of $(S,q)$ corresponding to $\ker(\phi_{hyp})$, and let $T$ denote the involution whose quotient corresponds to $\ker(\phi_{hol})$. By construction, $J$ and $T$ commute. 
We immediately obtain the following, where we define
$$a^{(1)} =
  \begin{cases}
                                  \{a, a\} & \text{ $a$ odd} \\
                                  \{2a+2\} & \text{ $a$ even} 
  \end{cases}, \quad
 a^{(2)} =
  \begin{cases}
                                  \{a+1\} & \text{ $a$ odd} \\
                                  \{\frac{a}{2}, \frac{a}{2}\} & \text{ $a$ even} 
  \end{cases}$$
for any integer $a\geq -1$. 

\begin{lem}\label{L:HyperellipticCheatSheet}
As indicated in Figure \ref{F:HyperellipticCheatSheet},
\begin{itemize}
\item $(X, \omega) \in \cH(a+1, a+1, b+1, b+1),$
\item $(X,\omega)/J\in \cQ(a^{(1)}, b^{(1)}),$
\item $(X, \omega)/JT \in \cQ(2a+2, 2b+2, -1^{2a+2b+8}),$
\item $(X, \omega)/T \in \cH(a^{(2)}, b^{(2)}).$
\end{itemize}
  \begin{figure}[h!]\centering
\begin{tikzpicture}[baseline= (a).base]
\node[scale=.9] (a) at (0,0){
\begin{tikzcd}
 & \cH(a+1, a+1, b+1, b+1) \arrow[dr, dash, "\operatorname{mod}  T"] \arrow[dl, dash, swap, "\operatorname{mod} J" ] \arrow[d, dash, "\operatorname{mod} JT"] &  \\
\cQ^{hyp}(a^{(1)}, b^{(1)})  \arrow[dr, dash] & \cQ(2a+2, 2b+2, -1^{2a+2b+8}) \arrow[d, dash] & \cH(a^{(2)}, b^{(2)}) \arrow[dl, dash]  \\
& \cQ(a, b, -1^{a+b+4}) &  
\end{tikzcd}
};
\end{tikzpicture}
\caption{
  The set $a^{(i)}$ gives the orders of the preimages of the zero of order $a$, and similarly for $b^{(i)}$. \textbf{Warning:} when $a=-1$ (resp. $b=-1$), then marking the points corresponding to $a+1$ and $2a+2$ (resp. $b+1$ and $2b+2$) is sometimes optional. }
  \label{F:HyperellipticCheatSheet}
\end{figure}
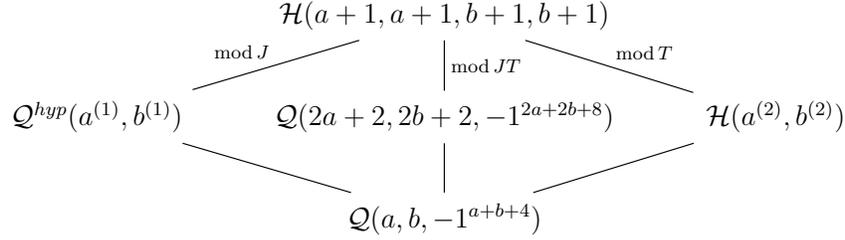
In particular, 
\begin{enumerate}
\item $(X,\omega)/T$ is the holonomy double cover of $(S,q)$, 
\item $(X,\omega)$ is the holonomy double cover of $(X,\omega)/J$, 
\item $(X, \omega)/JT$ has genus zero, so $JT$ is a hyperelliptic involution. 
\end{enumerate}
\end{lem}

As indicated in Figure \ref{F:HyperellipticCheatSheet}, $(X, \omega)/J$ is contained in a hyperelliptic connected component, and moreover  Lanneau \cite[Definition 1]{LanneauHyp} proved that all hyperelliptic components of strata of quadratic differentials are obtained in this way. More precisely, Lanneau proved that any surface in a hyperelliptic component of a stratum of quadratic differentials can be obtained as the cover of a surface in $\cQ(a, b, -1^{a+b+4})$ for some $a, b\geq -1$ corresponding to $\phi_{hyp}$. 

We now make some basic uniqueness observations, which are very intuitive but nonetheless require verification. 

\begin{lem}\label{L:HypSymmetries}
Suppose that $\cQ := \cQ^{hyp}(a^{(1)}, b^{(1)}) \ne \cQ(-1^4)$. Then the maps in $\{ \mathrm{id}, J, T, JT \}$ are the only affine self-maps of derivative $\pm \mathrm{Id}$ for a generic surface in a quadratic double of $\cQ$. In particular, a generic surface in a quadratic double of $\cQ$ has a unique hyperelliptic involution.
\end{lem}
The hyperelliptic involution is $JT$, which was proven above to be a hyperelliptic involution. 

\begin{rem}\label{R:HypSymmetry-SpecialCase}
When $(X, \omega)$ belongs to quadratic double of $\cQ(-1^4)$ and precisely two preimages of poles are marked on $(X, \omega)$ the result continues to hold if the affine symmetries are required to preserve the collection of marked points. An affine symmetry has derivative $\pm \mathrm{Id}$ and may either fix or exchange the two marked points on $(X, \omega)$. These four possibilities each correspond to an affine symmetry.
\end{rem}

\begin{proof}
Let $(X, \omega)$ be a generic surface in a quadratic double of $\cQ$, let $G = \{\mathrm{id}, J, T, JT\}$, and let $H$ be the collection of all affine self-maps of derivative $\pm \mathrm{Id}$. 

As seen in Figure \ref{F:HyperellipticCheatSheet}, $(X, \omega)/G$ is a generic surface in a genus zero stratum $\cQ' = \cQ(a, b, -1^{a+b+4})$. If $G \ne H$, then there is a non-bijective half-translation cover $(X, \omega)/G \ra (X, \omega)/H$. By Lemma \ref{L:InvolutionImpliesHyp}, $\For(\cQ')$ is hyperelliptic and hence equal to $\cQ(-1^4)$, which is the only hyperelliptic genus zero stratum. This shows that $a, b \in \{0, -1\}$ and so $\cQ = \cQ(2,-1^2)$ or $\cQ = \cQ(2,2)$.

In both cases, $(X, \omega)$ contains a subequivalence class $\bfC$ consisting of a pair of isometric simple cylinders. By perturbing (using the standard dilation in $\bfC$) we may suppose without loss of generality that there are no cylinders on $(X, \omega)$ that are parallel and isometric to those in $\bfC$. Therefore, every element of $H$ preserves the cylinders in $\bfC$. Since every element of $H$ is a holomorphic self-map of $X$, two elements of $H$ are equal if their action on $\bfC$ is identical. The only possible actions of an element of $H$ on $\bfC$ is that it fixes or exchanges the two cylinders and has derivative $\pm \mathrm{Id}$. These four possibilities correspond to the four elements of $G$, showing that $G = H$ as desired. 
\end{proof}

We now provide a proof of Lemma \ref{L:QHolonomy} that we previously deferred. We restate the lemma here for convenience. 

\begin{lem}\label{L:QHolonomy-redux}
Let $(X, \omega)$ be a generic surface in a quadratic double of a component $\cQ$ of a stratum of quadratic differentials. If $\For(\cQ) \ne \cQ(-1^4)$, then there is a unique involution $J$ of derivative $-\mathrm{Id}$ such that $(X, \omega)/J$ is a generic surface in a component of a stratum of quadratic differentials. 

If $\For(\cQ) = \cQ(-1^4)$ and $(X, \omega)$ has at least one marked point, then there is a unique marked-point preserving involution $J$ of derivative $-\mathrm{Id}$ such that $(X, \omega)/J$ is a generic surface in a component of a stratum of quadratic differentials.
\end{lem}
\begin{proof}
Suppose in order to derive a contradiction that $(X, \omega)$ is a generic surface in a quadratic double of $\cQ$ where $\For(\cQ) \ne \cQ(-1^4)$ and where there are two involutions, $J_1$ and $J_2$, of derivative $-\mathrm{Id}$ such that $(X, \omega)/J_i$ is a generic surface in a component of a stratum of quadratic differentials. Suppose that $(X, \omega)/J_1$ is a generic surface in $\cQ$. It admits a non-bijective half-translation cover, $(X, \omega)/J_1 \ra (X, \omega)/\langle J_1, J_2 \rangle$ and so (by Lemma \ref{L:InvolutionImpliesHyp}) $\For(\cQ)$ is a hyperelliptic component. 
Since $\For(\cQ) \ne \cQ(-1^4)$ there are exactly two affine self-maps of $(X, \omega)$ of derivative $-\mathrm{Id}$ (by Lemma \ref{L:HypSymmetries}). Using the notation of Lemma \ref{L:HypSymmetries} we will denote these involutions by $J$ and $JT$ and note that there are integers $a$ and $b$ both greater than or equal to $-1$ such that $\For(X, \omega)/J$ belongs to $\cQ^{hyp}(a^{(1)}, b^{(1)})$ and $\For(X, \omega)/JT$ belongs to $\cQ(2a+2, 2b+2, -1^{2a+2b+8})$. Since $\For(X, \omega)/JT$ is a generic genus zero surface that admits a map to a surface in $\cQ(a, b, -1^{a+b+4})$ we have by Lemma \ref{L:InvolutionImpliesHyp} that $a = b = -1$ (the only hyperelliptic genus zero stratum is $\cQ(-1^4)$). This contradicts the assumption that $\For(\cQ) \ne \cQ(-1^4)$.

Suppose now that $(X, \omega)$ belongs to a quadratic double of $\cQ$ where $\For(\cQ) = \cQ(-1^4)$. This implies that $(X, \omega)$ is a flat torus equipped with an involution $J$ of derivative $-\mathrm{Id}$ that preserves the marked points on $(X, \omega)$ and where $J$-invariance is the only constraint on these marked points. Suppose that $J'$ is also an involution of derivative $-\mathrm{Id}$ that preserves the set of marked points and such that $(X, \omega)/J'$ is a generic surface in a component of a stratum of quadratic differentials. We wish to show that $J = J'$.

If $(X, \omega)$ contains a slope $-1$ irreducible pair of marked points then it is possible to move these marked points while fixing all others. This shows that $J'$ would have to preserve this pair of points and hence $J' = J$. If $(X, \omega)$ contains no slope $-1$ irreducible pair of marked points, then all of its marked points are fixed points of $J$. If $J \ne J'$ then the fixed points of $J'$ are disjoint from the fixed points of $J$ and so the fixed points of $J$ map to periodic marked points on $(X, \omega)/J'$ contradicting the fact that this surface is generic in a stratum of quadratic differentials.
\end{proof}

\subsection{Simultaneous quadratic and Abelian doubles}\label{SS:simul} 
\begin{lem}\label{L:QuadAbDouble}
An invariant subvariety $\cM$ is both a quadratic and an Abelian double if and only if it is a quadratic double of $\cQ^{hyp}(a^{(1)}, -1^2)$ where $a \geq -1$ is an integer; if $a=-1$ then either the surfaces in $\cM$ have no marked points or the preimages of both poles in $a^{(1)}$ are the only marked points. Otherwise, the surfaces in $\cM$ have no marked points.
\end{lem}
\begin{rem}
By Figure \ref{F:HyperellipticCheatSheet}, when $a > -1$, $\cM$ is an Abelian double of $\cH^{hyp}(a^{(2)})$. When $a = -1$, if the surfaces in $\cM$ have marked points then $\cM$ is an Abelian double of $\cH(0)$; if not, then $\cM$ is an Abelian double of $\cH(\emptyset)$. 
\end{rem}

\begin{proof}
Suppose that $\cM$ is an Abelian double of $\cH$ and a quadratic double of $\cQ$. Let $(X, \omega)$ be a generic surface in $\cM$ with involutions $J$ and $T$ such that $(X, \omega)/T$ (resp. $(X, \omega)/J$) is a generic surface in $\cH$ (resp. $\cQ$). 

\begin{sublem}\label{SL:QuadAbDouble:MarkedPoints}
Any marked points on $(X, \omega)$ are preimages of poles on $(X, \omega)/J$. When $\For(\cQ) \ne \cQ(-1^4)$ there are no marked points at all. In all cases $\For(\cQ) = \cQ$.
\end{sublem}
\begin{proof}
If the marked points on $(X, \omega)$ include a slope $-1$ irreducible pair of points, then we may move these points while fixing all others. Since the collection of marked points must be fixed by $T$ (since $(X, \omega)$ belongs to an Abelian double), these two points must be exchanged by $T$ and hence form a slope $+1$ irreducible pair of marked points, a contradiction. Since $(X, \omega)$ belongs to a quadratic double and does not contain a slope $-1$ irreducible pair of points, the only marked points are preimages of poles. 

When $\For(\cQ) \ne \cQ(-1^4)$, $(X, \omega)$ is not a torus. Abelian doubles that are not loci of tori cannot contain periodic points as marked points. 
\end{proof}


Suppose first that $\cQ \ne \cQ(-1^4)$. The existence of the map $(X, \omega)/J \ra (X, \omega)/\langle J, T \rangle$ implies (by Lemma \ref{L:InvolutionImpliesHyp}) that $\cQ$ is hyperelliptic and hence that $\cQ = \cQ^{hyp}\left( a^{(1)}, b^{(1)} \right)$ for some integers $a \geq b \geq -1$. By Lemma \ref{L:HypSymmetries}, the affine symmetries of $(X, \omega)$ of derivative $\pm \mathrm{Id}$ are exactly those in $\{\mathrm{id}, J, T, JT \}$.

As illustrated in Figure \ref{F:HyperellipticCheatSheet}, $(X, \omega)/T$, which is a generic surface in $\cH$, is hyperelliptic and belongs to $\cH\left( a^{(2)}, b^{(2)} \right)$. Notice that when $a$ or $b$ is zero the points in $a^{(2)}$ (resp. $b^{(2)}$) must be marked since they are branch points of the map $(X, \omega) \ra (X, \omega)/T$ (see Figure \ref{F:HyperellipticCheatSheet}).

\begin{sublem}
The marked points on $(X, \omega)/T$ are invariant by the hyperelliptic involution. 
\end{sublem}
\begin{proof}
As illustrated in Figure \ref{F:HyperellipticCheatSheet}, since $J$ and $T$ commute, $J$ induces an involution on $(X, \omega)/T$ that is the hyperelliptic involution. The set of marked points on $(X,\omega)$ are invariant by $J$ since $(X, \omega)$ belongs to a quadratic double. Therefore, the set of marked points on $(X, \omega)/T$ are invariant by the hyperelliptic involution.
\end{proof}

 The only strata of genus $g$ translation surfaces, including those with marked points, in which the generic surface is hyperelliptic and in which the marked points are invariant by the hyperelliptic involution are $\cH^{hyp}(2g-2)$ and $\cH^{hyp}(g-1,g-1)$. In the latter case the two singularities of the metric are exchanged by the hyperelliptic involution. Since the set points in $a^{(2)}$ is invariant by the hyperelliptic involution (see Figure \ref{F:HyperellipticCheatSheet}) the points in $b^{(2)}$ cannot be marked and so $b = -1$. This shows that $\cQ = \cQ(a^{(1)}, -1^2)$ with the preimages of poles unmarked.

Suppose now that $\cQ = \cQ(-1^4)$. By Sublemma \ref{SL:QuadAbDouble:MarkedPoints}, the only marked points on $(X, \omega)$ are preimages of poles and so $\cM$, and hence also $\cH$, has rank one rel zero. The only strata of Abelian differentials that have rank one rel zero are $\cH(\emptyset)$ and $\cH(0)$ implying that $(X, \omega)$ has either no marked points or exactly two preimages of poles on $(X, \omega)/J$ marked. 
\end{proof}

\begin{cor}\label{C:HypSymmetries}
Suppose that $\cM$ is a quadratic and Abelian double and that $\cM \ne \cH(\emptyset)$. When $\cM$ is not rank one rel zero, the generic surface in $\cM$ has exactly four affine self-maps of derivative $\pm \mathrm{Id}$ (they are the ones in Lemma \ref{L:HypSymmetries}). When $\cM$ is rank one rel zero, the same statement holds if the affine self-maps are required to preserve the set of marked points. 
\end{cor}
\begin{proof}
This is an immediate application of Lemma \ref{L:QuadAbDouble} and \ref{L:HypSymmetries} with Remark \ref{R:HypSymmetry-SpecialCase} being used in the case that $\cM$ has rank one rel zero. 
\end{proof}

\subsection{Cylinders in hyperelliptic components of Abelian differentials}
Recall the definition of $S$-path from Definition \ref{D:Spath}. We conclude this section by showing the following.

\begin{lem}\label{L:MakingSigmaNice}
Let $(X,\omega)$ be a surface in $\cH^{hyp}(2g-2)$ or $\cH^{hyp}(g-1, g-1)$, excluding $\cH(0)$ but allowing $\cH(0,0)$. Let $S$ be a collection of one or two disjoint saddle connections, each of which is fixed by the hyperelliptic involution. Then there is an $S$-path $\gamma$ starting at $(X,\omega)$ such that on $\gamma(1)$ there is a simple cylinder that does not intersect a saddle connection in $S$ or contain one in its boundary.
\end{lem}
\begin{proof}
Let $\cH$ denote the stratum containing $(X, \omega)$. We begin with the following general observation.

\begin{sublem}\label{Slem1}
Suppose $(X, \omega)\notin \cH(0)$ has an involution negating $\omega$, and consider any triangulation of $(X,\omega)$ invariant under the involution. Then any triangle in this triangulation has at least one edge not fixed by the involution.
\end{sublem}

Here invariant means that the involution takes each triangle to another triangle. 

\begin{proof}[Proof of Sublemma \ref{Slem1}:]
Let $T$ be a triangle in an invariant triangulation, and let  $T'$ be its image under the involution. 
If all three edges are fixed, then $T$ and $T'$ share three edges and $(X,\omega)\in \cH(0)$.  
\end{proof}

We construct an invariant triangulation by iteratively picking saddle connections to add to the triangulation as follows, keeping in mind that any maximal collection of pairwise disjoint saddle connections is a triangulation. 

Begin with $S$. Pick any saddle connection disjoint from those already chosen. Add it to the collection of saddle connections already chosen, and, if the new saddle connection isn't fixed by the involution, add its image under the involution to the triangulation. 

Now, having constructed an invariant triangulation containing $S$, we appeal to Sublemma \ref{Slem1} to find a saddle connection $s$ disjoint from $S$, such that $s$ is not fixed by the involution. Let $s'$ be the image of $s$ under the involution. By Lemma \ref{L:HypCuttingS}, cutting $s\cup s'$ disconnects the surface into two components. One of the two components, call it $\Sigma$, contains at most one saddle connection in $S$.

Assume that we picked $s$ and $\Sigma$ as above to minimize the number of triangles in $\Sigma$. We will show that $\Sigma$ must be a simple cylinder. If $\Sigma$ consists of two triangles, these two triangles form the desired cylinder. Otherwise, we can find a triangle in $\Sigma$ not on the boundary of $\Sigma$. This triangle has an edge $t$ not fixed by the involution. Cutting $t$ and its image $t'$ again disconnects the surface. The component not containing $s\cup s'$ has fewer triangles than $\Sigma$ and still contains at most one of saddle connection in $S$, contradicting our choice of $s$. 

If $\Sigma$ contains neither saddle connection in $S$, then we are done (without having had to deform our original surface), by choosing $\Sigma$ to be the simple cylinder disjoint from the saddle connections in $S$. Therefore, we may suppose that $S$ contains two saddle connections, one of which is contained in $\Sigma$. Let $\Sigma'$ denote the translation surface with boundary that is the complement of $\Sigma$. 

Let $\cH\in \{\cH^{hyp}(2g-2), \cH^{hyp}(g-1, g-1)\}$ be the stratum containing $(X,\omega)$. We proceed by induction on the dimension of $\cH$.

\noindent \textbf{Base Case: $\cH = \cH(0,0)$}

As we have already seen, for $\cH(0,0)$ we have an invariant triangulation that contains two saddle connections $s$ and $s'$ that are exchanged by the hyperelliptic involution. Cutting along them produces two simple cylinders $\Sigma$ and $\Sigma'$, each of which contains exactly one saddle connection from $S$. 

The desired $S$-path can be obtained by individually twisting $\Sigma$ and $\Sigma'$ while fixing the unmarked surface $\cF(X,\omega)$ until both saddle connections in $S$ are nearly parallel to a cylinder direction on $\cF(X,\omega)$, so there is simple cylinder that is disjoint from (and nearly parallel to) both elements of $S$.  

\noindent \textbf{Inductive Step}

Let $\Sigma_{glue}'$ denote $\Sigma'$ with its boundary saddle connections glued together to form a translation surface without boundary. Let $\cH'$ be the component of the stratum of Abelian differentials containing $\Sigma'_{glue}$. Another way of forming $\Sigma'_{glue}$ is by collapsing $\Sigma$ on $(X, \omega)$. Since $\Sigma$ is a simple cylinder this shows that $\cH'$ is a hyperelliptic component that has dimension exactly one less than $\cH$. Since $\cH$ has complex dimension at least four, it follows that $\cH'$ has complex dimension at least three and therefore is not $\cH(0)$. Hence we can apply the induction hypothesis.

Let $S'$ denote the union of the saddle connection from $S$ contained in $\Sigma'$ and the saddle connection coming from the gluing of the boundary of $\Sigma'$. This is a set of at most two saddle connections fixed by the hyperelliptic involution. By hypothesis there is an $S'$-path $\gamma$ starting at $\Sigma'_{glue}$ such that $\gamma(1)$ contains a simple cylinder disjoint from $S'$.

Since our original surface $(X, \omega)$ can be formed by gluing together $\Sigma_{glue}$ and $\Sigma'_{glue}$, it follows by Lemma \ref{L:ExtendingPerturbations}, that there is a $S \cup \{s, s'\}$ path $\wt{\gamma}$ starting at $(X, \omega)$ and such that $\wt{\gamma}(1)$ contains a simple cylinder in $\Sigma'$ that does not intersect an element of $S$.
\end{proof}

\section{Diamonds with Abelian and quadratic doubles }\label{S:DiamondQDouble}

In this section we will classify the invariant subvarieties where one side is an Abelian double and the other is a quadratic double. In all figures in the sequel, all unlabelled sides of polygons will be tacitly identified with their opposites. We use the notation $\cF$ from Definition \ref{D:F} for forgetting marked points. 

\begin{thm}\label{TP2}
Let $((X,\omega), \cM, \bfC_1, \bfC_2)$ be a generic diamond of Abelian differentials. Suppose that $\MOne$ is an Abelian double and that $\MTwo$ is a quadratic double. Then $\cM$ is one of the following: 
\renewcommand\thesubfigure{\roman{subfigure}}
\begin{enumerate}
\item\label{I:AD} An Abelian double. See Figure \ref{F:AbDoubleCase}.
\item\label{I:QD} A quadratic double. See Figure \ref{F:QuadDoubleCase}.
\begin{figure}[h]\centering
    \centering
    \begin{subfigure}[t]{0.49\linewidth}
        \centering
        \includegraphics[width=\linewidth]{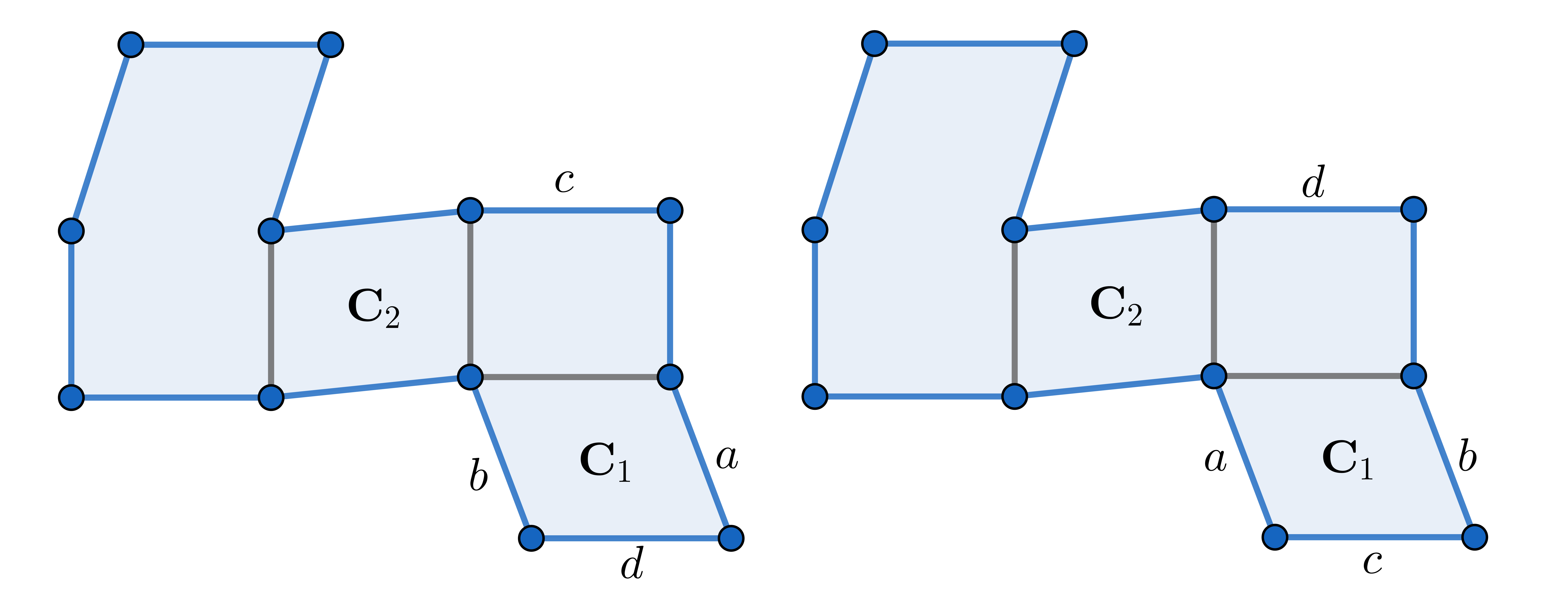}
        \caption{Abelian double case}
        \label{F:AbDoubleCase}
    \end{subfigure}
    \begin{subfigure}[t]{0.49\linewidth}
        \centering
        \includegraphics[width=\linewidth]{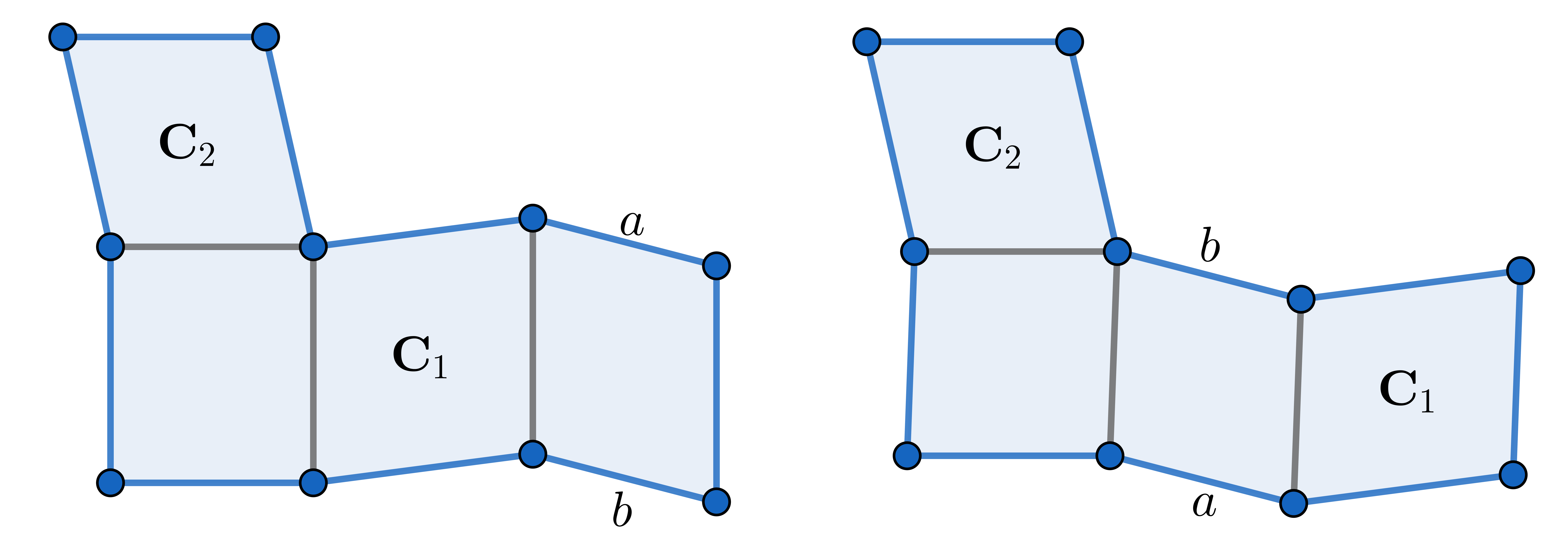}
        \caption{Quadratic double case}
        \label{F:QuadDoubleCase}
    \end{subfigure}
    \caption{An illustration of possibilities  (\ref{I:AD}) and (\ref{I:QD}) of Theorem \ref{TP2}. 
    }
\end{figure}
\item\label{I:BigExtraSymmetry} 
A full locus of \red double \black covers of a codimension one locus $\cN$ in a component of a stratum of Abelian differentials $\cH_0$. One of the following occurs:
\begin{enumerate}
    \item\label{I:ExtraSymmetry} $\For(\cN)$ is a codimension one hyperelliptic locus in $\For(\cH_0)$; in this case $\ColOneTwoX$ is disconnected and surfaces in $\cN$ either have no marked points or one free marked point. See Figure \ref{F:ExtraPossibility}. 
    \begin{figure}[h]\centering
\includegraphics[width=.7\linewidth]{ExtraPossibility.pdf}
\caption{An illustration of possibility (\ref{I:ExtraSymmetry}) of Theorem \ref{TP2}. Collapsing the pair of simple cylinders labelled $W$ gives an example of possibility (\ref{I:FinalPossibility}) with $\ColOneTwoX$ disconnected. If the examples in the previous two sentences are additionally modified by collapsing the horizontal cylinders below those labelled $\bfC_2$, new examples of (\ref{I:ExtraSymmetry}) and (\ref{I:FinalPossibility}) are formed where the surfaces in $\cN$ have no free marked points.  
}
\label{F:ExtraPossibility}
\end{figure}
    \item\label{I:FinalPossibility} $\For(\cN) = \For(\cH_0)$, $\For(\cH_0)$ is a hyperelliptic component of rank at least two, and surfaces in $\cN$ have at most three marked points at most one of which is free. See Figures \ref{F:ExtraPossibility}, \ref{F:NoExtraSymmetry}, \ref{F:EmptyP},  \ref{F:MC1notQdouble}, \ref{F:FinalPossibility}, and see Theorem \ref{T:3bSupplement} for additional information in this case.
\end{enumerate}
\end{enumerate}
\end{thm}

\begin{rem}
Recall that, by our definition of translation cover (Definition \ref{D:Covering}), every unbranched marked point must have a preimage that is marked.
\end{rem}

With Theorem \ref{TP2} in hand, the proof of Theorem \ref{T:DoubleIntro} is complete. 

\begin{proof}[Proof of Theorem \ref{T:DoubleIntro}]
If both $\MOne$ and $\MTwo$ are Abelian doubles, then Lemma \ref{L:StrongTheorem1.1} gives that $\cM$ is a full locus of double covers. We claim that moreover $\cM$ is an Abelian double; to show this we must show that for each marked point $p\in (X,\omega)$, the point $T(p)$ is also marked, where $T$ is the involution whose quotient is the double cover. Indeed, $p$ can be on the boundary of at most one of the $\bfC_i$, so without loss of generality assume it is not in the boundary of $\bfC_1$. Let $T_1$ be the involution on $\ColOneX$, and recall from the final claim of Lemma \ref{L:StrongTheorem1.1} that $\Col_{\bfC_1}(T)=T_1$. Because $\MOne$ is an Abelian double we have that $T_1(\Col_{\bfC_1}(p))$ is marked. Since $$T_1(\Col_{\bfC_1}(p))= \Col_{\bfC_1}(T(p)),$$  it follows that $T(p)$ is marked, because $p$ is not on the boundary of $\bfC_1$ and $(X,\omega)$ is obtained from $\ColOneX$ by gluing in $\bfC_1$. 


If both $\MOne$ and $\MTwo$ are quadratic doubles, then Theorem \ref{T:DoubleIntro} is an abbreviated form of Theorem \ref{P1}. In the remaining case, it is an abbreviated form of Theorem \ref{TP2}. In both cases one should keep in mind that a degree two cover of a hyperelliptic surfaces can be viewed as a degree four cover of a genus zero surface. 
\end{proof}

\begin{rem}
In Figure \ref{F:ExtraPossibility}, if we apply half a Dehn-twist to the cylinder $\bfC_2$, then we obtain a diamond with connected base surface. However, we  lose the property that $\MTwo$ is a quadratic double. This can be seen since, by Masur-Zorich (Theorem \ref{T:MZ}) if $\MTwo$ is a quadratic double, then cutting out a generically parallel pair of disjoint complex cylinders creates four connected components. However, if $\bfC_2$ is given half a Dehn-twisted before collapsing, then cutting out \red $\ColTwo(\bfC_1)$ \black produces only three connected components. 
\end{rem}

\begin{rem}\label{R:SpecialCase}
There is an interesting special case of case (\ref{I:ExtraSymmetry}), where $\cM$ is a locus that is simultaneously a holonomy double cover of codimension one hyperelliptic locus in a stratum of quadratic differentials and a translation double cover of a codimension one hyperelliptic locus in a stratum of Abelian differentials. In Figure \ref{F:ExtraPossibility},  collapsing the horizontal cylinder below the one labelled $\bfC_2$ yields an example. When the rank of $\cM$ is at least two, it is not hard to see that this case occurs precisely when the quotient by the translation involution does not have a free marked point. The proof of this claim is identical to that of Theorem \ref{T:3bSupplement} (\ref{I:ExtraSymmetry:NoP}). 
\end{rem}

We will now summarize the notation that we will adopt in the remainder of the section, and recall a few basic facts.

\begin{enumerate}
\item The surface $\ColOneX$ admits a translation involution $T_1$ such that $\ColOneX/T_1$ is a generic surface in a component $\cH_1$ of a stratum of connected Abelian differentials. 
\item The surface $\ColTwoX$ admits a half-translation involution $J_2$ such that $\ColTwoX/J_2$ is a generic surface in a component $\cQ_2$ of a stratum of connected quadratic differentials. 
\item On $\ColOneTwoX$, define $J:= \Col_{\ColTwo(\bfC_1)}(J_2)$ and $T := \Col_{\ColOne(\bfC_2)}(T_1)$. 
\item $\MOneTwo$ is a full locus of (possibly disconnected) covers of the stratum $$\cH := (\cH_1)_{\ColOne(\bfC_2)/T_1}$$ and a full locus of (possibly disconnected) covers of the stratum \red $$\cQ := (\cQ_2)_{\ColTwo(\bfC_1)/J_2}.$$ \black If $\ColOneTwoX$ is connected then $\MOneTwo$ is simultaneously an Abelian and quadratic double. 
\item We will presently show that $\cH$ is hyperelliptic; let $\tau$ denote the hyperelliptic involution on $\ColOneTwoX/T$.
\end{enumerate}

\begin{lem}\label{L:AllAboutThatBase}
The surface $\ColOneTwoX/T$ is connected and, letting $g$ denote its genus, $\cH = \cH^{hyp}(2g-2)$ or $\cH = \cH^{hyp}(g-1,g-1)$. 

When $g > 1$, for the generic surface in $\MOneTwo$, the only affine symmetries of derivative $\pm \mathrm{Id}$  are the four in $\{\mathrm{id}, J, T, JT\}$; the same statement holds when $g=1$ for marked point preserving symmetries. In particular, $J$ and $T$ commute. 

The involution $J$ descends to the hyperelliptic involution on $(X, \omega)/T$. 
\end{lem}
\begin{proof}
Since $\ColOne(\bfC_2)$ is a subequivalence class of generic cylinders, $\ColOne(\bfC_2)/T_1$ is a subequivalence class of generic cylinders on a surface in $\cH_1$ and hence a single simple cylinder (see Remark \ref{R:GenericCylInStrata}). Collapsing simple cylinders cannot disconnect a surface, so the surfaces in $\cH$ are connected. Thus, $\ColOneTwoX/T$ is connected.



We note that once we know that $J$ and $T$ commute \red and that $\cH$ is a hyperelliptic component, \black we will have that $J$ induces a marked point preserving involution of derivative $-\mathrm{Id}$ on \red $\ColOneTwoX/T$, \black which must be the hyperelliptic involution (for instance, by the uniqueness of the holonomy involution, as described in Lemma \ref{L:QHolonomy}).

If $\ColOneTwoX$ is connected, then $\MOneTwo$ is an Abelian double of $\cH$ and a quadratic double of $\cQ$. By Lemma \ref{L:QuadAbDouble}, $\cH = \cH^{hyp}(a^{(2)})$ and $\cQ = \cQ^{hyp}(a^{(1)}, -1^2)$ where $a \geq -1$ is an integer. The claim about affine symmetries holds by Corollary \ref{C:HypSymmetries}.

If $\ColOneTwoX$ is disconnected, then, by Lemma \ref{L:QBoundaryHolonomy}, $\MOneTwo$ is the antidiagonal embedding of $\cH$ into $\cH \times \cH$. When $g > 1$, there is an involution on a generic surface in $\MOneTwo$ that fixes each component if and only if $\For(\cH)$ is hyperelliptic; when $g > 1$ this involution is unique (by Lemma \ref{L:InvolutionImpliesHyp-background}), which implies the claim about affine symmetries when $g > 1$. The corresponding claim when $g = 1$ follows directly from Lemma \ref{L:QBoundaryHolonomy}.

It remains to show that, when $\ColOneTwoX$ is disconnected, $\cH = \cH^{hyp}(2g-2)$ or $\cH^{hyp}(g-1,g-1)$. This amounts to showing that there are no marked points on $\ColOneTwoX/T$ when $g > 1$ and at most two when $g = 1$. 

Since $\ColOneTwoX$ belongs to the boundary of an Abelian double, the only constraint on the marked points is that they are partitioned into pairs of points exchanged by $T$. This implies that the marked points on $\ColOneTwoX/T$ are free. Since $\ColOneTwoX$ belongs to the boundary of a quadratic double, the set of marked points must be preserved by $J$. Hence, the set of marked points on $\ColOneTwoX/T$ is invariant under $\tau$. This implies that either $g > 1$ and there are no marked points or $g = 1$ and $\cH = \cH(0)$ or $\cH = \cH(0,0)$.
\end{proof}

\begin{lem}\label{L:GenericParallelism}
Two saddle connections on $\ColOneTwoX$ are generically parallel if and only if there is an element of $\{\mathrm{id}, J, T, JT \}$ taking one to the other. 
\end{lem}
\begin{proof}
Two saddle connections on a surface in a hyperelliptic stratum are generically parallel if and only if they are exchanged by the hyperelliptic involution (Lemma \ref{L:HypParallelism}). Since $\ColOneTwoX/T$ is such a surface, with $J$ descending to the hyperelliptic involution on it (by Lemma \ref{L:AllAboutThatBase}), the result follows. 
\end{proof}

\begin{defn}\label{D:ExtraSymmetry}
We will say that $\cM$ has \emph{extra symmetry} if at least one of the following occurs: 
\begin{enumerate}
\item There is a translation involution $T_2$ on $\For(\ColTwoX)$ such that $\Col_{\ColTwo(\bfC_1)}(T_2) = T$ and $T_2\left( \ColTwo(\bfC_1) \right) = \ColTwo(\bfC_1)$.
\item There is a half-translation involution $J_1$ on $\For(\ColOneX)$ such that $\Col_{\ColOne(\bfC_2)}(J_1) = J$ and $J_2\left( \ColOne(\bfC_2) \right) = \ColOne(\bfC_2)$.
\end{enumerate}
We will call $T_2$ (resp. $J_1$) an \emph{extension} of $T$ (resp. $J$). We do not require the extensions to preserve the collection of marked points; this is in contrast to the definition of Abelian and quadratic doubles, which requires $T_1$ and $J_2$ to preserve the set of marked points. 
\end{defn}

While the previous figures in this section all contain examples where $\cM$ has extra symmetry, Figure \ref{F:NoExtraSymmetry} shows a case where $\cM$ has no extra symmetry.

\begin{figure}[h]\centering
\includegraphics[width=.6\linewidth]{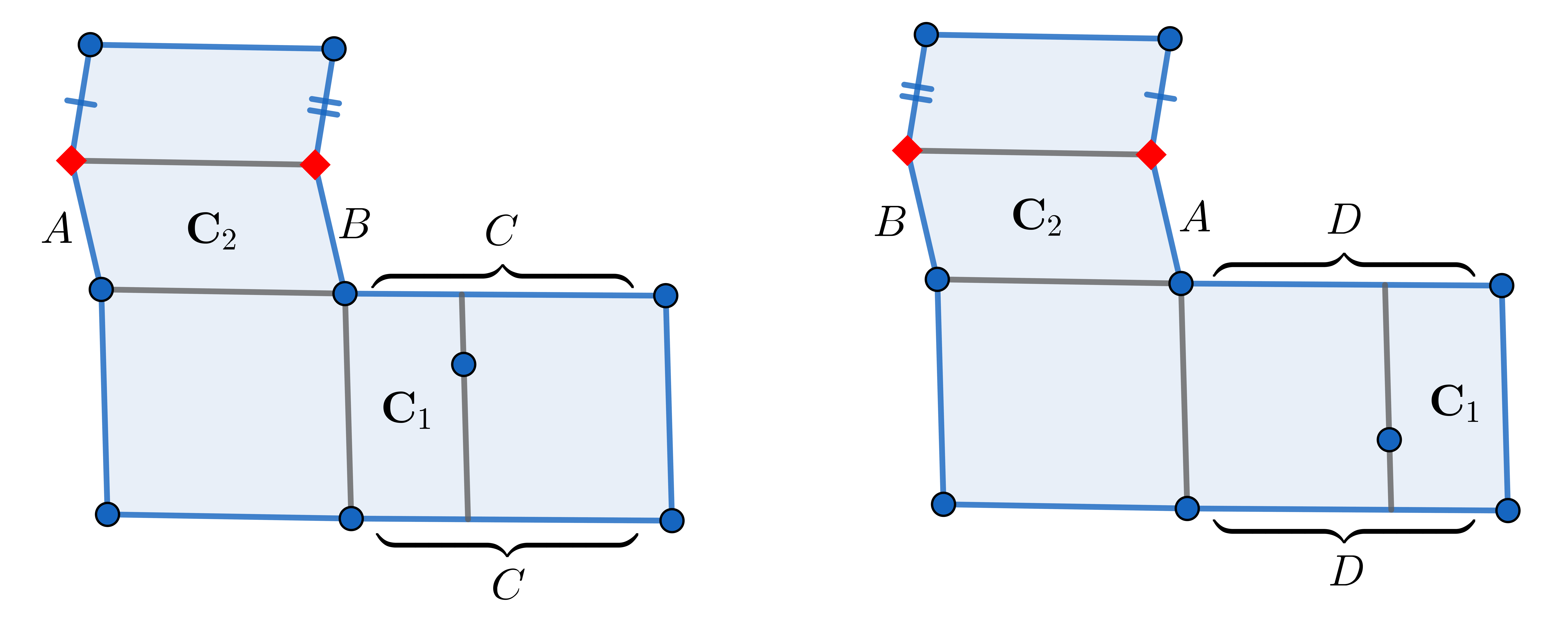}
\caption{Another illustration of Theorem \ref{TP2} (\ref{I:FinalPossibility}). Yet another example can be obtained by gluing the segments labelled $A$ and $B$ together in the other possible way. Notice that there is no extra symmetry in this diamond.}
\label{F:NoExtraSymmetry}
\end{figure}

\begin{rem}\label{R:NoExtraSymmetry}
The proof of Theorem \ref{TP2} will show that $\cM$ must have extra symmetry except sometimes in Case (\ref{I:FinalPossibility}). The examples with extra symmetry are analyzed in Section \ref{S:ExtraSymmetry}, and the remaining instances of Case (\ref{I:FinalPossibility}) involve marked points (see Lemma \ref{L:FinalCountdown} and Figure \ref{F:NoExtraSymmetry}). 
\end{rem}

In the sequel, we will make the following assumption (there is no loss of generality by Remark \ref{R:CouldBeDense}). 

\begin{ass}\label{A:DenseComplicated}
$(X, \omega)$ has dense $\GL$-orbit in $\cM$. 
\end{ass}

\subsection{Basic structural results}
\begin{lem}\label{L:StructureOfS}
$\bfC_2$ consists of either a pair of simple cylinders or a single complex cylinder. 

Moreover, if $\ColOneTwo(\bfC_2)$ is $J$-invariant, then there is an involution $J_1$ on $\ColOne(X, \omega)$ that is an extension of $J$ and that preserves the set of marked points; in particular, $\cM$ has extra symmetry. 
\end{lem}
\begin{proof}
By Definition \ref{D:GenericDiamond} (\ref{E:genericSE}), $\ColOne(\bfC_2)$ is a subequivalence class of generic cylinders (recall that subequivalence classes are defined in Definition \ref{D:subeq}). By Lemma \ref{L:AbelianDoubleSEC}, since $\MOne$ is an Abelian double, $\ColOne(\bfC_2)$ is either a single complex cylinder or two isometric simple cylinders with disjoint boundaries on $\ColOneX$. This establishes the first claim.

If $\ColOneTwo(\bfC_2)$ is $J$-invariant, then since $\ColOneX$ is formed by gluing in $\ColOne(\bfC_2)$ - a pair of simple cylinders or a single complex cylinder - to the pair of saddle connections comprising $\ColOneTwo(\bfC_2)$, the $J$-involution extends to $\ColOneX$.
\end{proof}

\begin{lem}\label{L:Cutting}
$\ColOneTwo(\bfC_2)/T$ and $\ColOneTwo(\bfC_1)/T$ are disjoint. 

Moreover, $\ColOneTwo(\bfC_1)/T$ is $\tau$-invariant and, in the case where $\cM$ has no extra symmetry, $\ColOneTwo(\bfC_2)/T$ is not. 
\end{lem}
\begin{proof}
Since $\ColOneTwo(\bfC_2)$ is $T$-invariant, it is the full preimage of $\ColOneTwo(\bfC_2)/T$. It follows that $\ColOneTwo(\bfC_2)/T$ and $\ColOneTwo(\bfC_1)/T$ are disjoint, since $\ColOneTwo(\bfC_2)$ and $\ColOneTwo(\bfC_1)$ are disjoint.

Since $\ColTwo(\bfC_1)$ is fixed by $J_1$, $\ColOneTwo(\bfC_1)$ is fixed by $J$. Since $J$ and $T$ commute on $\ColOneTwoX$, the $J$-involution descends to $\tau$ on $\ColOneTwoX/T$. Therefore, $\ColOneTwo(\bfC_1)/T$ is fixed by $\tau$. 

Similarly, if $\ColOneTwo(\bfC_2)/T$ were fixed by $\tau$ then $\ColOneTwo(\bfC_2)$ would be $J$-invariant, implying that $\cM$ has extra symmetry by  Lemma \ref{L:StructureOfS}.
\end{proof}

\begin{cor}\label{C:Cutting}
Each component of the following surface is fixed by $\tau$,
\[ \ColOneTwoX/T - \left( \left( \ColOneTwo(\bfC_1 \cup \bfC_2) \right)/T \cup \tau\left( \ColOneTwo(\bfC_2)/T \right) \right).\]
Moreover, when $\cM$ has no extra symmetry, there are $|\ColOneTwo(\bfC_1)/T|+1$ components where $|\ColOneTwo(\bfC_1)/T|\in \{1,2\}$ is the number of saddle connections in $\ColOneTwo(\bfC_1)/T$.
\end{cor}
\begin{proof}
By Lemma \ref{L:AllAboutThatBase}, $\ColOneTwo(X, \omega)/T$ belongs to $\cH^{hyp}(2g-2)$ or $\cH^{hyp}(g-1, g-1)$ for some  $g$. 

By Lemma \ref{L:HypParallelism}, the generically parallel saddle connections in a hyperelliptic component of a stratum of Abelian differentials are precisely the ones exchanged by the hyperelliptic involution.  Therefore, each $\ColOneTwo(\bfC_i)/T$ is either a single saddle connection or a pair of saddle connections exchanged by the hyperelliptic involution.  

We now make some basic observations.
\begin{enumerate}
\item By Lemma \ref{L:HypCuttingS}, if $s$ is any saddle connection on $\ColOneTwo(X, \omega)/T$ that is not fixed by the hyperelliptic involution $\tau$  then cutting along $s$ and $\tau(s)$ disconnects the surface into two subsurfaces with boundary, each fixed by the hyperelliptic involution. 
\item Similarly, cutting two disjoint pairs of exchanged saddle connections disconnects the surface into three components, each fixed by the hyperelliptic involution. 
\item By Lemma \ref{L:HypParallelism}, cutting two saddle connections that are not exchanged does not disconnect the surface.
\item Similarly, after cutting one exchanged pair of saddle connections, additionally cutting another saddle connection does not further disconnect the surface. 
\end{enumerate}
This implies the first claim. 

Assume that $\cM$ has no extra symmetry. By Lemmas \ref{L:StructureOfS} and \ref{L:Cutting}, $\ColOneTwo(\bfC_2)/T$ is a single saddle connection not fixed by the hyperelliptic involution, so cutting it and its image under $\tau$ disconnects the surface into two components. 

The set $\ColOneTwo(\bfC_1)/T$ is fixed by $\tau$, so the final claim follows from the basic observations above. 
\end{proof}

\begin{defn}\label{D:Cutting}
When $|\ColOneTwo(\bfC_1)/T| = 2$ and $\ColOneTwo(\bfC_2)/T$ is not $\tau$-invariant, the surface in the statement of Corollary \ref{C:Cutting} has three components. We will label them as follows:
\begin{enumerate}
    \item $\Sigma_1$ is the subsurface whose boundary is $\ColOneTwo(\bfC_1)/T$.
    \item $\Sigma_2$ is the subsurface whose boundary is $\ColOneTwo(\bfC_2)/T \cup \tau\left( \ColOneTwo(\bfC_2)/T \right)$.
    \item $\Sigma_3$ is the remaining subsurface.
\end{enumerate}
The preimage of $\Sigma_i$ in $\ColOneTwoX$ will be denoted $\wt{\Sigma}_i$ for $i \in \{1, 2, 3\}$. In the sequel, we will use $\wt{\Sigma}_i^{top}$ to refer to the corresponding subsurface on $(X, \omega)$. 

An equivalent definition of $\Sigma_2$, which we will use even in the case where $|\ColOneTwo(\bfC_1)/T| = 1$ \red as long as we continue to have that \black $\ColOneTwo(\bfC_2)/T$ is not $\tau$-invariant,  is that it is the component of 
\[ \ColOneTwoX/T - \left(\ColOneTwo(\bfC_2)/T \cup \tau\left( \ColOneTwo(\bfC_2)/T \right) \right) \]
that does not contain $\ColOneTwo(\bfC_1)/T$. 
\end{defn}

\begin{lem}\label{L:QuickSummary}
If $\ColOneTwo(\bfC_2)/T$ is not $\tau$-invariant, then $\cM$ has rank at least two and $\ColOneTwoX/T$ has marked points if and only if it belongs to $\cH(0,0)$.

If $\cM$ has no extra symmetry, then the previous two conclusions hold and additionally $\ColOneTwoX$ has no marked points. 
\end{lem}
\begin{proof}
By Lemma \ref{L:AllAboutThatBase}, $\ColOneTwo(X, \omega)/T$ belongs to $\cH^{hyp}(2g-2)$ or $\cH^{hyp}(g-1, g-1)$ for some  $g$. All claims are immediate when $g > 1$. So suppose that $g = 1$. 

When $\cM$ has no extra symmetry, $\ColOneTwo(\bfC_2)/T$ is not $\tau$-invariant by Lemma \ref{L:Cutting}. So suppose that $\ColOneTwo(\bfC_2)/T$ is not $\tau$-invariant.

Suppose first that $\ColOneTwoX/T$ belongs to $\cH(0)$. Then every saddle connection on $\ColOneTwoX/T$ is fixed by $\tau$. This contradicts our assumption on $\ColOneTwo(\bfC_2)/T$. Therefore, $\ColOneTwoX/T$ belongs to $\cH(0,0)$. 


Since $\ColOneTwo(\bfC_2)/T$ is not $\tau$-invariant, it consists of a saddle connection $s$ that joins a marked point to itself. Since $\ColOneTwo(\bfC_1)/T$ is $\tau$-invariant and disjoint from $\ColOneTwo(\bfC_2)/T$ (by Lemma \ref{L:Cutting}), we see that $\ColOneTwo(\bfC_1)/T$ consists of saddle connections disjoint from $s$ and $\tau(s)$. Every such saddle connection joins the two marked points and is fixed by $\tau$. Since the saddle connections in $\ColOneTwo(\bfC_1)/T$ are all generically parallel, it follows that it is a single saddle connection fixed by $\tau$, and which is not generically parallel to $s$.

In particular, $\bfC_1$ and $\bfC_2$ are subequivalence classes of disjoint non-parallel cylinders, which implies that $\cM$ has rank at least two by Lemma \ref{L:RankTest}. We have established the first sentence of the claim, so we suppose now that $\cM$ has no extra symmetry. We wish to show that $\ColOneTwoX$ contains no marked points.

Because $\ColOneTwoX/T\notin\cH(0)$,  if $\ColOneTwoX$ is connected, then it has no marked points by Lemma \ref{L:QuadAbDouble}. 

Suppose therefore that $\ColOneTwoX$ is disconnected. In particular, this means, by Lemma \ref{L:AllAboutThatBase}, that $\ColOneTwoX$ consists of two disjoint copies of $\ColOneTwoX/T$ that are exchanged by both $T$ and $J$. This means that $\ColOneTwo(\bfC_1)$ consists of two saddle connections exchanged by $T$ on two disjoint copies of $\ColOneTwoX/T$ (this follows since $\ColOneTwo(\bfC_1)/T$ is a single saddle connection and since $\ColOneTwo(\bfC_1)$ must contain at least two saddle connections since it is $J$-invariant and since $\ColOneTwoX$ consists of two components that are exchanged by $J$). By Masur-Zorich (Theorem \ref{T:MZ}; see also Figure \ref{F:CylinderTypes}), $\ColTwo(\bfC_1)/J_2$ is a complex envelope, as we now explain. By Masur-Zorich (Theorem \ref{T:MZ}) since $\ColTwo(\bfC_1)$ is a subequivalence class of generic cylinders whose collapse disconnects the surface it must be the preimage of a complex envelope or complex cylinder. In the latter case, $\ColTwo(\bfC_1)$ contains four saddle connections. So we see that $\ColTwo(\bfC_1)/J_2$ must be a complex envelope.

Let $P$ be the preimages of the poles on this complex envelope that are marked on $\ColTwoX$. Let $\For'\left( \ColTwoX \right)$ denote $\ColTwoX$ with the points in $P$ forgotten and let $\For'(\ColTwo(\bfC_1))$ denote the image of $\ColTwo(\bfC_1)$ on this surface. Since $\For'\left( \ColTwoX \right)$ is formed by gluing in a complex cylinder - $\For'\left( \ColTwo(\bfC_1) \right)$ - to two saddle connections exchanged by $T$, it follows that the $T$ involution extends (in the sense of Definition \ref{D:ExtraSymmetry}) to a translation involution $T_2$ on $\ColTwoX$ such that $T_2(\ColTwo(\bfC_1)) = \ColTwo(\bfC_1)$ and so $\Col_{\ColTwo(\bfC_1)}(T_2) = T$. This is the definition of having extra symmetry, which contradicts the assumption that $\cM$ does not have extra symmetry. \end{proof}

\subsection{The Collapsing Lemma}
In this section we will prove the main technical lemma for the sequel. Recall that cylinder adjacency was defined in Definition \ref{D:CylinderAdjacent}. 

\begin{lem}\label{L:Collapsing}
Let $C\subset \ColOneTwoX/T$ be a simple cylinder that is not adjacent to the saddle connections in $\ColOneTwo(\bfC_1 \cup \bfC_2)/T$. Let $\wt{C}$ be the preimage of $C$ on $\ColOneTwoX$, and let $\wt{C}_{top}$ be the corresponding cylinders on $(X, \omega)$. 
Then, after perhaps performing a half Dehn-twist in $\wt{C}_{top}$, 
$$\left(\Col_{\wt{C}_{top}}(X, \omega), \cM_{\wt{C}_{top}}, \Col_{\wt{C}_{top}}(\bfC_1), \Col_{\wt{C}_{top}}(\bfC_2)\right)$$
is a generic diamond where $\cM_{\wt{C}_{top}, \bfC_1}$ (resp. $\cM_{\wt{C}_{top}, \bfC_2}$) is an Abelian (resp. quadratic) double.  
\end{lem}

\begin{proof}
We begin by showing that the indicated $4$-tuple is a generic diamond. First, since $C$ is not adjacent to the saddle connections in $\ColOneTwo(\bfC_1 \cup \bfC_2)/T$, the saddle connections in $\ColOneTwo(\bfC_1 \cup \bfC_2)/T$ remain distinct \red when $C$ is collapsed, \black that is no two merge, and so $\Col_{\wt{C}_{top}}(\bfC_1)$ and $\Col_{\wt{C}_{top}}(\bfC_2)$ continue to neither intersect nor share boundary saddle connections. Moreover, $\Col_{\wt{C}_{top}}(\bfC_i)$ remains a subequivalence class, which verifies Definition \ref{D:GenericDiamond} \eqref{E:genericSE}. The following sublemma will verify Definition \ref{D:GenericDiamond} \eqref{E:one}; which completes the verification that we have a generic diamond.

\begin{sublem}
$\left( \cM_{\wt{C}_{top}} \right)_{\Col_{\wt{C}_{top}}(\bfC_i)}$ has dimension exactly one less than $\cM_{\wt{C}_{top}}$.
\end{sublem}
\begin{proof}
Since $C$ is a simple cylinder on $\ColOneTwoX/T$,  the dimension of $\cM_{\bfC_1, \bfC_2, \wt{C}_{top}}$ is exactly one less than that of $\MOneTwo$. Since $\cM_{\bfC_1, \bfC_2, \wt{C}_{top}}$ belongs to the boundary of $\cM_{\wt{C}_{top}, \bfC_1}$ and $\cM_{\wt{C}_{top}, \bfC_2}$, it follows that these two invariant subvarieties have dimension exactly one less than $\MOne$ and $\MTwo$ respectively. This implies that $\cM_{\wt{C}_{top}, \bfC_i}$ has dimension exactly one less than $\cM_{\wt{C}_{top}}$.
\end{proof}


By Lemma \ref{L:codim1doubles},  $\cM_{\wt{C}_{top}, \bfC_1}$ (resp. $\cM_{\wt{C}_{top}, \bfC_2}$) is an Abelian (resp. quadratic) double if the surfaces they contain are connected, which is what we will now show. 

Suppose first that $\ColOneTwoX$ is disconnected. Then $\wt{C}$ is a union of two simple cylinders, one on each component of $\ColOneTwoX$. Therefore, $\Col_{\bfC_1, \bfC_2, \wt{C}_{top}}(X, \omega)$ also has two connected components exchanged by the translation involution $\Col_{\wt{C}}(T)$. Since gluing in $\bfC_i$ to the saddle connections in $\ColOneTwo(\bfC_i)$ causes $\ColOneTwoX$ to become connected, the same holds for gluing in $\bfC_i$ to $$\Col_{\bfC_1, \bfC_2, \wt{C}_{top}}(\bfC_i)\subset \Col_{\bfC_1, \bfC_2, \wt{C}_{top}}(X, \omega).$$

Suppose now that $\ColOneTwoX$ is connected. If $\Col_{\bfC_1, \bfC_2, \wt{C}_{top}}(X, \omega)$ remains connected, then we are done. Otherwise, $\wt{C}$ is a complex cylinder and it is easy to see that performing a half-Dehn twist in $\wt{C}_{top}$ causes $\Col_{\bfC_1, \bfC_2, \wt{C}_{top}}(X, \omega)$ to remain connected as desired.
\end{proof}

\begin{defn}\label{D:Parallelogram}
A \emph{slit torus} is a translation surfaces with boundary formed by cutting a line segment between two distinct points in a flat torus. We will say that a subsurface with boundary of a translation surface is a \emph{parallelogram} if it \red is isometric to a parallelogram in the plane. \black 
\end{defn}

\begin{cor}\label{C:Collapsing2-0}
Let $\Sigma$ be a component of 
\[ \ColOneTwoX/T - \left( \left( \ColOneTwo(\bfC_1 \cup \bfC_2) \right)/T \cup \tau\left( \ColOneTwo(\bfC_2)/T \right) \right).\]
Let $S$ denote the boundary saddle connections of $\Sigma$. Suppose that $\Sigma$ is not a cylinder or parallelogram. Let $\wt{\Sigma}$ and $\wt{S}$ denote the preimages of $\Sigma$ and $S$ respectively on $\ColOneTwoX$. Let $\wt{\Sigma}_{top}$  be the subsurface isometric to $\wt{\Sigma}$ on $(X, \omega)$ and \red $\wt{S}_{top}$ \black its boundary saddle connections. 

Then there is an $\wt{S}_{top}$-path $\gamma:[0,1] \ra \cM$ such that: 
\begin{enumerate}
\item The component of $\gamma(t) - \wt{S}_{top}$ corresponding to the complement of $\wt{\Sigma}_{top}$ is a rotated and scaled copy of the complement of $\wt{\Sigma}_{top}$ on $(X, \omega)$. 
\item There is a subequivalence class of cylinders $\wt{C}_{top}$ on \red the component of $\gamma(1)$ corresponding to $\wt{\Sigma}_{top}$ \black such that $\ColOneTwo(\wt{C}_{top})/T$ is a simple cylinder that is not adjacent to the saddle connections in $\ColOneTwo(\bfC_1 \cup \bfC_2)/T$.
\end{enumerate}
\end{cor}
By Lemma \ref{L:Collapsing}, after perhaps performing a half Dehn-twist in $\wt{C}_{top}$, $$\left(\Col_{\wt{C}_{top}}(X, \omega), \cM_{\wt{C}_{top}}, \Col_{\wt{C}_{top}}(\bfC_1), \Col_{\wt{C}_{top}}(\bfC_2)\right)$$ is a generic diamond where $\cM_{\wt{C}_{top}, \bfC_1}$ and $\cM_{\wt{C}_{top}, \bfC_2}$ are Abelian (resp. quadratic) doubles.
\begin{proof}
By Corollary \ref{C:Cutting}, $\Sigma$ is a subsurface with boundary that is fixed by the hyperelliptic involution. Let $\Sigma^{glue}$ be the translation surface formed by gluing together boundary saddle connections of $\Sigma$ exchanged by the hyperelliptic involution. The resulting surface belongs to a hyperelliptic component of a stratum of Abelian differentials by Lemma \ref{L:InvolutionImpliesHyp}. On $\Sigma^{glue}$ let $S^{glue}$ denote the saddle connections that were identified on the boundary of $\Sigma$ to form $\Sigma^{glue}$. The set $S^{glue}$ contains at most two saddle connections each of which is fixed by the hyperelliptic involution. 

By Lemma \ref{L:MakingSigmaNice}, if $\Sigma^{glue}$ does not belong to $\cH(0)$, then we can find an $S^{glue}$-path $\gamma_0$ such that on $\gamma_0(1)$ there is a simple cylinder $C$ that is not adjacent to $S^{glue}$. By Lemma \ref{L:ExtendingPerturbations}, there is an $S$-path $\gamma_1: [0,1] \ra \cH$ where $\gamma_1(0) = \ColOneTwoX/T$ and $C$ is a simple cylinder, not adjacent to $S$, in $\Sigma$ on $\gamma_1(1)$. By Lemma \ref{L:ExtendingPerturbations}, on $\gamma_1(t)$ the complement of $\Sigma$ is a rotated and scaled copy of the complement of $\Sigma$ on $\gamma_1(0)$.   Let $\gamma_2: [0,1] \ra \MOneTwo$ be the corresponding $\wt{S}$-path in $\MOneTwo$. 

Let $\lambda_i: [0,1] \ra \mathbb{C}^\times$ be the scalar such that the period of a saddle connection in $\ColOneTwo(\bfC_i)$ on $\gamma_2(t)$ is $\lambda_i(t)$ times its period on $\gamma_2(0)$. Let $\gamma(t)$ denote the surface formed by gluing $\lambda_i(t) \cdot \bfC_i$ into the saddle connections in $\ColOneTwo(\bfC_i)$ on $\gamma_2(t)$. By Lemma \ref{L:ExtendingPaths},  $\gamma(t)$ belongs to $\cM$ for all $t$. 
\end{proof}

\begin{cor}[The Collapsing Lemma]\label{C:Collapsing2}
Under the hypotheses of Corollary \ref{C:Collapsing2-0}, we may iterate the collapsing in Corollary \ref{C:Collapsing2-0} to replace the generic diamond $\left( (X, \omega), \cM, \bfC_1, \bfC_2 \right)$ with another generic diamond $\left( (X', \omega'), \cM', \bfC_1', \bfC_2', \right)$ where the following hold:
\begin{enumerate}
    \item\label{C:Collapsing2:AbelianDouble} $\cM'_{\bfC_1'}$ is an Abelian double.
    \item\label{C:Collapsing2:QuadraticDouble} $\cM'_{\bfC_2'}$ is a quadratic double
    \item\label{C:Collapsing2:RotatedScaledCi} There is a subsurface $\wt{\Sigma'}_{top}$ of $(X', \omega')$ such that $(X, \omega) - \wt{\Sigma}_{top}$ is a rotated and scaled copy of  $(X', \omega') - \wt{\Sigma'}_{top}$. Under this rotation and scaling, $\bfC_i'$ corresponds to $\bfC_i$ for $i \in \{1,2\}$.
    \item\label{I:Stop} Letting $J' = \Col(J)$ and $T' = \Col(T)$ we have that $\Sigma' := \Col_{\bfC_1', \bfC_2'}(\wt{\Sigma'}_{top})/T'$ is a cylinder if $|S|=2$ and a parallelogram if $|S|= 4$. Moreover, if $|S|=2$ and $\Sigma$ is not a cylinder, we can alternatively arrange for  $\Sigma'$ to be a slit torus.
    \item\label{I:S1} $|\Col_{\bfC_1', \bfC_2'}(\bfC_1')/T'| = |\ColOneTwo(\bfC_1)/T|$.
    \item\label{I:S2} $|\Col_{\bfC_1', \bfC_2'}(\bfC_2')/T' \cup \Col(\tau)(\Col_{\bfC_1', \bfC_2'}(\bfC_2')/T')| =   |\Col_{\bfC_1, \bfC_2}(\bfC_2)/T \cup \tau(\Col_{\bfC_1, \bfC_2}(\bfC_2)/T)|$.
\end{enumerate}
\end{cor}
\begin{proof}
The only statements that are not immediate are the final sentence of (\ref{I:Stop}), (\ref{I:S1}) and (\ref{I:S2}). We will prove these sequentially. 

Using the notation introduced in the proof of Corollary \ref{C:Collapsing2-0}, $\Sigma^{glue}$ belongs to a hyperelliptic component of a stratum of Abelian differentials. Therefore, $\Col_{C}(\Sigma^{glue})$ also belongs to a hyperelliptic component. Specifically, if $\Sigma^{glue}$ belongs to $\cH(2g-2)$ (resp. $\cH(g-1,g-1)$), then $\Col_{C}(\Sigma^{glue})$ belongs to $\cH(g-2, g-2)$ (resp. $\cH(2g-2)$). Notice that, when $|S|=2$, $\Sigma^{glue}$ belongs to $\cH(0)$ if and only if $\Sigma$ is a cylinder. When this is not the case, we can stop the iterated collapsing when $(\Sigma')^{glue}$ belongs to $\cH(0,0)$. When $|S| = 2$ this is equivalent to saying that $\Sigma'$ is a slit torus. This proves the final sentence of (\ref{I:Stop}).

Both (\ref{I:S1}) and (\ref{I:S2})  follow immediately from the fact that, in Corollary \ref{C:Collapsing2-0}, $C$ is not adjacent to the saddle connections in $\ColOneTwo(\bfC_1 \cup \bfC_2)/T$ (since $C$ is a simple cylinder on a surface in a hyperelliptic component, $C$ is fixed by the hyperelliptic involution and hence is not adjacent to $\tau\left( \ColOneTwo(\bfC_2)/T \right)$ either; \red note that these claims would not necessarily hold if $C$ were not fixed by the hyperelliptic involution, as in Figure \ref{F:sJs}). \black 
\end{proof}

\subsection{Weak translation symmetry}\label{S:ExtraSymmetryPrep}

We begin with the following simple lemma.

\begin{lem}\label{L:MarkedPointsAndInvariance}
Suppose that $\bfC$ is a subequivalence class of cylinders on a surface $(Y, \eta)$ in an invariant subvariety $\cN$. Suppose that $f$ is a translation involution on $(Y, \eta)$ that remains defined on all surfaces in a neighborhood of $(Y, \eta)$ in $\cN$. Then if every boundary of every cylinder in $\bfC$ contains a singularity of the flat metric (not simply a marked point), then $\bfC$ is $f$-invariant.  
\end{lem}
\begin{rem}
The hypothesis that $f$ remains defined on all surfaces in a neighborhood of $(Y, \eta)$ holds whenever $(Y, \eta)$ has dense orbit in $\cN$. 
\end{rem}
\begin{proof}
Since we have not assumed that $f$ preserves the marked points on $(Y, \eta)$, $f$ does not necessarily take cylinders to cylinders. However, $f$ does have this property on $\For(Y, \eta)$. By assumption, after forgetting the marked points in $(Y, \eta)$, the cylinders in $\bfC$ remain cylinders (in particular, they are still maximal). Therefore, it suffices to prove the claim on $\For(Y, \eta)$ where we can assume that $f$ takes cylinders to cylinders. 

Since $f$ remains defined in a neighborhood of $(Y, \eta)$ and since $\bfC$ is a subequivalence class of cylinders, we may apply the standard dilation in $\bfC$ to pass to an arbitrarily close surface $(Y', \eta')$ where the cylinders in $\bfC$ have moduli that are distinct from those of any parallel cylinder. It is now clear that $f$ preserves $\bfC$ as desired.
\end{proof}



Recall that by Assumption \ref{A:DenseComplicated}, $(X, \omega)$ has a dense orbit in $\cM$.  We make the following standing assumption for this subsection: 


\begin{ass}\label{A:T0Involution}
There is a translation involution $T_0$ on $(X, \omega)$ such that $(X, \omega)/T_0$ is not a translation cover of degree greater than one of another surface. 
\end{ass}

$T_0$ is not assumed to respect marked points or to preserve $\bfC_i$, so Assumption \ref{A:T0Involution} does not imply that $\cM$ has extra symmetry in the sense of Definition \ref{D:ExtraSymmetry}. Indeed, the main results of this subsection, Lemmas \ref{L:C2TInvariant} and \ref{L:C1TInvariant}, will determine the extent to which $T_0$ preserves $\bfC_i$ and the collection of marked points. The results in the remainder of this subsection will only be used in Subsections \ref{S:ExtraSymmetry} and \ref{S:ExtraSymmetrySupplement}.


It may be useful to recall that the slope of an irreducible pair of marked points is defined in Definition \ref{D:slope}.

\begin{lem}\label{L:C2TInvariant}
$\bfC_2$ is $T_0$-invariant and $T_0$ preserves the set of marked points on its boundary. 

Moreover, there are no free marked points on $(X, \omega)$, and the boundary of $\bfC_2$ contains all marked points that are part of an irreducible pair of points of slope $+1$.
\end{lem}
\begin{proof}
We begin by establishing the claim that $\bfC_2$ is $T_0$-invariant. By Lemma \ref{L:MarkedPointsAndInvariance}, the claim can only fail if there is a cylinder in $\bfC_2$ that has a boundary that contains no singularities of the flat metric - only  marked points. For convenience, we assume $\bfC_2$ is horizontal.

\begin{sublem}\label{SL:C2Invariance1}
Suppose that $p$ is a marked point on the top (resp. bottom) boundary of a cylinder in $\bfC_2$. Then both $\ColOne(p)$ and $p':=T_1(\ColOne(p))$ are marked points. 

Each of these marked points is on the top (resp. bottom) of a cylinder in $\ColOne(\bfC_2)$, and there are no other singularities or marked points in the top (resp. bottom) of cylinders in $\ColOne(\bfC_2)$.
\end{sublem}
\begin{proof}
$\ColOne(p)$ remains a marked point since the point $p$ cannot belong to the boundary of $\bfC_1$, since $\bfC_1$ and $\bfC_2$ are disjoint. 

By definition of Abelian double, $p'=T_1(\ColOne(p))$ is also marked. Since $\ColOne(\bfC_2)$ is a subequivalence class, it is $T_1$ invariant. Hence, the fact that $\ColOne(p)$ is on the top (resp. bottom) of a cylinder in $\ColOne(\bfC_2)$ implies the same statement for $p'$.

The final claim follows because the cylinders in $\bfC_2$ are generic.
\end{proof}

\begin{sublem}\label{SL:C2Invariance2}
Under the same assumptions as Sublemma \ref{SL:C2Invariance1}, if there is a marked point $q$ on the boundary of $\bfC_2$ such that $\ColOne(q) = p'$, then the cylinders in $\bfC_2$ are $T_0$ invariant and $T_0(p) = q$.
\end{sublem}
\begin{proof}
Let $q$ be a marked point, without loss of generality on the top boundary of $\bfC_2$, such that $\ColOne(q) = p'$. By Sublemma \ref{SL:C2Invariance1},  the only marked points or zeros on the top boundary of cylinders in $\bfC_2$ are the points in $\{p, q\}$, which are an irreducible pair of slope $+1$. By Lemma \ref{L:ApisaWright}, $p$ and $q$ are mapped to the same point on $(X_{min}, \omega_{min})$ (see Theorem \ref{T:MinimalCover} where this is defined). By Assumption \ref{A:T0Involution}, $(X_{min}, \omega_{min}) = (X, \omega)/T_0$ so we see that $p = T_0(q)$. Therefore, $T_0$ preserves the top boundary of cylinders in $\bfC_2$ and, since $\bfC_2$ contains either one cylinder or two isometric cylinders, preserves $\bfC_2$ as well.
\end{proof}

\begin{sublem}\label{SL:C2Invariance3}
Suppose that $p$ is a marked point contained on the top boundary of a cylinder in $\bfC_2$, then the top boundaries of cylinders in $\bfC_2$ do not contain singularities of the flat metric, only a pair of marked points exchanged by $T_0$. 
\end{sublem}
\begin{proof}
We will adopt the notation of Sublemma \ref{SL:C2Invariance1}. 
We first argue that the claim is true if $\ColOne(\bfC_1)$ does not have $p'$ as an endpoint. Indeed, in this case gluing in $\bfC_1$ to $\ColOneX$ does not affect $p'$, so there is a corresponding marked point $q$ on $(X,\omega)$ with $\ColOne(q)=p'$. The second statement of Sublemma \ref{SL:C2Invariance1} gives that $p$ and $q$ are the only marked points on the top boundary of $\bfC_2$, and  Sublemma \ref{SL:C2Invariance2} gives that $q=T_0(p)$.

Suppose therefore, in order to derive a contradiction, that $\ColOne(\bfC_1)$ has $p'$ as an endpoint. Since saddle connections in $\ColOne(\bfC_1)$ are generically parallel and cannot have $\ColOne(p)$ as an endpoint we see that $\ColOne(\bfC_1)$ consists of a single saddle connection (since $\MOne$ is an Abelian double, $\ColOne(p)$ and $p'$ can be moved as a pair while fixing the underlying surface and position of all other marked points). By Masur-Zorich (Theorem \ref{T:MZ}), $\ColTwo(\bfC_1)/J_2$ is a simple envelope. It follows that $\bfC_1$ consists of either a single simple cylinder (in which case we will let $P$ denote the empty set) or a pair of adjacent simple cylinders whose common boundary contains a marked point $P$. Let $\For'(X, \omega)$ denote $(X, \omega)$ with the points in $P$ forgotten and let $\For'(\bfC_1)$ denote the image of $\bfC_1$ on $\For'(X, \omega)$. The image is a simple cylinder. 

We will show that $\For'(\bfC_1)$ has singularities of the metric on both of its boundaries. If not, then one boundary of $\For'(\bfC_1)$ contains a marked point, and it is easy to see that, up to a cylinder deformation, $\ColOneX$ is simply $(X, \omega)$ with some marked points forgotten. This contradicts the fact that the top boundary of $\bfC_2$ contains a singularity of the metric and $\ColOne(\bfC_2)$ does not (by Sublemma \ref{SL:C2Invariance1}).

Since $\bfC_1$ is a subequivalence class and since  $\For'(\bfC_1)$ is a simple cylinder with singularities on both of its boundaries, it follows from Lemma \ref{L:MarkedPointsAndInvariance} that $T_0$ fixes $\For'(\bfC_1)$, which is impossible since $\For'(\bfC_1)$ is a simple cylinder.
\end{proof}

Therefore, if the top (resp. bottom) boundary of $\bfC_2$ contains a marked point $p$, then that entire boundary does not contain a singularity of the flat metric (by Sublemma \ref{SL:C2Invariance3}) and so $\bfC_2$ is $T_0$ invariant and the entire top (resp. bottom) boundary of $\bfC_2$ only contains saddle connections \red with endpoints in $\{p, T_0(p) \}$ \black (by Sublemma \ref{SL:C2Invariance2}). In particular, $T_0$ preserves the collection of boundary saddle connections on $\bfC_2$. If no boundary of $\bfC_2$ contains a marked point, then it is clear that $T_0(\bfC_2) = \bfC_2$ by Lemma \ref{L:MarkedPointsAndInvariance}.


Now we turn to the second claim. First we observe that by Assumption \ref{A:T0Involution}, $(X_{min}, \omega_{min}) = (X, \omega)/T_0$ and so by Apisa-Wright (Lemma \ref{L:ApisaWright}), any irreducible pair of points of slope $+1$ is a pair of points exchanged by $T_0$. 

Suppose that there is a marked point $p$ on $(X, \omega)$ that is either free or such that $T_0(p)$ is also marked. Suppose to a contradiction that $p$ does not belong to the boundary of $\bfC_2$. By $T_0$-invariance of $\bfC_2$, $T_0(p)$ also does not belong to the boundary of $\bfC_2$.

When $p$ is free, $\ColTwo(p)$ remains a free point; otherwise,  $\{ \ColTwo(p), \ColTwo(T_0(p))\}$ is a slope $+1$ irreducible pair of marked points (the fact that these points remain marked points follows from the fact that they do not belong to the boundary of $\bfC_2$). Such collections of marked points are not permitted on a generic surface, e.g. $\ColTwoX$, in a quadratic double. This is a contradiction. \red (Note that there must be at least one marked point or singularity not in $\{ \ColTwo(p), \ColTwo(T_0(p))\}$, namely one of the endpoints of $\ColTwo(\bfC_2)$, which rules out the possibility that $\MTwo$ is $\cH(0)$ or the locus in $\cH(0,0)$ where difference of the two marked points is two-torsion.) \black
\end{proof}



\begin{cor}\label{C:Commuting}
When $\bfC_1$ is $T_0$-invariant, $\ColOne(T_0) = T_1$. 
\end{cor}

\red The $T_0$-invariance of $\bfC_1$ is required to define $\ColOne(T_0)$. \black

\begin{proof}


By Lemma \ref{L:C2TInvariant}, $\ColOne(T_0)$ preserves $\ColOne(\bfC_2)$ and the saddle connections on its boundary. By Lemma \ref{L:StructureOfS}, $\ColOne(\bfC_2)$ is either a pair of simple cylinders or a complex cylinder, so there is a unique translation involution on $\ColOneX$ with this property. Since $T_1$ also has this property we have that $\ColOne(T_0) = T_1$. 
%
%
\end{proof}

\begin{lem}\label{L:C1TInvariant}
$\bfC_1$ is either $T_0$-invariant or contains a slope $-1$ irreducible pair of marked points in its boundary.
\end{lem}

For an example of the second possibility, see Figure \ref{F:NoExtraSymmetry}. 

\begin{proof}
If the boundary of $\bfC_1$ contained no marked points then $\bfC_1$ would be $T_0$ invariant by Lemma \ref{L:MarkedPointsAndInvariance}. Suppose therefore that $p$ is a marked point contained in the boundary of a cylinder in $\bfC_1$. By Lemma \ref{L:C2TInvariant}, $p$ cannot be a free point or part of a slope $+1$ irreducible pair of marked points. By \cite[Theorem 1.3]{ApisaWright}, if $p$ is not free, then either $p$ is periodic or there is a marked point $p'$ such that $\{p, p'\}$ is a slope $\pm 1$ irreducible pair (in fact the cited result says that there is another point $p'$ that has the same image as $p$ in a map to a quadratic differential).


We begin by showing that, 
\red if $p$ and $p'$ are a slope $-1$ pair with $p$ on the boundary of $\bfC_1$, then \black $p'$ also belongs to the boundary of $\bfC_1$. If not, then $\ColOne(p')$ remains a marked point on $\ColOneX$ and either $\{\ColOne(p), \ColOne(p') \}$ is a slope $-1$ irreducible pair of marked points or $\ColOne(p')$ is a periodic point. However, surfaces in Abelian doubles, such as $\ColOneX$, do not have such collections of marked points, \red provided that the Abelian double is not a double of $\cH(0)$ (which is not the case here since $\MOne$ contains $\MOneTwo$ in its boundary, whereas the boundary of a rank one rel zero invariant subvariety contains no finite area translation surfaces). \black So we have a contradiction. Therefore, the slope $-1$ irreducible pair $\{p, p'\}$ is contained in the boundary of $\bfC_1$. 

By the preceding paragraph \red it remains only to consider the case \black that the only marked points on the boundary of $\bfC_1$, the collection of which we denote by $R$, are periodic points. We will now show that $\bfC_1$ is $T_0$-invariant. Since $\ColTwo(\bfC_1)$ is a subequivalence class of generic cylinders in a quadratic double with a periodic point in its boundary, it follows that $\ColTwo(\bfC_1)/J_2$ is a generic envelope and the periodic points, including those in $\ColTwo(R)$,  are preimages of poles. Let $\For'(X, \omega)$ denote $(X, \omega)$ with the points in $R$ forgotten and let $\For'(\bfC_1)$ denote the image of $\bfC_1$ under this map. It is then clear that $\For'(\bfC_1)$ is a single cylinder whose boundary does not contain marked points and hence which is fixed by $T_0$ by Lemma \ref{L:MarkedPointsAndInvariance}.
\end{proof}

\subsection{Extra symmetry case}\label{S:ExtraSymmetry}
The goal of this subsection is to show that if $\cM$ has extra symmetry, then $\cM$ is as described in Theorem \ref{TP2}. We begin with the following special case, which can occur even when $\cM$ does not have extra symmetry (see Figures \ref{F:NoExtraSymmetry} and \ref{F:FinalPossibility}).


\begin{lem}\label{L:MarkedPointsFTW}
If $\cM$ has rank at least two and the boundary of $\bfC_1$ contains a slope $-1$ irreducible pair of marked points, then $\cM$ is as in Theorem \ref{TP2} (\ref{I:FinalPossibility}).
\end{lem}
\begin{proof}

Recall that by Assumption \ref{A:DenseComplicated}, $(X, \omega)$ has a dense orbit in $\cM$.  By Apisa-Wright (Lemma \ref{L:ApisaWright}) since $\cM$ has rank at least two and since $(X, \omega)$ contains a slope $-1$ irreducible pair of marked points, $(X, \omega)$ admits a map to a quadratic differential. 

\red Since $\ColTwo(\bfC_1)$ is a generic subequivalence class with a slope $-1$ irreducible pair of points on its boundary in a quadratic double, our explicit understanding of such subequivalence classes implies that collapsing $\ColTwo(\bfC_1)$ is, up to a cylinder deformation, the same as forgetting the pair of marked points. The same holds when collapsing $\bfC_1$ on $(X, \omega)$. \black



We will now establish the three claims made in Theorem \ref{TP2} (\ref{I:FinalPossibility}). 
\begin{itemize}
    \item By Theorem \ref{T:MinimalCover}, the generic surface in $\cH_1$ - the stratum containing $\ColOneX/T_1$ - also admits a map to a quadratic differential and hence $\For(\cH_1)$ is hyperelliptic by Lemma \ref{L:InvolutionImpliesHyp-background}.
    \item Since $\ColOneX$ is, up to a cylinder deformation, simply $(X, \omega)$ with the slope $-1$ pair of marked points forgotten, it follows that $\cM$ is a full locus of covers of a locus $\cN$ in a component of a  stratum of Abelian differentials $\cH_0$ where $\For(\cN) = \For(\cH_0) = \For(\cH_1)$ is hyperelliptic. The slope $-1$ irreducible pair must correspond to a pair of points on surfaces in $\cN$ that are exchanged by the hyperelliptic involution, and in fact $\cN$ must be exactly the locus in $\cH_0$ where this pair of marked points is exchanged by the involution. In particular, $\cN$ has codimension 1. 
    \item Since $\ColOneTwoX/T$ belongs to $\cH^{hyp}(2g-2)$ or $\cH^{hyp}(g-1, g-1)$ for some $g \geq 1$ this shows that there is at most one free marked point on $\ColOneX/T_1$ and hence on surfaces in $\cN$.\qedhere
\end{itemize}
%
\end{proof}

\begin{lem}\label{L:JExtraSymmetry}
Suppose that $J$ extends to $\ColOneX$ in the sense of Definition \ref{D:ExtraSymmetry}. Then one of the following holds:
\begin{enumerate}
    \item $\cM$ is a quadratic double, 
    \item $\cM$ is as in Theorem \ref{TP2} (\ref{I:FinalPossibility}), or
    \item $T$ extends to $\ColTwoX$ in the sense of Definition \ref{D:ExtraSymmetry}.
\end{enumerate}
\end{lem}
\begin{proof}
Let $J_1$ denote the extension of $J$ to $\ColOneX$. By the Diamond Lemma (Lemma \ref{L:diamond}), there is an involution $J_0$ on $\For(X, \omega)$ such that $(X, \omega)/J_0$ is a surface in a component $\cQ'$ of a stratum of quadratic differentials. 

To proceed, one might be tempted to try to show that $\MOne$ is a quadratic double. However, as illustrated in Figure \ref{F:EmptyP} \red (after applying a half-Dehn twist to the cylinders labelled $\bfC_1$) \black this is not necessarily the case, which requires us to use a slightly more involved argument.  

By Lemma \ref{L:StructureOfS}, $\ColOne(\bfC_2)$ either contains one complex cylinder or two simple cylinders; and $J_1$ preserves the collection of saddle connections on the boundary of $\ColOne(\bfC_2)$. Therefore, $\ColOne(\bfC_2)/J_1$ is one of the following: a complex envelope, a single simple cylinder, or a pair of isometric simple envelopes. We will single out the case where $\ColOne(\bfC_2)/J_1$ is a pair of isometric simple envelopes for special attention.

\begin{sublem}\label{SL:PairIsometricEnvelopes}
If $\ColOne(\bfC_2)/J_1$ is a pair of simple envelopes, then either $\cM$ is as in Theorem \ref{TP2} (\ref{I:FinalPossibility}) or $T$ extends to $\ColTwoX$ in the sense of Definition \ref{D:ExtraSymmetry}.
\end{sublem}
For an illustration of a diamond where this phenomenon occurs, see Figure \ref{F:NoExtraSymmetry}. It is necessary to modify the figure by collapsing the horizontal cylinder above $\bfC_2$ after performing a half Dehn twist in it.

\begin{proof}
The assumption on $\bfC_2$ implies that $\ColTwo(\bfC_2)/J_2$ is a pair of generically parallel saddle connections, each with one endpoint at a pole. By Masur-Zorich (Theorem \ref{T:HatHomologous}), the complement of $\ColTwo(\bfC_2)/J_2$ is a connected translation surface. The main consequences of this that will be used are the following:
\begin{enumerate}
    \item $\ColTwoX - \ColTwo(\bfC_2)$ is disconnected.
    \item\label{I:Simple} $\ColTwo(\bfC_1)/J_2$ is a simple cylinder. So $\ColTwo(\bfC_1)$, and hence also $\bfC_1$, consists of a pair of simple cylinders.
    \item\label{I:TwoComponents} The two components of $\ColOneTwoX - \ColOneTwo(\bfC_2)$ are exchanged by $T$, since $\ColOneTwo(\bfC_2)$ is a pair of saddle connections exchanged by $T$.
\end{enumerate}  



If $\ColOneTwo(\bfC_1)$ consists of two saddle connections exchanged by $T$, then since $\ColTwo(\bfC_1)$ is a pair of simple cylinders, there is an involution $T_2$ on $\ColTwoX$ such that $\Col(T_2) = T$ and $T_2(\ColTwo(\bfC_1)) = \ColTwo(\bfC_1)$. Therefore we may assume that $\ColOneTwo(\bfC_1)$ consists of two saddle connections that are $J$-invariant but not $T$-invariant. 

Denote the image of $\ColOneTwo(\bfC_1)$ on $\ColOneTwoX/T$ by $\ColOneTwo(\bfC_1)/T$, even though $\ColOneTwo(\bfC_1)$ is not $T$-invariant. This image consists of two saddle connections, which by Lemma \ref{L:HypParallelism}, must be exchanged by the hyperelliptic involution. Cutting these saddle connections disconnects $\ColOneTwoX/T$ into two components. Let $\Sigma$ denote the component that does not contain $\ColOneTwo(\bfC_2)/T$, keeping in mind that $\ColOneTwo(\bfC_2)/T$ is a single saddle connection. The preimage $\wt{\Sigma}\subset \ColOneTwoX$ of $\Sigma$  consists of two components each isometric to $\Sigma$, by the observation \eqref{I:TwoComponents} above.


Before proceeding we will prove that $\cM$ has rank at least two. If this were not the case, then $\ColOneTwo(\bfC_1)$ and $\ColOneTwo(\bfC_2)$ would be generically parallel to each other (for instance by Lemma \ref{L:RankTest}). Hence the saddle connections in $\ColOneTwo(\bfC_1)/T$ and $\ColOneTwo(\bfC_2)/T$ would all be generically parallel to each other. Since $\ColOneTwo(\bfC_1)/T$ is a set of two saddle connections exchanged by the hyperelliptic involution,  Lemma \ref{L:HypParallelism} gives that there cannot be any other saddle connection generically parallel to $\ColOneTwo(\bfC_1)/T$, so we get a contradiction. 

Our strategy is to show that, when $\Sigma$ is a simple cylinder, there is a slope $-1$ irreducible pair of marked points on the boundary of $\bfC_1$ (and so we can conclude by Lemma \ref{L:MarkedPointsFTW}), and otherwise to deduce a contradiction.




\bold{Case 1: $\Sigma$ is a simple cylinder.} Then $\wt{\Sigma}$ is a pair of simple cylinders. Let $\bfC_3$ denote the corresponding pair of simple cylinders on $\ColOneX$, so $\Col_{\ColOne(\bfC_2)}(\bfC_3) = \wt{\Sigma}$. It is possible to find such cylinders since the subsurface $\Sigma$ does not contain the saddle connection in $\ColOneTwo(\bfC_2)/T$.

Therefore, $(X, \omega)$ can be formed by gluing in a pair of simple cylinders, those in $\bfC_1$, to the boundary of the pair of simple cylinders in $\bfC_3$. This produces a slope $-1$ irreducible pair of marked points on the boundary of $\bfC_1$, so we are done by Lemma \ref{L:MarkedPointsFTW}.

\bold{Case 2: $\Sigma$ is not a simple cylinder.} We will derive a contradiction. 

$\ColOneTwoX/T-\ColOneTwo(\bfC_1)/T$ has two components, and $\ColOneTwo(\bfC_2)/T$ is a single saddle connection. Hence  
$\ColOneTwoX/T - \ColOneTwo(\bfC_1 \cup \bfC_2)/T$ 
also has two components. Let $\Sigma'$ denote the one that is not $\Sigma$. 

By the Collapsing Lemma (Corollary \ref{C:Collapsing2}), we may assume that $\Sigma$ is a slit torus and $\Sigma'$ is a parallelogram (these were defined in Definition \ref{D:Parallelogram}). 

We will now prove that the surface $\ColOneTwoX$ is as depicted in Figure \ref{F:SecondAttackPrep} (left). To start, since $\ColOneTwoX/T = \overline{\Sigma \cup \Sigma'}$ is obtained by gluing a parallelogram into a slit torus, we see that $\ColOneTwoX/T\in \cH(2)$. This surface is depicted in Figure \ref{F:SecondAttackPrep} (right). 
\begin{figure}[h!]\centering
\includegraphics[width=\linewidth]{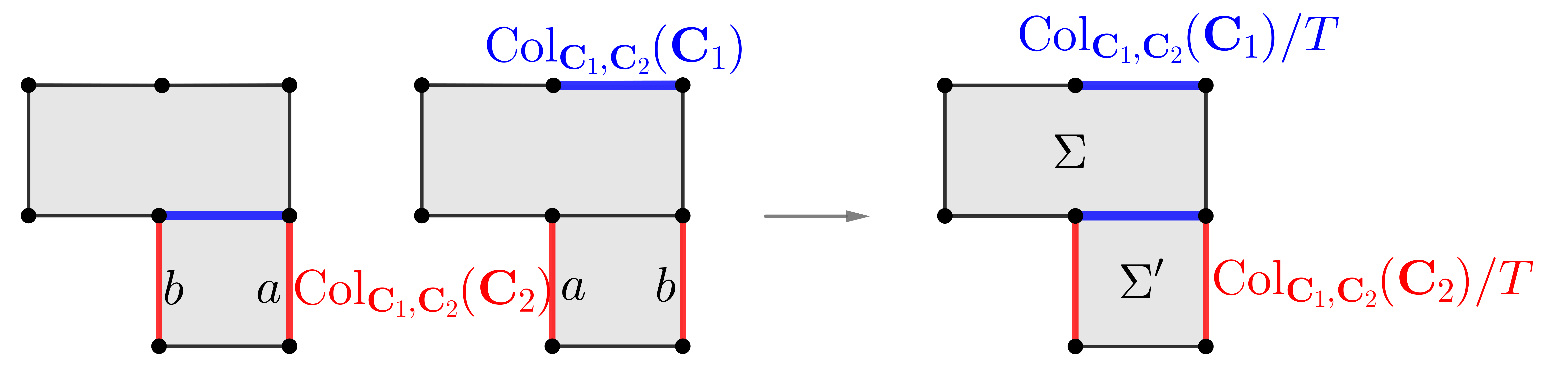}
\caption{An illustration of $\ColOneTwoX$ in Case 2.}
\label{F:SecondAttackPrep}
\end{figure}

The parallelogram $\Sigma'$ has its boundary comprised of the saddle connections in $\ColOneTwo(\bfC_1)/T$ and $\ColOneTwo(\bfC_2)/T$, as in Figure \ref{F:SecondAttackPrep}. The saddle connections in $\ColOneTwo(\bfC_2)$ are exchanged by $T$ and cutting them disconnects the surface into two surfaces with boundary that are exchanged by $T$ and isometric to $\ColOneTwoX/T - \ColOneTwo(\bfC_2)/T$. Since $\bfC_1$ and $\bfC_2$ both consist of simple cylinders (for $\bfC_1$ this follows by \eqref{I:Simple} and for $\bfC_2$ it follows by assumption), it follows that $(X, \omega)$ is the surface shown in Figure \ref{F:SecondAttack} (top).

\begin{figure}[h!]\centering
\includegraphics[width=.7\linewidth]{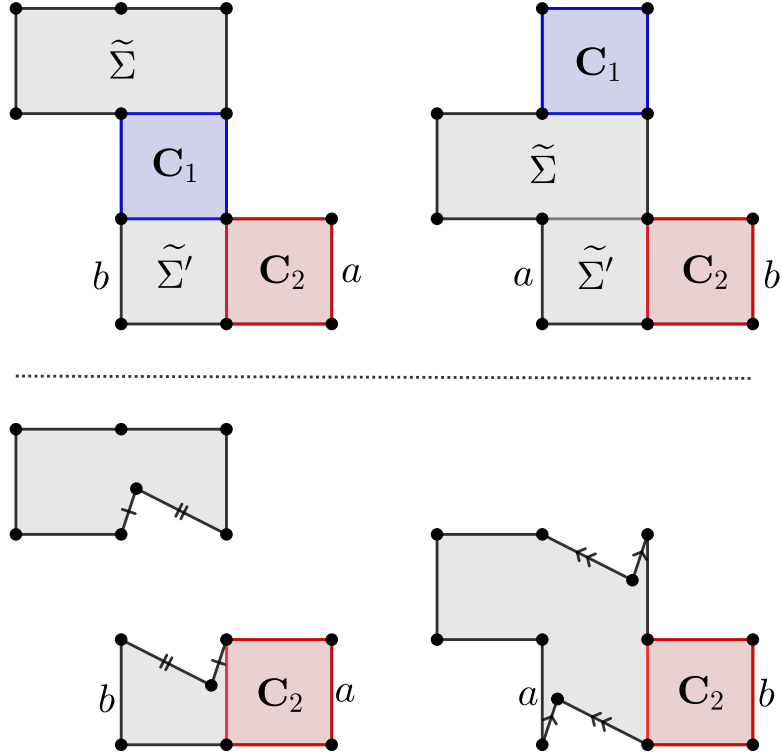}
\caption{An illustration of overcollapsing $\bfC_1$ to derive a contradiction in Sublemma \ref{SL:PairIsometricEnvelopes}. The upper part of the figure is $(X, \omega)$, and the lower part is the ``attack" deformation in the final paragraph of the proof.}
\label{F:SecondAttack}
\end{figure}

We now claim that $\cM$ has no rel. Using Figure \ref{F:SecondAttack}, one may check that $\cM$ is contained in the quadratic double of $\cQ(8, -1^4)$, which has rank three rel one (by Lemma \ref{L:Q-rank}). Since $\MOneTwo$ is an Abelian double of $\cH(2)$, $\MOneTwo$ has complex dimension four and so, by definition of generic diamond (Definition \ref{D:GenericDiamond} \eqref{E:one}), $\cM$ has complex dimension six. Since $\cM$ is contained in an invariant subvariety of rank three rel one, $\cM$ has rank three rel zero, as desired.

We're now going to describe a deformation of the surface in which we use the cylinders in $\bfC_1$ to ``attack" those in $\bfC_2$, resulting in a family of surfaces where the modulus of one of the two cylinders in $\bfC_2$ changes but not other. This will be a contradiction since $\cM$ has rel zero (see for instance Theorem \ref{T:CylECTwistSpace}).

We will describe this deformation by making reference to Figure \ref{F:SecondAttack}. While keeping all other edges constant, take the corner of the cylinders in $\bfC_1$ and move it into the subsurface labelled $\wt{\Sigma'}$ as shown in Figure \ref{F:BigAttack}. Continue moving these corners into $\bfC_2$. It is clear that two corners enter one cylinder in $\bfC_2$, changing its modulus, but not the other. This is the desired deformation, which produces the desired contradiction.
\end{proof}

Letting $\cM'$ denote the orbit closure of $(X, \omega)/J_0$ we have that $$\left( (X, \omega)/J_0, \cM', \bfC_1/J_0, \bfC_2/J_0 \right)$$ is a generic diamond where $\bfC_2/J_0$ is either a simple cylinder or complex envelope. (The final possibility for $\bfC_2/J_0$ - a pair of isometric simple envelopes - has been dealt with in Sublemma \ref{SL:PairIsometricEnvelopes}.)

When $\bfC_2/J_0$ is a simple cylinder, since $\ColTwoX/J_2$ is a generic surface in $\cQ_2$, Corollary \ref{C:codim1} gives that $\cM' = \cQ'$. When $\bfC_2/J_0$ is a complex envelope,  Theorem \ref{T:complex-gluing0}  gives that one of the following occurs: $\cM' = \cQ'$, $\For(\cM') = \For(\cQ')$ is a hyperelliptic component and all marked points on surfaces in $\cM'$ are free except for one pair of points in the boundary of $\bfC_2/J_0$ exchanged by the hyperelliptic involution, or $\For(\cM')$ is a codimension one hyperelliptic locus in $\For(\cQ')$ (which is non-hyperelliptic) and all marked points are free. 


\begin{sublem}\label{SL:MarkedPointReduction}
The collection of marked points on $(X, \omega)$ is invariant under $J_0$ and no marked points are contained in the boundary of $\bfC_2$. 

\end{sublem}
\begin{proof}
Let $p$ be a marked point on $(X, \omega)$. Suppose first that $p$ does not belong to $\overline{\bfC}_2$. Then $\ColTwo(p)$ does not belong to $\ColTwo(\bfC_2)$. Notice that $\ColTwo(\bfC_2)$ is $J_2$-invariant since $\bfC_2$ is $J_0$-invariant (which in turn follows from $\ColOne(\bfC_2)$ being $J_1$-invariant). So $J_2(\ColTwo(p))$ is disjoint from $\ColTwo(\bfC_2)$ and marked since $\ColTwoX$ is a quadratic double. It follows that $J_0(p)$ is also a marked point on $(X, \omega)$. 

Now suppose to a contradiction that $p$ does belong to the closure of $\bfC_2$. By our conventions, $p$ lies in the boundary of $\bfC_2$; a marked point by definition cannot be on the interior of a cylinder. 

Since $\bfC_2$ is fixed by $J_0$ it follows that $J_0(p)$ is marked since, by Lemma \ref{L:StructureOfS}, $J_1$ and hence $J_0$ preserves the collection of saddle connections on the boundary of $\ColOne(\bfC_2)$ (resp. $\bfC_2$). But then $\{ \ColOne(p), J_1\left( \ColOne(p) \right) \}$ is a slope $-1$ irreducible pair of marked points on $\ColOneX$. Such collections of marked points do not exist on surfaces in Abelian doubles, so we have a contradiction. 
\end{proof}

By Sublemma \ref{SL:MarkedPointReduction}, if $\For(\cM') = \For(\cQ')$ then $\cM' = \cQ'$ and $\cM$ is a quadratic double.  We will assume for the remainder of the proof that $\For(\cM')$ is a codimension one hyperelliptic locus in $\For(\cQ')$, which implies that there is a translation involution $T_0$ on $(X, \omega)$ (since half translation maps between quadratic differentials induce maps between their holonomy double covers).

Since any component of a rank one stratum of quadratic differentials is hyperelliptic, our assumption implies that $\For(\cQ')$ has rank at least two. Since $\For(\cM')$ is codimension one in $\For(\cQ')$ it follows that $\For(\cM')$, $\cM'$, and $\cM$ all have rank at least two. Therefore, $(X, \omega)/T_0$ belongs to a quadratic double of a genus zero stratum of rank at least two. Since no genus zero stratum of rank at least two is hyperelliptic it follows from Lemma \ref{L:InvolutionImpliesHyp-background} that $(X, \omega)/T_0$ is not a translation cover of degree greater than one of another surface. We have shown that Assumption \ref{A:T0Involution} is satisfied. 

By Lemma \ref{L:C1TInvariant}, either $\bfC_1$ is $T_0$-invariant or $(X, \omega)$ contains a slope $-1$ irreducible pair of marked points contained in the boundary of $\bfC_1$. In the latter case, we are done by Lemma \ref{L:MarkedPointsFTW}. In the former case, $\ColOne(T_0) = T_1$ (by Corollary \ref{C:Commuting}) and $\bfC_2$ is $T_0$-invariant (by Lemma \ref{L:C2TInvariant}), implying that $T_2 := \ColTwo(T_0)$ is defined, $\Col(T_2) = T$, and $T_2(\ColTwo(\bfC_1)) = \ColTwo(\bfC_1)$.
\end{proof}

\begin{prop}\label{PP2}
If $\cM$ has extra symmetry, then it is as in Theorem \ref{TP2}.
\end{prop}
\begin{proof}
By Lemma \ref{L:JExtraSymmetry}, it remains to study the case where there is an involution $T_2$ on $\For(\ColTwoX)$ such that $\Col(T_2) = T$ and $T_2(\ColTwo(\bfC_1)) = \ColTwo(\bfC_1)$.

By the Diamond Lemma (Lemma \ref{L:diamond}), $\cM$ is a locus of translation double covers. Let $T_0$ denote the translation involution on $\For(X, \omega)$. (Even though we do not by default assume that involutions preserve marked points, we use $\For$ here to emphasize that marked points need not be preserved.) 

Let $\cM'$ denote the orbit closure of $(X, \omega)/T_0$ and let $\cH'$ denote the component of the stratum of Abelian differentials containing it. Then $\left( (X, \omega)/T_0, \cM', \bfC_1/T_0, \bfC_2/T_0 \right)$ is a generic diamond where $\cM'_{\bfC_1/T_0} = \cH_1$. 

\begin{sublem}\label{SL:Commuting}
$\cM'_{\bfC_2/T_0}$ is a quadratic double.
\end{sublem}

\begin{proof}
We begin by showing that $J_2$ and $T_2$ commute. Notice that $\ColTwoX - \ColTwo(\bfC_1)$ is isometric to $\ColOneTwoX - \ColOneTwo(\bfC_1)$ and that under this identification the restriction of $J_2$ and $T_2$ agree with the restriction of $J$ and $T$ respectively. This shows that $J_2$ and $T_2$ commute on an open set (since this is true of $J$ and $T$ by Lemma \ref{L:AllAboutThatBase}) and hence they must commute since they are holomorphic maps.

Since the set of marked points (and zeros) on $\ColTwoX$ is $J_2$ invariant, and since $J_2$ and $T_2$ commute, the set of marked points is also invariant on $\ColTwoX/T_2$ by the involution induced by $J_2$. Since $\ColTwoX/T_2$ has a degree two map to the quadratic differential $\ColTwoX/\langle J_2, T_2\rangle$, we get that $\ColTwoX/T_2$ is contained in a quadratic double. 

By perturbing we may assume that $\ColTwoX$ is generic in $\MTwo$. This implies that $\ColTwoX/J_2$ has dense $\GL$-orbit in its stratum. The same must hold for $\ColTwoX/\langle J_2, T_2 \rangle$ since it is the image of $\ColTwoX/J_2$. Since $\ColTwoX/T_2$ is the holonomy double cover of a quadratic differential with dense orbit in its stratum, i.e. $\ColTwoX/\langle J_2, T_2 \rangle$, \red $\cM'_{\bfC_2/T_0}$ \black is a quadratic double. 
%
%
\end{proof}



By Proposition \ref{P:DiamondWithH}, $\cM'$ (and hence also $\cM$) has rank at least two and one of the following occurs:
\begin{enumerate}
    \item $\cM' = \cH'$ and there is at most one marked point on surfaces in $\cM'$.
    \item $\For(\cM') = \For(\cH')$ is a hyperelliptic component and there is at most one free marked point on surfaces in \red $\cM'$ \black with the remaining marked points being a collection of either one marked point fixed or two marked points exchanged by the hyperelliptic involution. 
    \item $\For(\cM')$ is a codimension one hyperelliptic locus in $\For(\cH')$ and there is at most one marked point, which is free.
\end{enumerate}

\noindent In the second possibility listed above, $\cM$ is described by Theorem \ref{TP2} (\ref{I:FinalPossibility}).

\begin{sublem}\label{SL:FreePoints}
All preimages of any free marked point on $(X, \omega)/T_0$ are marked on $(X, \omega)$. Thus, when $\cM' = \cH'$, $\cM$ is an Abelian double. 
\end{sublem}
\begin{proof}
Suppose that there is a marked point on $(X, \omega)/T_0$ and that its preimages contain marked points. Suppose too that exactly one preimage is marked and call this point $p$. 


\red Suppose first that $p$ does not lie on the boundary of $\bfC_2$. Since $p$ is free, $\ColTwo(p)$ is a free marked point in a quadratic double. Such a point only exists in $\cH(0,0)$ - the quadratic double coming from $\cQ(-1^4,0)$ with no preimages of poles marked. Therefore, $\MTwo = \cH(0,0)$ and $\Col_{\bfC_2}(\bfC_2)$ would be a saddle connection starting and ending at the marked point that isn't $p$. This implies that $\Col_{\bfC_2}(\bfC_2)$ is a boundary saddle connection of $\ColTwo(\bfC_1)$, which is a contradiction. Therefore, $p$ must lie on the boundary of $\bfC_2$.\black 

However, $p$ must also lie on the boundary of $\bfC_1$ since otherwise $\ColOne(p)$ would be a free marked point in an Abelian double. However, no marked point lies on the boundaries of two disjoint cylinders that don't share boundary saddle connections.
\end{proof}


\begin{sublem}\label{SL:Case3}
When $\For(\cM')$ is a codimension one hyperelliptic locus, $\cM$ is as in (\ref{I:ExtraSymmetry}) of Theorem \ref{TP2}.
\end{sublem}
\begin{proof}
By Sublemma \ref{SL:FreePoints}, it remains to show that $\ColOneTwoX$ is disconnected. If $(X, \omega)/T_0$ contains a marked point $p$ (which would necessarily be free), then it cannot belong to the boundary of $\bfC_1/T_0$ since if it did then $\ColOneX/T_1$ would be formed by colliding the marked point with a singularity and hence would be a surface in a codimension one hyperelliptic locus and hence not a generic surface in $\cH_1$ as required.

Since the generic cylinders on a genus zero half-translation surface,  are simple cylinders and simple envelopes, it follows that $\bfC_1/T_0$ is either a simple cylinder or two simple cylinders with disjoint boundaries. Since $\For(\cM') \ne \cH'$ it follows that $\bfC_1/T_0$ is not a simple cylinder (gluing in a simple cylinder to a generic surface in a component of a stratum of Abelian differentials produces a component of a stratum of Abelian differentials). 

Thus, $\bfC_1/T_0$ consists of two simple cylinders with disjoint boundary. Since $\bfC_1/T_0$ does not contain a marked point in its boundary and since $\bfC_1$ consists of at most two cylinders, $\bfC_1$ is a pair of complex cylinders with disjoint boundary. The same is true of $\ColTwo(\bfC_1)$ and so by Masur-Zorich (Theorem \ref{T:MZ}), $\ColOneTwoX$ is disconnected. 
\end{proof}
The proof of Proposition \ref{PP2} is now complete.
\end{proof}

\subsection{When $\bfC_1$ consists of half-simple or complex cylinders}
The goal of this subsection is to show that Theorem \ref{TP2} holds when $\bfC_1$ consists of half-simple or complex cylinders. Lemmas \ref{L:FinalCountdown}, \ref{L:ForgettingMarkedPoints}, and \ref{L:P(S1)} will each be used once in a later subsection. 

Before we begin, we point out that when $\bfC_1$ consists of two adjacent complex cylinders, $\ColOneTwo(\bfC_1)$ might consist of 2, 3, or 4 saddle connections; see Figure \ref{F:CylinderType5HalfTwist}. 
\begin{figure}[h]\centering
\includegraphics[width=0.2\linewidth]{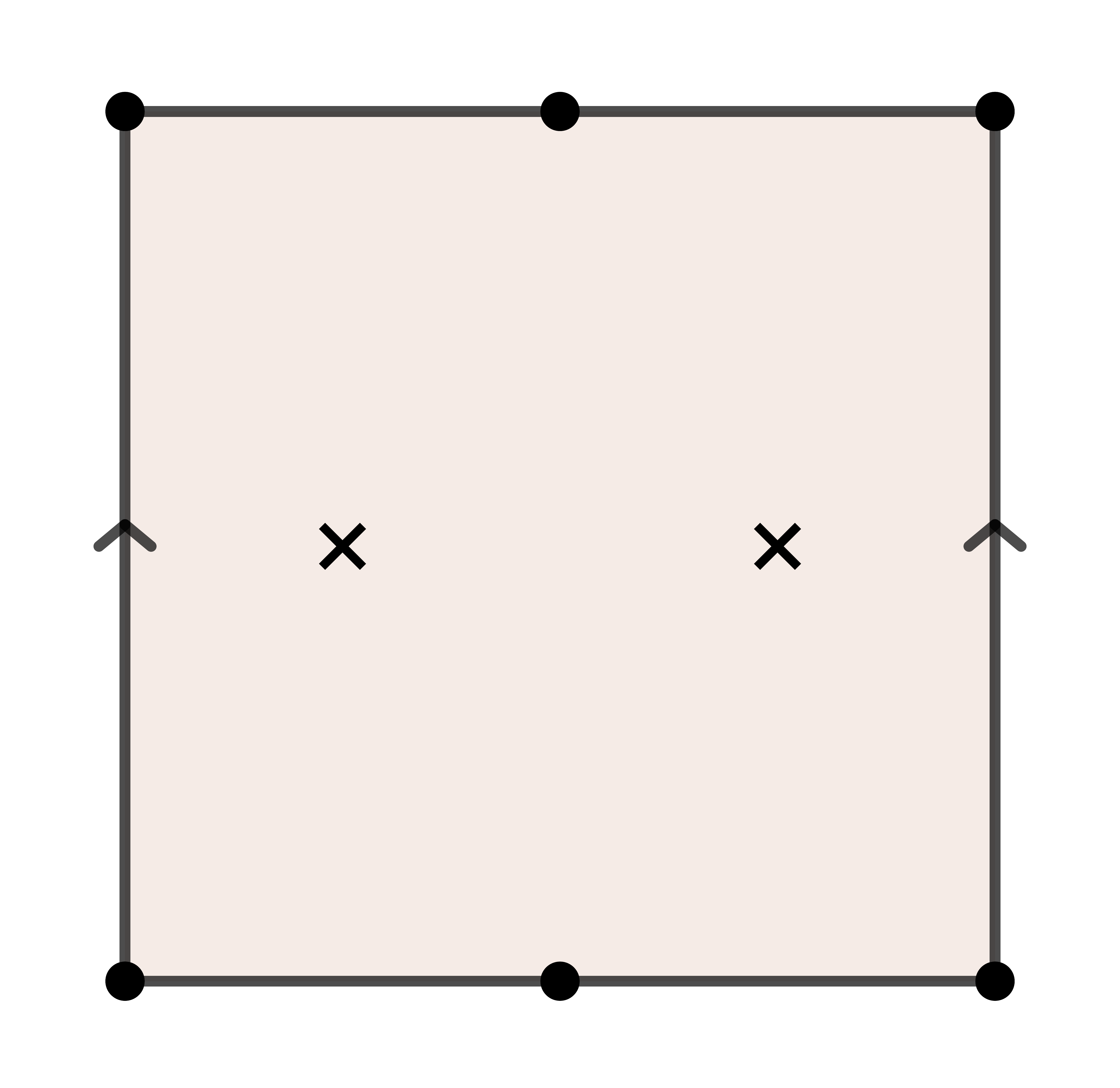}
\caption{A half Dehn twist applied to the bottom right possibility in Figure \ref{F:CylinderTypes} moves the two preimages of poles, which can be marked or unmarked.}
\label{F:CylinderType5HalfTwist}
\end{figure}

\begin{lem}\label{L:NoPairComplexCylinder}
If $\bfC_1$ is a pair of non-adjacent complex cylinders, then $\cM$ has extra symmetry. 
\end{lem}
\begin{proof}
Since $\ColOneTwo(\bfC_1)$ consists of four generically parallel saddle connections, by Lemma \ref{L:GenericParallelism} there is a saddle connection $s$ such that $\ColOneTwo(\bfC_1) = \{s, Ts, Js, JTs\}$. Since $J_2$ exchanges the two cylinders in $\ColTwo(\bfC_1)$, there are only two possibilities for how the two complex cylinders in $\bfC_1$ could be glued into the four saddle connections of $\ColOneTwo(\bfC_1)$ to form $\ColTwoX$. These are illustrated in Figure \ref{F:ComplexCyl}, and we observe that in both cases $\cM$ has extra symmetry since the $T$ involution extends to $\ColTwoX$. (In the right subfigure, the extension of the $T$-involution exchanges the two cylinders in $\ColTwo(\bfC_1)$; in the left subfigure, it fixes each individual cylinder in $\ColTwo(\bfC_1)$.) 
\begin{figure}[h]\centering
\includegraphics[width=0.9\linewidth]{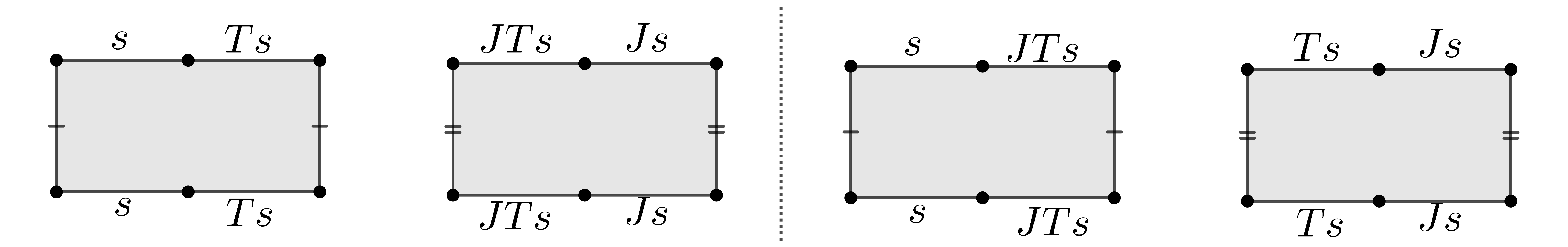}
\caption{Gluing in a pair of non-adjacent complex cylinders produces extra symmetry.}
\label{F:ComplexCyl}
\end{figure}
\end{proof}

\begin{lem}\label{L:FinalCountdown}
Suppose that  $\ColTwo(\bfC_1)$ consists of two non-adjacent cylinders each of which contains a boundary component that is made up of a single saddle connection joining a marked point to itself. Then either $\cM$ has extra symmetry or is as in Theorem \ref{TP2} (\ref{I:FinalPossibility}). 
\end{lem} 
\begin{proof}
Suppose that $\cM$ has no extra symmetry. This implies by Lemma \ref{L:QuickSummary} that $\ColOneTwoX$ has no marked points. Therefore, for each cylinder in $\ColTwo(\bfC_1)$ there is a unique boundary component that is made up of a single saddle connection joining a marked point to itself, since if there were two then there would be a marked point on $\ColOneTwoX$. Let $\bfC$ be the (one or two) cylinders bordering $\ColTwo(\bfC_1)$ along these unique boundary components. Since $\ColTwo(\bfC_1)$ is $J_2$-invariant, so is $\bfC$.

Define $\bfC' := \Col_{\ColTwo(\bfC_1)}(\bfC)$. We will show that $\bfC'/T$ consists of a single cylinder. (Even though $\bfC'$ need not be $T$-invariant, we will use the notation $\bfC'/T$ to denote its image in the quotient.) If not, then $\bfC$ consists of two cylinders exchanged by $J_2$ (this follows since $J_2$ must exchange the unique boundary components of the cylinders in $\ColTwo(\bfC_1)$ and hence exchange the two cylinders bordering them). Therefore, the cylinders in $\bfC'$ consist of two cylinders exchanged by $J$. Thus, the hyperelliptic involution exchanges the two cylinders in $\bfC'/T$, which contradicts the fact that the hyperelliptic involution fixes all cylinders. 

Therefore, $\bfC'/T$ is a single cylinder with one component of its boundary consisting of a single saddle connection. Since $\bfC'/T$ and $\ColOneTwo(\bfC_1)/T$ are fixed by the hyperelliptic involution (which exchanges the two boundaries of $\bfC'/T$), it follows that the boundary saddle connections of $\bfC'/T$ are exactly the saddle connections in $\ColOneTwo(\bfC_1)/T$. Therefore, $\bfC'/T$ consists of one simple cylinder. 

Since, by Lemma \ref{L:Cutting}, $\ColOneTwo(\bfC_2)/T$ consists of a single saddle connection that is not fixed by the hyperelliptic involution and that is disjoint from $\ColOneTwo(\bfC_1)/T$ it follows that $\ColOneTwo(\bfC_2)/T$ does not intersect $\bfC'/T$. To see this, notice that any saddle connection contained in a simple cylinder must be fixed by the hyperelliptic involution. 

Since $\ColTwo(\bfC_2)$ does not intersect $\ColTwo(\bfC_1)$ or $\bfC$, it follows that $\bfC_1$ contains a slope $-1$ irreducible pair of marked points in its boundary, because we have assumed that $\ColTwo(\bfC_1)$ has such a pair, and since $(X,\omega)$ is formed by gluing the cylinders in $\bfC_2$ into $\ColTwo(\bfC_2)$. Noting that $\cM$ has rank at least two since it has no extra symmetry (by Lemma \ref{L:QuickSummary}), Lemma \ref{L:MarkedPointsFTW} implies that $\cM$ is as in Theorem \ref{TP2} (\ref{I:FinalPossibility}).
\end{proof}

\begin{lem}\label{L:NoHalfSimpleCylinders}
If $\bfC_1$ is a pair of non-adjacent half-simple cylinders, then either $\cM$ has extra symmetry or is as in Theorem \ref{TP2} (\ref{I:FinalPossibility}).
\end{lem}
\begin{proof}
Suppose that $\cM$ does not have extra symmetry. Let $s$ and $s'$ be two saddle connections in $\ColOneTwo(\bfC_1)$ corresponding to the boundary of one of the cylinders in $\bfC_1$. These saddle connections must meet at a point $p$ such that the angle between them is $\pi$ (either measured clockwise or counterclockwise but not necessarily both; without loss of generality suppose that the angle measured clockwise is $\pi$). Since $s$ and $s'$ are $\MOneTwo$-generically parallel, by Lemma \ref{L:GenericParallelism} it follows that $s' \in \{Js, JTs, Ts\}$. If $s' = Js$ or $s'= JTs$, then $p$ is a fixed point of $J$ or $JT$ respectively. Since $J$ and $JT$ act by $-I$ in a neighborhood of $p$, and since the clockwise angle between $s$ and $s'$ is $\pi$, it would follow that $p$ is a marked point, contradicting Lemma \ref{L:QuickSummary}.

Therefore, $s = Ts'$. This implies that $s$ and $s'$ also meet at $Tp$ and that the angle between them there measured clockwise is also $\pi$. This implies that $s$ and $s'$ form the bottom boundary of a cylinder. Passing from $\ColOneTwoX$ to $\ColTwoX$, which involves cutting along $s$ and $s'$ and gluing in a half-simple cylinder, necessarily produces marked points on $\ColTwoX$ (see Figure \ref{F:HalfSimpleMarkedPoint}) and we are done by Lemma \ref{L:FinalCountdown}.
\begin{figure}[h]\centering
\includegraphics[width=.6\linewidth]{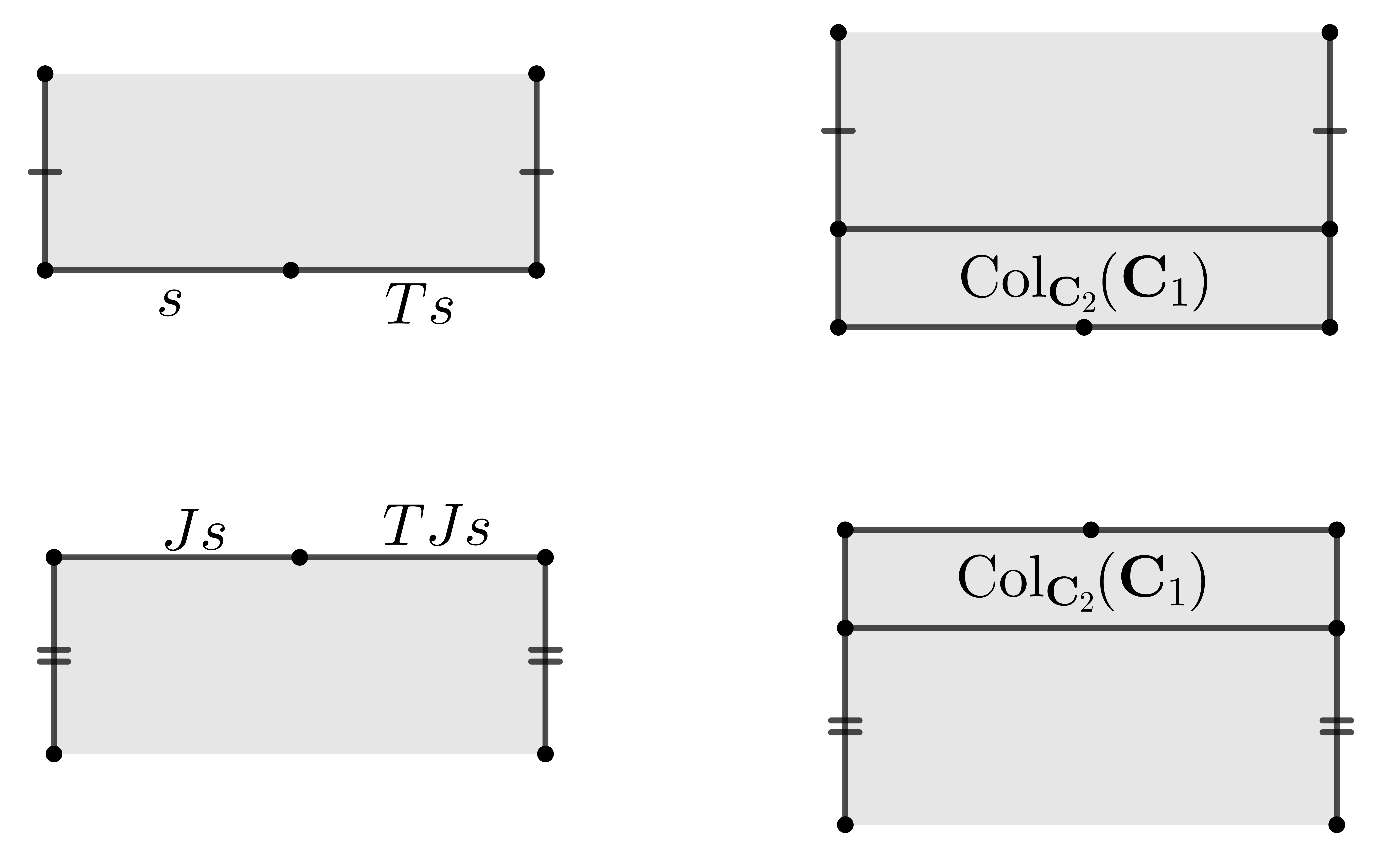}
\caption{Gluing in a half-simple cylinder to the boundary of a cylinder produces marked points.}
\label{F:HalfSimpleMarkedPoint}
\end{figure}
\end{proof}

\begin{lem}\label{L:ForgettingMarkedPoints}
If the cylinders in $\bfC_1$ are adjacent, then the zeros on their common boundary are a collection $P$ of marked points. 

Let $\For'(X, \omega)$ and $\For'(\bfC_1)$ denote $(X, \omega)$ and $\bfC_1$ respectively  once the marked points in $P$ are forgotten. Let $\cM'$ denote its orbit closure. The collection of marked points $P$ are periodic on $\For'(X, \omega)$. 

If $\cM$ has no extra symmetry, then collapsing $\bfC_1$ causes $P$ to merge with other zeros or marked points and  $\left(\For'(X, \omega), \cM', \For'(\bfC_1), \bfC_2 \right)$ is a generic diamond with $\cM'_{\For'(\bfC_1)}$ and $\cM'_{\bfC_2}$ an Abelian (resp. quadratic) double.
\end{lem}
\begin{rem}
Notice that $\For'(X, \omega)$ may still have marked points. The statement that the collection of points $P$ is periodic on $\For'(X, \omega)$ means that $\cM$ and $\cM'$ have the same dimension. For instance, if $\For'(X, \omega)$ is a generic surface in a hyperelliptic component with a point $p$ marked, then the image of $p$ under the hyperelliptic involution would be a periodic point for $\For'(X, \omega)$.
\end{rem}
\begin{proof}
By Masur-Zorich (Theorem \ref{T:MZ}), if the cylinders in $\ColTwo(\bfC_1)$ are adjacent, then the zeros on their common boundary are a collection of at most two preimages of poles. Since $\ColTwo(\bfC_2)$ is disjoint from $\ColTwo(\bfC_1)$, it follows that the zeros on the common boundaries of the cylinders in $\bfC_1$ remain a collection $P$ of marked points on $(X,\omega)$. Since the two cylinders in $\ColTwo(\bfC_1)$ have $\MTwo$-generically identical heights, the cylinders in $\bfC_1$ have $\cM$-generically identical heights, from which it follows that the marked points in $P$ are periodic on $\For'(X, \omega)$. 

Suppose now that $\cM$ has no extra symmetry. It is easy to see $\ColTwo\left(\For'(X, \omega)\right)$ remains a generic surface in a quadratic double (the preimages of poles that were marked before are simply not marked now). We will now show that $\Col_{\For'(\bfC_1)}\left(\For'(X, \omega)\right) = \ColOneX$. It suffices to show that all the points in $\ColOne(P)$ are marked points or singularities on $\ColOne\left(\For'(X, \omega)\right)$. If this is not the case, then $\ColOneTwo(P)$ is a collection of marked points on $\ColOneTwoX$, which contradicts Lemma \ref{L:QuickSummary}.
\end{proof}

\begin{lem}\label{L:P(S1)}
If $\cM$ has no extra symmetry and $\ColOneTwo(\bfC_1)/T$ is a single saddle connection then $\ColOneTwo(\bfC_1)$ is also a single saddle connection. 
\end{lem} 
\begin{proof}
Suppose that $\ColOneTwo(\bfC_1)/T$ is a single saddle connection and that $\ColOneTwo(\bfC_1)$ has more than one saddle connection. (Again, even though $\ColOneTwo(\bfC_1)$ is not assumed to be $T$-invariant, we use $\ColOneTwo(\bfC_1)/T$ to denote its image in $\ColOneTwoX$.) It follows that $\ColOneTwo(\bfC_1)$ consists of two saddle connections exchanged by $T$. 

By Lemma \ref{L:QuickSummary}, $\ColOneTwoX$ does not contain marked points. Therefore, by Theorem \ref{T:MZ} (see Figure \ref{F:CylinderTypes}), $\bfC_1$ consists of either a pair of simple cylinders or a single complex cylinder with at most $n$ preimages of poles marked along its core curve where $n \in \{0, 1, 2\}$. (If $\ColOneTwoX$ had had marked points, it might have also been the case that $\bfC_1$ was a single simple cylinder with a \red pair of marked points \black in its interior.) 

Since $\ColTwoX$ is formed by gluing in $\ColTwo(\bfC_1)$ - which is, up to forgetting marked points, a pair of simple cylinders or a single complex cylinder - to the pair of saddle connections comprising $\ColOneTwo(\bfC_2)$, which are exchanged by $T$; it follows that there is an involution $T_2$ on $\For(\ColTwoX)$ such that $T_2(\ColTwo(\bfC_1)) = \ColTwo(\bfC_1)$ and $\Col_{\ColTwo(\bfC_1)}(T_2) = T$. This implies that $\cM$ has extra symmetry.
\end{proof}

\begin{prop}\label{P:NoSingleComplexCylinder}
If $\bfC_1$ consists of half-simple or complex cylinders, then Theorem \ref{TP2} holds. 
\end{prop}
\begin{proof}
Suppose that $\cM$ does not have extra symmetry. Note that by Masur-Zorich (Theorem \ref{T:MZ}) $\bfC_1$ consists of two half-simple cylinders, one complex cylinder, or two complex cylinders.

By Lemmas \ref{L:NoPairComplexCylinder} and \ref{L:NoHalfSimpleCylinders}, if $\bfC_1$ consists of half-simple cylinders or a pair of complex cylinders we may suppose that they are adjacent and we will let $P$ denote the collection of marked points on their common boundary. If $\bfC_1$ consists of two adjacent half-simple cylinders coming from marking two points on the core curve of a simple cylinder, then $\ColOneTwoX$ would have a marked point contradicting Lemma \ref{L:QuickSummary}. By Masur-Zorich (Theorem \ref{T:MZ}, see Figure \ref{F:CylinderTypes}), it follows that $\For'\left(\ColTwo(\bfC_1) \right)$ is a single complex cylinder ($\For'$ was defined in Lemma \ref{L:ForgettingMarkedPoints}). 


\begin{rem}
Recall that, according to Definition \ref{D:CylAndBoundary}, the boundary of a collection of cylinders $\bfC$ is the  union of the boundary of the cylinders in $\bfC$. This convention might be confusing in the case that $\bfC$ contains two cylinders that are adjacent and separated by marked points. In that case we say that the marked points are on the boundary even though they are in the interior of a cylinder after forgetting marked points.
\end{rem}

By Lemma \ref{L:ForgettingMarkedPoints}, the collection $P$ of marked points on the boundary of $\bfC_1$ are periodic on $\For'(X, \omega)$ and $\left(\For'(X, \omega), \cM', \For'(\bfC_1), \bfC_2 \right)$ is a generic diamond with $\cM'_{\For'(\bfC_1)}$ and $\cM'_{\bfC_2}$ an Abelian (resp. quadratic) double. We will show that this is not possible. 



Notice that there are two saddle connections in $\ColOneTwo(\For'(\bfC_1))$. By Lemma \ref{L:P(S1)}, $\Col_{\For'(\bfC_1), \bfC_2}(\For'(\bfC_1))/T$ also contains at least two saddle connections. By Corollary \ref{C:Cutting}, 
\[ \Col_{\For'(\bfC_1), \bfC_2}\left( \For'(X, \omega) \right)/T - \left( \left( \Col_{\For'(\bfC_1), \bfC_2}(\For'(\bfC_1) \cup \bfC_2) \right)/T \cup \tau\left( \Col_{\For'(\bfC_1), \bfC_2}(\bfC_2)/T \right) \right)\]
consists of three subsurfaces fixed by the hyperelliptic involution. Applying the Collapsing Lemma (Corollary \ref{C:Collapsing2}) we may reduce to a generic diamond $$\left( (X'', \omega''), \cM'', \bfC_1'', \bfC_2'' \right)$$ where two of the previously mentioned subsurfaces (the ones denoted $\Sigma_1$ and $\Sigma_2$ in the notation of Definition \ref{D:Cutting}) are simple cylinders and the third (denoted $\Sigma_3$) is a parallelogram. Let $\{ \mathrm{id}, J'', T'', J''T'' \}$ denote the affine symmetries of $\Col_{\bfC_1'', \bfC_2''}(X'', \omega'')$ as in Lemma \ref{L:AllAboutThatBase}.

\begin{sublem}
$(X'', \omega'')$ is the surface depicted in Figure \ref{F:DisconnectedSurface}.
\end{sublem}


\begin{figure}[h]\centering
\includegraphics[width=.8\linewidth]{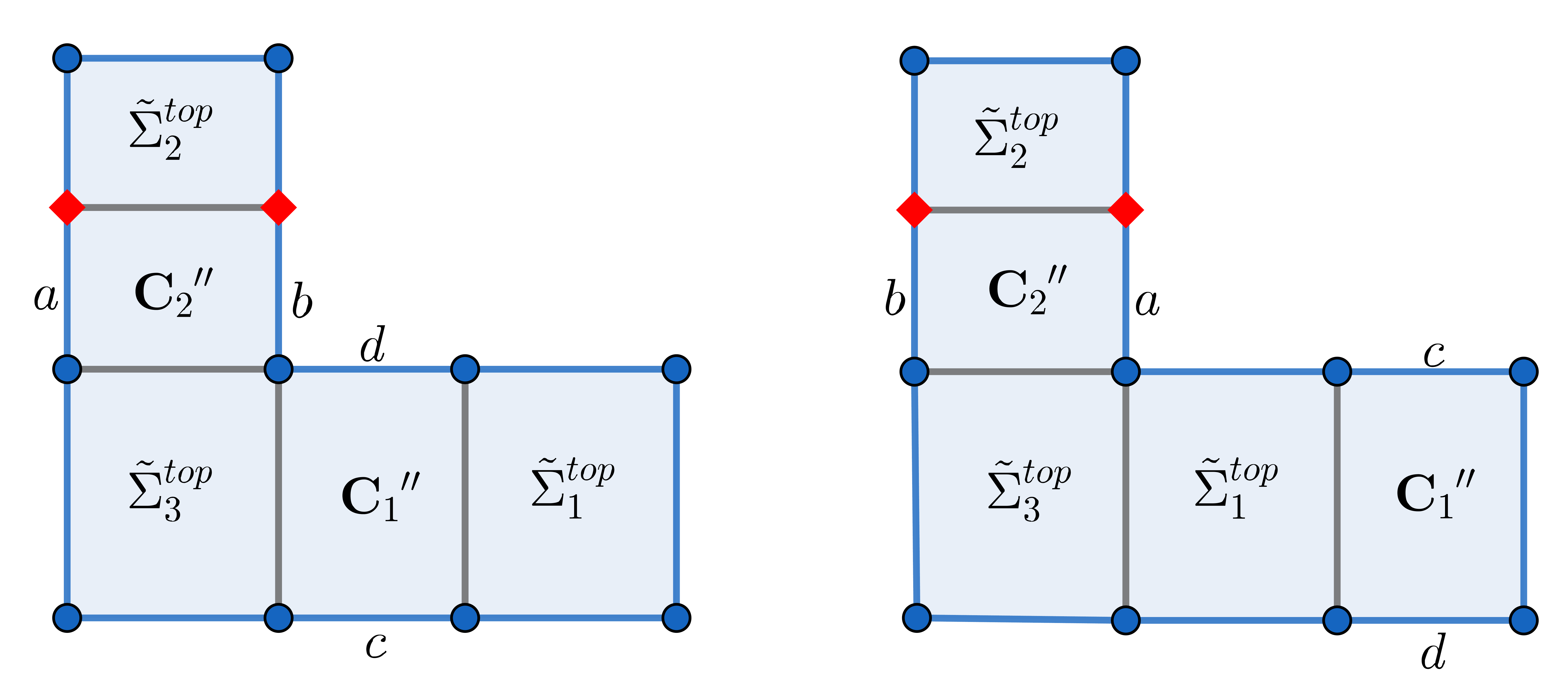}
\caption{A surface in $\cH(7, 1)$. Recall the convention that any unlabelled side is glued to the opposite side of the polygon containing it.}
\label{F:DisconnectedSurface}
\end{figure}

\begin{proof}
By Lemma \ref{L:AllAboutThatBase}, $\Col_{\bfC_1'', \bfC_2''}(X'', \omega'')/T''$ is a connected surface that (by Definition \ref{D:Cutting}) is formed by gluing together $\Sigma_1, \Sigma_2$ and $\Sigma_3$. Since these subsurfaces are two simple cylinders and a parallelogram respectively, it follows that $\Col_{\bfC_1'', \bfC_2''}(X'', \omega'')/T''$ belongs to $\cH(2)$. We have seen that $\Col_{\bfC_1'', \bfC_2''}(\bfC_1'')/T''$ consists of two saddle connections, which, by Definition \ref{D:Cutting}, form the boundary of $\Sigma_1$. Similarly, $\Col_{\bfC_1'', \bfC_2''}(\bfC_2'') \cup \tau\left( \Col_{\bfC_1'', \bfC_2''}(\bfC_2'') \right)$ forms the boundary of $\Sigma_2$. 

Since $\bfC_1''$ is a complex cylinder and $\bfC_2''$ is either a pair of simple cylinders or a single complex cylinder, we see that $(X'', \omega'')$ can be presented as the two polygons in Figure \ref{F:DisconnectedSurface} with some edge identifications that will be deduced now.

First notice that $\Col_{\bfC_2''}(\bfC_1'')$ is the preimage of a complex envelope on a generic surface in a stratum of quadratic differentials. By Masur-Zorich (Theorem \ref{T:MZ} \eqref{T:MZ:TrivHolComponent}), $\Col_{\bfC_2''}(X'', \omega'') - \Col_{\bfC_2''}(\bfC_1'')$ has two connected components. This is only possible if the unlabelled opposite edges in Figure \ref{F:DisconnectedSurface} are glued to each other (if not it is clear that $\Col_{\bfC_2''}(X'', \omega'') - \Col_{\bfC_2''}(\bfC_1'')$ is connected).

Finally, $\bfC_2''$ is a complex cylinder since otherwise it would be a pair of simple cylinders (by Lemma \ref{L:StructureOfS}) and $\Col_{\bfC_1''}(X'', \omega'')$ would be disconnected, contradicting our assumption that it belongs to an Abelian double (which by definition must contain connected surfaces).
\end{proof}

Since $\cM''_{\bfC_1'', \bfC_2''}$ is an Abelian double of $\cH(2)$, it follows that $\cM''$ has rank three rel zero (it cannot be rank two rel two since it is contained in $\cH(7,1)$ which is rel one; and these are the only two possibilities for $\cM''$ since $\cM''_{\bfC_1'', \bfC_2''}$ has dimension exactly two less than $\cM''$). However, the core curve of the cylinder $\bfC_2''$ is homologous to the sum of core curves of the cylinders labelled $\wt{\Sigma}_2^{top}$. Since $\wt{\Sigma}_2^{top}$ and $\bfC_2''$ are both subequivalence classes of cylinders, it follows that $\cM''$ has rel greater than zero, which is a contradiction. 
\end{proof}

\subsection{When $\bfC_1$ is a simple cylinder}
The goal of this section is to show that Theorem \ref{TP2} holds when $\bfC_1$ is a simple cylinder. For convenience we isolate the following several lemmas and show how they imply the main result of the section. These lemmas will also be used in the next subsection. The hypotheses of these lemmas are more general than just assuming that $\bfC_1$ is a simple cylinder. 

\begin{lem}\label{L:ColOneTwoXConnected}
When $\bfC_1$ is a collection of simple cylinders, $\MOneTwo$ is an Abelian and quadratic double. 

\end{lem}
\begin{proof}
It is only necessary to check that $\ColOneTwoX$ is connected, which holds since $\ColTwoX$ is connected and $\ColTwo(\bfC_1)$ is a collection of simple cylinders. \qedhere 
%
%
\end{proof}

\begin{lem}\label{L:Connectedness-New}
Suppose that $\bfC_1$ is a collection of simple cylinders and that $\ColOneTwo(\bfC_2)/T$ is not $\tau$-invariant, then $\ColOneTwoX - \ColOneTwo(\bfC_2)$ is connected. 

Moreover, if $s$ is a $J$-invariant saddle connection, then $$\ColOneTwoX - \{s, Ts\}$$ 
has two connected components exchanged by $T$, each isometric to $\ColOneTwoX/T - s/T$.
\end{lem}
\begin{proof}
We will prove the second claim first. Let $s$ be a $J$-invariant saddle connection. By Lemma \ref{L:ColOneTwoXConnected}, $\MOneTwo$ is an Abelian and quadratic double. The saddle connections in $S := \{s, Ts\}$ are generically parallel and fixed by $J$. Since $S/J$ consists of two hat-homologous saddle connections, it follows, by Masur-Zorich (Theorem \ref{T:HatHomologous}), that $\ColOneTwoX/J - S/J$ has trivial linear holonomy and hence that $\ColOneTwoX - S$ is disconnected. The second claim now follows.


We now turn to the first claim. Notice that $\ColOneTwo(\bfC_2)/T$ is a single saddle connection that is not fixed by the hyperelliptic involution. Let $w$ be any fixed point of $J$ that is not a singularity (such points exist by Lemma \ref{L:QuadAbDouble}). Since $J$ descends to $\tau$ on $\ColOneTwoX/T$ (by Lemma \ref{L:AllAboutThatBase}), $w/T$ is a fixed point of $\tau$ that is not a singularity. Let $\sigma$ be any saddle connection on $\ColOneTwoX/T$ that is disjoint from $\ColOneTwo(\bfC_2)/T$ and that passes through $w/T$. Let $s$ be any saddle connection in the preimage of $\sigma$ on $\ColOneTwoX$. We have already shown (by the second part of the claim) that $\ColOneTwoX - \{s, Ts\}$ consists of two isometric copies of $\ColOneTwoX/T - \sigma$.

Notice that $\ColOneTwoX/T - (\ColOneTwo(\bfC_2)/T \cup \sigma)$ is not disconnected since $\sigma$ is a saddle connection fixed by the hyperelliptic involution and $\ColOneTwo(\bfC_2)/T$ is a saddle connection that is not (by Lemma \ref{L:HypParallelism}, cutting two saddle connections that are not exchanged by $\tau$ does not disconnect the surface). Therefore, $\ColOneTwoX - \{s, Ts, \ColOneTwo(\bfC_2)\}$ consists of two isometric copies of the connected surface $\ColOneTwoX/T - (\ColOneTwo(\bfC_2)/T \cup \sigma)$. Since adding $s$ back in joins the two components together, $\ColOneTwoX - \ColOneTwo(\bfC_2)$ is connected.
\end{proof}

Recall from Definition \ref{D:Cutting}, that when $|\ColOneTwo(\bfC_1)/T| = 2$, $\Sigma_1$ denotes the component of $$\ColOneTwoX/T -  \ColOneTwo(\bfC_1)/T$$ that does not contain $\ColOneTwo(\bfC_2)/T$. Its preimage on $\ColOneTwoX$ is denoted $\wt{\Sigma}_1$. The following three lemmas will make use of the following assumption.

\begin{ass}\label{A:NoTInvolution}
Suppose that $\ColOneTwo(\bfC_2)/T$ is not $\tau$-invariant and that one of the following holds: 
\begin{enumerate}
    \item $\bfC_1$ is a simple cylinder;
    \item $\bfC_1$ is a pair simple cylinders, $|\ColOneTwo(\bfC_1)/T| = 2$, and $\wt{\Sigma}_1$ is not a pair of non-adjacent simple cylinders. (Notice that the cylinders in $\bfC_1$ are non-adjacent since if the cylinders in $\bfC_1$ were adjacent then $\ColOneTwo(\bfC_1)$ would have only one saddle connection.)
\end{enumerate}
\end{ass}
\begin{rem}\label{R:NoTInvolution}
A crucial point is that whether or not $\ColOneTwo(\bfC_2)/T$ is $\tau$-invariant is unchanged by applying the Collapsing Lemma (Corollary \ref{C:Collapsing2}, see \eqref{I:S2}). However, a priori, whether $\cM$ has extra symmetry might change after applying the Collapsing Lemma. 

Similarly, we note that the assumptions that $\bfC_1$ is a simple cylinder, a pair of simple cylinders, or that $|\ColOneTwo(\bfC_1)/T| = 2$ all continue to hold after applying the Collapsing Lemma by Corollary \ref{C:Collapsing2} \eqref{C:Collapsing2:RotatedScaledCi} and \eqref{I:S1}.

If the Collapsing Lemma is applied to a subsurface other than $\wt{\Sigma}_1$, then the Collapsing Lemma only alters $\wt{\Sigma}_1$ by rotating and scaling it (see Corollary \ref{C:Collapsing2} \eqref{C:Collapsing2:RotatedScaledCi}), and so Assumption \ref{A:NoTInvolution} continues to hold.
\end{rem}

\begin{lem}\label{L:NoTInvolution}
If Assumption \ref{A:NoTInvolution} holds, then a generic perturbation of $\For(\ColTwoX)$ in $\MTwo$  does not admit a translation involution. 
\end{lem}
\begin{proof}
Suppose to a contradiction that (a generic perturbation of)  $\For(\ColTwoX)$ admits a translation involution $T_2$. 

\begin{sublem}\label{SL:LicenseToApplyEarlierLemma}
There are no marked points on the boundary of $\ColTwo(\bfC_1)$.
\end{sublem}
\begin{proof}
Since the cylinders in $\ColTwo(\bfC_1)$ are non-adjacent, a boundary of a cylinder in $\ColTwo(\bfC_1)$ contains a marked point if and only if it is glued to the full boundary of another cylinder $\bfC$ on $\ColTwoX$. \red(If there is more than one boundary component of $\ColTwo(\bfC_1)$ containing marked points, we pick one, and use that choice to pick the single cylinder $\bfC$.) \black By assumption one of the boundaries of $\bfC$ is a single saddle connection. Define $\bfC' := \Col_{\ColTwo(\bfC_1)}(\bfC)$. Since $\bfC_1$ contains only simple cylinders, $\bfC'$ also has a boundary that is a single saddle connection. Although $\bfC'$ is not $T$-invariant we will nonetheless use the notation $\bfC'/T$ to denote its image in the quotient.

Since every cylinder on \red $\ColOneTwo(X, \omega)/T$ \black is fixed by the hyperelliptic involution, and since $\ColOneTwo(\bfC_1)/T$ is invariant by the hyperelliptic involution, we see that $\bfC'/T$ is a simple cylinder whose boundary is $\ColOneTwo(\bfC_1)/T$. This is impossible when $|\ColOneTwo(\bfC_1)/T| = 1$, so suppose that $|\ColOneTwo(\bfC_1)/T| = 2$.

Notice that any saddle connection in $\bfC'/T$ is fixed by the hyperelliptic involution. Since $\ColOneTwo(\bfC_2)/T$ is not fixed by the hyperelliptic involution (by Assumption \ref{A:NoTInvolution}), $\ColOneTwo(\bfC_2)/T$ is not contained in $\bfC'/T$. Therefore, $\Sigma_1 = \bfC'/T$. Since $\Sigma_1$ is a simple cylinder, $\wt{\Sigma}_1$ is either a pair of non-adjacent simple cylinders or a complex cylinder. Since $\bfC'$ is a cylinder contained in $\wt{\Sigma}_1$ that has a boundary consisting of one saddle connection, it follows that $\wt{\Sigma}_1$ consists of two non-adjacent simple cylinders, contradicting Assumption \ref{A:NoTInvolution}. 
%
\end{proof}

If $\bfC_1$ is a simple cylinder, then so is $\ColTwo(\bfC_1)$, which must be fixed by $T_2$ by Lemma \ref{L:MarkedPointsAndInvariance} (which may be applied by Sublemma \ref{SL:LicenseToApplyEarlierLemma}). However, no simple cylinder can be fixed by a translation involution, so we have a contradiction. 

Similarly, if $\bfC_1$ is a pair of non-adjacent simple cylinders, so is $\ColTwo(\bfC_1)$. The $T_2$-involution must exchange those two cylinders by Lemma \ref{L:MarkedPointsAndInvariance} (which may be applied by Sublemma \ref{SL:LicenseToApplyEarlierLemma}).

Let $T':=\Col_{\ColTwo(\bfC_1)}(T_2)$ (which is defined, by Lemma \ref{L:DefinitionCol(f)}, since $\ColTwo(\bfC_1)$ is fixed by $T_2$). Since $\ColOneTwo(\bfC_1)$ and $\ColOneTwo(\bfC_1)/T$ both consist of two saddle connections, it follows that $\ColOneTwo(\bfC_1)$ is not $T$-invariant. It is equally clear, since $T_2$ fixes $\ColTwo(\bfC_1)$ that $\ColOneTwo(\bfC_1)$ is $T'$-invariant. We will deduce a contradiction by showing that $T = T'$.

Since $\MOneTwo$ is an Abelian and quadratic double (by Lemma \ref{L:ColOneTwoXConnected}) that is not a double of $\cH(0)$ (since $\ColOneTwo(\bfC_2)/T$ is not $\tau$-invariant) it follows that $T = T'$ by Corollary \ref{C:HypSymmetries}. This is a contradiction. 
\end{proof}


\noindent Recall from Definition \ref{D:Cutting}, that, when $\ColOneTwo(\bfC_2)/T$ is not $\tau$-invariant,  $\Sigma_2$ is defined as the component of
\[ \ColOneTwo(X, \omega)/T - \left( \ColOneTwo(\bfC_2)/T \cup \tau\left( \ColOneTwo(\bfC_2)/T\right) \right) \]  not containing $\ColOneTwo(\bfC_1)/T$. Recall too that $\wt{\Sigma}_2$ denotes the preimage of $\Sigma_2$ on $\ColOneTwoX$. We will use $\wt{\Sigma}_2^{top}$ to refer to the corresponding subsurface on $(X, \omega)$ as in Definition \ref{D:Cutting}. 

\begin{lem}\label{L:FinalMarkedPointSL}
If Assumption \ref{A:NoTInvolution} holds, then $\cM$ has rank at least two and there is a singularity that is not a marked point in the intersection of the boundaries of $\bfC_2$ and $\wt{\Sigma}_2^{top}$. 
\end{lem}
\begin{proof}
By Assumption \ref{A:NoTInvolution}, $\ColOneTwo(\bfC_2)$ is not $\tau$-invariant, so $\cM$ has rank at least two by Lemma \ref{L:QuickSummary}. Suppose now, in order to derive a contradiction, that the intersection of the boundaries of $\bfC_2$ and $\wt{\Sigma}_2^{top}$ contains no singularities (only marked points).

Let $P$ be the set of marked points on the common boundary of $\bfC_2$ and $\wt{\Sigma}_2^{top}$. Because $\MOne$ is an Abelian double and $\ColOne(\bfC_2)$ is a subequivalence class of generic cylinders, $\ColOne(P)$ is a slope $+1$ irreducible pair of marked points on $\ColOneX$. The same holds on $(X, \omega)$.  

Since $\cM$ has rank at least two, Apisa-Wright \cite{ApisaWright} (Lemma \ref{L:ApisaWright}) implies that $\For(X, \omega)$ admits a translation cover of a distinct surface. Since $\ColTwoX$ is formed by merging $P$ with zeros and/or marked points, it follows that $\ColTwoX$ (and any generic perturbation thereof)  also admits a translation cover of a distinct surface and that $\MTwo$ has rank at least two since $\cM$ does. By Lemma \ref{L:InvolutionImpliesHyp}, this translation cover is the quotient by a translation involution. But no translation involution exists on $\ColTwoX$ by Lemma \ref{L:NoTInvolution}, a contradiction.
\end{proof}

\begin{lem}\label{L:DisconnectedImpliesSimple}
If Assumption \ref{A:NoTInvolution} holds and $\wt{\Sigma}_2$ is disconnected, then $\bfC_2$ is a pair of simple cylinders.
\end{lem}

Before we prove this lemma, we will show how it can be used to conclude this section. 

\begin{prop}\label{P:NoSimpleCylinder}
If $\bfC_1$ is a simple cylinder, then $\ColOneTwo(\bfC_2)/T$ is $\tau$-invariant. In particular, $\cM$ has extra symmetry. 
\end{prop}
\begin{proof}
Suppose to a contradiction $\bfC_1$ is a simple cylinder, and that $\ColOneTwo(\bfC_2)/T$ is not $\tau$-invariant. Note that once we establish that $\ColOneTwo(\bfC_2)/T$ is $\tau$-invariant, it will follow that $\cM$ has extra symmetry (by Lemma \ref{L:Cutting}).

Let $s$ denote the single saddle connection in $\ColOneTwo(\bfC_1)$, which is $J$-invariant. By Lemma \ref{L:Connectedness-New}, $\ColOneTwoX - \{s, Ts\}$ is disconnected. Since $\Sigma_2$ is disjoint from $\ColOneTwo(\bfC_1)/T$ by construction, $\wt{\Sigma}_2$ is disconnected. By Lemma \ref{L:DisconnectedImpliesSimple}, $\bfC_2$ is a pair of simple cylinders.



By the Collapsing Lemma (Corollary \ref{C:Collapsing2}), we may reduce to a generic diamond $\left( (X', \omega'), \cM', \bfC_1, \bfC_2 \right)$ where $\Sigma_2$ is a simple cylinder and its complement is a parallelogram (see Figure \ref{F:C1Simple}). As observed in Remark \ref{R:NoTInvolution}, Assumption \ref{A:NoTInvolution} continues to hold after applying the Collapsing Lemma. Since $\bfC_2$ is a pair of simple cylinders glued into $\wt{\Sigma}_2$, which is also a pair of simple cylinders, there are no singularities (only marked points) on their common boundary, contradicting Lemma \ref{L:FinalMarkedPointSL}. 
%
%
%
%
\begin{figure}[h!]\centering
\includegraphics[width=.4\linewidth]{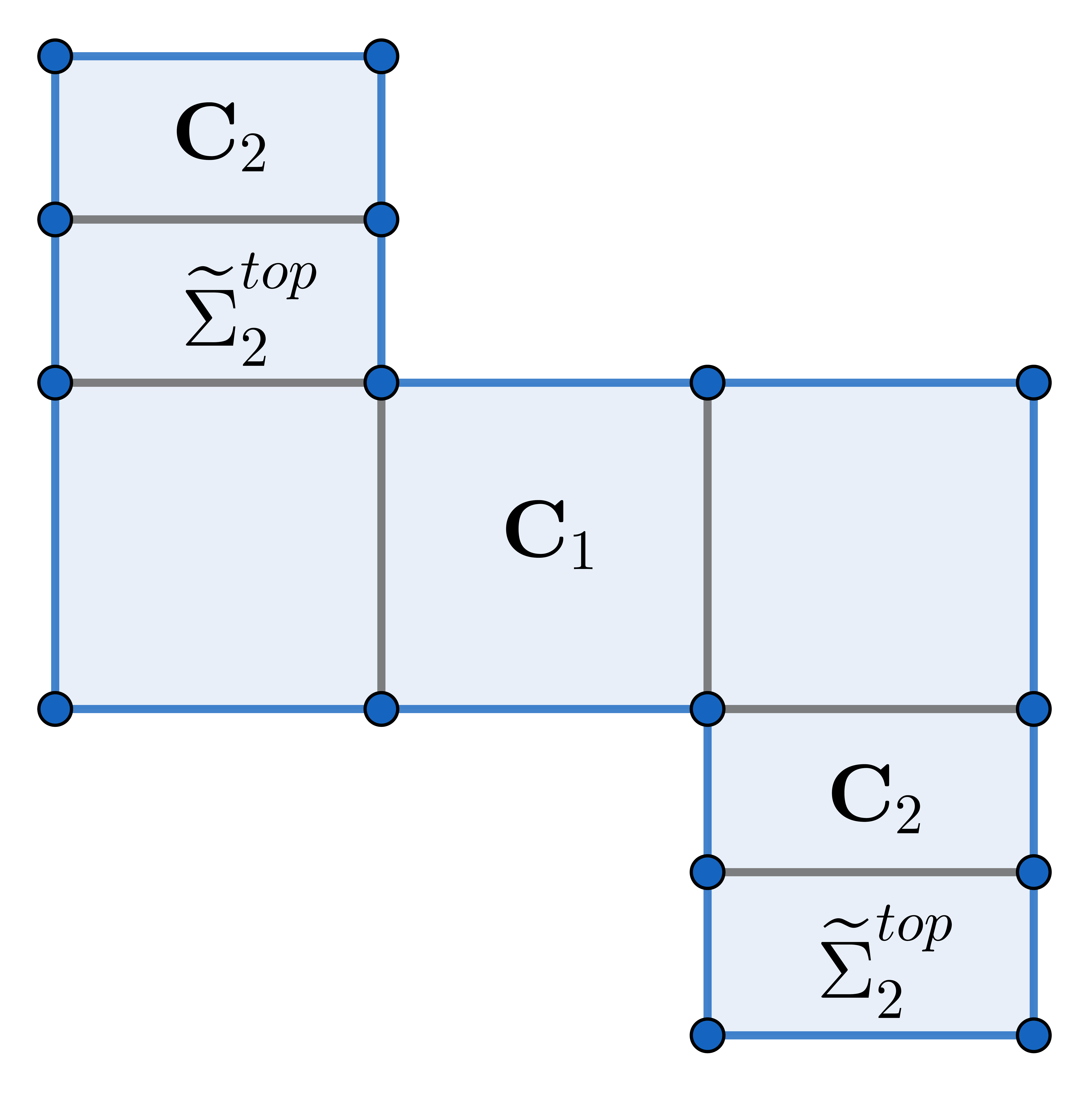}
\caption{The surface that results from applying the Collapsing Lemma. Recall the convention that any unlabelled side is glued to the opposite side of the polygon containing it.}
\label{F:C1Simple}
\end{figure}
%
%
\end{proof}

\begin{proof}[Proof of Lemma \ref{L:DisconnectedImpliesSimple}]
Suppose to a contradiction that $\wt{\Sigma}_2$ is disconnected and that $\bfC_2$ is not a pair of simple cylinders. By Lemma \ref{L:StructureOfS}, $\bfC_2$ is a single complex cylinder. 

By applying the Collapsing Lemma (Corollary \ref{C:Collapsing2}) to $\Sigma_2$, we may replace $(X, \omega)$ with another surface, which we will give the same name for simplicity, that fits into a generic diamond with $\bfC_1$ and $\bfC_2$ and so $\Sigma_2$ is a simple cylinder on $\ColOneTwoX/T$. By Remark \ref{R:NoTInvolution}, Assumption \ref{A:NoTInvolution} continues to hold. Moreover, after collapsing, it is clear that $\wt{\Sigma}_2$ remains disconnected and that $\bfC_2$ remains a complex cylinder (by Corollary \ref{C:Collapsing2} \eqref{C:Collapsing2:RotatedScaledCi}).



Since $\wt{\Sigma}_2$ is disconnected it follows that it is a pair of non-adjacent simple cylinders both of which are isometric to $\Sigma_2$. A specific example of a surface satisfying the hypotheses of this lemma is shown in Figure \ref{F:Impossibility}.

\begin{figure}[h!]\centering
\includegraphics[width=.7\linewidth]{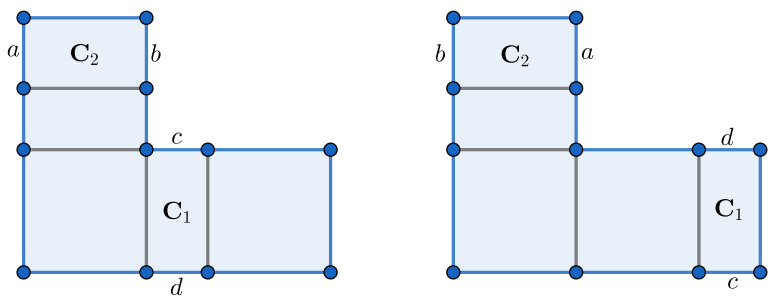}
\caption{An example of a surface that will be ruled out by the proof.}
\label{F:Impossibility}
\end{figure}

Let $\bfC_3$ denote the pair of non-adjacent simple cylinders on $(X, \omega)$ corresponding to $\wt{\Sigma}_2$. We use this notation, instead of the $\wt{\Sigma}_2^{top}$, to emphasize the fact that, in this case, $\wt{\Sigma}_2^{top}$ is a subequivalence class of cylinders.   Notice that $\left( (X, \omega), \cM, \bfC_2, \bfC_3 \right)$ is not a diamond since $\bfC_2$ and $\bfC_3$ share boundary saddle connections. Nevertheless, the invariant subvariety $\cM_{\bfC_3, \bfC_2}$ is well-defined.

\begin{sublem}\label{SL:SecondDiamond1}
The orbit closure $\cM_{\bfC_3, \bfC_2}$ is a quadratic double of a component $\cQ'$ of a stratum. If $\cQ'$ is hyperelliptic, then $\Col_{\bfC_3, \bfC_2}(\bfC_3)/\Col_{\ColTwo(\bfC_3)}(J_2)$ is not fixed by the hyperelliptic involution. 
\end{sublem}
\begin{proof}
Since $\bfC_3$ consists of a pair of simple cylinders and since $\ColTwoX$ is connected, $\Col_{\bfC_3, \bfC_2}(X, \omega)$ is also connected. Moreover, $\cM_{\bfC_3, \bfC_2}$ has dimension exactly one less than $\MTwo$ since $\ColTwo(\bfC_3)$ consists of two simple cylinders exchanged by the holonomy involution. By Lemma \ref{L:codim1doubles}, it follows that $\cM_{\bfC_3, \bfC_2}$ is a quadratic double of a component $\cQ'$ of a stratum. By Lemma \ref{L:QHolonomy}, $\Col_{\ColTwo(\bfC_3)}(J_2)$ is the holonomy involution on $\Col_{\bfC_3, \bfC_2}(X, \omega)$. 

Suppose to a contradiction that $\cQ'$ is hyperelliptic, and that $$\Col_{\bfC_3, \bfC_2}(\bfC_3)/\Col_{\ColTwo(\bfC_3)}(J_2)$$ is fixed by the hyperelliptic involution. Then $\ColTwoX/J_2$ belongs to a hyperelliptic component since it is formed by gluing in a simple cylinder, namely $\ColTwo(\bfC_3)/J_2$, to a saddle connection fixed by the hyperelliptic involution on $\Col_{\bfC_3, \bfC_2}(X, \omega)$. This implies that $\For(\ColTwoX)$ (and any generic perturbation thereof) admits a translation involution, which contradicts Lemma \ref{L:NoTInvolution}.
\end{proof}

\begin{sublem}\label{SL:SecondDiamond2}
$\cM_{\bfC_3}$ is properly contained in a quadratic double. 
\end{sublem}
\begin{proof}
Since $\ColOneTwo(\bfC_2)$ consists of boundary saddle connections of $\wt{\Sigma}_2$, it follows that $\ColTwo(\bfC_2)$ consists of boundary saddle connections of $\ColTwo(\bfC_3)$. Notice that $\ColOneTwo(\bfC_2)$ is $T$-invariant and so is $\ColOneTwo(\bfC_3)$, since it is the full preimage of $\Sigma_2$ under the quotient by $T$. The $T$-involution exchanges the two cylinders in $\ColOneTwo(\bfC_3)$. Therefore, for each cylinder in $\ColOneTwo(\bfC_3)$ there is a saddle connection in $\ColOneTwo(\bfC_2)$ that borders it. Notice too that since $\bfC_2$ is a complex cylinder, $\ColOneTwo(\bfC_2)$ consists of exactly two saddle connections. Therefore, $\Col_{\bfC_3, \bfC_2}(\bfC_2) = \Col_{\bfC_3, \bfC_2}(\bfC_3)$.

Since $\ColOneTwo(\bfC_3)/T = \Sigma_2$, which is invariant by $\tau$, it follows that $\ColOneTwo(\bfC_3)$ is invariant by $J$. Therefore, $\ColTwo(\bfC_3)$ is fixed by $J_2$. Since $\ColTwo(\bfC_3)$ is fixed by $J_2$, it follows that $\Col_{\bfC_3, \bfC_2}(\bfC_3)$ is fixed by $\Col_{\bfC_3}(J_2)$. Therefore, $\Col_{\bfC_3}(X, \omega)$ is formed by gluing the complex cylinder  $\Col_{\bfC_3}(\bfC_2)$  into the two saddle connections in $\Col_{\bfC_3, \bfC_2}(\bfC_2)$, which are exchanged by the holonomy involution $\Col_{\ColTwo(\bfC_3)}(J_2)$. It is clear that there is a marked-point preserving involution $J_3$ on $\Col_{\bfC_3}(X, \omega)$ such that $\Col_{\bfC_2}(J_3)$ is the holonomy involution on $\Col_{\bfC_3, \bfC_2}(X, \omega)$ (see Lemma \ref{L:StructureOfS} where an identical argument is given). This implies that $\cM_{\bfC_3}$ is contained in a quadratic double. 

Suppose to a contradiction that the $\cM_{\bfC_3}$ is not a proper invariant subvariety in the quadratic double containing it. Since $\Col_{\bfC_3}(\bfC_2)$ is a complex cylinder whose saddle connections are generically parallel to it, Masur-Zorich (Theorem \ref{T:MZ} \eqref{T:MZ:TrivHolComponent}) implies that there are two connected components of \[ \Col_{\bfC_3}(X, \omega) - \Col_{\bfC_3}(\bfC_2). \]
This implies that there are two connected components of 
\[ \Col_{\bfC_2, \bfC_3}(X, \omega) - \Col_{\bfC_2, \bfC_3}(\bfC_2). \]
Since $\ColTwo(\bfC_3)$ consists of two simple cylinders and since gluing in simple cylinders does not cause the number of connected components that a translation surface with boundary has to decrease, there are at least two connected components of 
\[ \Col_{\bfC_2}(X, \omega) - \Col_{\bfC_2}(\bfC_2). \]
This implies that there are at least two connected components of 
\[ \ColOneTwoX - \ColOneTwo(\bfC_2). \]
However, this contradicts Lemma \ref{L:Connectedness-New} (note that $\ColOneTwo(\bfC_2)/T$ is not $\tau$-invariant by the fact that Assumption \ref{A:NoTInvolution} holds). 
%
\end{proof}

Let $J_3$ denote the holonomy involution on $\Col_{\bfC_3}(X, \omega)$. Notice that $\Col_{\bfC_3}(X, \omega)/J_3$ is formed by gluing in the complex envelope  $\Col_{\bfC_3}(\bfC_2)/J_3$  to the saddle connection  $\Col_{\bfC_3, \bfC_2}(\bfC_3)/\Col_{\ColTwo(\bfC_3)}(J_2)$. By Sublemma \ref{SL:SecondDiamond2} and Theorem \ref{T:complex-gluing0}, it follows that $\cM_{\bfC_3, \bfC_2}$ is a quadratic double of a hyperelliptic component $\cQ'$ and that $\Col_{\bfC_3, \bfC_2}(\bfC_3)/\Col_{\ColTwo(\bfC_3)}(J_2)$ is fixed by the hyperelliptic involution, contradicting Sublemma \ref{SL:SecondDiamond1}.
\end{proof}

\subsection{When $\bfC_1$ is a pair of simple cylinders}

\begin{prop}\label{P:NoPairSimple}
If $\bfC_1$ is a pair of simple cylinders, then $\cM$ is as in Theorem \ref{TP2}. 
\end{prop}
\begin{proof}
Suppose that $\bfC_1$ is a pair of simple cylinders and suppose to a contradiction that $\cM$ has no extra symmetry. By Lemma \ref{L:Cutting}, $\ColOneTwo(\bfC_2)/T$ is not $\tau$-invariant. If $\bfC_1$ is a pair of adjacent simple cylinders, then by Lemma \ref{L:ForgettingMarkedPoints}, forgetting the marked point on their common boundary is a diamond in which $\bfC_1$ is a simple cylinder, which is impossible by Proposition \ref{P:NoSimpleCylinder}. 
 
Therefore, $\bfC_1$ consists of two non-adjacent simple cylinders. By Lemma \ref{L:P(S1)}, $|\ColOneTwo(\bfC_1)/T| = 2$. If $\wt{\Sigma}_1$ is a pair of simple cylinders, then we are done by Lemma \ref{L:FinalCountdown}. Therefore, suppose that $\wt{\Sigma}_1$ is not a pair of non-adjacent simple cylinders. Notice that Assumption \ref{A:NoTInvolution} now holds.


\noindent \textbf{Case 1: $\wt{\Sigma}_2$ is disconnected.}

Since $\wt{\Sigma}_2$ is disconnected, $\bfC_2$ is a pair of simple cylinders by Lemma \ref{L:DisconnectedImpliesSimple}. By applying the Collapsing Lemma (Corollary \ref{C:Collapsing2}) to $\Sigma_2$  we may reduce to the case where $\wt{\Sigma}_2$ is a pair of non-adjacent  simple cylinders. Assumption \ref{A:NoTInvolution} continues to hold by Remark \ref{R:NoTInvolution}.  Since $\bfC_2$ is a pair of simple cylinders, its boundary with $\wt{\Sigma}_2^{top}$ consists solely of marked points, contradicting Lemma \ref{L:FinalMarkedPointSL}.

\noindent \textbf{Case 2: $\wt{\Sigma}_2$ is connected.}

By applying the Collapsing Lemma (Corollary \ref{C:Collapsing2}) to $\Sigma_2$ and $\Sigma_3$, we may reduce to the case where $\Sigma_2$ is a simple cylinder and $\Sigma_3$ is a parallelogram.  We continue to call this surface $(X, \omega)$. By Remark \ref{R:NoTInvolution}, Assumption \ref{A:NoTInvolution} continues to hold. \red Case 1 gives a contradiction if $\wt{\Sigma}_2$ is disconnected, so we can assume that $\wt{\Sigma}_2$ is a complex cylinder. \black

Since both $\wt{\Sigma}_2$ and $\ColOneTwo(\bfC_2)$ are $T$-invariant, $\ColOneTwo(\bfC_2)$ comprises either the entire top or the entire bottom boundary of $\wt{\Sigma}_2$. Suppose that $\ColOneTwo(\bfC_2)$ is the bottom boundary of $\wt{\Sigma}_2$. The argument is identical when $\ColOneTwo(\bfC_2)$ is the top boundary.

By Lemma \ref{L:FinalMarkedPointSL}, $\bfC_2$ cannot be a single complex cylinder (since then the common boundary of $\bfC_2$ and $\wt{\Sigma}_2^{top}$ would not contain a singularity, only marked points) and so $\bfC_2$ must be a pair of simple cylinders by Lemma \ref{L:StructureOfS}. 

\begin{sublem}\label{SL:PictureTheBase}
There is a $J$-invariant saddle connection $s$ contained in $\wt{\Sigma}_2$ such that $\ColOneTwoX - \{s, Ts\}$ consists of two isometric copies of $\ColOneTwoX/T - s/T$ exchanged by $T$. 
\end{sublem}
\begin{proof}
Since $\wt{\Sigma}_2$ is connected, it is a complex cylinder that is fixed by $J$ (since $\Sigma_2$ is fixed by $\tau$ and since $J$ descends to $\tau$ on $\ColOneTwoX/T$). Therefore, $\wt{\Sigma}_2$ contains two fixed points of $J$. Let $s$ be a saddle connection in $\wt{\Sigma}_2$ that contains one of these fixed points. We are now done by Lemma \ref{L:Connectedness-New} (note that $\ColOneTwo(\bfC_2)/T$ is not $\tau$-invariant by the fact that Assumption \ref{A:NoTInvolution} holds).
%
%
\end{proof}

By Sublemma \ref{SL:PictureTheBase}, $\wt{\Sigma}_1$ is disconnected. Since $\wt{\Sigma}_1$ is not a non-adjacent pair of simple cylinders, it follows that $\Sigma_1$ is not a simple cylinder. By applying the Collapsing Lemma (Corollary \ref{C:Collapsing2}) to $\Sigma_1$ we may reduce to the case where $\Sigma_1$ is a slit torus (and hence $\wt{\Sigma}_1$ is a pair of disjoint slit tori). We continue to call this surface $(X, \omega)$. 




\begin{sublem}\label{SL:M-Dimension}
$\cM$ has dimension seven and $(X, \omega)$ is depicted in Figure \ref{F:BigAttack} (top).
\end{sublem}

\begin{figure}[h!]\centering
\includegraphics[width=.6\linewidth]{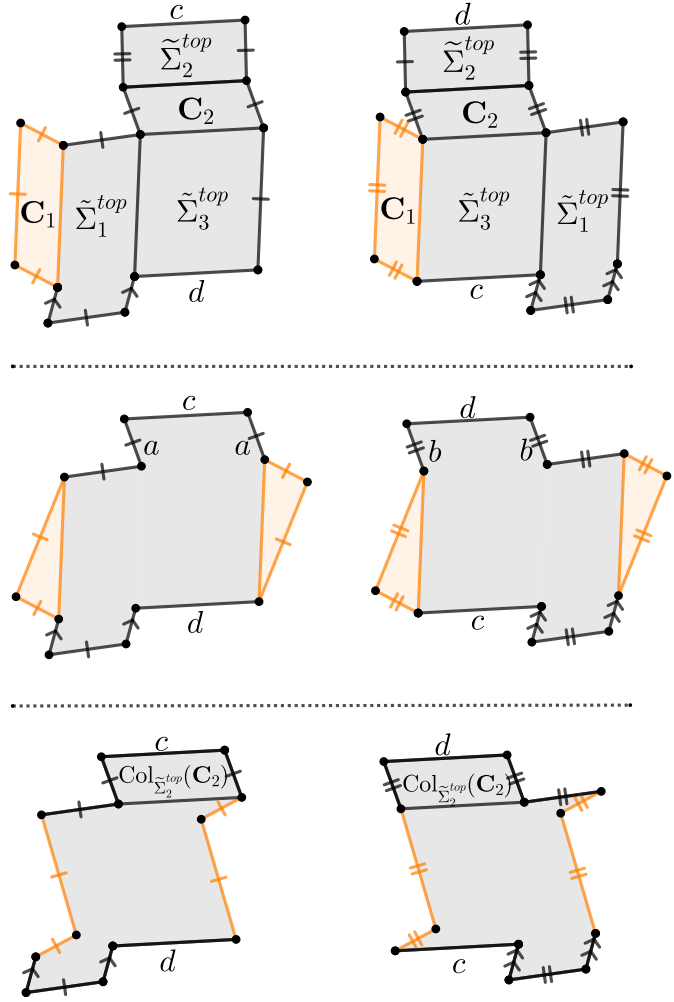}
\caption{The top figure is the surface in Case 2 of Proposition \ref{P:NoPairSimple}. The middle figure is $\Col_{\wt{\Sigma}_2^{top}}(X, \omega)$. The bottom figure illustrates the ``attack".  }
\label{F:BigAttack}
\end{figure}

\begin{proof}
The surface $\ColOneTwoX/T$ is comprised of three subsurfaces - $\Sigma_1, \Sigma_2, \Sigma_3$ - consisting of a slit torus, simple cylinder, and parallelogram respectively - that are glued together. This shows that $\ColOneTwoX/T$ belongs to $\cH(1,1)$ (see Lemma \ref{L:AllAboutThatBase}) and hence that $\MOneTwo$ has dimension five. By definition of generic diamond (Definition \ref{D:GenericDiamond} \eqref{E:one}), $\cM$ has dimension seven.

At this point, appealing to Sublemma \ref{SL:PictureTheBase}, we have shown that $\ColOneTwoX$ is the surface in Figure \ref{F:BigAttack} (top) once the cylinders labelled $\bfC_1$ and $\bfC_2$ have been collapsed. Since $\bfC_2$ is a pair of simple cylinders glued into the bottom of $\wt{\Sigma}_2$ and since $\bfC_1$ is a pair of simple cylinder such $\ColOneTwo(\bfC_1)/T$ forms the boundary of $\wt{\Sigma}_1$ the surface $(X, \omega)$ is as depicted in Figure \ref{F:BigAttack} (top). 
\end{proof}

Collapsing $\wt{\Sigma}_2^{top}$ on $(X, \omega)$, as in the middle subfigure of Figure \ref{F:BigAttack}, we see that $\cM_{\wt{\Sigma}_2^{top}}$ belongs to a quadratic double of $\cQ(8, -1^4)$, which has rank three and rel one (and hence dimension seven) by Lemma \ref{L:Q-rank}. By Sublemma \ref{SL:M-Dimension}, $\cM$ has dimension seven. Since $\wt{\Sigma}_2^{top}$ is a generic cylinder, $\cM_{\wt{\Sigma}_2^{top}}$ has dimension six and hence is codimension one in the quadratic double (in fact, using the saddle connections labelled $a$ and $b$ in the middle subfigure of Figure \ref{F:BigAttack}, the additional equation defining $\cM_{\wt{\Sigma}_2^{top}}$ is $a = b$).

Notice that $\cM_{\wt{\Sigma}_2^{top}}$ has no rel, as can be seen either by directly working with the surface or since $\cM_{\wt{\Sigma}_2^{top}}$ is codimension one in an invariant subvariety with rel one. By Corollary \ref{C:ConstantRatio}, this implies that equivalent cylinders, for instance the ones in $\Col_{\wt{\Sigma}_2^{top}}\left(\bfC_2\right)$, have a constant ratio of moduli. 


We're now going to describe a deformation of the surface in which we use the cylinders in $\Col_{\wt{\Sigma}_2^{top}}\left(\bfC_1\right)$ to ``attack" those in $\Col_{\wt{\Sigma}_2^{top}}\left(\bfC_2\right)$, resulting in a family of surfaces where the modulus of one of the two cylinders in $\Col_{\wt{\Sigma}_2^{top}}\left(\bfC_2\right)$ changes but not the other. This will contradict Corollary \ref{C:ConstantRatio}. We will describe this deformation by making reference to Figure \ref{F:BigAttack}. While keeping all other edges constant, take the corner of the orange triangle and move it as shown in Figure \ref{F:BigAttack}. Moving the corner vertically causes the corner to enter one  $\Col_{\wt{\Sigma}_2^{top}}\left(\bfC_2\right)$, changing the height of one cylinder of $\Col_{\wt{\Sigma}_2^{top}}\left(\bfC_2\right)$ but not the other.  This is the desired deformation, which produces the desired contradiction.
\end{proof}

\begin{proof}[Proof of Theorem \ref{TP2}:]
This is immediate from Propositions \ref{PP2},  \ref{P:NoSingleComplexCylinder}, \ref{P:NoSimpleCylinder}, and \ref{P:NoPairSimple}.
\end{proof}

\subsection{Supplemental statements in case \eqref{I:FinalPossibility} }\label{S:ExtraSymmetrySupplement}

In this section, we deduce the following supplemental claims from the statement of Theorem \ref{TP2} \eqref{I:FinalPossibility} and the results of Section \ref{S:ExtraSymmetryPrep}. 

\begin{thm}\label{T:3bSupplement}
The following hold in Case (\ref{I:FinalPossibility}) when $\cM$ is not a quadratic double. 
\begin{enumerate}[label=\ref{I:FinalPossibility}-\arabic*]
    \item\label{I:ExtraSymmetry:NoP} If the surfaces in $\cN$ do not contain a free marked point, then, defining 
    \[ a^{(3)} :=
  \begin{cases}
                                  \{a+1, a+1\} & \text{ $a$ odd} \\
                                  \{2a+4\} & \text{ $a$ even} 
  \end{cases} \]
  $\cM$ is contained in the quadratic double of 
  \begin{enumerate}[label=(\roman*)]
  \item\label{I:Case1} $\cQ_0=\cQ\left( a^{(3)}, 2, -1^2 \right)$ with no preimages of poles marked, or 
  \item\label{I:Case2} $\cQ_0=\cQ\left( a^{(3)}, -1^4 \right)$ with one or two preimages of poles marked (see Figure \ref{F:EmptyP}), or 
  \item\label{I:Case3} $\cQ_0=\cQ\left( a^{(3)}, 0, -1^4 \right)$ with no preimages of poles marked (see Figure \ref{F:MC1notQdouble}), 
  \end{enumerate}
where $a\geq 0$.
Moreover, $\For(\cM)$ is the preimage of a codimension one hyperelliptic locus in $\For(\cQ_0)$. 
  \begin{figure}[h!]\centering
\includegraphics[width=.6\linewidth]{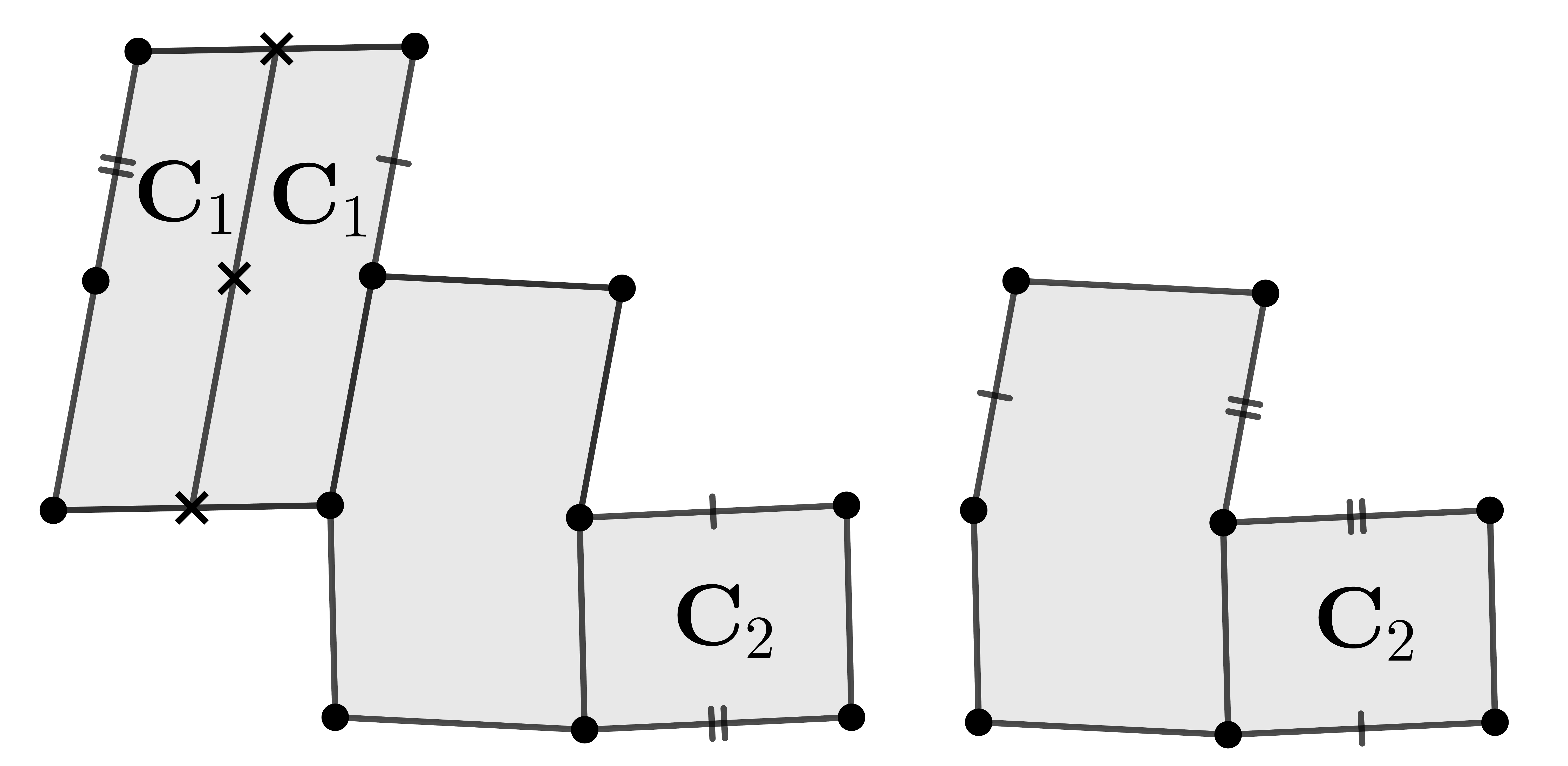}
\caption{ An illustration of Case (\ref{I:FinalPossibility}) when the surfaces in $\cN$ contain no free marked points. Either 1 or 2 of the points labelled $x$ can be marked. More specifically see \eqref{I:ExtraSymmetry:NoP} \ref{I:Case2}.
}
\label{F:EmptyP}
\end{figure}
  \begin{figure}[h!]\centering
\includegraphics[width=.6\linewidth]{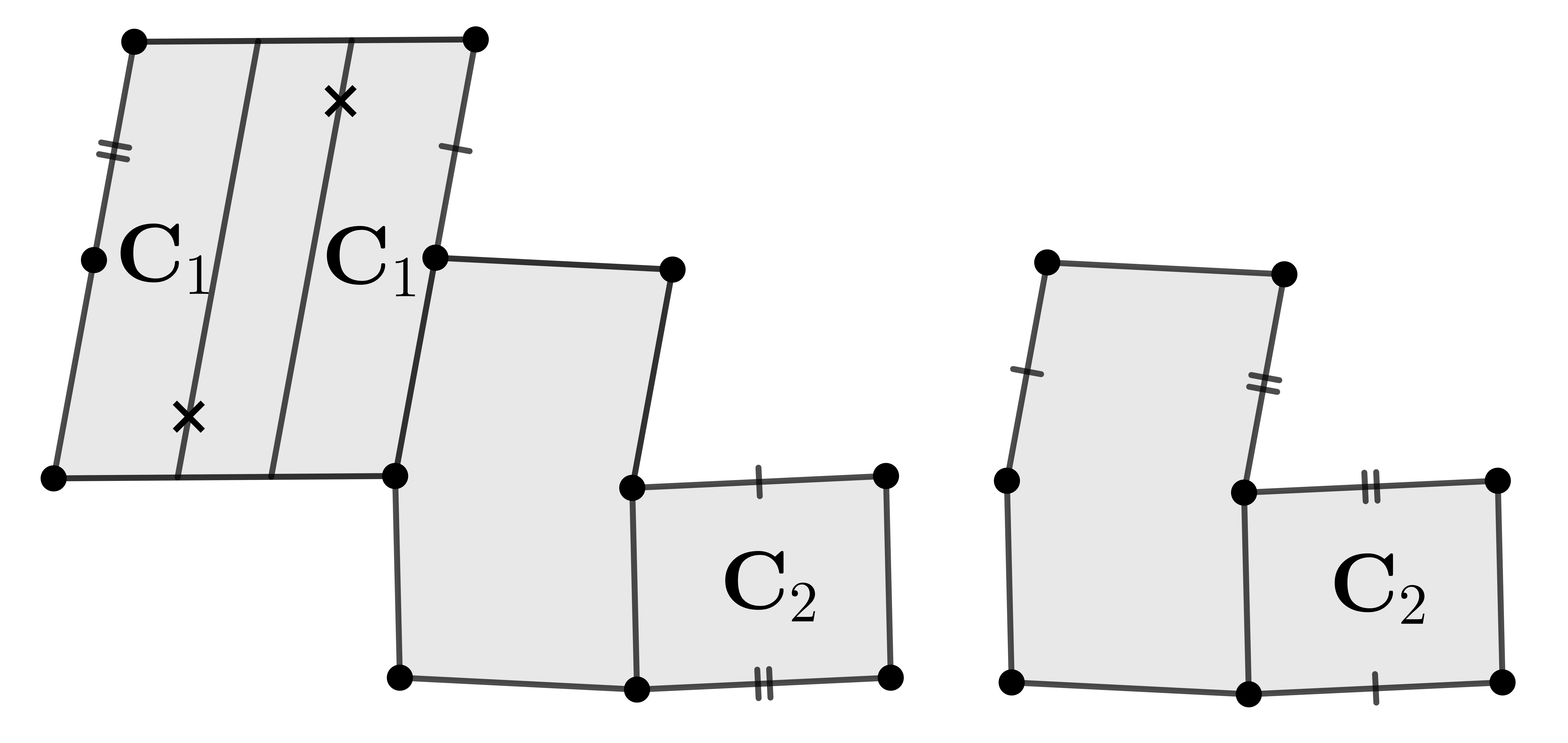}
\caption{ An illustration of Case (\ref{I:FinalPossibility}) when the surfaces in $\cN$ contain no free marked points. Both points labelled $x$ are marked. More specifically see \eqref{I:ExtraSymmetry:NoP} \ref{I:Case3}. 
} 
\label{F:MC1notQdouble}
\end{figure}
    \item\label{I:ExtraSymmetry:PMarkedPoints} If the surfaces in $\cN$ contain a free marked point that is not a branch point, then $\For(\cM) = \MOneTwo$ and $\MOneTwo$ is a quadratic double of $\cQ^{hyp}( a^{(1)}, -1^2)$ for some $a \geq 1$. Up to two cylinder deformations, $\ColOneTwoX$ is $(X, \omega)$ with four marked points forgotten - a slope $+1$ irreducible pair of points exchanged by $T$ and a slope $-1$ irreducible pair of points exchanged by $J$. 
  \item \label{I:ExtraSymmetry:ConnectedBase} If $\ColOneTwoX$ is connected, then surfaces in $\cM$ contain a slope $-1$ irreducible pair of marked points (see Figure \ref{F:FinalPossibility}).
   \begin{figure}[h!]\centering
\includegraphics[width=.6\linewidth]{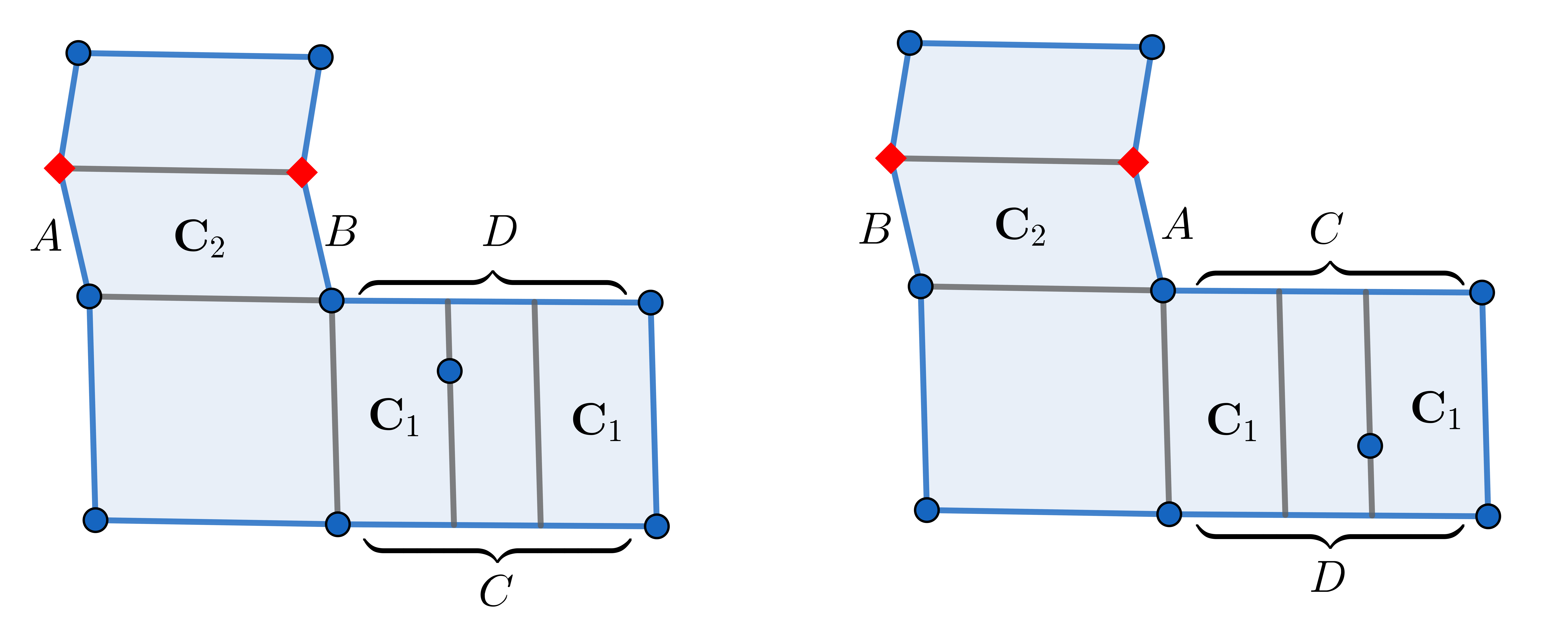}
\caption{An illustration of possibility (\ref{I:FinalPossibility}) of Theorem \ref{TP2}. Here there is a slope $-1$ irreducible pair of marked points. There is also an additional free point that is a branch point on surfaces in $\cN$. An additional example can be obtained by gluing the segments labelled $A$ and $B$ together in the other possible way (so each segment labelled $A$ gets glued to one labelled $B$); in this example the points there is still an additional free point on surfaces in $\cN$, but it is not in the branch locus. A third example can be obtained by collapsing the vertical cylinder between the cylinders labelled $\bfC_1$. On this new surface $\ColOneTwoX$ is disconnected and one or two periodic points are marked.}
\label{F:FinalPossibility}
\end{figure}
  \item \label{I:ExtraSymmetry:DisconnectedBase} If $\ColOneTwoX$ is disconnected, then if surfaces in $\cN$ contain a free point, it belongs to the branch locus. Letting $Q'$ denote the non-free marked points on surfaces in $\cN$, one of the following occurs:
  \begin{enumerate}
      \item $Q'$ consists of one point that is not in the branch locus (see Figures \ref{F:FinalPossibility} and \ref{F:EmptyP} for the cases where surfaces in $\cN$ contain (resp. do not contain) a free marked point).
      \item $Q'$ consists of two points that are in the branch locus (see Figure \ref{F:ExtraPossibility} for the cases where $\cN$ does and does not contain a free marked point).
  \end{enumerate}
\end{enumerate}
\end{thm}

We make a standing assumption arising from the statement of Theorem \ref{TP2} \eqref{I:FinalPossibility}.

\begin{ass}\label{A:TP2-3-b}
$\left( (X, \omega), \cM, \bfC_1, \bfC_2 \right)$ is a generic diamond, and  $\cM$ is a full locus of \red double \black covers of a codimension one locus $\cN$ in a component of a stratum of Abelian differentials $\cH_0$, where $\For(\cN) = \For(\cH_0)$ is a hyperelliptic connected component of rank at least two, and surfaces in $\cN$ have at most three marked points at most one of which is free. Additionally, $\cM$ is not a quadratic double. 
\end{ass}

Throughout this subsection we will let $T_0$ denote the involution on $(X, \omega)$ such that $(X, \omega)/T_0$ is a surface in $\cH_0$.  By Lemma \ref{L:InvolutionImpliesHyp-background}, since $\For\left( (X, \omega)/T_0 \right)$ is a generic surface in a stratum of Abelian differentials of rank at least two, $(X, \omega)$ is not a translation cover of degree greater than one of another surface. Therefore, Assumption \ref{A:T0Involution} holds, as do Lemmas \ref{L:C2TInvariant} and \ref{L:C1TInvariant}. We will let $Q'$ denote the marked points on $(X, \omega)/T_0$ that are not free. 

\begin{lem}\label{L:TheNatureOfQ'}
$Q'$ consists of either a single point fixed or two points exchanged by the hyperelliptic involution.
\end{lem}
\begin{proof}
First, if $Q'$ contains a single point, then since $\cN$ is codimension one in $\cH_0$ and $\cF(\cN)=\cF(\cH_0)$ this point must be a periodic point and hence a Weierstrass point by Theorem \ref{T:StrataMarkedPoints}. 

Second, if $Q'$ consists of two points, then for the same reasons they must be an irreducible pair. The other possibility would have been that the points in $Q'$ were both periodic points, but then $\cN$ would have codimension two. By Lemma \ref{L:ApisaWright}, an irreducible pair of marked points must belong to the same fiber of a map to a lower genus surface. By Lemma \ref{L:InvolutionImpliesHyp-background}, since $\For(\cH_0)$ is a hyperelliptic component of rank at least two, the only such map is the quotient by the hyperelliptic involution. Therefore, the points in $Q'$ consist of two points exchanged by the hyperelliptic involution. 
\end{proof}

We will let $P'$ denote the marked points on $(X, \omega)/T_0$ in the complement of $Q'$. The set $P'$ is either empty or contains a single free marked point by Assumption \ref{A:TP2-3-b}. Finally, we will let $P$ (resp. $Q$) denote the marked points and singularities of $(X, \omega)$ that are preimages of $P'$ (resp. $Q'$) under the map from $(X, \omega)$ to $(X, \omega)/T_0$. 

\begin{lem}\label{L:AdjacencyOfC2AndQ}
The points in $Q'$ do not belong to the boundary of cylinders in $\bfC_2/T_0$. When $\bfC_1$ is $T_0$-invariant, the points in $Q'$ are contained in the boundary of $\bfC_1/T_0$ and $\bfC_1/T_0$ consists of two simple cylinders.
\end{lem}
\begin{proof}
Suppose first that $\bfC_1$ is $T_0$-invariant. This implies that $\ColOne(T_0)$ is defined and that $\ColOne(T_0) = T_1$ (by Corollary \ref{C:Commuting}).

The codomain of $T_1$ is $\Col_{\bfC_1/T_0}\left( (X, \omega)/T_0 \right)$ (by Lemma \ref{L:DefinitionCol(f)}). By definition of Abelian double all marked points on this surface are free, which implies that $\Col_{\bfC_1/T_0}(Q')$ no longer contains marked points. In particular, this implies that the points in $Q'$ belonged to the boundary of $\bfC_1/T_0$ and that $\bfC_1/T_0$ consists of two simple cylinders (the other possibility - that $\bfC_1/T_0$ consisted of a single simple cylinder, with both boundaries consisting of a loop from a marked point to itself - would lead to periodic marked points on $\ColOneX/T_1$). In this case, it is also easy to see that $\bfC_2/T_0$ cannot contain a point in $Q'$ in its boundary since then $\bfC_1/T_0$ and $\bfC_2/T_0$ would not be disjoint; and so, $\bfC_1$ and $\bfC_2$ would not be disjoint (since both $\bfC_1$ and $\bfC_2$ are $T_0$-invariant). This contradicts the definition of a diamond. We have established the second claim in general and the first claim in the case that $\bfC_1$ is $T_0$-invariant.

By Lemma \ref{L:C1TInvariant}, either $\bfC_1$ is $T_0$-invariant or $(X, \omega)$ contains an irreducible pair of marked points in $Q$ that belong to the boundary of $\bfC_1$ and map to distinct points of $Q'$. We may suppose that the latter is the case. Suppose to a contradiction that the boundary of $\bfC_2/T_0$ does contain a point in $Q'$. 

Because each point of $Q'$ has a marked point in its preimage (in $Q$), neither point of $Q'$ is a branch point. So all preimages of points in $Q'$ are non-singular points.


We have assumed that $\bfC_2$ has a point of $Q$ in its boundary. Since $\bfC_2$ is $T_0$-invariant (by Lemma \ref{L:C2TInvariant}), the boundary of $\bfC_2$ contains both preimages of a point in $Q'$. Since the boundary of $\bfC_1$ contains one of these preimages, we have that $\overline{\bfC}_2$ and $\overline{\bfC}_1$ contain a non-singular point in their intersection. This contradicts the fact that $\bfC_1$ and $\bfC_2$ are disjoint. 
%
\end{proof}

Define $T_2 := \ColTwo(T_0)$, which is defined since $\bfC_2$ is $T_0$-invariant (by Lemma \ref{L:C2TInvariant}). By Lemma \ref{L:InvolutionImpliesHyp-background}, the existence of the map $\ColTwoX/J_2 \ra \ColTwoX/\langle J_2, T_2 \rangle$ implies that $\For(\cQ_2)$ is a hyperelliptic component. In particular, there are integers $a \geq b \geq -1$ such that $\For(\cQ_2) = \cQ^{hyp}\left( a^{(1)}, b^{(1)} \right)$ (see Subsection \ref{SS:Hyp} where this notation is first defined). 

\begin{lem}\label{L:Commuting2}
If $\For(\cQ_2) \ne \cQ(-1^4)$, then $\{\mathrm{id}, J_2, T_2, J_2 T_2\}$ are the affine symmetries identified in Lemma \ref{L:HypSymmetries}. 
\end{lem}
\begin{proof}
Suppose that $\For(\cQ_2) \ne \cQ(-1^4)$. By Lemma \ref{L:HypSymmetries}, the generic surface in $\For(\cQ_2)$ has four affine symmetries of derivative $\pm \mathrm{Id}$. These are: the identity, the holonomy involution, a unique translation involution, and the composition of the previous two involutions. Since $J_2$ is the holonomy involution on $\ColTwoX$ and $T_2$ is a translation involution, the claim follows. 
%
%
\end{proof}

\begin{lem}\label{L:ConstantsComputation}
$b \in \{-1, 0\}$. 
\end{lem}
While technically $a^{(1)}$ is a set of integers we will also use it to denote the corresponding points on quadratic differentials with zeros of the corresponding order. Similarly for $a^{(2)}$. 
\begin{proof}
If not, then, by Lemma \ref{L:Commuting2}, $\For(\ColTwoX)/T_2$ belongs to $\cH\left( a^{(2)}, b^{(2)} \right)$ (by Lemma \ref{L:HyperellipticCheatSheet}) \red with $a\geq b>0$. \black By Assumption \ref{A:TP2-3-b}, $\cF((X,\omega)/T_0)$ belongs to a hyperelliptic connected component of a stratum of Abelian differentials, so the same is true for $\cF(\ColTwoX/T_2)$. Hyperelliptic connected components of strata of Abelian differentials have either a single zero or a pair of zeros that differ by the hyperelliptic involution, which describes the points in $a^{(2)}$ (since $a>0$). So $b^{(2)}$ must consist of marked points and hence $b \in \{-1, 0\}$, a contradiction. 
\end{proof}

\begin{lem}\label{L:Con1}
Suppose $Q'$ consists of a single point. Then it is not a branch point, $Q$ entirely consists of one or two marked points, and $\ColOneTwoX$ is disconnected.
\end{lem}
\begin{proof}
By Lemma \ref{L:TheNatureOfQ'}, since $Q'$ is a single point, it is a periodic point and hence its preimages cannot contain a slope $-1$ irreducible pair of marked points.  Similarly, if the preimage of $P'$ contains a marked point it is either free or part of a slope $1$ irreducible pair of marked points. In particular, $(X, \omega)$ does not contain a slope $-1$ irreducible pair of marked points. 

Hence Lemma \ref{L:C1TInvariant} gives that $\bfC_1$ is $T_0$-invariant. By Lemma \ref{L:AdjacencyOfC2AndQ}, $Q'$ is contained in the boundary of a cylinder in $\bfC_1/T_0$ and $\bfC_1/T_0$ is a pair of simple cylinders. The cylinders in $\bfC_1/T_0$ share a boundary saddle connection since $Q'$ is a non-singular fixed point of the hyperelliptic involution (by Lemma \ref{L:TheNatureOfQ'}). 


Since every subequivalence class in a quadratic double consists of at most two cylinders, it follows that $\ColTwo(\bfC_1)$ also consists of two cylinders. These two cylinders are adjacent, so by Masur-Zorich (Theorem \ref{T:MZ}), $\ColTwo(\bfC_1)/J_2$ is an envelope and the points in $\ColTwo(Q)$ are preimages of poles on the boundary of the envelope. \red Since $\bfC_1$ and $\bfC_2$ are not adjacent, \black this shows that $Q'$ is not a branch point and that $Q$ entirely consists of one or two marked points.

Since $\bfC_1$ is the preimage of two simple cylinders under the quotient by $T_0$, \red there are four saddle connections at the interface of $\bfC_1$ and the rest of the surface. (We will only need that there are at least three.) Hence $\ColTwo(\bfC_1)/J_2$ has at least two saddle connections at its interface with the rest of the surface, and  we \black see that $\ColTwo(\bfC_1)/J_2$ is a complex envelope. Let $Z$ denote the singularities and marked points of $\ColTwoX$ excluding points in $\ColTwo(Q)$. Notice that $\ColOneTwo(Q)$ must be contained in $\Col_{\ColTwo(\bfC_1)}(Z)$ since $\ColOneTwoX/T$ does not contain marked points that are fixed points of $\tau$ that lie outside of $\Col_{\ColTwo(\bfC_1)}(Z)/T$ (by Lemma \ref{L:AllAboutThatBase}). By Remark \ref{R:TwoDegenerations}, this implies that $\ColOneTwoX$ is disconnected.
%
\end{proof}

\begin{lem}\label{L:Con3}
If $\ColOneTwoX$ is disconnected, then $P'$ is contained in the branch locus.
\end{lem}
\begin{proof}
Suppose to a contradiction that $P$ contains marked points, but that $\ColOneTwoX$ is disconnected. Since $P$ is a collection of two marked points exchanged by $T_0$ and that belong to the boundary of $\bfC_2$ (by Lemma \ref{L:C2TInvariant}), collapsing $\ColOne(\bfC_2)$ entails moving \red (at least one of) \black these marked points into other marked points or zeros. However, this shows that $\ColOneTwoX$ is connected since $\ColOneX$ is, a contradiction. 
\end{proof}

\begin{lem}\label{L:Con2}
Suppose that $Q'$ consists of two points, at least one of which is a branch point. Then $P' \cup Q'$ is contained in the branch locus and $\ColOneTwoX$ is disconnected.
\end{lem}
\begin{proof}
First we will show that both points in $Q'$ are contained in the branch locus. By Lemma \ref{L:AdjacencyOfC2AndQ}, $\bfC_2$ does not contain points of $Q$ in its boundary. Therefore, both points in $Q'$ are branch points if and only if both points in $\Col_{\bfC_2/T_0}(Q')$ are branch points, which is what we will show. 

If this were not the case, then $\ColTwo(Q)$ contains a marked point. By definition of quadratic double, $\ColTwo(Q)$ contains two points exchanged by the holonomy involution, which would project to two points on $\ColTwoX/T_2$ that were not branch points.  Therefore neither point in $\Col_{\bfC_2/T_0}(Q')$ would be a branch point, which is a contradiction. 

Notice that $(X, \omega)$ does not contain a slope $-1$ irreducible pair of marked points, since they would have to project to such a pair of points on $(X, \omega)/T_0$. But the only such pair of points is $Q'$, which are branch points.

Next we will show that $\bfC_1$ consists of two \red (non-adjacent) \black complex cylinders. By Lemma \ref{L:C1TInvariant}, we have that $\bfC_1$ is $T_0$-invariant and, by Lemma \ref{L:AdjacencyOfC2AndQ}, that $\bfC_1/T_0$ consists of two \red (non-adjacent) \black simple cylinders each of which contains a boundary made up of a saddle connection joining a point in $Q'$ to itself \red (the non-adjacency follows since $Q'$ consists of two points, see Lemma \ref{L:TheNatureOfQ'}). \black Since the points in $Q'$ are branch points and since $\ColTwo(\bfC_1)$, which is a subequivalence class of generic cylinders, consists of at most two cylinders by Masur-Zorich (Theorem \ref{T:MZ}), it follows that $\ColTwo(\bfC_1)$ consists of two \red (non-adjacent) \black complex cylinders. By Masur-Zorich (Theorem \ref{T:MZ}), $\ColOneTwoX$ is disconnected. 


It follows that $P'$ is a subset of the branch locus by Lemma \ref{L:Con3}.
\end{proof}

\begin{proof}[Proof of \ref{I:ExtraSymmetry:ConnectedBase} and \ref{I:ExtraSymmetry:DisconnectedBase}:]
Suppose first that $\ColOneTwoX$ is connected. By Lemma \ref{L:Con1}, $Q'$ is not a fixed point of the hyperelliptic involution. So by Lemma \ref{L:TheNatureOfQ'}, $Q'$ consists of two points exchanged by the hyperelliptic involution. By Lemma \ref{L:Con2}, these points are not branched points, so their preimage consists of a slope $-1$ irreducible pair of marked points. This proves \ref{I:ExtraSymmetry:ConnectedBase}.

Suppose now that $\ColOneTwoX$ is disconnected. The claim that if $P'$ is nonempty, then it is contained in the branch locus is Lemma \ref{L:Con3}. If $Q'$ consists of one point, then it is not in the branch locus (by Lemma \ref{L:Con1}). If $Q'$ consists of two points, one of which belongs to the branch locus, then both belong to the branch locus (by Lemma \ref{L:Con2}). If $Q'$ consists of two points, neither of which belongs to the branch locus, then $Q$ consists entirely of marked points, which (by Lemma \ref{L:AdjacencyOfC2AndQ}) do not belong to the boundary of $\bfC_2$. This shows that, up to a cylinder deformation, $\ColOneTwoX$ is formed from $\ColTwoX$ by forgetting the marked points in $\ColTwo(Q)$; in particular, $\ColOneTwoX$ is connected (since $\ColTwoX$ is), which is a contradiction. This proves  \ref{I:ExtraSymmetry:DisconnectedBase}. 
\end{proof}

\begin{proof}[Proof of \ref{I:ExtraSymmetry:PMarkedPoints}:]
Suppose now that $P'$ contains a single point that is not part of the branch locus. Since $P'$ is a free point, $P$ consists of either a free point or a pair of points that differ by $T_0$. By Lemma \ref{L:C2TInvariant} only the second possibility occurs, so $(X, \omega)$ contains a slope $+1$ irreducible pair of marked points $P$. 

Without loss of generality (up to replacing the word ``top" with ``bottom"), the points in $P$ are the only singularities or marked points on the top boundary of $\bfC_2$ (by Lemma \ref{L:C2TInvariant}). Therefore, passing from $\ColOneX$ to $\ColOneTwoX$ amounts to, up to a single cylinder deformation, simply forgetting the points in $\ColOne(P)$. In particular, $\ColOneTwoX$ is connected. This implies that $Q$ is a slope $-1$ irreducible pair of marked points on $(X, \omega)$ by Theorem \ref{T:3bSupplement} \eqref{I:ExtraSymmetry:ConnectedBase}. Since passing from $(X, \omega)$ to $\ColTwoX$ amounts to, up to a single cylinder deformation, simply forgetting the points in $P$, we see that there is a holonomy involution $J_0$ on $(X, \omega)$ that comes from the holonomy involution $J_2$ on $\ColTwoX$. Since $\ColTwo(Q)$ is $J_2$ invariant, it follows that $Q$ is a $J_0$-invariant pair of marked points.

Therefore, up to two cylinder deformations, $\ColOneTwoX$ is $(X, \omega)$ with $P \cup Q$ forgotten. This shows that $\MOneTwo$ has rank at least two and hence that $\MOneTwo$ is a quadratic double of $\cQ(a^{(1)}, -1^2)$ \red for some \black $a > 0$ and with no marked points (by Lemma \ref{L:QuadAbDouble}). So $\For(\cM) = \MOneTwo$. 
\end{proof}

\begin{proof}[Proof of \ref{I:ExtraSymmetry:NoP}:]
By assumption (the statement of \ref{I:ExtraSymmetry:NoP}), $P'$ is empty. Since $\bfC_2/T_0$ is a subequivalence class of generic cylinders that do not border the points in $Q'$ (by Lemma \ref{L:AdjacencyOfC2AndQ}), $\bfC_2/T_0$ consists of a simple cylinder that is fixed by the hyperelliptic involution. The fact that $\bfC_2/T_0$ is fixed by the hyperelliptic involution is the key point here and it is the emptiness of $P'$ that accounts for it. As a consequence, $\ColTwo(\bfC_2)/T_2$  is a single saddle connection.

\begin{sublem}\label{SL:ConstantsComputation2}
$\For(\cQ_2) \ne \cQ(-1^4)$ and the single saddle connection $\ColTwo(\bfC_2)/T_2$ has both of its endpoints in $a^{(2)}$. Moreover, $\ColTwo(\bfC_2)/J_2$ consists of one of the following:
\begin{enumerate}
    \item one or two saddle connection whose endpoints lie in $a^{(1)}$, or
    \item two saddle connections each of which has one endpoint in $a^{(1)}$ and the other a pole.
\end{enumerate} 
(When $a = b = 0$ there is no way to differentiate $a^{(i)}$ and $b^{(i)}$ so this might require exchanging these two sets). 
\end{sublem}
\begin{proof}
We will begin by explaining how the first claim implies the second. By the first claim, $\ColTwo(\bfC_2)$ consists of two saddle connections whose endpoints are preimages of $a^{(2)}$. Notice that the preimages of $a^{(2)}$ project to the points in $a^{(1)}$ under the quotient by $J_2$. The second listed possibility for $\ColTwo(\bfC_2)/J_2$ occurs when $J_2$ fixes each saddle connections in $\ColTwo(\bfC_2)$ and the first possibility occurs otherwise. 

We will prove the first claim in two cases.

\noindent \textbf{Case 1: $a > 0$.}

Since $a>0$, on $\ColTwoX/T_2$, which belongs to $\cH\left( a^{(2)}, b^{(2)} \right)$, there are singularities of the flat metric and they are precisely the points in $a^{(2)}$. Notice that $\bfC_2/T_0$ is a simple cylinder on a surface that, after forgetting the points in $Q$ (which do not belong to the boundary of $\bfC_2/T_0$ by Lemma \ref{L:AdjacencyOfC2AndQ}), belongs to a hyperelliptic connected component. For any such cylinder, if collapsing it does not give rise to a genus one surface, then the saddle connection resulting from collapsing the cylinder joins singularities of the metric. Applying this to $\bfC_2/T_0$ gives the result. 



\noindent \textbf{Case 2: $a \leq 0$.}

Since $a \leq 0$, $\MTwo$ has rank one. Since $\cM$ has rank at least two it follows by definition of generic diamond (Definition \ref{D:GenericDiamond} \eqref{E:one}) that $\For\left( (X, \omega)/T_0 \right)$ belongs to $\cH(2)$. By assumption (the statement of \ref{I:ExtraSymmetry:NoP}), the marked points on $(X, \omega)/T_0$ are precisely the point(s) in $Q'$. It follows that $\Col_{\bfC_2/T_0}\left( \For\left(X, \omega)/T_0\right) \right)$ belongs to $\cH(0,0)$ and that $\Col_{\bfC_2/T_0}(\bfC_2/T_0)$ is a saddle connection joining these two zeros. Therefore, $\ColTwoX/T_2$ belongs to $\cH(0^2, 0^{|Q'|})$ and $\MTwo$ has rel $|Q'|$. 

If $\For(\cQ_2) = \cQ(-1^4)$, then \red $\cQ_2 = \cQ(-1^4, 0^{|Q'|})$.  \black 
It follows that $\ColTwoX$ is a torus with \red one or two pair(s) \black of slope $-1$ irreducible marked points and with the remaining marked points being periodic points (in particular, fixed points of $J_2$). Since the endpoints of $\ColTwo(\bfC_2)/T_2$ consist of two distinct points that are a slope $-1$ irreducible pair, there is a saddle connection $s$ in $\ColTwo(\bfC_2)$ whose endpoints are a slope $-1$ irreducible pair of marked points. But $\ColTwo(\bfC_2)$ is $T_2$-invariant (which follows from Lemma \ref{L:C2TInvariant}) and $T_2$ cannot fix $s$ (no translation involution fixes a saddle connection) nor can it send it to another saddle connection (since no other saddle connection has generically the same length). This is a contradiction.


If $\For(\cQ_2) \ne \cQ(-1^4)$, then $\{\mathrm{id}, J_2, T_2, J_2 T_2\}$ are the affine symmetries described in Lemma \ref{L:HypSymmetries} (by Lemma \ref{L:Commuting2}). We have already seen that $\ColTwo(\bfC_2)/T_2$ is a single saddle connection joining two points exchanged by the hyperelliptic involution. Since $\ColTwoX/T_2$ belongs to $\cH\left( a^{(2)}, b^{(2)} \right)$ and the points in $a^{(2)}$ and $b^{(2)}$ consist of one point fixed or two points exchanged by the hyperelliptic involution, the claim follows. 
\end{proof}

Since $\bfC_2/T_0$ is fixed by the hyperelliptic involution, $\ColTwo(\bfC_2)/T_2$ is fixed by the hyperelliptic involution on $\ColTwoX/T_2$. Since $J_2$ and $T_2$ commute and $J_2$ induces the hyperelliptic involution on $\ColTwoX/T_2$ (by Lemma \ref{L:Commuting2}), this implies that $\ColTwo(\bfC_2)$ is fixed by the holonomy involution. Since $(X, \omega)$ is formed by gluing $\bfC_2$, which is either a pair of simple cylinder or a complex cylinder (by Lemma \ref{L:StructureOfS}), into $\ColTwo(\bfC_2)$ - a pair of saddle connections fixed by the holonomy involution - there is an involution $J_0$ on $(X, \omega)$ that preserves marked points, preserves $\bfC_1$ and $\bfC_2$,  and such that $\ColTwo(J_0) = J_2$. This shows that $\cM$ is contained in a quadratic double of a stratum $\cQ_0$.

Letting $\cM'$ be the orbit closure of $(X, \omega)/J_0$, $$\left( (X, \omega)/J_0, \cM', \bfC_1/J_0, \bfC_2/J_0 \right)$$ is a generic diamond. Since $\bfC_2$ is either a pair of simple cylinders or a single complex cylinder, $\bfC_2/J_0$ is one of the following: a simple cylinder, a pair of isometric simple envelopes, or a single complex envelope.

\begin{sublem}
$\For(\cM')$ is a codimension one hyperelliptic locus in $\cQ(a^{(3)}, b^{(1)}, -1^2)$ where $b \in \{-1, 0\}$ and $a \geq 0$. 
\end{sublem}

\noindent Note that if $b=0$ then 
$$\cQ(a^{(3)}, b^{(1)}, -1^2)=\cQ\left( a^{(3)}, 2, -1^2 \right)$$ 
and  if $b=-1$ then  
$$\cQ(a^{(3)}, b^{(1)}, -1^2)=\cQ\left( a^{(3)}, -1^4 \right).$$

\begin{proof}
By Assumption \ref{A:TP2-3-b}, $\cM$ is not a quadratic double, which lets us discard the possibility that $\bfC_2/J_0$ is a simple cylinder (by Corollary \ref{C:codim1}). Similarly, our analysis of gluing in a complex envelope (Theorem \ref{T:complex-gluing0}) implies that if $\bfC_2/J_0$ is a complex envelope, then either $\cM$ is a quadratic double, a case that we again discard, or that $\ColTwo(\bfC_2)/J_2$ is a single saddle connection $s$ fixed by the hyperelliptic involution. By Sublemma \ref{SL:ConstantsComputation2}, this saddle connection must join points corresponding to those in $a^{(1)}$ to other such points and so, by Remark \ref{R:GluingStratum}, $\For(\cM')$ is contained in $\For(\cQ')$ where $\cQ' = \cQ\left( a^{(3)}, b^{(1)}, -1^2 \right)$. Since $s$ is not a saddle connection joining two poles (since $a \ne -1$ by Sublemma \ref{SL:ConstantsComputation2}), it follows that the boundary of $\bfC_2/J_0$ does not contain marked points and so (by Theorem \ref{T:complex-gluing0}) $\For(\cM')$ is a codimension one hyperelliptic locus in $\For(\cQ')$. 


It remains to consider the case where $\bfC_2/J_0$ consists of two isometric simple envelopes. In this case, it is clear that $\cQ_2$ contains at least two poles. Since $\cQ_2= \cQ(a^{(1)}, b^{(1)})$ and $a\geq b$, this implies $b=-1$. By Sublemma \ref{SL:ConstantsComputation2}, $a \geq 0$ and $\ColTwo(\bfC_2)/J_2$ contains two saddle connections that join a point corresponding to $a^{(1)}$ to a pole. Recall that $\ColTwo(\bfC_2)$ is $T_2$-invariant (by Lemma \ref{L:C2TInvariant}). Therefore, $\ColTwo(\bfC_2)/J_2$ is invariant by the hyperelliptic involution since $T_2$ induces this involution on $\ColTwoX/J_2$ (by Lemma \ref{L:Commuting2}).  

Gluing in a simple envelope changes the stratum by adding a pole and increasing the cone angle around one point by $\pi$, so $\For(\cQ_0) = \cQ(a^{(3)}, -1^4)$. Note that gluing in the two isometric simple envelopes in $\ColTwo(\bfC_2)/J_2$ to a pair of saddle connections that are exchanged by the hyperelliptic involution on $\ColTwoX/J_2$  produces a surface in the hyperelliptic locus, and moreover the two envelopes are exchanged by the hyperelliptic involution. The condition guaranteeing the existence of the hyperelliptic involution is that the two envelopes are isometric, so in fact $\For(\cM')$ is a codimension one hyperelliptic locus in $\For(\cQ_0)$. (The fact that the envelopes have the same circumference is automatic, because the saddle connections they are glued into are generically parallel.) 
%
%
%
%
\end{proof}

It remains to establish the claim in (\ref{I:ExtraSymmetry:NoP}) about marked points on $(X, \omega)$, the set of which we know is invariant under $J_0$. By Assumption \ref{A:TP2-3-b}, since $P$ is empty, the set of marked points on $(X, \omega)$ is a subset of $Q$.

\begin{sublem}
When $b = 0$, $(X, \omega)$ has no marked points. 
\end{sublem}
\begin{proof}
Recall $\For(\cQ_2) = \cQ(a^{(1)}, b^{(1)})$. Since $b=0$, $b^{(1)}=\{2\}$ corresponds to single point on $\ColTwoX/J_2$. Let  $R$ denote the preimage of this point  on $\ColTwoX$, so $R$ consists of two singularities of the flat metric. 

By Lemma \ref{L:Commuting2}, $T_2$ is the translation involution studied in Figure \ref{F:HyperellipticCheatSheet} and so the points in $R$ are ramification points for the map from $\ColTwoX$ to $\ColTwoX/T_2$; moreover, the image of the points in $R$ on $\ColTwoX/T_2$ are two points exchanged by the hyperelliptic involution (see Figure \ref{F:HyperellipticCheatSheet}). 

On $\ColTwoX/T_2$, the singularities and marked points consist of the points in $\Col_{\bfC_2/T_0}(Q')$ (these remain marked points by Lemma \ref{L:AdjacencyOfC2AndQ}) and the point in $a^{(2)}$, i.e. the zeros and marked points coming from the fact that, after forgetting the points in $\Col_{\bfC_2/T_0}(Q')$, $\ColTwoX/T_2$ belongs to $ \cH^{hyp}\left( a^{(2)} \right)$. 

The points in $a^{(2)}$ consist of either one point fixed or two points exchanged by the hyperelliptic involution on $\Col_{\bfC_2/T_0}(X, \omega)$. Since $\bfC_2/T_0$ is fixed by the hyperelliptic involution, the same is true of the saddle connection in $\Col_{\bfC_2/T_0}(\bfC_2/T_0)$. The endpoints of this saddle connection belong to $a^{(2)}$ (by Sublemma \ref{SL:ConstantsComputation2}). It follows that every point in $a^{(2)}$ is an endpoint of $\Col_{\bfC_2/T_0}(\bfC_2/T_0)$. 

Since $\Col_{\bfC_2/T_0}(Q')$ belongs to the complement of $a^{(2)}$ (by Lemma \ref{L:AdjacencyOfC2AndQ}) and contains at most two points (by Lemma \ref{L:TheNatureOfQ'}) it follows that this set coincides with $R/T_2$, which also belongs to the complement of $a^{(2)}$ and contains exactly two points. Therefore, $\ColTwo(Q) = R$ and so the points in $\ColTwo(Q)$ are singularities of the flat metric. 

Since the boundary of $\bfC_2$ does not contain points in $Q$ (by Lemma \ref{L:AdjacencyOfC2AndQ}) the number of points in $Q$ and the cone angles around these points are the same on $(X, \omega)$ as on $\ColTwoX$. This shows that $Q$ consists of two singularities of the flat metric. Since $Q$ contains no marked points and any marked point on $(X, \omega)$ belongs to $Q$, we are done.
\end{proof}

Suppose finally that $b = -1$. Suppose first that $Q'$ consists of two points. By Lemma \ref{L:Con2}, either both points in $Q$ are singularities or the points in $Q$ are not ramification points, and hence are a slope $-1$ irreducible pair of marked points (by Lemma \ref{L:TheNatureOfQ'}). In the latter case, $\cQ_0 = \cQ(a^{(3)}, 0, -1^4)$ and the points in $Q$ are precisely the preimages of the free marked points in $(X, \omega)/J_0$; in particular there are no other marked points on $(X, \omega)$. The case that the points in $Q$ are both singularities cannot occur since, by virtue of $\For\left( (X, \omega)/J_0 \right)$ belonging to $\cQ\left( a^{(3)}, -1^4 \right)$, all the singularities of the flat metric belong to the boundary of $\bfC_2$ whereas the points in $Q$ do not (by Lemma \ref{L:AdjacencyOfC2AndQ}). 

Suppose finally that $b = -1$ and that $Q'$ consists of one point. By Lemma \ref{L:Con1}, the points in $Q$ are periodic marked points. It suffices to show that these points are fixed points of $J_0$. Since the points in $Q$ do not belong to the boundary of $\bfC_2$, it suffices to show that the points in $\ColTwo(Q)$ are fixed points of $J_2$, which follows by definition of quadratic double. Therefore, $(X, \omega)$ belongs to a quadratic double of $\cQ\left( a^{(3)}, -1^4 \right)$ and the marked points consist of one or two preimages of poles. 
\end{proof}

\bibliography{mybib}{}

\newcommand{\etalchar}[1]{$^{#1}$}
\providecommand{\bysame}{\leavevmode\hbox to3em{\hrulefill}\thinspace}
\providecommand{\MR}{\relax\ifhmode\unskip\space\fi MR }
\providecommand{\MRhref}[2]{%
  \href{http://www.ams.org/mathscinet-getitem?mr=#1}{#2}
}
\providecommand{\href}[2]{#2}
\begin{thebibliography}{BCG{\etalchar{+}}18}

\bibitem[AEM17]{AEM}
Artur Avila, Alex Eskin, and Martin M\"{o}ller, \emph{Symplectic and isometric
  {${\rm SL}(2,\Bbb R)$}-invariant subbundles of the {H}odge bundle}, J. Reine
  Angew. Math. \textbf{732} (2017), 1--20.

\bibitem[AN]{aulicino2016rank}
David Aulicino and Duc-Manh Nguyen, \emph{Rank two affine submanifolds in genus
  $3$}, to appear in J. Differential Geom.

\bibitem[AN16]{AN}
\bysame, \emph{Rank two affine submanifolds in {$\mathcal{H}(2,2)$} and
  {$\mathcal{H}(3,1)$}}, Geom. Topol. \textbf{20} (2016), no.~5, 2837--2904.

\bibitem[ANW16]{ANW}
David Aulicino, Duc-Manh Nguyen, and Alex Wright, \emph{Classification of
  higher rank orbit closures in {$\mathcal{H}^{\rm{odd}}(4)$}}, J. Eur. Math.
  Soc. (JEMS) \textbf{18} (2016), no.~8, 1855--1872.

\bibitem[Api18]{Apisa2}
Paul Apisa, \emph{{${\rm GL}_2\Bbb R$} orbit closures in hyperelliptic
  components of strata}, Duke Math. J. \textbf{167} (2018), no.~4, 679--742.

\bibitem[Api19]{Apisa:Rank1}
\bysame, \emph{Rank one orbit closures in {$\mathcal{H}^{hyp}(g-1, g-1)$}},
  Geom. Funct. Anal. \textbf{29} (2019), no.~6, 1617--1637.

\bibitem[Api20]{Apisa}
\bysame, \emph{{${\rm GL}_2\Bbb{R}$}-invariant measures in marked strata:
  generic marked points, {E}arle-{K}ra for strata and illumination}, Geom.
  Topol. \textbf{24} (2020), no.~1, 373--408.

\bibitem[AWa]{ApisaWrightGemini}
Paul Apisa and Alex Wright, \emph{Generalizations of the
  {E}ierlegende-{W}ollmilchsau}, preprint, arXiv:2011.09452 (2020).

\bibitem[AWb]{ApisaWrightHighRank}
\bysame, \emph{High rank invariant subvarieties}, preprint, arXiv:2102.06567
  (2021).

\bibitem[AWc]{ApisaWright}
\bysame, \emph{Marked points on translation surfaces}, preprint,
  arXiv:1708.03411 (2017), to appear in Geom. Top.

\bibitem[BCG{\etalchar{+}}]{lms}
Matt Bainbridge, Dawei Chen, Quentin Gendron, Samuel Grushevsky, and Martin
  M\"oller, \emph{The moduli space of multi-scale differentials}, preprint,
  arXiv:1910.13492 (2019).

\bibitem[BCG{\etalchar{+}}18]{Many}
\bysame, \emph{Compactification of strata of {A}belian differentials}, Duke
  Math. J. \textbf{167} (2018), no.~12, 2347--2416.

\bibitem[Ben]{Benirschke}
Frederik Benirschke, \emph{The boundary of linear subvarieties}, preprint,
  arXiv:2007.02502v1 (2020).

\bibitem[BHM16]{BHM}
Matt Bainbridge, Philipp Habegger, and Martin M\"{o}ller, \emph{Teichm\"{u}ller
  curves in genus three and just likely intersections in {${\bf G}^n_m\times
  {\bf G}^n_a$}}, Publ. Math. Inst. Hautes \'{E}tudes Sci. \textbf{124} (2016),
  1--98.

\bibitem[BSW]{BSW}
Matt Bainbridge, John Smillie, and Barak Weiss, \emph{Horocycle dynamics: new
  invariants and eigenform loci in the stratum $\mathcal{H}(1,1)$}, preprint,
  arXiv:1603.00808v2 (2020).

\bibitem[CDF]{CDF}
Gabriel Calsamiglia, Bertrand Deroin, and Stefano Francaviglia, \emph{A
  transfer principle: from periods to isoperiodic foliations}, preprint,
  arXiv:1511.07635v2 (2016).

\bibitem[CM14]{CM}
Dawei Chen and Martin M\"{o}ller, \emph{Quadratic differentials in low genus:
  exceptional and non-varying strata}, Ann. Sci. \'{E}c. Norm. Sup\'{e}r. (4)
  \textbf{47} (2014), no.~2, 309--369.

\bibitem[CSW]{CSW}
Jon Chaika, John Smillie, and Barak Weiss, \emph{Tremors and horocycle dynamics
  on the moduli space of translation surfaces}, preprint, arXiv:2004.04027v1
  (2020).

\bibitem[CW19]{ChenWright}
Dawei Chen and Alex Wright, \emph{The {WYSIWYG} compactification}, 2019,
  arXiv:1908.07436v2.

\bibitem[EFW18]{EFW}
Alex Eskin, Simion Filip, and Alex Wright, \emph{The algebraic hull of the
  {K}ontsevich-{Z}orich cocycle}, Ann. of Math. (2) \textbf{188} (2018), no.~1,
  281--313.

\bibitem[EM18]{EM}
Alex Eskin and Maryam Mirzakhani, \emph{Invariant and stationary measures for
  the {${\rm SL}(2,\Bbb R)$} action on moduli space}, Publ. Math. Inst. Hautes
  \'{E}tudes Sci. \textbf{127} (2018), 95--324.

\bibitem[EMM15]{EMM}
Alex Eskin, Maryam Mirzakhani, and Amir Mohammadi, \emph{Isolation,
  equidistribution, and orbit closures for the {${\rm SL}(2,\Bbb R)$} action on
  moduli space}, Ann. of Math. (2) \textbf{182} (2015), no.~2, 673--721.

\bibitem[EMMW]{EMMW}
Alex Eskin, Curtis~T. McMullen, Ronen~E. Mukamel, and Alex Wright,
  \emph{Billiards, quadrilaterals and moduli spaces}, preprint, to appear in J.
  Amer. Math. Soc.

\bibitem[Fil16]{Fi2}
Simion Filip, \emph{Semisimplicity and rigidity of the {K}ontsevich-{Z}orich
  cocycle}, Invent. Math. \textbf{205} (2016), no.~3, 617--670.

\bibitem[FM14]{ForniMatheus:Intro}
Giovanni Forni and Carlos Matheus, \emph{Introduction to {T}eichm\"{u}ller
  theory and its applications to dynamics of interval exchange transformations,
  flows on surfaces and billiards}, J. Mod. Dyn. \textbf{8} (2014), no.~3-4,
  271--436.

\bibitem[Ham18]{HamErg}
Ursula Hamenst\"{a}dt, \emph{Ergodicity of the absolute period foliation},
  Israel J. Math. \textbf{225} (2018), no.~2, 661--680.

\bibitem[HS06]{HS-IntroVeech}
Pascal Hubert and Thomas~A. Schmidt, \emph{An introduction to {V}eech
  surfaces}, Handbook of dynamical systems. {V}ol. 1{B}, Elsevier B. V.,
  Amsterdam, 2006, pp.~501--526. \MR{2186246}

\bibitem[HW18]{HooperWeiss}
W.~Patrick Hooper and Barak Weiss, \emph{Rel leaves of the {A}rnoux-{Y}occoz
  surfaces}, Selecta Math. (N.S.) \textbf{24} (2018), no.~2, 875--934, With an
  appendix by Lior Bary-Soroker, Mark Shusterman, and Umberto Zannier.

\bibitem[KZ03]{KZ}
Maxim Kontsevich and Anton Zorich, \emph{Connected components of the moduli
  spaces of {A}belian differentials with prescribed singularities}, Invent.
  Math. \textbf{153} (2003), no.~3, 631--678.

\bibitem[Lan04]{LanneauHyp}
Erwan Lanneau, \emph{Hyperelliptic components of the moduli spaces of quadratic
  differentials with prescribed singularities}, Comment. Math. Helv.
  \textbf{79} (2004), no.~3, 471--501.

\bibitem[Lan08]{Lconn}
\bysame, \emph{Connected components of the strata of the moduli spaces of
  quadratic differentials}, Ann. Sci. \'Ec. Norm. Sup\'er. (4) \textbf{41}
  (2008), no.~1, 1--56.

\bibitem[LM]{LanneauMoller}
Erwan Lanneau and Martin Moeller, \emph{Non-existence and finiteness results
  for {T}eichmueller curves in {P}rym loci}, preprint, arXiv:1704.03210v2
  (2017).

\bibitem[LN]{Wg4}
Erwan Lanneau and Duc-Manh Nguyen, \emph{Weierstrass {P}rym eigenforms in genus
  four}, preprint, arXiv:1802.04879v1 (2018).

\bibitem[LN18]{LNcomponents}
\bysame, \emph{Connected components of {P}rym eigenform loci in genus three},
  Math. Ann. \textbf{371} (2018), no.~1-2, 753--793.

\bibitem[LNW17]{LNW}
Erwan Lanneau, Duc-Manh Nguyen, and Alex Wright, \emph{Finiteness of
  {T}eichm\"{u}ller curves in non-arithmetic rank 1 orbit closures}, Amer. J.
  Math. \textbf{139} (2017), no.~6, 1449--1463.

\bibitem[McM07]{Mc5}
Curtis~T. McMullen, \emph{Dynamics of {${\rm SL}_2(\Bbb R)$} over moduli space
  in genus two}, Ann. of Math. (2) \textbf{165} (2007), no.~2, 397--456.

\bibitem[McM14]{McM:iso}
\bysame, \emph{Moduli spaces of isoperiodic forms on {R}iemann surfaces}, Duke
  Math. J. \textbf{163} (2014), no.~12, 2271--2323.

\bibitem[MMW17]{MMW}
Curtis~T. McMullen, Ronen~E. Mukamel, and Alex Wright, \emph{Cubic curves and
  totally geodesic subvarieties of moduli space}, Ann. of Math. (2)
  \textbf{185} (2017), no.~3, 957--990.

\bibitem[M{\"o}l06]{M2}
Martin M{\"o}ller, \emph{Periodic points on {V}eech surfaces and the
  {M}ordell-{W}eil group over a {T}eichm\"uller curve}, Invent. Math.
  \textbf{165} (2006), no.~3, 633--649.

\bibitem[MT02]{MT}
Howard Masur and Serge Tabachnikov, \emph{Rational billiards and flat
  structures}, Handbook of dynamical systems, {V}ol.\ 1{A}, North-Holland,
  Amsterdam, 2002, pp.~1015--1089.

\bibitem[MW17]{MirWri}
Maryam Mirzakhani and Alex Wright, \emph{The boundary of an affine invariant
  submanifold}, Invent. Math. \textbf{209} (2017), no.~3, 927--984.

\bibitem[MW18]{MirWri2}
\bysame, \emph{Full-rank affine invariant submanifolds}, Duke Math. J.
  \textbf{167} (2018), no.~1, 1--40.

\bibitem[MZ08]{MZ}
Howard Masur and Anton Zorich, \emph{Multiple saddle connections on flat
  surfaces and the principal boundary of the moduli spaces of quadratic
  differentials}, Geom. Funct. Anal. \textbf{18} (2008), no.~3, 919--987.

\bibitem[NW14]{NW}
Duc-Manh Nguyen and Alex Wright, \emph{Non-{V}eech surfaces in
  {$\mathcal{H}^{\rm hyp}(4)$} are generic}, Geom. Funct. Anal. \textbf{24}
  (2014), no.~4, 1316--1335.

\bibitem[SW04]{SW2}
John Smillie and Barak Weiss, \emph{Minimal sets for flows on moduli space},
  Israel J. Math. \textbf{142} (2004), 249--260.

\bibitem[Wri14]{Wfield}
Alex Wright, \emph{The field of definition of affine invariant submanifolds of
  the moduli space of abelian differentials}, Geom. Topol. \textbf{18} (2014),
  no.~3, 1323--1341.

\bibitem[Wri15a]{Wcyl}
\bysame, \emph{Cylinder deformations in orbit closures of translation
  surfaces}, Geom. Topol. \textbf{19} (2015), no.~1, 413--438.

\bibitem[Wri15b]{Wsurvey}
\bysame, \emph{Translation surfaces and their orbit closures: {A}n introduction
  for a broad audience}, EMS Surv. Math. Sci. \textbf{2} (2015), no.~1,
  63--108.

\bibitem[Ygoa]{Ygouf2}
Florent Ygouf, \emph{A criterion for density of the isoperiodic leaves in rank
  1 affine invariant orbifolds}, preprint, arXiv:2002.01186v1 (2020).

\bibitem[Ygob]{Ygouf}
\bysame, \emph{Non arithmetic affine invariant orbifolds in
  $\mathcal{H}^{odd}(2, 2)$ and $\mathcal{H}(3, 1)$}, preprint,
  arXiv:2002.02004v1 (2020).

\bibitem[Zor06]{Z}
Anton Zorich, \emph{Flat surfaces}, Frontiers in number theory, physics, and
  geometry. {I}, Springer, Berlin, 2006, pp.~437--583.

\end{thebibliography}
\bibliographystyle{amsalpha}
\end{document}